\documentclass[11pt]{article}
\usepackage[utf8x]{inputenc}
\usepackage{microtype}
\usepackage[T1]{fontenc}
\usepackage{ae,aecompl}
\usepackage{times}
\usepackage{mathrsfs}  
\usepackage{leftidx}
\usepackage{verbatim,amsmath,amsthm,amsfonts,amssymb,latexsym,graphicx,mathtools,extpfeil,color} 
\usepackage{epstopdf,pinlabel}
\usepackage[all]{xy}
\usepackage{graphicx}
\usepackage{caption}
\usepackage{subcaption}
\usepackage{tikz,xcolor}
\tikzset{pics/.cd,
handle/.style={code={
\draw (1,2)  coordinate (-left) 
to (0,2)
to (-.5,2)
to [out=180, in=-20] (-2,3) 
to [out=160, in=20] (-4,3) 
to [out=200,in=90] (-6,0) 
to [out=270,in=160] (-4,-3) 
to [out=340,in=200] (-2,-3) 
to [out=20,in=180] (-.5,-2)
to (0,-2) 
to (1,-2) coordinate (-right);
}}}
\tikzset{pics/.cd,
hole/.style={code={
\pgfgettransformentries{\tmpa}{\tmpb}{\tmp}{\tmp}{\tmp}{\tmp}
\pgfmathsetmacro{\myrot}{-atan2(\tmpb,\tmpa)}
\draw[rotate around={\myrot:(0,-2.5)}] (-1.2,-2.4) to[bend right]  (1.2,-2.4);
\draw[fill=white,rotate around={\myrot:(0,-2.5)}] (-1,-2.5) to[bend right] (1,-2.5) 
to[bend right] (-1,-2.5);
}}}
\tikzset{pics/.cd,
thole/.style={code={
\pgfgettransformentries{\tmpa}{\tmpb}{\tmp}{\tmp}{\tmp}{\tmp}
\pgfmathsetmacro{\myrot}{-atan2(\tmpb,\tmpa)}
\draw[rotate around={\myrot:(-2.5,0)}] (-2.4,1.2) to[bend right]  (-2.4,-1.2);
\draw[fill=white,rotate around={\myrot:(-2.5,0)}] (-2.5,1) to[bend right] (-2.5,-1) 
to[bend right] (-2.5,1);
}}}
\tikzset{pics/.cd,
thandle/.style={code={
\draw (2,-1)  coordinate (-left) 
to (2,0)
to (2,0.5)
to [out=90, in=-110] (3,2) 
to [out=70, in=-70] (3,4) 
to [out=110,in=0] (0,6) 
to [out=180,in=70] (-3,4) 
to [out=250,in=110] (-3,2) 
to [out=-70,in=90] (-2,.5)
to (-2,0) 
to (-2,-1) coordinate (-right);
}}}
\DeclareMathAlphabet{\mathpzc}{OT1}{pzc}{m}{it}
\newcommand\rSI{{\rm SI}}
\newcommand\Hess{{\rm Hess}}

\newtheorem{theorem-intro}{Theorem}
\newtheorem{cor-intro}{Corollary}
\newtheorem{hypo-intro}{Hypothesis}

\usepackage[colorlinks,pagebackref,hypertexnames=false]{hyperref} \usepackage[alphabetic,backrefs,msc-links]{amsrefs}

\usepackage{aliascnt}
\numberwithin{equation}{section}

\newcommand{\rd}{{\rm d}}
\newcommand{\re}{{\rm e}}

\newcommand{\rv}{{\rm v}}

\newcommand{\rE}{{\rm E}}

\newcommand{\rI}{{\rm I}}

\newcommand{\rL}{{\rm L}}
\newcommand{\rM}{{\rm M}}

\newcommand{\bi}{{\bf i}}

\newcommand{\br}{{\bf r}}
\newcommand{\bs}{{\bf s}}

\newcommand{\bA}{{\bf A}}
\newcommand{\bB}{{\bf B}}
\newcommand{\bC}{{\bf C}}

\newcommand{\bH}{{\bf H}}

\newcommand{\bK}{{\bf K}}
\newcommand{\bL}{{\bf L}}
\newcommand{\bM}{{\bf M}}
\newcommand{\bN}{{\bf N}}

\newcommand{\bQ}{{\bf Q}}
\newcommand{\bR}{{\bf R}}
\newcommand{\bS}{{\bf S}}

\newcommand{\bW}{{\bf W}}

\newcommand{\bZ}{{\bf Z}}

\newcommand{\fb}{{\mathfrak b}}
\newcommand{\fc}{{\mathfrak c}}

\newcommand{\ff}{{\mathfrak f}}
\newcommand{\fg}{{\mathfrak g}}

\newcommand{\fq}{{\mathfrak q}}

\newcommand{\fs}{{\mathfrak s}}

\newcommand{\fB}{{\mathfrak B}}
\newcommand{\fC}{{\mathfrak C}}
\newcommand{\fD}{{\mathfrak D}}
\newcommand{\fE}{{\mathfrak E}}

\newcommand{\fM}{{\mathfrak M}}
\newcommand{\fN}{{\mathfrak N}}

\newcommand{\fQ}{{\mathfrak Q}}

\newcommand{\fU}{{\mathfrak U}}

\newcommand{\N}{\bN}
\newcommand{\Z}{\bZ}
\newcommand{\Q}{\bQ}
\newcommand{\R}{\bR}
\newcommand{\C}{\bC}

\newcommand{\su}{\mathfrak{su}}

\renewcommand{\O}{{\rm O}}
\newcommand{\SO}{{\rm SO}}

\newcommand{\SU}{{\rm SU}}

\newcommand{\U}{{\rm U}}

\DeclareMathOperator{\coker}{coker}

\DeclareMathOperator{\ind}{index}

\DeclareMathOperator{\tr}{tr}

\renewcommand{\det}{\operatorname{det}}

\newcommand{\id}{{\rm id}}

\newcommand{\dvol}{{\rm dvol}}

\renewcommand{\epsilon}{\varepsilon}

\def\({\mathopen{}\left(}
\def\){\right)\mathclose{}}
\def\<{\mathopen{}\left<}
\def\>{\right>\mathclose{}}

\usepackage{multicol, color}

\definecolor{gold}{rgb}{0.85,.66,0}
\definecolor{cherry}{rgb}{0.9,.1,.2}
\definecolor{burgundy}{rgb}{0.8,.2,.2}
\definecolor{orangered}{rgb}{0.85,.3,0}
\definecolor{orange}{rgb}{0.85,.4,0}
\definecolor{olive}{rgb}{.45,.4,0}
\definecolor{lime}{rgb}{.6,.9,0}
\definecolor{green}{rgb}{.2,.7,0}
\definecolor{grey}{rgb}{.4,.4,.2}
\definecolor{brown}{rgb}{.4,.3,.1}

\def\makeautorefname#1#2{\AtBeginDocument{\expandafter\def\csname#1autorefname\endcsname{#2}}}

\newcommand{\mynewtheorem}[2]{
  \newaliascnt{#1}{equation}          
  \newtheorem{#1}[#1]{#2}
  \aliascntresetthe{#1}
  \makeautorefname{#1}{#2}
}
\mynewtheorem{theorem}{Theorem}
\mynewtheorem{prop}{Proposition}
\mynewtheorem{cor}{Corollary}
\mynewtheorem{construction}{Construction}
\mynewtheorem{lemma}{Lemma}
\mynewtheorem{conjecture}{Conjecture}
\mynewtheorem{hypo}{Hypothesis}

\numberwithin{substep}{step}
\makeautorefname{step}{Step}
\makeautorefname{substep}{Step}

\numberwithin{subcase}{case}
\makeautorefname{case}{Case}
\makeautorefname{subcase}{case}

\theoremstyle{remark}
\mynewtheorem{remark}{Remark}

\theoremstyle{definition}
\mynewtheorem{definition}{Definition}
\mynewtheorem{example}{Example}
\mynewtheorem{exercise}{Exercise}
\mynewtheorem{convention}{Convention}
\newtheorem*{convention*}{Convention}
\newtheorem*{conventions*}{Conventions}
\mynewtheorem{question}{Question}
        
\makeautorefname{chapter}{Chapter}
\makeautorefname{section}{Section}
\makeautorefname{subsection}{Section}
\makeautorefname{subsubsection}{Section}

\title{Lagrangians, SO(3)-instantons and the Atiyah-Floer Conjecture}
\author{\bf \sc \large Aliakbar Daemi\thanks{The work of AD was supported by NSF Grant DMS-1812033 and NSF FRG Grant DMS-1952762.} \hspace{1cm} Kenji Fukaya\thanks{The work of KF was supported by the Simons Foundation through its Homological Mirror Symmetry Collaboration grant.} \hspace{1cm} Maksim Lipyanskiy}
\date{}

\begin{document}
\maketitle

\begin{abstract}
	A version of the Atiyah-Floer conjecture, adapted to admissible $\SO(3)$-bundles, is established.
\end{abstract}
{
  \hypersetup{linkcolor=black}
  \tableofcontents
}
\newpage

\section{Introduction}

Gauge theoretic methods in low dimensional topology and holomorphic curve methods in symplectic geometry are responsible for many revolutionary advancements in respective fields. These two approaches have many formal similarities, despite the fact that low dimensional topology and symplectic geometry have different origins. The Atiyah--Floer conjecture is a manifestation of these similarities. According to this conjecture, instanton Floer homology, a 3-manifold invariant constructed in the context of Yang--Mills gauge theory can be recovered using holomorphic curve methods. The main goal of the present paper is to prove a version of the Atiyah--Floer conjecture.

\subsection*{Main Results}

Suppose $Y_\#$ is an orientable connected closed 3-manifold and $E_\#$ is an $\SO(3)$-bundle on $Y_\#$. The pair $(Y_\#,E_\#)$ is {\it admissible} if the Stiefel-Whitney class $w_2(E_\#)$ lifts to a non-torsion element of $H^2(Y_\#;\Z)$. This condition is equivalent to the existence of an element $R\in H_2(Y_\#;\Z)$ such that the pairing of $w_2(E_\#)$ and $R$ is non-trivial. Any such $R$ is called a {\it nice} homology class for the pair $(Y_\#,E_\#)$. Associated to any such admissible pair $(Y_\#,E_\#)$, we have the {\it instanton Floer homology} of $(Y_\#,E_\#)$, which is a relatively $\Z/8\Z$-graded group \cite{Floer:inst2,BD:surgery}. A nice homology class $R$ induces a degree $4$ involution on the instanton Floer homology of $(Y_\#,E_\#)$, and the invariant subspace with respect to this involution determines a relatively $\Z/4\Z$-graded group $\rI_*(Y_\#,E_\#)$. 

Any admissible pair with a choice of a nice homology class admits an {\it admissible splitting}. An admissible splitting 
\begin{equation}\label{ad-split}
	(Y,E)\cup_{(\Sigma,F)}(Y',E')
\end{equation}	
consists of connected 3-manifolds $Y$, $Y'$ with boundary $\Sigma$ and $\SO(3)$-bundles $E$, $E'$ on $Y$, $Y'$ whose restrictions to $\Sigma$ are identified with an $\SO(3)$-bundle $F$. The restriction of $w_2(F)$ to each connected component of $\Sigma$ is required to be non-trivial. This assumption implies that $\Sigma$ has an even number of connected components and as an additional assumption we require that $\Sigma$ has exactly two connected components $\Sigma_0$ and $\Sigma_1$. 
We also assume that $Y$ and $Y'$ are oriented such that the induced orientations on their boundaries are the orientation on $\Sigma$. Thus, after reversing the orientation of $Y'$, we may glue $Y$ and $Y'$ to form an oriented closed 3-manifold $Y_\#$. The bundles $E$ and $E'$ can be also glued to from an $\SO(3)$-bundle $E_\#$ on $Y_\#$. Connected components of $\Sigma$ determine homologous homology classes in $Y_\#$ which are denoted by $R$. Since $w_2(E_\#)$ has a non-trivial pairing with $R$, the pair $(Y_\#,E_\#)$ is admissible and $R$ is a nice homology class for this pair. We say \eqref{ad-split} is an admissible splitting of $(Y_\#,E_\#)$ compatible with $R$.

Suppose $(Y_\#,E_\#)$ is a pair with an admissible splitting as in \eqref{ad-split}, and $\mathcal M(\Sigma,F)$ denotes the moduli space of flat connections on $F$ modulo determinant one automorphisms of $F$ \cite{AB:YM}. The manifold $\mathcal M(\Sigma,F)$ admits a canonical symplectic structure $\Omega$. (See Subsection \ref{flat-surface} for our conventions.) Flat connections on $E$ after a small perturbation gives rise to an immersed Lagrangian submanifold $L(Y,E)$, which is called the {\it 3-manifold Lagrangian} associated to $(Y,E)$ \cite{He:3-man-lag}. (See Subsection \ref{3-man-lag-sec} for more details.) Similarly, we can associate a 3-manifold Lagrangian $L(Y',E')$ to $(Y',E')$. We say that the 3-manifold Lagrangians $L(Y,E)$ and $L(Y',E')$ are embedded if there are arbitrarily small perturbations such that the associated Lagrangians are embedded. A more precise version of this assumption is stated as Hypothesis \ref{main-hypo}. For instance, if the fundamental group of one of the connected components of $\Sigma$ surjects into the fundamental group of $Y$ (resp. $Y'$), then $L(Y,E)$ (resp. $L(Y',E'))$ is embedded. (See Proposition \ref{emb-fund}.)
\begin{theorem-intro}
	Suppose $L(Y,E)$ and $L(Y',E')$ are embedded submanifolds of $\mathcal M(\Sigma,F)$. 
	Then the pair of Lagrangians $(L(Y,E), L(Y',E'))$ is monotone with minimal Maslov number $4$, 
	and configuration space of strips associated to these Lagrangians can be coherently oriented. 
\end{theorem-intro}

Let $(L,L')$ be a pair of embedded Lagrangians with minimal Maslov number $N$. Building on Floer's work \cite{Fl:LFH}, Oh defines a $\Z/N\Z$-graded Lagrangian Floer homology group ${\rm HF}(L,L')$ in \cite{Oh:LFH}, which is a vector space over $\Z/2$. {\it Coherent orientations} for configuration spaces of strips allow us to work with integer coefficients (see Definition \ref{coh-ori-def}). Let $L(Y,E)$ and $L(Y',E')$ be embedded. Then Oh's Lagrangian Floer homology of these two Lagrangians is a $\Z/4\Z$-graded abelian group which is called {\it symplectic instanton Floer homology} of $(Y_\#,E_\#)$ and is denoted by $\rSI_*(Y_\#,E_\#)$. The following theorem is our main result.
\begin{theorem-intro}\label{main-thm}
	Suppose an admissible splitting for a pair $(Y_\#,E_\#)$ is given such that 
	the associated 3-manifold Lagrangians are embedded.
	Then there is an isomorphism of relatively $\Z/4\Z$-graded abelian groups $\bN: \rI_*(Y_\#,E_\#)\to \rSI_*(Y_\#,E_\#)$.
\end{theorem-intro}
\noindent 
The proof of Theorem \ref{main-thm} modulo some analytical results is given in Section \ref{main-thm-sec}. Sections \ref{lin-analysis}, \ref{reg-comp-exp-dec-sec} and \ref{perturbations} of the paper are devoted to verifying the analytical results which are used in Section \ref{main-thm-sec}. 

Instanton Floer homology group $\rI_*(Y_\#,E_\#)$ is an invariant of the topological type of $(Y_\#,E_\#)$. However, symplectic instanton Floer homology $\rSI_*(Y_\#,E_\#)$, a priori, depends on the choice of an admissible splitting. As a consequence of Theorem \ref{main-thm}, we have the following corollary.
\begin{cor-intro}
	Symplectic instanton Floer homology group $\rSI_*(Y_\#,E_\#)$ depends only on the topological type of $(Y_\#,E_\#)$. In particular, it is independent of the admissible splitting of $(Y_\#,E_\#)$.
\end{cor-intro}

\begin{cor-intro}\label{DS}
	Let  $\phi:\Sigma_g\to \Sigma_g$ be a diffeomorphism and $Y_\phi$ be the mapping cylinder 
	$[0,1]\times Y/\{(x,1)\sim (\phi(x),0)\}$. Let also $E_\phi$ be the $\SO(3)$-bundle on $Y_\phi$ which is induced by 
	the non-trivial $\SO(3)$-bundle on $\Sigma_g$. The map $\phi$ induces a symplectomorphism
	$\phi_*:\mathcal M(\Sigma_{g},F_{g}) \to \mathcal M(\Sigma_{g},F_{g})$. Then instanton Floer homology group 
	$\rI_*(Y_\phi,E_\phi)$ is isomorphic to the Lagrangian Floer homology of 
	the diagonal and the graph $\Gamma_{\phi_*}$ of  $\phi_*$, which are Lagrangians in 
	$(\mathcal M(\Sigma_{g},F_{g})\times \mathcal M(\Sigma_{g},F_{g}),-\Omega\times \Omega)$.
\end{cor-intro}

This corollary of Theorem \ref{main-thm} is essentially the same as Dostoglou and Salamon's celebrated result in \cite{DS:At-Fl}. It is shown in \cite{DS:At-Fl} that $\rI_*(Y_\phi,E_\phi)$ is isomorphic to the fixed point Floer homology of $\phi_*$. It is a folklore theorem that fixed Floer homology of a symplectomorphism is isomorphic to the Lagrangian Floer homology of the diagonal and the graph of the symplectomorphism. 	

\begin{proof}
	The pair $(Y_\phi,E_\phi)$ has an obvious admissible splitting as the union of $(Y,E)$ and $(Y',E')$
	where $Y$, $Y'$ are diffeomorphic to $[0,1]\times \Sigma_g$ and $E$, $E'$ are pull-backs of the non-trivial 
	$\SO(3)$-bundle $F_g$ on $\Sigma_g$. The Lagrangian submanifolds associated to $Y$, $Y'$
	are the diagonal $\Delta$ and $\Gamma_{\phi_*}$. 
\end{proof}

In \cite{KM:YAFT}, Kronheimer and Mrowka use instanton Floer homology to define an invariant of 3-manifolds which is called {\it framed Floer homology}. Let $T^3=S^1\times T^2$ be the 3-dimensional torus and $E_1$ be the $\SO(3)$-bundle on $T^3$ which is the pullback of the non-trivial bundle $F_1$ on $T^2$. The trivial $\SO(3)$-bundle on a 3-manifold $M$ and $E_1$ induces a bundle $E_\#$ on $Y_\#=M\#T^3$. The pair $(Y_\#,E_\#)$ is admissible and the factor $T^2$ of $T^3$ determines a nice homology class for this pair. The framed Floer homology $\rI_*^\sharp(M)$ of $M$ is defined to be the associated instanton Floer homology. Given a Heegaard splitting $H\cup_{\Sigma_g} H'$ of $Y$, we can obtain an admissible splitting of $(Y_\#,E_\#)$. Let $Y$ (resp. $Y'$) be the boundary sum of $H$ (resp. $H'$) and $[0,1]\times T^2$. Then $Y$ and $Y'$ are 3-manifolds whose boundary components are $\Sigma=\Sigma_{g+1}\sqcup T^2$. The non-trivial $\SO(3)$-bundle on $[0,1]\times T^2$ induces $\SO(3)$-bundles $E$ and $E'$ on $Y$ and $Y'$. The restriction $F$ of $E$ (or equivalently $E'$) to $\Sigma$ is given by the non-trivial $\SO(3)$-bundle $F_{g+1}$ on $\Sigma_{g+1}$ and the bundle $F_1$ on $T^2$. In particular, $\mathcal M(\Sigma,F)=\mathcal M(\Sigma_{g+1},F_{g+1})$. The subspaces of elements of $\mathcal M(\Sigma_{g+1},F_{g+1})$ which extend as flat connections to $E$ and $E'$ determine embedded Lagrangian submanifolds $L$ and $L'$ of $\mathcal M(\Sigma_{g+1},F_{g+1})$. In \cite[Definition 4.4.1]{WW:FFT-coprime}, Wehrheim and Woodward, define a 3-manifold invariant as the Lagrangian Floer homology of $L$ and $L'$. We call this invariant {\it symplectic framed Floer homology} of $M$, and denote it by $\rSI_*^\sharp(M)$. 
\begin{theorem-intro} \label{framed-AF}
	The 3-manifold invariants $\rI_*^\sharp(M)$ and $\rSI_*^\sharp(M)$ together with their Chern-Simons filtrations, are isomorphic to each other.
\end{theorem-intro}

The Chern-Simons filtrations on $\rI_*^\sharp(M)$ and $\rSI_*^\sharp(M)$ are defined in Section \ref{action}, where the proof of Theorem \ref{framed-AF} is given. Forgetting this additional structure, Theorem \ref{framed-AF} is a special case of Theorem \ref{main-thm}. 

For a pair of monotone Lagrangians $(L,L')$ in a symplectic manifold $(M,\omega)$, ${\rm HF}(L,L')$ is a module over the quantum cohomology ring $QH^*(M)$. See, for example, \cite{Fl:cup-Lag,Sei:van-mut,Al:PSS-mor,FOOO:HF1,BC:ruling} for this structure on Lagrangian Floer homology and related constructions in symplectic geometry. 
For our purposes in the present paper, this structure determines an action of $QH^*(\mathcal M(\Sigma,F))$ on $\rSI_*(Y_\#,E_\#)$. To simplify the discussion we work with the coefficients in $\Q$ for the rest of the introduction. An explicit set of generators for the ring $QH^*(\mathcal M(\Sigma,F))$ (or equivalently $H^*(\mathcal M(\Sigma,F))$) is given in \cite{AB:YM}. There is a {\it universal $\SO(3)$-bundle} $\mathbb F$ over the product $\mathcal M(\Sigma,F)\times \Sigma$, and the slant products 
\begin{equation}\label{mu-class-symp}
	\hspace{2cm}\frac{1}{4}p_1(\mathbb F)\backslash \sigma,\hspace{1cm}\sigma \in H_*(\Sigma)
\end{equation}	
defines an element of $H^*(\mathcal M(\Sigma,F))$. These cohomology classes as $\sigma$ varies over a generating set for $H_*(\Sigma)$ determine multiplicative generating for the cohomology ring of $\mathcal M(\Sigma,F)$. We write $m^S_\sigma:\rSI_*(Y_\#,E_\#) \to \rSI_*(Y_\#,E_\#)$ for the induced action of \eqref{mu-class-symp}, as an element of $QH^*(\mathcal M(\Sigma,F))$, on $\rSI_*(Y_\#,E_\#)$.

On the gauge theoretical side, there is an action of $H_*(Y_\#)$ on $\rI_*(Y_\#,E_\#)$ for any admissible pair $(Y_\#,E_\#)$. For $\sigma\in H_*(Y_\#)$, we denote the corresponding action by $m^G_\sigma:\rI_*(Y_\#,E_\#) \to \rI_*(Y_\#,E_\#)$. This action plays a crucial role in certain topological applications of instanton Floer homology of admissible pairs (see \cite{km-sutures} for some instances of such topological applications), and it is related to $\mu$-classes and polynomial invariants in Donaldson theory of smooth closed 4-manifolds. 

\begin{theorem-intro}\label{module-over-quantum}
	Suppose $\sigma\in H_1(\Sigma)\oplus H_2(\Sigma)$, and the endomorphism $m^G_\sigma$ of $\rI_*(Y_\#,E_\#)$ is defined using the inclusion of $\Sigma$ in $Y_\#$. Then the isomorphism $\bN$ of Theorem \ref{main-thm} is compatible with $m^G_\sigma$ and the the homomorphism
	 $m^S_\sigma$ on the symplectic side. That is to say, $\bN\circ m^G_\sigma=m^S_\sigma\circ \bN$.
\end{theorem-intro}

This theorem is meant to feature an instance of a more general result. In particular, we believe that this theorem should generalize to the case that one uses arbitrary homology classes in $H_*(\Sigma)$ and the homology classes are defined with arbitrary coefficient ring. 

The original version of the Atiyah-Floer conjecture was stated in \cite{At:AT-Fl}. This version of the Atiyah-Floer conjecture concerns the invariants of 3-manifolds with the same integral homology as $S^3$ \cite{floer:inst1}. One can generalize this conjecture so that it has the conjecture of \cite{At:AT-Fl} and Theorem \ref{framed-AF} as two special cases. A strategy to approach the Atiyah-Floer conjecture for admissible bundles, similar to the method of this paper, was proposed in \cite{Fuk:Fl-boundary}. The key geometrical tool to prove the results of this paper is the mixed equation \cite{Max:GU-comp,DFL:mix} whereas the proposal of \cite{Fuk:Fl-boundary} is based on a version of the ASD equation defined using degenerate metrics. Another major approach to the Atiyah-Floer conjecture makes use of adiabatic limits. The adiabatic limits method was already used in \cite{DS:At-Fl} and it forms a crucial part of the programs of \cite{We:AF-exp,duncan2015higherrank}. Other attempts to the Atiyah-Floer conjecture can be found in \cite{Yos:AF,LL:AF}.

\paragraph{Notations.} In this paper we shall be concerned with connections on manifolds of dimensions $2$, $3$ and $4$. To avoid confusion, we denote a typical connection on a $4$-manifold by $A$ (possibly with an index), on a $3$-manifold by $B$ (possibly with an index) and on a surface by a greek letter.

The Euclidean space $\R^3$ with the standard cross product defines a Lie algebra with an action of $\SO(3)$. This Lie algebra with the action of $\SO(3)$ is isomorphic to $\mathfrak{so}(3)$, linear space of skew-adjoint endomorphisms of $\R^3$, and $\mathfrak{su}(2)$, the linear space of trace free skew-Hermitian endomorphisms of $\C^2$. The Lie algebra structure and the action of $\SO(3)$ on $\mathfrak{so}(3)$ and $\mathfrak{su}(2)$ are respectively given by the commutator map and the adjoint action. Throughout this paper, we use this isomorphism to identify an $\SO(3)$ vector bundle $V$ with the bundle $\mathfrak{so}(V)$ of skew-adjoint endomorphisms of $V$. We also define a bi-linear form $\tr:\R^3\times \R^3\to \R$ given by $-\frac{1}{2}$ of the inner product. Using the identification with $\mathfrak{su}(2)$, this bi-linear form can be identified with $\tr:\mathfrak{su}(2) \times \mathfrak{su}(2)\to \R$ which maps a pair of a skew-Hermitian matrices $A$ and $B$ to $\tr(AB)$. The bi-linear form $\tr$ induces a bi-liner form on sections of any $\SO(3)$-vector bundle, which is denoted by the same notation.

\section{Floer homology groups} \label{HF}

In this section, we recall the definitions of Floer homology groups $\rI_*(Y_\#,E_\#)$ and $\rSI_*(Y_\#,E_\#)$ for an admissible pair $(Y_\#,E_\#)$. The definition of the latter Floer homology group requires some preparation. First we recall the definition of the symplectic manifold $\mathcal M(\Sigma,F)$. Then we define the 3-manifold Lagrangian $L(Y,E)$, which is an immersed Lagrangian of $\mathcal M(\Sigma,F)$. In the case that $F$ is replaced with the trivial bundle, analogues of the 3-manifold Lagrangian $L(Y,E)$ are the main subject of study in \cite{He:3-man-lag}. Our case of interest is less complicated because there is no singular point in $\mathcal M(\Sigma,F)$. 

\subsection{3-manifolds and $\SO(3)$-bundles}\label{3-man-bdles}

Suppose $\Sigma$ is a Riemann surface with two connected components. We assume that $Y$, $Y'$ are oriented connected 3-manifolds such that an identification of collar neighborhoods of their boundary components with $[-1,2)\times \Sigma$, $(-2,1]\times \Sigma$ are fixed, which are respectively orientation preserving and orientation reversing. Throughout this paper, we use outward-normal-first convention to orient the boundary of an oriented 3-manifold and the first-factor-first convention to orient the product of two oriented manifolds. We reverse the orientation of $Y'$ and glue it to $Y$ using the rule
\[
  (t,x)\in [-1,1]\times \Sigma \subset Y \sim (t,x)\in [-1,1]\times \Sigma \subset Y'
\]
to form a closed oriented 3-manifold $Y_\#$. It will be useful to fix a notation for the following subspaces of $Y$, $Y'$ 
\begin{equation}\label{Y0'}
	Y_0:= Y\setminus [-1,1)\times \Sigma, \hspace{1cm} Y_0':= Y'\setminus (-1,1]\times \Sigma
\end{equation}
which are clearly diffeomorphic to $Y$, $Y'$. Let $g$, $g'$ be Riemannian metrics on $Y$, $Y'$ that restrict to the product metric on the collar neighborhoods of the boundaries of $Y$ and $Y'$ corresponding to a fixed metric on $\Sigma$. Gluing these metrics produces a metric $g_\#$ on $Y_\#$, and in the following we use $g$, $g'$ or $g_\#$ when we need a metric on $Y$, $Y'$ or $Y_\#$. We will impose further constraints on the metrics $g$ and $g'$ in Section \ref{perturbations}.

Suppose $F$ is an $\SO(3)$-bundle on $\Sigma$ with non-trivial restrictions to the connected components of $\Sigma$. Suppose $\SO(3)$-bundles $E$, $E'$ on $Y$, $Y'$ are also given such that their restrictions to $[-1,2)\times \Sigma\subset Y$, $(-2,1]\times \Sigma\subset Y'$ are identified with $[-1,2)\times F$, $(-2,1]\times F$. Then we can glue these two $\SO(3)$-bundles to form the bundle $E_\#$ on $Y_\#$. In particular, $(Y_\#,E_\#)$ is an admissible pair with an admissible splitting as in \eqref{ad-split} determined by $(Y,E)$ and $(Y',E')$. This admissible pair with the given splitting shall be fixed for the rest of the paper.

\subsection{The moduli space $\mathcal M(\Sigma,F)$}\label{flat-surface}

Throughout this section, we need to consider the space of connections on $\SO(3)$-bundles over manifolds of dimensions $2$, $3$ and $4$. As the first instance, let $\mathcal A(\Sigma,F)$ be the space of connections on the bundle $F$. For analytical purposes, it is convenient to allow for Sobolev connections. To that end, we fix an integer $l\geq 2$, and assume that $\mathcal A(\Sigma,F)$ is defined using $L^2_{l-1}$ connections.\footnote{The same integer $l$ is used throughout the paper for Sobolev spaces associated to manifolds of various dimensions. The exact choice of the Sobolev parameter $l$ does not play any role in this paper.} Let $F \times_{\rm ad} \SU(2)$ be the fiber bundle on $\Sigma$ associated to the framed bundle of $F$ via the adjoint action of $\SO(3)$ on $\SU(2)$. Any section of this fiber bundle is called a {\it determinant one gauge transformation} of $F$. We use this standard terminology for $\SO(3)$ bundles over manifolds of any dimension. The space of $L^2_{l}$ sections of $F \times_{\rm ad} \SU(2)$ forms the Banach Lie group $\mathcal G(F)$. Taking pullbacks with respect to elements of $\mathcal G(F)$ determines an action of this group on $\mathcal A(\Sigma,F)$, and the quotient space, the {\it configuration space} of connections on $\Sigma$, is denoted by $\mathcal B(\Sigma,F)$. Embedded in this infinite dimensional Banach manifold, there is the moduli space $\mathcal M(\Sigma,F)$, which consists of the elements of $\mathcal B(\Sigma,F)$ that are represented by flat connections. The dimension of $\mathcal M(\Sigma,F)$ is equal to $-3\chi(\Sigma)$.

\begin{remark}\label{ext-gauge-p}
	An alternative  gauge group $\mathcal G_{\rm ex}(F)$ can be defined by considering the sections of the 
	fiber bundle $F \times_{\rm ad} \SO(3)$ induced by the adjoint action of $\SO(3)$ on itself. 
	There is an obvious map from $\mathcal G(F)$ to $\mathcal G_{\rm ex}(F)$ induced by the quotient map 
	$\SU(2)$ to $\SO(3)$. This map fits into an exact sequence:
	\[
	  \Z/2\oplus \Z/2 \hookrightarrow{}\mathcal G(F)\xrightarrow{}\mathcal G_{\rm ex}(F)\twoheadrightarrow{}
	  H^1(\Sigma;\Z/2).
	\]
	The first map is the inclusion of the elements of $\mathcal G(F)$ which are locally equal to $\pm 1$, and the last map for $g\in \mathcal G_{\rm ex}(F)$
	 is given as the obstruction of lifting $g$ to $\mathcal G(F)$ over the 
	1-skeleton of $\Sigma$.
\end{remark}

For a flat connection $\sigma$ on $F$, the tangent space of the smooth manifold $\mathcal M(\Sigma,F)$ at $[\sigma]$ is given by
\begin{equation}\label{tangent=space}
	\mathcal H^1(\Sigma;\sigma)=\{c\in \Omega^1(\Sigma,\Lambda^1\otimes F)\mid d_\sigma c=0,\,d_\sigma^* c=0\}.
\end{equation}
We consider the $L^2$ inner product
\begin{equation}\label{L^2-metric-dim-2}
	\langle c,c'\rangle:=-\int_{\Sigma} \tr(c\wedge *_2c')
\end{equation}
on  \eqref{tangent=space}, where $\tr$ is defined by applying the inner product of $F$ to the vector factor of $c$ and $c'$.
In this paper, we use a similar convention to define inner products of differential forms of any degree with values in an $\SO(3)$-bundle over a Riemannian manifold. We similarly define a symplectic form $\Omega$ on $\mathcal H^1(\Sigma;\sigma)$:
\begin{equation}\label{symplectic-forms}
	\Omega(c,c'):=-\int_{\Sigma} \tr(c\wedge c').
\end{equation}
The complex structure $J_*:\mathcal H^1(\Sigma;\sigma) \to \mathcal H^1(\Sigma;\sigma)$, defined as $J_*(c)=*_2c$, can be used to relate the metric and the symplectic form:
\[
  \langle c,c'\rangle=\Omega(c,J_*c').
\]

The basic topological invariants of $\mathcal M(\Sigma,F)$ are well understood. This manifold is simply connected and $\pi_2(\mathcal M(\Sigma,F))=\pi_2(\mathcal M(\Sigma_0,F_0))\oplus \pi_2(\mathcal M(\Sigma_1,F_1))$ where $\pi_2(\mathcal M(\Sigma_i,F_i))=0$ if the genus of the connected component $\Sigma_i$ of $\Sigma$ is $1$, and $\pi_2(\mathcal M(\Sigma_i,F_i))=\Z$ otherwise \cite{AB:YM,DS:-CR-SF}. In fact, for a surface $\Sigma_g$ of genus $g\geq 2$, a generator of $\pi_2(\mathcal M(\Sigma_g,F_g))$ can be constructed in the following way. There is a connection $A$ on the pullback of $F_g$ to $D^2\times \Sigma_g$ such that for any point $z\in D^2$ the restriction $A\vert_{\{z\}\times \Sigma_g}$ is flat, the restriction of $A$ to the boundary $S^1\times \Sigma_g$ is flat and 
\[
  \frac{1}{8\pi^2}\int_{D^2\times \Sigma_g}\tr\(F_A\wedge F_A\)=\frac{1}{2}.
\]
In particular, for any $z\in S^1$, the flat connection $A\vert_{\{z\}\times \Sigma_g}$ represents a fixed element $\alpha$ of $\mathcal M(\Sigma_g,F_g)$. Therefore, $A$ induces a map $s:(D^2,S^1)\to (\mathcal M(\Sigma_g,F_g),\alpha)$ determining a generator of $\pi_2(\mathcal M(\Sigma_g,F_g))$. Since the connections $A\vert_{\{z\}\times \Sigma_g}$ for $z\in S^1$ are gauge equivalent to a fixed (irreducible) flat connection, we also obtain a loop in $\mathcal G(F_g)/\{\pm 1\}$ from the restriction of $A$ to $S^1\times \Sigma_g$ which gives a generator of $\pi_1(\mathcal G(F_g)/\{\pm 1\})$.

\subsection{3-manifold Lagrangians}\label{3-man-lag-sec}

Suppose $(Y,E)$ is as in Subsection \ref{3-man-bdles}. Fix a smooth connection $B_0$ on $E$, and define 
\begin{equation*}
	\mathcal A(Y,E):=\{B_0+b \mid b \in L^2_l(Y,\Lambda^1 \otimes E)\}
\end{equation*}
The group $\mathcal G(E)$ of determinant one gauge transformations of $E$ with finite $L^2_{l+1}$ norm acts smoothly on $\mathcal A(Y,E)$ by taking pull back. We denote the quotient space by $\mathcal B(Y,E)$.  Let $B\in \mathcal A(Y,E)$ be an irreducible connection, that is to say, the stabilizer of $B$ with respect to the action of $\mathcal G(E)$ consists of only $\pm 1$. Then $\mathcal B(Y,E)$ is a Banach smooth manifold at the class of $B$. The tangent space of $\mathcal B(Y,E)$ at this point can be identified with 
the Banach space
\begin{equation}  \label{gauge-slice}
	X_B:=\{b \in L^2_l(Y,\Lambda^1 \otimes E) \mid d_B^*b=0,\, *_3b|_{\Sigma}=0\}.
\end{equation}
In particular, for any $b$ in $L^2_l(Y,\Lambda^1 \otimes E)$, the tangent space of $\mathcal A(Y,E)$ at $B$, there is $\zeta \in L^2_{l+1}(Y,E)$ such that $b-d_B\zeta$ belongs to $X_B$. There is a variation of this fact which shall be useful. Let $b\in L^2_l(Y,\Lambda^1 \otimes E)$ such that $b|_{\Sigma}=0$. Then there is $\zeta \in L^2_{l+1}(Y,E)$ with $\zeta\vert_{\Sigma}=0$ such that $d_B^*(b-d_B\zeta)=0$.

Define a $\mathcal G(E)$-equivariant map $\phi:\mathcal A(Y,E) \to L^2_{l-1}(Y,\Lambda^1 \otimes E)$ by
\begin{equation}\label{flat-3d}
  \phi(B)=*_3F_B.
\end{equation}
In general the space $L(Y,E):=\phi^{-1}(0)/\mathcal G(E)$ might not be smooth because $\phi$ might have zeros which are not cut down transversely. To achieve transversality,  we perturb $\phi$ following a standard scheme used in various places including \cite{donaldson:orientations,floer:inst1,Tau:Cass,He:3-man-lag,KM:YAFT}.

In Section \ref{hol-pert-sect}, we review the definition of a family of functions defined on $\mathcal A(Y,E)$, which are known as {\it cylinder functions}. Given a cylinder function $h:\mathcal A(Y,E)\to \R$, we may define its formal gradient $\nabla h:\mathcal A(Y,E) \to L^2_{l}(Y,\Lambda^1 \otimes E)$ with respect to the $L^2$ metric on $\mathcal A(Y,E)$. This determines a gauge invariant perturbation of \eqref{flat-3d}:
\begin{equation}\label{flat-pert-3d}
  \phi_h(B)=*_3F_B+\nabla_B h.
\end{equation}
The function $h$ depends only on the restriction of the connection $B$ to the interior of $Y_0$. In particular, $\nabla_B h$ vanishes on a neighborhood of the boundary of $Y$. Moreover, invariance of $h$ with respect to the action of $\mathcal G(E)$ implies that $d_B^*\nabla_B h$ vanishes.  This together with Bianchi identity implies that $\phi_h(B)$ belongs to the kernel of $d_B^*$. We write $L_h(Y,E)$ for the quotient space $\phi_h^{-1}(0)/\mathcal G(E)$. Any element of $L_h(Y,E)$ restricts to a flat connection on the boundary Riemann surface $\Sigma$. In particular, this defines a map $r:L_h(Y,E) \to \mathcal M(\Sigma,F)$. 

The linearization of $\phi_h$ at any connection $B$ modulo the action of the gauge group defines a map from $X_{B}$ to $L^2_{l-1}(Y,\Lambda^1\otimes E)$ as follows:
\begin{equation}\label{linearized}
	b\to*_3d_{B}b+{\rm Hess}_{B}h(b).
\end{equation}
This map has an infinite dimensional co-kernel and is not Fredholm. To resolve this issue, let $\Pi_{B}$ be the projection to the kernel of $d_{B}^*$ acting on $L^2_{l-1}(Y,\Lambda^1\otimes E)$. Since $d_{B'}^*\phi_h(B')=0$ for any $B'\in \mathcal A(Y,E)$, the zeros of $\phi_h$ and  $\Pi_{B}\circ \phi_h$, in a neighborhood of $B$, agree with each other. Therefore, we may consider linearization of the operator $\Pi_{B}\circ \phi_h$ to study the deformation theory of the space $L_h(Y,E)$.

\begin{prop}[\cite{He:3-man-lag}]\label{linearized-3-man-Lag}
	Suppose $B$ represents an element of $L_h(Y,E)$. Then
	\[
	  d_{B}^*(*_3d_{B}+{\rm Hess}_{B}h)=0.
	\]
	In particular, the linearization of $\Pi_{B}\circ \phi_h$, 
	denoted by $L_{B}:X_{B}\to \ker(d_{B}^*)$, is given by \eqref{linearized}.
	The operator $L_{B}$ is Fredholm with index $-\frac{3}{2}\chi(\Sigma)$. 
	The kernel of this operator can be identified with
	\begin{equation}\label{tang-per}
		\mathcal H^1_h(Y;B):=\{b\in L^2_l(Y,\Lambda^1 \otimes E) \mid *_3b|_{\Sigma}=0,\, d_B^*b=0,\,*_3d_B(b)+{\rm Hess}_{B}h(b)=0 \},
	\end{equation}
	and its cokernel is given by
	\begin{equation}\label{tang-per}
		\mathcal H^1_h(Y,\Sigma;B):=\{b \in L^2_l(Y,\Lambda^1 \otimes E) \mid b|_{\Sigma}=0,\, d_B^*b=0,\,*_3d_B(b)+{\rm Hess}_{B}h(b)=0 \}.
	\end{equation}
\end{prop}
\begin{proof}
	Let $\Psi$ be a smooth 0-form on $Y$ with values in $E$ which is supported in the interior of $Y$,
	 and $b \in L^2_l(Y,\Lambda^1 \otimes E)$. If $B_t$ is the connection $B+tb$, then the inner product of 
	$*_3F_{B_t}+\nabla_{{B_t}} h$ and $d_{B_t} \Psi$ vanishes. Taking derivative with respect to $t$ implies that:
	\begin{equation*}
		\langle*_3d_{B}b+{\rm Hess}_{B}h(b),d_{B} \Psi\rangle_{L^2(Y)}+\langle*_3F_B+\nabla_{{B}} h,[b,\Psi]\rangle_{L^2(Y)}=0
	\end{equation*}	
	Now the claim follows from the assumption that $*_3F_B+\nabla_{{B}} h$ vanishes. The remaining claims can be treated as in 
	\cite{He:3-man-lag}.
\end{proof}
As another useful property of cylinder functions, we record the following lemma. It is a consequence of the symmetric property of hessians.
\begin{lemma}\label{H-prop}
	Suppose $B\in \mathcal A(Y,E)$ and $b, b'\in L^2_l(Y,\Lambda^1 \otimes E)$. Then we have
 	\begin{equation}\label{H-prop-id}
	  \int_{Y}\tr(b\wedge *_3{\rm Hess}_{B}h (b'))=\int_{Y}\tr(*_3{\rm Hess}_{B}h (b)\wedge b').
	\end{equation}
\end{lemma}

To study regularity properties of $L_h(Y,E)$ for a general choice of $h$, it is useful to consider the family version of this construction. Let $\mathcal P$ denote the {\it parameter space} of cylinder functions (to be defined in Section \ref{hol-pert-sect}). Define a gauge invariant map $\Phi:\mathcal A(Y,E)\times \mathcal P\to L^2_{l-1}(Y,\Lambda^1 \otimes E)$ as follows:
\begin{equation}\label{flat-pert-3d}
	\Phi(B,\rho)=*_3F_B+\nabla_B h_\rho.
\end{equation}
Then $\bL(Y,E)=\Phi^{-1}(0)/\mathcal G(E)$ determines the family version of $L_h(Y,E)$. There is an obvious map $\pi:\bL(Y,E) \to \mathcal P$ such that the fibers of this map are the spaces $L_h(Y,E)$. Any element of $\bL(Y,E)$ restricts to a flat connection on the boundary Riemann surface $\Sigma$. In particular, we have a map $\br:\bL(Y,E) \to \mathcal M(\Sigma,F)$. Obviously the restriction of this map to each subspace $L_h(Y,E)$ is $r$. The proof of the following proposition will be given in Section \ref{hol-pert-sect}.

\begin{prop}[\cite{He:3-man-lag}]\label{family-lag}
	The map $\Phi$ is a submersion, and hence the space $\bL(Y,E)$ is a smooth Banach manifold. The projection map $\pi:\bL(Y,E) \to \mathcal P$
	 is a proper Fredholm map of index $-\frac{3}{2}\chi(\Sigma)$. 
	 Moreover, the restriction map $\br:\bL(Y,E) \to \mathcal M(\Sigma,F)$ is a submersion.
\end{prop}

Proposition \ref{family-lag} and Sard-Smale theorem imply that the space $L_h(Y,E)$ is a smooth compact manifold of dimension $-\frac{3}{2}\chi(\Sigma)$ for a generic choice of $h$. Thus, Proposition \ref{linearized-3-man-Lag} implies that for any such $h$ and any $[B]\in L_h(Y,E)$, $\mathcal H_{h}^1(Y;B)$ has dimension $-\frac{3}{2}\chi(\Sigma)$ and $\mathcal H_{h}^1(Y,\Sigma;B)$ is trivial. For any such $B$ and any $b\in \mathcal H_{h}^1(Y;B)$, the restriction $c:=b\vert_{\Sigma}$ is $d_\sigma$ closed where $\sigma=B\vert_{\Sigma}$. Thus, there is $\xi$ such that $c-d_\sigma\xi$ belongs to $\mathcal H^1(\Sigma;\sigma)$, and the derivative of the restriction map $r:L_h(Y,E) \to \mathcal M(\Sigma,F) $ at $B$ is given by $b\to c-d_\sigma\xi$. Thus, if $b$ belongs to the kernel of this map, then there is $\zeta \in L^2_{l+1}(Y,E)$ such that $b-d_B\zeta$ restricts to the trivial 1-form on $\Sigma$. In particular, by applying the fact mentioned at the beginning of this subsection, we may assume that $b-d_B\zeta\in \mathcal H^1_h(Y,\Sigma;B)$, which implies that $b=d_B\zeta$ because of our assumption on $h$. Since $b\in X_b$, this implies that $b=0$. Thus, the restriction map $r$ is an immersion.

\begin{prop} \label{immersed-lag}
	Suppose $h$ is a regular value of the projection map $\pi:\bL(Y,E) \to \mathcal P$. Then the immersion 
	$r:L_{h}(Y,E)\to \mathcal M(\Sigma,F)$ is Lagrangian. Moreover, given a finite dimensional smooth manifold $N$ and a 
	smooth map $s:N\to \mathcal M(\Sigma,F)$,
	for a generic $h$, the map $r$ is transversal to $s$.
\end{prop}
\noindent
In the following, if $h$ satisfies the assumption of the first part of this proposition, we say $L_{h}(Y,E)$ is {\it regular}.
\begin{proof}
	Suppose $B$ represents an element of $L_{h}(Y,E)$ and $b, b'\in \mathcal H_{h}^1(Y;B)$.
	Suppose also $c$ and $c'$ denote the restrictions of these elements to $\Sigma$:
	\begin{align*}
		\int_\Sigma \tr(c\wedge c')&=\int_Y \tr(d_Bb\wedge b')-\tr(b\wedge d_Bb')\\
		&=-\int_Y \tr(*_3{\rm Hess}_{B}h(b) \wedge b')-\tr(b\wedge *_3{\rm Hess}_{B}h(b'))\\
		&=0
	\end{align*}
	where the last identity is a consequence of Lemma \ref{H-prop}. To verify the second part, notice that 
	$\br:\bL(Y,E) \to \mathcal M(\Sigma,F)$ is transversal to $s:N\to \mathcal M(\Sigma,F)$ because the former map is a submersion.
	In particular, the following space is a smooth Banach manifold:
	\[
	  \bL(Y,E) \leftidx{_\br}{\times}{_s} N:=\{([A],x)\mid [A]\in \bL(Y,E),\, x\in N,\, \br([A])=s(x)\}.
	\]
	The map $p$ induces a map from $\bL(Y,E)\leftidx{_\br}{\times}{_s}N$ to $\mathcal P$ which is Fredholm (with 
	index $-\frac{3}{2}\chi(\Sigma)+\dim(N)-\dim(\mathcal M(\Sigma,F))$). 
	If $h$ is a regular value of this map, then $r:L_{h}(Y,E)\to \mathcal M(\Sigma,F)$ is transversal to
	$s:N\to \mathcal M(\Sigma,F)$. Therefore, the second part is a consequence of the Sard-Smale theorem.
\end{proof}

\begin{prop}\label{pert-3-man-Lag}
	Suppose $h, h'\in \mathcal P$ are two cylinder functions such that $L_{h}(Y,E)$ and $L_{h'}(Y,E)$ 
	are immersed Lagrangians of $\mathcal M(\Sigma,F)$. Then $L_{h}(Y,E)$ and $L_{h'}(Y,E)$ are 
	Lagrangian cobordant. That is to say, there is a Lagrangian immersion 
	$R:V\to \R^2\times \mathcal M(\Sigma,F)$ such that there are subspaces 
	$V_-,V_+\subset V$ with the property that 
	\[
	  R^{-1}((-\infty,-1]\times \R \times \mathcal M(\Sigma,F))=V_-,\hspace{1cm}
	  R^{-1}([1,\infty)\times \R \times \mathcal M(\Sigma,F))=V_+,
	\]
	$V_-$, $V_+$ can be identified with $(-\infty,-1]\times L_{h}(Y,E)$, $[1,\infty)\times L_{h'}(Y,E)$, and the 
	restrictions of $R$ to $V_\pm$ is given by $(\id,0,r)$.
	Here the symplectic structure on the product space $\R^2\times \mathcal M(\Sigma,F)$ is induced 
	by the standard symplectic structure $dx\wedge dy$ on $\R^2$ and the symplectic form $\Omega$ on $\mathcal M(\Sigma,F)$.
\end{prop}
A closely related result is proved in \cite{He:3-man-lag}. In the case that the bundles $F$ and $E$ over $\Sigma$ and $Y$ are trivial, Herlad shows that $L_h(Y,E)$, which is a singular Lagrangian in the singular symplectic manifold $\mathcal M(\Sigma,F)$, has a Legendrian lift to a certain $S^1$-bundle over $\mathcal M(\Sigma,F)$, and there is a Legendrian cobordism from the Legendrian lift of $L_h(Y,E)$ to the Legendrian lift of $L_{h'}(Y,E)$ for any two choices of perturbation functions $h$ and $h'$. Herlad's result has a counterpart in the admissible setup of this paper. However, we content ourselves with Proposition \ref{pert-3-man-Lag}.
\begin{proof}
	For $t\in \R$, suppose $h_t\in \mathcal P$ is a 1-parameter family of cylinder functions depending smoothly on 
	$t$ such that $h_t=h$ for $t\leq -1$ and $h_t=h'$.  Let also $g_t$ be the function 
	$\frac{dh_\tau}{d\tau}\vert_{\tau=t}$.
	Define
	\[
	  V=\{(t,[B])\in \R\times \mathcal B(Y,E)\mid *_3F_B+\nabla_B h_t=0 \}.
	\]
	As a consequence of Proposition \ref{family-lag}, this family of cylinder functions can be chosen such 
	that $V$ is a smooth manifold, and its tangent space at $(t,[B])$ is given as
	\begin{equation}\label{tangent-space}
	  T_{(t,[B])}V=\{(s,b)\mid *_3b|_{\Sigma}=0,\, d_B^*b=0,\,*_3d_B(b)+{\rm Hess}_{B}h(b)+s\nabla_B g_t=0\}.
	\end{equation}
	Since $V$ is cut down transversely, the vector space
	\begin{equation}\label{cokernel-family}
	  \{b \in L^2_l(Y,\Lambda^1 \otimes E) \mid b|_{\Sigma}=0,\, d_B^*b=0,\,*_3d_B(b)+{\rm Hess}_{B}h(b)=0
	  ,\,\langle\nabla_B g_t,b\rangle=0 \}
	\end{equation}
	 is trivial.
	Consider the map $R:V\to \R^2\times \mathcal M(\Sigma,F)$ defined as
	\[
	  R(t,[B])=(t,g_t([B]),r([B])).
	\]
	Analogous to the map $r$ and using the triviality of \eqref{cokernel-family}, one can see that $R$ 
	is an immersion. Let $(s,b)$ and $(s',b')$ be
	two vectors in \eqref{tangent-space}, and $c$, $c'$ denote the restrictions of $b$, $b'$ to $\partial Y$.
	Then we have
	\begin{align*}
		\int_\Sigma \tr(c\wedge c')&=\int_Y \tr(d_Bb\wedge b')-\tr(b\wedge d_Bb')\\
		&=-\int_Y \tr((*_3{\rm Hess}_{B}h(b)+s*_3\nabla_B g_t) \wedge b')-
		\tr(b\wedge (*_3{\rm Hess}_{B}h(b')+s'*_3\nabla_B g_t))\\
		&=s\langle\nabla_B g_t,b'\rangle-s'\langle\nabla_B g_t,b\rangle.
	\end{align*}
	It is easy to see from this identity that $R$ induces a Lagrangian immersion.
\end{proof}

\subsection{Floer homology groups}\label{Fl-hom-subsec}

In this subsection, the pairs $(Y,E)$, $(Y',E')$ and the glued up pair $(Y_\#,E_\#)$ is given as in Subsection \ref{3-man-bdles}. The construction of Subsection \ref{3-man-lag-sec} can be used to form two immersed Lagrangians $L_{h}(Y,E)$ and $L_{h'}(Y',E')$ of $\mathcal M(\Sigma,F)$. The essential assumption that we make throughout the paper is:

\begin{hypo-intro}\label{main-hypo}
	For any positive real number $\epsilon$, there is $h$ such that the associated parameter in $\mathcal P$ 
	is smaller than $\epsilon$,
	$L_{h}(Y,E)$ is regular, and $r:L_{h}(Y,E) \to \mathcal M(\Sigma,F)$ defines a submanifold. 
	The pair $(Y',E')$ satisfies a similar property.
\end{hypo-intro}

The following proposition provides a special case where the assumptions in Hypothesis \ref{main-hypo} are always satisfied.
\begin{prop}\label{emb-fund}
	If the inclusion map induces a surjection of the fundamental group of a connected
	component 
	$\Sigma_0$ of $\Sigma$ into $\pi_1(Y)$, then 
	$L_{h}(Y,E)$ with $h=0$, denoted by $L_0(Y,E)$, is regular and $r:L_{0}(Y,E) \to \mathcal M(\Sigma,F)$ defines an 
	embedded submanifold.
\end{prop}
\begin{proof}
	It is clear that $r:L_{0}(Y,E) \to \mathcal M(\Sigma,F)$ is injective. For $B\in L_{0}(Y,E)$, it suffices to show that $\mathcal H^1(Y,\Sigma;B)$ is trivial. Let 
	$b$ be a 1-form with values in $E$ such that 
	\[b|_{\Sigma}=0, \hspace{1cm} d_B^*b=0, \hspace{1cm}d_Bb=0.\]
	Let $\zeta$ be the section of $E$ defined as 
	\[
	  \zeta(p):=\int_{\gamma}\gamma^*b \hspace{1cm}\text{for $p\in Y$}.
	\]
	Here $\gamma$ is a path from a base point $p_0\in \Sigma_0$ to $p$ and to define the integral, we 
	trivialize $E$ along $\gamma$ using parallel transport with respect to $B$. This integral 
	does not depend on the choice of $\gamma$, because the integral of $b$ over any closed path
	based at $p_0$ vanishes. To see the latter claim, note that any closed path based at $p_0$ can 
	be homotoped into a closed path in $\Sigma_0$ and the Stokes theorem and the 
	assumption $d_Bb=0$ show that the integral 
	does not change throughout the homotopy. Since $b|_{\Sigma_0}=0$, the integral over a path in 
	$\Sigma$ clearly vanishes. The definition of $\zeta$ implies that 
	$b=d_B\zeta$. Using the Stokes theorem we have
	\begin{align}
	  |\!|b|\!|_{L^2}&=-\int_{Y}\tr(d_B\zeta\wedge *_3d_B\zeta)\nonumber\\
	  &=-\int_{\Sigma}\tr(\zeta\wedge *_3d_B\zeta)+\int_{Y}\tr(\zeta\wedge d_B*_3d_B\zeta) \label{L2normb}.
	\end{align}
	The second term in the last expression vanishes because $d_B^*b=0$. Since $d_B\zeta|_{\Sigma}=0$ and the restriction of $B$ to each boundary component of $\Sigma$
	is irreducible, $\zeta|_{\Sigma}=0$, which implies the vanishing of the first term in \eqref{L2normb}. 
\end{proof}

Suppose $h$ and $h'$ are chosen such that $L_{h}(Y,E)$ and $L_{h'}(Y',E')$ are embedded Lagrangian submanifolds of $\mathcal M(\Sigma,F)$. Proposition \ref{immersed-lag} implies that by a small perturbation of the cylinder function $h$ (or $h'$), we may assume that the submanifolds $L_{h}(Y,E)$ and $L_{h'}(Y',E')$ are transversal to each other. To define symplectic instanton Floer homology  $\rSI_*(Y_\#,E_\#)$ as the Lagrangian Floer homology of these two Lagrangians, we need to guarantee that $L_{h}(Y,E)$ and $L_{h'}(Y',E')$ satisfy further restrictive assumptions.

\begin{definition}\label{mon-prop-def}
	Suppose $L$ is a Lagrangian submanifold of a symplectic manifold $(M,\omega)$. Suppose $\mu:\pi_2(X,L)\to \Z$ is the Maslov index function.
	Integrating the symplectic form $\omega$ on discs with boundary values in $L$ defines another map $[\omega]:\pi_2(X,L)\to \R$.
	The Lagrangian $L$ is {\it monotone} if there is a positive constant $c$ such that
	\begin{equation}\label{mon-1}
		[\omega]=c\mu.
	\end{equation}
	The positive generator of the image $\mu$ is called the {\it minimal Maslov number} of $L$.
\end{definition}

\begin{definition}\label{mon-prop-2-def}
	Suppose $L$ and $L'$ are Lagrangian submanifolds of a symplectic manifold $(M,\omega)$. 
	Suppose $\Omega(L,L')$ denotes the space of all paths from $L'$ to $L$. Any $\alpha\in L\cap L'$ determines a constant path $o_\alpha\in \Omega(L,L')$.
	Any element of $\pi_1(\Omega(L,L'),o_\alpha)$
	determines a continuous map from $S^1\times [-1,1]$ to $M$ with boundary components $S^1\times \{1\}$, $S^1\times \{-1\}$
	mapped to $L$ and $L'$. In particular, the Maslov index and the symplectic form induces maps 
	$\mu:\pi_1(\Omega(L,L'),o_\alpha)\to \Z$ and $[\omega]:\pi_1(\Omega(L,L'),o_\alpha)\to \R$. We say the pair $(L, L')$ is 	{\it monotone}, if there 
	is a positive constant $c$ such that:
	\begin{equation}\label{mon-2}
		[\omega]=c\mu.
	\end{equation}	
	for any choice of $\alpha$.
	The positive generator of the image $\mu$ is called the {\it minimal Maslov number} of $(L,L')$.
\end{definition}

The following proposition will be proved in Subsection \ref{monotone-lag-subs}. The proof of this proposition utilizes the linear theory of the mixed equation, discussed in the next subsection.
\begin{prop}\label{monotone-lag}
	The Lagrangians $L_h(Y,E)$ and $L_{h'}(Y',E')$ are orientable and monotone with minimal Maslov numbers $4$. 
	In fact, the pair $(L_h(Y,E), L_{h'}(Y',E'))$ is monotone.
\end{prop}

\begin{remark}
	If $L$ and $L'$ are monotone Lagrangians in a simply connected symplectic manifold $(M,\omega)$, then $(L,L')$ is also a monotone pair. 
	In particular, the second part of Proposition \ref{monotone-lag} is an immediate consequence of the first part.
\end{remark}

Now we review how Proposition \ref{monotone-lag} allows us to define Lagrangian Floer homology of the monotone pair $(L_h(Y,E), L_{h'}(Y',E'))$ \cite{Oh:LFH}. Suppose $\alpha$, $\beta$ belong to the finite set $\fC_S:=L_{h}(Y,E)\cap L_{h'}(Y',E')$. Let $u:\R\times [-1,1] \to \mathcal M(\Sigma,F)$ be a smooth map such that 
\begin{equation} \label{constant-ends}
	\hspace{2cm}  u(-s,\theta)=\alpha,\hspace{.5cm}u(s,\theta)=\beta, \hspace{2cm}\forall (s,\theta)\in [1,\infty)\times [-1,1]
\end{equation}
and it satisfies the boundary conditions
\begin{equation}\label{BC}
	u\vert_{\R\times \{1\}}\subset L_{h}(Y,E),\hspace{1cm} u\vert_{\R\times \{-1\}}\subset L_{h'}(Y',E').
\end{equation}
Given another map $u':\R\times [-1,1] \to \mathcal M(\Sigma,F)$ with similar properties, we say $u$ and $u'$ are homotopic, if there is a smooth map $U:\R\times [-1,1]\times [-1,1] \to \mathcal M(\Sigma,F)$ such that 
\begin{itemize}
	\item[(i)] for any $t\in [-1,1]$, $U\vert_{\R\times [-1,1]\times \{t\}}$ satisfies \eqref{constant-ends} and \eqref{BC};
	\item[(ii)] $U\vert_{\R\times [-1,1]\times \{-1\}}=u$ and $U\vert_{\R\times [0,1]\times \{1\}}=u'$.
\end{itemize}
Equivalence classes of this relation can be regarded as homotopy classes of paths from $\alpha$ to $\beta$ in $\Omega((L_h(Y,E), L_{h'}(Y',E')))$. The set of all such homotopy classes is denoted by $\pi_2(\alpha,\beta)$. Maslov class of any $p\in \pi_2(\alpha,\beta)$, denoted by $\mu(p)$, is defined to be the Maslov class of any representative of $p$. For a path $p$, it is helpful to form a space $\mathcal B_S(\alpha,\beta)_{p}$ consisting of more general representatives of $p$.

To define $\mathcal B_S(\alpha,\beta)_{p}$, let $\exp$ denote the exponential map with respect to an arbitrary metric on $\mathcal M(\Sigma,F)$. A continuous map $u:\R\times [-1,1] \to \mathcal M(\Sigma,F)$ satisfying \eqref{BC} is an element of $\mathcal B_S(\alpha,\beta)_{p}$ if the following conditions hold. The covariant derivatives $\nabla^k (du)$ have finite $L^2$ norms for $0\leq k\leq l-1$. Moreover, there is $t_0$ such that the restriction of $u$ to $(-\infty,-t_0]$ and $[t_0,\infty)$ is given by $\exp_{\alpha}v_-$ and $\exp_{\beta}v_+$ for $v_-\in L^2_{l}((-\infty,-t_0]\times [-1,1],T_{\alpha}\mathcal  M(\Sigma,F))$ and $v_+\in L^2_{l}([t_0,\infty)\times [-1,1],T_{\beta}\mathcal  M(\Sigma,F))$. In particular, $u$ determines an element of $\pi_2(\alpha,\beta)$, which we denote by $p$. We call $\mathcal B_S(\alpha,\beta)_{p}$ the {\it configuration space of strips} corresponding to the path $p$.

The set $\fC_S$ can be decomposed as
\[
  \fC_S:=\bigcup_{o} \fC_{S,o},
\]
where $o$ runs over the set of the connected components of $\Omega(L_h(Y,E), L_{h'}(Y',E'))$ and $\fC_{S,o}$ consists of $\alpha\in \fC_S$ such that $o_\alpha=o$. There is a relative $\Z/4\Z$-grading $\deg_S$ on each $\fC_{S,o}$, which is called {\it Floer grading}. Let $\alpha,\beta\in \fC_{S,o}$ and $p$ be a path from $\alpha$ to $\beta$ in $\Omega(L_h(Y,E), L_{h'}(Y',E'))$. Then define
\begin{equation}\label{Deg-S}
	\deg_S(\alpha)-\deg_S(\beta)\equiv \mu(p)  \mod 4
\end{equation}
Proposition \ref{monotone-lag} implies that the value of $\mu(p)$ mod $4$ is independent of the choice of $p$ and hence $\deg_S$ is well-defined. A relative $\Z/4\Z$-grading on $\fC_S$ is {\it compatible with the Floer grading} if its restriction to each $\fC_{S,o}$ agrees with the Floer grading.

\begin{remark}
	Since $\mathcal M(\Sigma,F)$ is simply connected, the connected component of $o_\alpha$ 
	is determined by 
	the connected components of $L_h(Y,E)$ and $L_{h'}(Y',E')$ that contain $\alpha$. 
	In particular, if $L_h(Y,E)$ and $L_{h'}(Y',E')$
	are connected, then $\Omega(L_h(Y,E), L_{h'}(Y',E'))$ is path connected. Consequently,
	there is a unique relative $\Z/4\Z$-grading on $\fC_S$ compatible with the Floer grading. This, for example, happens for the Lagrangians involved in
	the definition of symplectic framed Floer homology.
\end{remark}

Fix a 1-parameter family of $\Omega$-compatible almost complex structures $\mathcal J=\{J_\theta\}_{\theta\in [-1,1]}$ on $\mathcal M(\Sigma,F)$, and consider the Cauchy-Riemann equation
\begin{equation}\label{CR}
	\frac{\partial u}{\partial \theta}-J_\theta\frac{\partial u}{\partial s}=0
\end{equation}
where $u:\R\times [-1,1] \to \mathcal M(\Sigma,F)$ satisfies the Lagrangian boundary condition of \eqref{BC}. Any solution of \eqref{CR}, with $\vert\!\vert d u\vert\!\vert_{L^2}$ being finite, belongs to $\mathcal B_S(\alpha,\beta)_{p}$ for some choice of $\alpha$, $\beta$ and the homotopy class of a path $p$ from $\alpha$ to $\beta$. The space of all such solutions of \eqref{CR} is denoted by $ \rM_S(\alpha,\beta)_p$. Translation along the $\R$ factor defines an $\R$-action on $\rM_S(\alpha,\beta)_p$, which is free unless $\rM_S(\alpha,\beta)_p$ contains the constant map to $\alpha$. The quotient space by this action is denoted by $\breve \rM_S(\alpha,\beta)_p$. 

For any $u\in \mathcal B_S(\alpha,\beta)_{p}$, let $\mathcal D_u$ denote the linearization of \eqref{CR}	. Then $\mathcal D_u$ is an operator acting on $L^2_l$ sections of $u^*T\mathcal M(\Sigma,F)$ with the boundary condition that the restriction of $u$ to $\R\times \{1\}$ and $\R\times \{-1\}$ belong to $TL_{h}(Y,E)$ and $TL_{h'}(Y',E')$, and $L^2_{l-1}(\R\times [-1,1], u^*T\mathcal M(\Sigma,F))$ is the target of this operator. For a section $\zeta$ of $u^*T\mathcal M(\Sigma,F)$ in the domain of $\mathcal D_u$, we have
\begin{equation}\label{linear-CR}
	\mathcal D_u \zeta=\nabla_\theta \zeta-J_{\theta}(u)\nabla_{s}\zeta-(\nabla_\zeta J_{\theta})\frac{du}{ds},   
\end{equation}
where the connection $\nabla$ is defined by pulling back the Levi-Civita connection on $\mathcal M(\Sigma,F)$. The index of this elliptic operator is equal to $\mu(p)$. The equation \eqref{CR} is cut down transversely at $u$ if $\mathcal D_u$ is surjective. In a neighborhood of $u$, the moduli space $ \rM_S(\alpha,\beta)_p$ is a smooth manifold of dimension $\mu(p)$, which is equal to $ \deg_S(\alpha)-\deg_S(\beta)$ mod 4.

\begin{lemma}[\cite{Oh:LFH}]\label{alm-cx-str}
	There is a family of almost complex structures $\mathcal J=\{J_\theta\}_{\theta\in [-1,1]}$ such that 
	the moduli space $\rM_S(\alpha,\beta)_p$ is cut-down transversely. 
\end{lemma}

The moduli spaces $\rM_S(\alpha,\beta)_p$ are orientable. Using a standard construction, we may define the {\it determinant line bundle}  $\delta^S_p$ on $ \mathcal B_S(\alpha,\beta)_{p}$, where the fiber over $u$ is given by
\[
  \Lambda^{\rm max}\ker(\mathcal D_u)\otimes (\Lambda^{\rm max}\coker(\mathcal D_u))^*.
\]
If $\rM_S(\alpha,\beta)_p$ is cut down transversely at $u$, then an orientation of the fiber of $\delta^S_p$ determines an orientation of $T_u\rM_S(\alpha,\beta)_p$. Thus, to orient $\rM_S(\alpha,\beta)_p$ it suffices to fix a trivialization of $\delta^S_p$, which always exists (see Proposition \ref{orientaion-delta-p}). We denote the set of trivializations of this bundle by $\Lambda^S_p$, which is a $\Z/2\Z$-torsor. The set $\Lambda^S_p$ can be identified with the trivializations of $\delta^S_p$ over the subspace $\mathcal B^c_S(\alpha,\beta)_p$ of $\mathcal B_S(\alpha,\beta)_p$ consisting maps $u$ which satisfy \eqref{constant-ends}. 

If $p$ is a path from $\alpha_0$ to $\alpha_1$ and $p'$ is a path $\alpha_1$ to $\alpha_2$, then there is an obvious strip gluing map $\mathcal B^c_S(\alpha_0,\alpha_1)_p\times \mathcal B^c_S(\alpha_1,\alpha_2)_{p'} \to \mathcal B^c_S(\alpha_0,\alpha_2)_{p\sharp p'}$ which induces the map
\begin{equation}\label{strip-gluing-ori}
  \Phi_{p,p'}:\Lambda^S_p\otimes_{\Z/2\Z}\Lambda^S_{p'}\to \Lambda^S_{p\sharp p'}.
\end{equation}
using additivity of the index of the Fredholm operator $\mathcal D_u$ with respect to gluing strips.

\begin{definition}\label{coh-ori-def}
	A {\it coherent system of orientations for strips associated to the Lagrangians $L_{h}(Y,E)$ and 
	$L_{h'}(Y',E')$} is an association of an element $\lambda_p\in \Lambda^S_p$ to each homotopy class $p$ 
	of a path between two elements of $\fC_S$ which is compatible with the map
	$\Phi_{p,p'}$. That is to say, for any two paths $p$ and $p'$, 
	where the terminal point of $p$ is equal to the initial point of $p'$, we have
	\[
	  \Phi_{p,p'}(\lambda_p\otimes \lambda_{p'})=\lambda_{p\sharp p'}.
	\]
	Two systems of coherent orientations 
	$\{\lambda_p\}$ and $\{\lambda_p'\}$ are {\it $\epsilon$-equivalent} if there is $\epsilon:\fC_S\to \Z/2$
	such that for any path $p$ from $\alpha$ to $\beta$
	\[
	  \lambda_p'=(-1)^{\epsilon(\beta)-\epsilon(\alpha)}\lambda_p.
	\]
\end{definition}

\begin{prop}\label{orientaion-delta-p}
	The line bundles $\delta^S_p$ are orientable. Moreover, there is a coherent system of orientations for strips 
	associated to the Lagrangians $L_{h}(Y,E)$ and $L_{h'}(Y',E')$. 
\end{prop}

A proof of this proposition will be given in Subsection \ref{monotone-lag-subs}. In fact, we will also give a recipe in the proof of Proposition \ref{orientaion-delta-p} to fix a coherent system of orientations for strips associated to the Lagrangians $L_{h}(Y,E)$ and $L_{h'}(Y',E')$.

\begin{remark}\label{deficiency-ori}
	Although Proposition \ref{orientaion-delta-p} is sufficient for our purposes here, there is still room to improve this proposition.
	For instance, the Lagrangians $L_h(Y,E)$ and $L_{h'}(Y',E')$ are in fact spin, and the spin structure can be used to fix orientations for the line bundles $\delta^S_p$ following \cite{FOOO:HF1,FOOO:HF2}.
	The authors expect that there is a preferred choice of spin structures for $L_h(Y,E)$ and $L_{h'}(Y',E')$ and the induced orientations by these spin structures agree with the coherent system of orientations $\{\lambda_p\}$ constructed in Subsection \ref{monotone-lag-subs}. 
	Another issue related to orientations of the determinant bundles which is not completely addressed here is compatibility of the coherent system of orientations $\{\lambda_p\}$ with {\it gluing spheres}.
	For an arbitrary smooth map $s:S^2\to \mathcal M(\Sigma,F)$, we may similarly define the linearized Cauchy-Riemann operator $\mathcal D_s$ as a map
	\[
	  L^2_l(S^2,s^*T\mathcal M(\Sigma,F)) \to L^2_{l-1}(S^2,\Lambda^{0,1}\otimes s^*T\mathcal M(\Sigma,F)).
	\]
	Since $S^2$ is a closed Riemann surface, the operator $\mathcal D_s$ is complex linear up to compact terms and hence its determinant line has a canonical orientation. Gluing $s$ to elements of $\mathcal B^c_S(\alpha,\beta)_p$ determines a map $\mathcal B^c_S(\alpha,\beta)_p \to 	\mathcal B^c_S(\alpha,\beta)_{p'}$ where $p':=p\sharp s$ is the induced path from $\alpha$ to $\beta$. This gluing map gives 
	\[
	  \Psi_{p,s}:\Lambda^S_p \to \Lambda^S_{p'},
	\]
	which depends only on the homotopy class of $s$. Recall that $\pi_2(\mathcal M(\Sigma,F))=\Z^{i}$ where $i$ is the number of the connected components of $\Sigma$ which have genus greater than $1$.
	We can guarantee that the system of orientations $\{\lambda_p\}$ given by Proposition \ref{orientaion-delta-p} is compatible with the map $\Psi_{p,s}$ in the case that $i=1$. For $i=2$, we can only obtain a system of orientations $\{\lambda_p\}$ compatible with $\Psi_{p,s}$
	when $s$ belongs to one of the summands of $\pi_2(\mathcal M(\Sigma,F))$. See Remark \ref{comp-Z-summand} for more details. 
\end{remark}

Let $C_S((Y,E),(Y',E'))$ be the abelian group freely generated by the elements of $\fC_S$. Fix a family of almost complex structures as in Lemma \ref{alm-cx-str} and orient the smooth manifolds $\rM_S(\alpha,\beta)_p$ using the orientation given by Proposition \ref{orientaion-delta-p}. The space $\rM_S(\alpha,\beta)_p$ is a fiber bundle over $\breve \rM_S(\alpha,\beta)_p$ with fiber $\R$, and the total space and the fiber of this bundle are oriented. We orient $\breve \rM_S(\alpha,\beta)_p$ such that the orientation of $\rM_S(\alpha,\beta)_p$ is obtained from those of $\R$ and $\breve \rM_S(\alpha,\beta)_p$ using the fiber-first convention. Let $d_S:C_S((Y,E),(Y',E'))\to C_S((Y,E),(Y',E'))$ be the linear map whose value at $\alpha\in \fC_S$ is given by
\[
  d_S(\alpha):=\sum_{p:\alpha\to \beta}\#\breve \rM_S(\alpha,\beta)_p \cdot \beta,
\]
where the above sum is taken over all paths $p$ such that $\breve \rM_S(\alpha,\beta)_p$ is zero dimensional, and $\#\breve \rM_S(\alpha,\beta)_p$ denotes the signed count of the elements of $\breve \rM_S(\alpha,\beta)_p$. In particular, $d_S$ decreases the $\Z/4\Z$-grading by $1$. This map is a differential, i.e., $d_S^2=0$, and the homology of the chain complex $(C_S((Y,E),(Y',E')),d_S)$ is independent of the choice of the family of almost complex structures $\mathcal J$. In fact, our main theorem shows that this homology group depends only on $(Y_\#,E_\#)$, and the symplectic instanton Floer homology $\rSI_*(Y_\#,E_\#)$ of the pair $(Y_\#,E_\#)$ is defined to be this relatively $\Z/4\Z$-graded homology group.

\begin{prop}\label{inv-symp-ins}
	The chain homotopy type of the chain complex $(C_S((Y,E),(Y',E')),d_S)$ is an invariant of the pair
	$(Y_\#,E_\#)$. In particular, it does not depend on the family of almost complex strictures
	$\mathcal J$, the cylinder functions $h$ and $h'$, and the coherent system of orientations provided by Proposition 
	\ref{orientaion-delta-p}..
\end{prop}
\begin{proof}
	This is a consequence of Proposition \ref{top-inv-gauge} below and Theorem 
	\ref{main-thm-detailed}, proved in the next section.
\end{proof}

\begin{remark}
	It is desirable to give a direct proof for the above proposition. The invariance with 
	respect to the choice of almost complex structures is standard. Proposition \ref{pert-3-man-Lag}
	asserts that changing cylinder functions $h$ and $h'$ gives rise to cobordant Lagrangians. 
	Thus, one would expect that the results of \cite{BC:lag-cob}
	imply that changing $h$ and $h'$ give
	chain homotopy equivalent chain complexes $(C_S((Y,E),(Y',E')),d_S)$. However, 
	it is not clear that the Lagrangian cobordism $V$ provided by Proposition \ref{pert-3-man-Lag}
	is embedded. If so, then it is reasonable to expect that the 
	analogue of Proposition \ref{monotone-lag} 
	holds, and $V$ is monotone. We hope to come back to this issue in a sequel where we pursue 
	generalization of the results of this paper to the case of immersed Lagrangians. 
\end{remark}

\begin{remark}\label{simply-con-Lag-ori}
	Suppose the Lagrangians $L_h(Y,E)$ and $L_{h'}(Y',E')$ are simply connected, 
	and $\{\lambda_p\}$ and $\{\lambda_p'\}$ are two systems of coherent orientations, which are compatible with the maps $\Psi_{p,s}$. 
	(For example, the Lagrangians involved in the definition of symplectic framed Floer homology together with the systems of orientations of \cite{WW:FFT-coprime} and Proposition \ref{orientaion-delta-p} have this property.)
	Then there 
	is a map $\kappa:\fC_S\times \fC_S\to \Z/2\Z$ such that for any path $p$ from $\alpha$ to $\beta$
	we have
	\[
	  \lambda_p'=(-1)^{\kappa(\alpha,\beta)}\lambda_p.
	\]
	This follows from the fact that any two paths from $\alpha$ to $\beta$ are related to each other by
	gluing an element of $\pi_2(\mathcal M(\Sigma,F))$. Since $\{\lambda_p\}$ and $\{\lambda_p'\}$ are both
	systems of coherent orientations, we have
	\[
	  \kappa(\alpha_1,\alpha_3)=\kappa(\alpha_1,\alpha_2)+\kappa(\alpha_2,\alpha_3).
	\]
	Therefore, there exists $\epsilon: \fC_S\to \Z/2\Z$ such that 
	\[
	  \kappa(\alpha,\beta)=\epsilon(\beta)-\epsilon(\alpha).
	\]
	That is to say, $\{\lambda_p\}$ and $\{\lambda_p'\}$ are $\epsilon$-equivalent. Thus, Lagrangian Floer homology groups 
	of $L_h(Y,E)$ and $L_{h'}(Y',E')$ with respect to $\{\lambda_p\}$ and $\{\lambda_p'\}$ are isomorphic to each other. 
	In particular, our definition of symplectic framed Floer homology agree with \cite{WW:FFT-coprime}.
\end{remark}

Next, we turn into the gauge theoretical component of our main theorem. For the pair $(Y_\#,E_\#)$, the class $w_2(E_\#)\in H^2(Y_\#,\Z/2\Z)$ has a non-trivial pairing with each connected component of the copy of $\Sigma$ in $Y_\#$. In particular, $E_\#$ is admissible in the sense of \cite{BD:surgery}. We review the definition of a version of instanton Floer homology for the admissible pair $(Y_\#,E_\#)$ which is more suitable for our purposes.

Suppose $\mathcal A(Y_\#,E_\#)$ is the space of all $L^2_l$ connections on $E_\#$. This is an affine space modeled on $L^2_l(Y_\#,\Lambda^1 \otimes E_\#)$. Let $\mathcal G(E_\#)$ be the space of global sections of the fiber bundle $E_\#\times_{\rm ad}\SU(2)$ of class $L^2_{l+1}$. As in the case of 3-manifolds with boundary, the gauge group $\mathcal G(E_\#)$ acts on $\mathcal A(Y_\#,E_\#)$, and we denote the quotient space by $\mathcal B(Y_\#,E_\#)$. Analogous to Remark \ref{ext-gauge-p}, we may form the gauge group $\mathcal G_{\rm ex}(E_\#)$ using sections of $E_\#\times_{\rm ad}\SO(3)$. There is an obvious homomorphism $\mathcal G(E_\#) \to \mathcal G_{\rm ex}(E_\#)$, whose cokernel can be identified with $H^1(Y,\Z/2\Z)$. The group $\mathcal G_{\rm ex}(E_\#)$ acts on $\mathcal A(Y_\#,E_\#)$, extending the action of $\mathcal G(E_\#)$. In particular, there is an action of $H^1(Y,\Z/2\Z)$ on $\mathcal B(Y_\#,E_\#)$. We shall be interested in the action of $l_\Sigma\in H^1(Y,\Z/2\Z)$ given as the Poincar\'e dual of any of the connected components of $\Sigma$. We write $\iota$ for the involution determined by $l_\Sigma$.

\begin{lemma}\label{fixed-iota}
	An element of $\mathcal B(Y_\#,E_\#)$ is fixed by the action of $\iota$ if it is represented 
	by an $\O(2)$-connection such that its orientation bundle is determined by the cohomology class
	$l_\Sigma$. In particular, $\iota$ does not have any fixed point which restricts to a flat connection 
	on one of the connected components of $\Sigma$.	
\end{lemma}
\begin{proof}
	Suppose $\widetilde E_\#$ denotes a $\U(2)$-bundle such that $c_1(\widetilde E_\#)$ is a lift of $w_2(E_\#)$. Then the vector bundle associated to $\widetilde E_\#$ by the adjoint map $\U(2) \to \SO(3)$ is isomorphic to $E_\#$. Moreover, the determinant 
	map $\U(2)\to \U(1)$ induces a complex line bundle $\det(\widetilde E_\#)$, and we fix a connection $b_0$ on this line bundle. The configuration space of connections on $\widetilde E_\#$ with the induced connection on $\det(\widetilde E_\#)$ being $b_0$
	can be identified with $\mathcal B(Y_\#,E_\#)$. Using this identification, the involution $\iota$ on $\mathcal B(Y_\#,E_\#)$ is given by taking the tensor product with a real line bundle determined by $l_\Sigma$. For a $\U(2)$ connection $\widetilde B$ 
	representing an element of $\mathcal B(Y_\#,E_\#)$ and a loop $\gamma$ based at a point $x\in Y_\#$, if ${\rm hol}_\gamma(\widetilde B)$ denotes the holonomy of $\widetilde B$ along $\gamma$, then the holonomy of $\iota(\widetilde B)$ 
	is $(-1)^{l_\Sigma(\gamma)}{\rm hol}_\gamma(\widetilde B)$. Thus $\widetilde B$ represents a fixed point of $\iota$, if and only if there is $g\in \SU(2)$ such that for any loop $\gamma$ based at $x$ we have
	\[
	  (-1)^{l_\Sigma(\gamma)}{\rm hol}_\gamma(\widetilde B)=g{\rm hol}_\gamma(\widetilde B)g^{-1}.
	\]
	By picking a loop $\gamma_0$ with $l_\Sigma(\gamma_0)=1$, we conclude that $\tr(g)=0$. Now, if $\gamma$ represents an element in $\ker(l_\Sigma)$, then ${\rm hol}_\gamma(\widetilde B)$ commutes with $g$, and otherwise 
	${\rm hol}_{\gamma_0}(\widetilde B){\rm hol}_\gamma(\widetilde B)$ commutes with $g$. It is easy to see from this that the $\SO(3)$ connection induced by $\widetilde B$ is an $\O(2)$ connection with orientation bundle $l_\Sigma$. In particular, the restriction of 
	any such connection to a connected component of $\Sigma$ is an $S^1$ connection on $F$, and hence this restriction cannot be flat.
\end{proof}

The cylinder functions $h$ and $h'$ may be used to define a perturbation of the flat equation on $\mathcal A(E_\#)$:
\begin{equation}\label{flat-3d-closed}
	\phi_{h,h'}(B)=*_3F_B+\nabla_B h+\nabla_B h'.
\end{equation}
This map is equivariant with respect to the automorphisms of $E_\#$. Any solution of $\phi_{h,h'}(B)=0$ restricts to a flat connection on a neighborhood of $\Sigma$. In particular, such connections are irreducible and $\iota$ acts freely on them by Lemma \ref{fixed-iota}. We write $\widetilde \fC_G$ and $\fC_G$ respectively for the subspaces of $\mathcal B(Y_\#,E_\#)$ and $\mathcal B(Y_\#,E_\#)/\iota$ which are represented by the solutions of \eqref{flat-3d-closed}. If $B_\#\in \mathcal A(Y_\#,E_\#)$ represents an element of $\widetilde \fC_G$, then the restrictions of $B_\#$ to $Y$ and $Y'$ determine an element of $\fC_S= L_h(Y,E) \cap L_{h'}(Y',E')$. Since the restrictions of $l_\Sigma$ to $Y$ and $Y'$ are trivial, this element of $\fC_S$  depends only on the equivalence class of $A_\#$ in $\fC_G$. Moreover, if the pair of $[B]\in  L_h(Y,E) $ and $[B']\in L_{h'}(Y',E')$ represents an element of $\fC_S$, then gluing these connections gives rise to two equivalence classes of connections in $\widetilde \fC_G$ which are related to each other by $\iota$. Thus, any element of $\fC_S$ determines a well-defined element of $\fC_G$. We summarize this discussion in the following lemma.
 
\begin{lemma}
	The space $\fC_G$ is compact and can be identified with $\fC_S$. 
\end{lemma}

As the first step to study \eqref{flat-3d-closed}, we may consider the linearization of \eqref{flat-3d-closed} as in the case of 3-manifolds with boundary. For any connection $\alpha$ representing an element of $\fC_G$, we may define $X_\alpha$ as in \eqref{gauge-slice} where the condition $*_3a\vert_\Sigma=0$ is dropped. There is also a Fredholm operator $L_B:X_B\to X_B$ as in Proposition \ref{linearized-3-man-Lag} with index $0$. An element of $\mathcal B(Y_\#,E_\#)$ represented by $B$ is regular if $L_B$ is a surjective operator. This is equivalent to say that the kernel of $L_B$ given as follows
\begin{equation}\label{tang-closed}
	\mathcal H^1_{h,h'}(Y_\#;B):=\{b \in L^2_l(Y_\#,\Lambda^1 \otimes E_\#) \mid d_B^*b=0,\,
	*_3d_B(b)+{\rm Hess}_{B}h(b)+{\rm Hess}_{B}h'(b)=0 \}
\end{equation}
 is trivial. The following lemma is a consequence of Mayer-Viertoris principle for the space $\mathcal H^1_{h,h'}(Y_\#,B)$.

\begin{lemma}\label{comp-per-reg}
	An element of $\fC_G$ is regular if the corresponding element of $\fC_S$ is given by a transversal intersection of $L_h(Y,E)$ and $L_{h'}(Y',E')$.
\end{lemma}
\begin{proof}
	Let $B_\#$ represent an element of $\fC_G$, and $B$, $B'$ denote its restrictions to $Y$ and $Y'$. 
	Since the connection $B_\#$ is flat on 
	the overlapping region $[-1,1]\times \Sigma$ of $Y$ and $Y'$, we may assume that $B_\#$
	is the pull-back of a flat connection $\sigma$ on the bundle $F$.
	Suppose $b_\#\in \mathcal H^1_{h,h'}(Y_\#;B_\#)$ whose restrictions to $Y$ and $Y'$ are denoted by 
	$b$ and $b'$. 
	There are $\zeta\in L^2_{l+1}(Y,E)$ and $\zeta'\in L^2_{l+1}(Y',E')$ such that
	\[
	  b-d_B\zeta\in \mathcal H^1_h(Y;B),\hspace{1cm}b'-d_{B'}\zeta'\in \mathcal H^1_{h'}(Y';B').
	\]
	The restrictions of $b-d_B\zeta$ and $b'-d_{B'}\zeta'$ to $\partial Y$ and $\partial Y'$
	represent the same element of $\mathcal H^1(\Sigma;\sigma)$ because they are equal to the 
	cohomology classes represented by $b_\#\vert_{\{t\}\times \Sigma}$ for any $t\in [-1,1]$.
	Transversality of the intersection of Lagrangians $L_h(Y,E)$ and $L_{h'}(Y',E')$ implies that
	$b-d_B\zeta=0$ and $b'-d_{B'}\zeta'=0$. 
	In particular, on the overlap region $[-1,1]\times \Sigma$, we have $d_{B_\#}(\zeta-\zeta')=0$.
	Since $B_\#$ restricted to the overlap region is irreducible, 
	$\zeta=\zeta'$, and hence $\zeta$ and $\zeta'$ determine a 0-form $\zeta_\#$ on $Y_\#$ such that
	$b_\#=d_{B_\#}\zeta_\#$. This identity and the condition $d_{B_\#}^*b_\#=0$ imply that
	$\zeta_\#=0$. Thus, $b_\#$ vanishes.
\end{proof}

Fix $\alpha,\,\beta\in \fC_G$, and let $A_0$ be a smooth connection on the bundle $E_\#\times \R$ over the cylinder 4-manifold $\R\times Y_\#$, such that the restriction of $A_0$ to $(-\infty,-1]\times Y_\#$ (resp. $[1,\infty)\times Y_\#$) is the pull-back of a representative $B$ of $\alpha$ (resp. $B'$ of $\beta$). We say two such connections $A_0$ and $A_1$ represent the same path, if there is a smooth section $g$ of $(\R\times E_\#)\times_{\rm ad} \SU(2)$ over $\R\times Y_\#$ such that $A_1-g^*A_0$ or $A_1-\iota (g^*A_0)$ is compactly supported. This defines an equivalence relation, and any equivalence class of this relation is called a {\it path} along $\R\times Y_\#$ from $\alpha$ to $\beta$. For a fixed pair $\alpha$ and $\beta$, a set of representatives for all paths from $\alpha$ to $\beta$ can be given as follows. Suppose $A_0$ is a connection as above which restricts to the pullbacks of the representative $B$ for $\alpha$ and the representative $B'$ for $\beta$. Suppose also $A_1$ is a connection on $\R\times Y_\#$ which restricts to the pullbacks of $B'$ and $\iota(B')$ on the ends $(-\infty,-1]\times Y_\#$ and $[1,\infty)\times Y_\#$. Then any path from $\alpha$ to $\beta$ is represented by gluing an $\SU(2)$ connection on $S^4$ with $c_2=k\in \Z$ to $A_0$ and then possibly gluing the resulting connection to $A_1$.

Given a path $p$ from $\alpha$ to $\beta$, fix a representative $A_0$ for $p$, and let $\mathcal A_G(\alpha,\beta)_p$  be the space of connections of the form $A_0+a$ where $a\in L^2_{l}(\R\times Y,\Lambda^1\otimes E_\#)$. The configuration space $\mathcal B_G(\alpha,\beta)_p$ is defined as the quotient of $\mathcal A_G(\alpha,\beta)_p$ with respect to the action of the sections $g$ of $(\R\times E_\#)\times_{\rm ad}\SU(2)$ over $\R\times Y_\#$ such that $\nabla_{A_0}g$ is in  $L^2_{l}$.
For $A\in \mathcal A_G(\alpha,\alpha_+)_p$, define the {\it perturbed ASD equation}
\begin{equation}\label{ASD}
	F^+_A+(*_3\nabla_{A_t}h)^++(*_3\nabla_{A_t'}h')^+=0
\end{equation}
where $F^+_A$ denotes the self-dual part of the curvature of $A$, defined with respect to the product metric on $\R\times Y_\#$. For each $t\in \R$, $A_t$ (resp. $A_t'$) denotes the restriction of $A$ to $\{t\}\times Y$ (resp. $\{t\}\times Y'$). Thus $\nabla_{A_t}h$ (resp. $\nabla_{A_t'}h'$) can be regarded as a 1-form on $\{t\}\times Y$ (resp. $\{t\}\times Y'$) with values in $E$ (resp. $E'$), and $*_3\nabla_{A_t}h$ (resp. $*_3\nabla_{A_t'}h'$) is the Hodge dual of $\nabla_{A_t}h$ (resp. $\nabla_{A_t'}h'$) with respect to the metric on $Y$ (resp. $Y'$). This equation is gauge invariant and determines a subspace of $\mathcal B_G(\alpha,\beta)_p$ which is denoted by $\rM_G(\alpha,\beta)_p$. Translation along the first factor of $\R\times Y_\#$ determines an action of $\R$ on $\rM_G(\alpha,\beta)_p$ and the quotient space with respect to this action is denoted by $\breve \rM_G(\alpha,\beta)_p$. 

For a connection $A\in \mathcal A_G(\alpha,\beta)_p$, define the ASD operator
\[
  \mathcal D_A:L^2_{l}(\R\times Y,\Lambda^1\otimes E_\#)\to 
  L^2_{l-1}(\R\times Y,(\Lambda^0\otimes \Lambda^+)\otimes E_\#)
\]
as follows:
\begin{equation}\label{ASD-op}
	\mathcal D_A(a):=(d_A^*a,d_A^+a+(*_3\Hess_{A_t}h(a_t))^++(*_3\Hess_{A_t'}h'(a_t'))^+).
\end{equation}
The first component of $\mathcal D_A$ takes into account the gauge fixing condition and the second component is given by the linearization of \eqref{ASD}. This operator $\mathcal D_A$ is the perturbation of the standard ASD operator by a compact term induced by $h$ and $h'$. Since $\alpha$ and $\beta$ are regular, $\mathcal D_A$ is a Fredholm operator. The index of this operator depends only on the path $p$ and otherwise is independent of $A$. Moreover, switching $p$ with another path from $\alpha$ to $\beta$ changes the index by a multiple of $4$. 
In fact, gluing a connection on $S^4$ with $c_2=k$ to the path $p$ changes the index by $8k$, and gluing $p$ to a path from a representative $B'$ of $\beta$ to $\iota(B')$, changes the index by an integer of the form $8k+4$. In particular, we may use the index of the path $p$ to define a relative $\Z/4\Z$-grading $\deg_G$ on $\fC_G$:
\begin{equation}\label{Deg-G}
	\deg_G(\alpha)-\deg_G(\beta)\equiv \ind(\mathcal D_A)  \mod 4.
\end{equation}

There is another useful number associated to a path $p$ from $\alpha$ to $\beta$. For any connection $A\in \mathcal A_G(\alpha,\beta)_p$, define the topological energy of $A$ as follows:
\begin{equation}\label{top-energy-connection}
	\mathcal E(A):=\frac{1}{8\pi^2}\int_{\R\times Y}\tr\((F_A+*_3\nabla_{A_t}h+*_3\nabla_{A_t'}h')\wedge (F_A+*_3\nabla_{A_t}h+*_3\nabla_{A_t'}h')\)
\end{equation}	
It is straightforward to check that
\begin{equation}\label{rewrite-top-energy}
  \mathcal E(A)=\frac{1}{4\pi^2}\(h(\alpha)+h'(\alpha)\)-\frac{1}{4\pi^2}\(h(\beta)+h'(\beta)\)+\frac{1}{8\pi^2}\int_{\R\times Y}\tr\((F_A)\wedge (F_A)\).
\end{equation}
The last term in the above sum, which is the more standard definition for the topological energy of $A$, depends only on the path. This implies that $\mathcal E(A)$ also depends only on $p$. For a connection $A$ that represents an element of $\rM_G(\alpha,\beta)_p$, $\mathcal E(A)$ is non-negative and is zero if and only if $\alpha=\beta$, $p$ is the constant path and $A$ is the pullback of a representative of $\beta$. Another straightforward observation about topological energy is that $2\mathcal E(A)\in \Z$ for any connection $A\in \mathcal A_G(\alpha,\alpha)_p$. This is a consequence of \eqref{rewrite-top-energy} and the fact that the Chern-Weil integral in \eqref{rewrite-top-energy} satisfies a similar property.

The following proposition gives a relationship between $\mathcal E(A)$ and the index of $\mathcal D(A)$.

\begin{prop}\label{top-energy-index-gauge}
	To each $\alpha \in \fC_G$, we can associate a real number $\epsilon(\alpha)$ such that for any 
	$A\in \mathcal A_G(\alpha,\beta)_p$, we have
	\[
	  \ind(\mathcal D_{A}):=8\mathcal E(A)+\epsilon(\beta)-\epsilon(\alpha).
	\]
\end{prop}
\begin{proof}
	The standard index formula for the ASD operator on a manifold with cylindrical end \cite{Taubes:L2-mod-space,MMR:L2-mod-space} asserts that
	\[
	  \ind(\mathcal D_{A})=\frac{\rho_{\beta}-\rho_{\alpha}}{2}+\frac{1}{\pi^2}\int_{\R\times Y}\tr\(F_A\wedge F_A\),
	\]
	where for $\alpha\in \fC_G$, $\rho_\alpha$ is the $\rho$-invariant associated to the connection $\alpha$. This identity and \eqref{rewrite-top-energy} give the desired result.
\end{proof}

Following proposition can be regarded as a linear version of our main theorem. It is also a variation of the main result of \cite{Tau:Cass} for the admissible setting. Proof of this result will be given in Subsection \ref{monotone-lag-subs}.
\begin{prop} \label{linear-AF-gradings}
	The relative $\Z/4$-grading $\deg_G$ on $\fC_S\cong \fC_G$ is compatible with the Floer grading $\deg_S$. In particular, in the case of framed Floer homology, the two gradings $\deg_G$ and $\deg_S$ agree with each other.
\end{prop}

As in the symplectic case, orientations of determinant lines of $\mathcal D_A$ for connections $A$ representing elements of $\mathcal B_G(\alpha,\beta)_p$ determine a real line bundle $\delta_p^G$ on $\mathcal B_G(\alpha,\beta)_p$. This line bundle is oriented \cite{floer:inst1,Don:YM-Floer} and the set of the two orientations of this line bundle is denoted by $\Lambda_p^G$. For paths $p$ from $\alpha_0\in \fC_G$ to $\alpha_1\in \fC_G$ and $p'$ from $\alpha_1$ to $\alpha_2\in \fC_G$ along $\R\times Y$, we may again define an isomorphism
\begin{equation}\label{cyl-gluing-ori}
  \Phi_{p,p'}:\Lambda^G_p\otimes_{\Z/2\Z} \Lambda^G_{p'}\to \Lambda^G_{p\sharp p'}.
\end{equation}
For a path $p$ from $\alpha$ to $\beta$ and the path $p'$ obtained by gluing a connection on $S^4$ with $ ,kc_2=k$ to $p$, there is an isomorphism obtained from gluing the standard orientations of the ASD complexes for $S^4$:
\[\Psi_{p,k}:\Lambda_{p}^G\to \Lambda_{p'}^G.\]
We can fix a system of orientations for the line bundles $\delta_p^G$, which is compatible with the maps $\Phi_{p,p'}$ and $\Psi_{p,k}$ \cite[Section 5.4]{Don:YM-Floer}. To achieve this goal, fix $\alpha_0\in \fC_G$ with a representative connection $B_0$ on $E_\#$. We also fix a path $p_0$ from $\alpha_0$ to $\alpha_0$ represented by a connection $A_0$ whose restrictions to the ends $(-\infty,-1]\times Y_\#$ and $[1,\infty)\times Y_\#$ are pullback of $B_0$ and $\iota(B_0)$. Fix an element $\lambda_{p_0}\in \Lambda_{p_0}^G$. For any $\alpha\in \fC_G$, we pick an arbitrary path $p$ from $\alpha$ to $\alpha_0$, and pick an element $\lambda_{p}\in \Lambda_p^G$. Then we extend this choice of orientations of the line bundle $\delta_p^G$ to all paths from $\alpha$ to $\alpha_0$ using the maps $\Phi_{p,p_0}$ and $\Psi_{p,k}$ and the orientation element $\lambda_{p_0}$. Finally for $\alpha$, $\beta\in \fC_G$ and a path $p$ from $\alpha$ to $\beta$, we pick an arbitrary path $p_+$ from $\beta$ to $\alpha_0$, and pick $\lambda_p\in \Lambda_p^G$ such that 
\[
  \Phi_{p,p_+}(\lambda_p\otimes \lambda_{p_+})=\lambda_{p_-},
\]
where $p_-$ is the path obtained by gluing $p$ to $p_+$.

 A connection $A$ representing an element of $\rM_G(\alpha,\beta)_p$ is regular if $\mathcal D_A$ is surjective. The moduli space $\rM_G(\alpha,\beta)_p$ in a neighborhood of a regular connection $[A]$ is a smooth manifold of dimension $\ind(\mathcal D_A)$, and a trivialization of $\delta_p^G$ fixes an orientation of this manifold. The following lemma, which will be proved in Subsection \ref{hol-pert-sect}, asserts that we can ensure regularity of the elements of moduli spaces which are essential for the definition of instanton Floer homology.
\begin{lemma} \label{reg-ASD}
	There are Riemannian metrics $g$, $g'$ on $Y$, $Y'$ and small enough perturbations of the cylinder functions $h$ and $h'$
	such that the sets $\fC_S$ and $\fC_G$ do not change and all solutions of 
	\eqref{ASD} with index at most seven are regular.
\end{lemma}

From now on, we assume that $h$ and $h'$ are chosen such that the spaces $L_h(Y,E)$, $L_{h'}(Y',E')$ are smooth embedded Lagrangians which intersect transversely and the claim of Lemma \ref{reg-ASD} holds. We also drop $h$ and $h'$ from our notations for the 3-manifolds Lagrangians.

Let $C_G(Y_\#,E_\#)$ be the abelian group freely generated by the elements of $\fC_G$. Fix orientations of the determinant line bundles $\delta_p^G$ as above, and use them to orient the moduli spaces $\breve \rM_G(\alpha,\beta)_p$. Let $d: C_G(Y_\#,E_\#) \to C_G(Y_\#,E_\#)$ be the map defined as 
\[
  d(\alpha):=\sum_{p:\alpha\to \beta}\#\breve \rM_G(\alpha,\beta)_p \cdot \beta.
\]
where the sum is over all paths $p$ from $\alpha$ to another element $\beta\in \fC_G$ such that the associated ASD operator has index $1$. Thus, the operator $d$ decreases the relative grading by $1$. Moreover, $d^2=0$ and the instanton Floer group $\rI_*(Y_\#,E_\#)$ is defined to be the homology of the relatively $\Z/4$-graded  complex $(C_G(Y_\#,E_\#),d)$. 

\begin{prop}\label{top-inv-gauge}
	The chain homotopy type of the chain complex $(C_G(Y_\#,E_\#),d)$ is independent of the choice of the 
	Riemannian metrics on $Y$, $Y'$, 
	the cylinder functions $h$ and $h'$ and the choices of orientation elements $\lambda_p\in \Lambda_p^G$.
	In particular, $\rI_*(Y_\#,E_\#)$ is a topological invariant of $(Y_\#,E_\#)$.
\end{prop}

The proof is standard and we refer the reader to \cite{floer:inst1,Don:YM-Floer} for more details. We only remark on the dependence of $\rI_*(Y_\#,E_\#)$ on the involution $\iota$, which is determined by the $\Z/2$ cohomology class dual to the  connected component $\Sigma_0$ of $\Sigma$. We may define a variant of $\rI_*(Y_\#,E_\#)$, where we do not pass to the quotient by the action of $\iota$. The resulting invariant is a $\Z/8\Z$-graded chain complex, which is a topological invariant of $(Y_\#,E_\#)$ and does not depend on the $\Z/2$ cohomology class of $\Sigma_0$. Moreover, $\iota$ induces an involution of degree $4$ on this complex and the quotient space is isomorphic to $\rI_*(Y_\#,E_\#)$. Thus, the isomorphism type of $\rI_*(Y_\#,E_\#)$ does not depend on $\iota$.

\section{Proof of the Main Theorem}\label{main-thm-sec}

In this section, we prove our main result, Theorem \ref{main-thm}. Our key tool in the proof is the {\it mixed equation}, which is defined using a combination of the Cauchy-Riemann equation and the ASD equation. In the prequel to this paper \cite{DFL:mix} and following \cite{Max:GU-comp}, we defined mixed equation for any {\it quintuple}. We recall the notion of quintuples in Subsection \ref{quintuples}, and introduce {\it special quintuples}, which are the specific type of quintuples used in our proof. In the next subsection, we use the moduli spaces of solutions to the mixed equation associated to special quintuples, and construct the desired isomorphism for Theorem \ref{main-thm}.

\subsection{Special quintuples}\label{quintuples}
A quintuple $\fq=(X,V,S,\mathcal M(\Sigma,F), \mathbb L)$ consists of a Riemannian $4$-manifold $X$, an $\SO(3)$-bundle $V$ over $X$, a Riemann surface $(S,j)$, the symplectic manifold $\mathcal M(\Sigma,F)$ and a collection of Lagrangians $\mathbb L=\{L_1,L_2,\dots,L_k\}$ in $\mathcal M(\Sigma,F)$. There is a (possibly non-compact) oriented 1-manifold $\gamma$ such that the boundary of the $4$-manifold $X$ is identified with $\gamma\times \Sigma$ where $\Sigma$ is the disconnected Riemann surface that we fixed in the previous section. Moreover, the restriction of $V$ to $\partial X$ is identified with the pullback of the $\SO(3)$-bundle $F$ on $\Sigma$ to $\gamma\times \Sigma$. The boundary components of the Riemann surface $S$ are given as 
\begin{equation}\label{boundary-S}
  \partial S=-\gamma\sqcup \eta_1 \sqcup \eta_2\sqcup \dots \sqcup \eta_k.
\end{equation}
In particular, we regard $L_i$ as a Lagrangian attached to the boundary component $\eta_i$. In \cite{DFL:mix}, we considered quintuples in the more general case that $\mathcal M(\Sigma,F)$ is replaced with an arbitrary symplectic manifold $(M,\omega)$. In that case, $\mathbb L$ includes some additional information in the form a certain type of Lagrangian correspondence from $\mathcal A(\Sigma,F)$ to $M$.

The mixed equation associated to the quintuple $\fq$ is defined for a pair of a connection $A$ on $V$ and a map $u:S\to \mathcal M(\Sigma,F)$:
\begin{equation}\label{mixed-eq-pre}
	\left\{\begin{array}{c}
		F^+_A=0,\\\overline \partial_{J_*}u=0,
	\end{array}	\right.
\end{equation}
where $\overline \partial_{J_*}u=1/2(du+J_*\circ du\circ j)$. The restriction of $A$ to $\{x\}\times \Sigma\subset \gamma \times \Sigma$, for each $x\in \gamma$, is required to be a flat connection representing $u(x)$. Moreover, $u(x)\in L_i$ for $x\in \eta_i$.
These two conditions are respectively called the {\it matching} and the Lagrangian boundary conditions. The Cauchy-Riemann-equation in \eqref{mixed-eq-pre} is defined using the standard complex structure $J_*$ on $\mathcal M(\Sigma,F)$. Eventually, we shall be interested in the case that the ASD equation in \eqref{mixed-eq-pre} is perturbed and the Cauchy-Riemann equation is defined by a domain dependent family of almost complex structures.

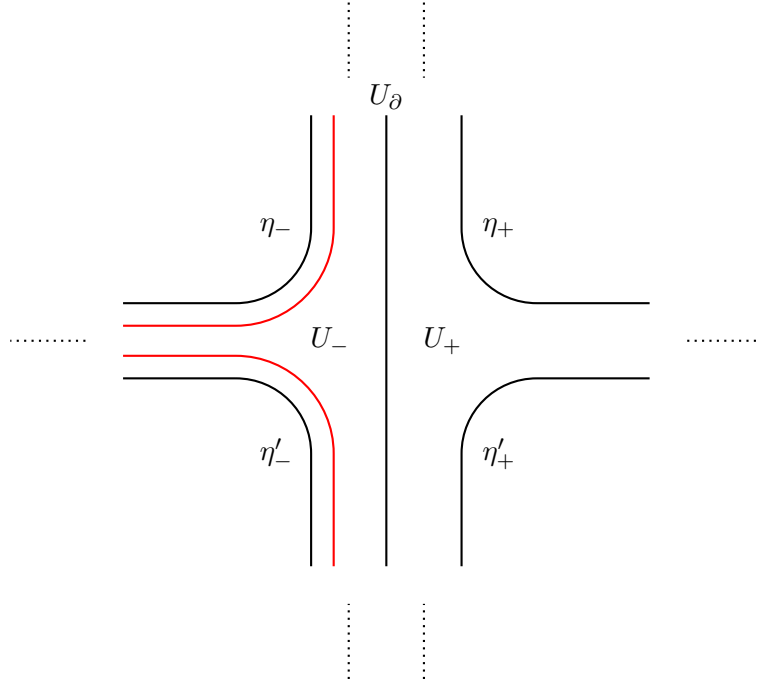
\begin{figure}
	\begin{center}
	\begin{tikzpicture}[thick]
\draw (0,3)  --  (0,-3)
	(-1,3)  --  (-1,1.5)
	(-1,1.5) coordinate (-left) 
	to [out=-90, in=0] (-2,.5)
	 (-2,.5) -- (-3.5,0.5)
	 (-2,-.5) -- (-3.5,-0.5)
	(-2,-.5) coordinate (-left) 
	to [out=0, in=90] (-1,-1.5)
	to (-1,-3);
\draw (1,3)  --  (1,1.5)
	(1,1.5) coordinate (-left) 
	to [out=-90, in=180] (2,.5)
	 (2,.5) -- (3.5,0.5)
	 (2,-.5) -- (3.5,-0.5)
	(2,-.5) coordinate (-left) 
	to [out=180, in=90] (1,-1.5)
	to (1,-3);
\draw[red]	(-.70,3) to (-.70,1.5)
	to [out=-90, in=0] (-2,.2)
	to (-3.5,0.2)
	(-3.5,-0.2) to (-2,-.2)
	to [out=0, in=90] (-.70,-1.5)
	to (-0.7,-3); 
\node at (-.75,0){$U_-$};	
\node at (.75,0){$U_+$};	
\node at (1.5,1.5){$\eta_+$};	
\node at (1.5,-1.5){$\eta_+'$};
\node[left] at (-1.1,1.5){$\eta_-$};
\node[left] at (-1.1,-1.5){$\eta_-'$};
\node at (0,3.25){$U_\partial$};
\draw[dotted] (-.5,4.5)  --  (-0.5,3.5);	
\draw[dotted] (.5,4.5)  --  (0.5,3.5);	
\draw[dotted] (-.5,-4.5)  --  (-0.5,-3.5);	
\draw[dotted] (.5,-4.5)  --  (0.5,-3.5);	
\draw[dotted] (4,0)  --  (5,0)
		     (-4,0)  --  (-5,0);	
	\end{tikzpicture}
	\end{center}
		\caption{The space $U$, a subspace of the complex plane $\C$. The chosen regular neighborhoods of $\eta_-$ and $\eta_-'$ are determined by the red curves.}
	\label{region-U}
\end{figure}

To define special quintuples, let $U$ be the domain in the complex plane which is sketched in Figure \ref{region-U}. This non-compact space has four boundary components, denoted by $\eta_+$, $\eta_-$, $\eta_+'$, $\eta_-'$, and contains the following subspaces of $\C$:
\[
  [-2,2]\times [2,\infty),\hspace{1cm}[-2,2]\times (-\infty,-2],\hspace{1cm}[3,\infty)\times [-1,1],\hspace{1cm}(-\infty,-3]\times [-1,1].
\] 
This space is decomposed as the union of the regions $U_+$ and $U_-$ which share the imaginary line in $\C$, denoted by $U_\partial$, as their common boundary components. We identify a regular neighborhood of the boundary components $\eta_-$, $\eta_-'$ with $\R \times (\frac{1}{2},1]$, $\R \times [-1,-\frac{1}{2})$ and fix a Riemannian metric $g_-$ on $U_-$ which is equal to product metrics on these regular neighborhoods, and is equal to the standard metric of the complex plane on the subsets 
\[
  [-2,0]\times [2,\infty),\hspace{1cm}[-2,0]\times (-\infty,-2],\hspace{1cm}(-\infty,-3]\times [-1,1],\hspace{1cm}[-1,0]\times [-2,2].
\]

Let $X$ be the oriented smooth 4-manifold given by gluing the following 4-manifolds along their common boundaries
\[\R\times Y_0\cup U_-\times \Sigma \cup \R \times -Y_0',\]
where $Y_0\subset Y$, $Y_0'\subset Y'$ are given in \eqref{Y0'}.
The subspaces $[-2,0]\times [2,\infty)$, $[-2,0]\times (-\infty,-2]$ and $(-\infty,-3]\times [-1,1]$ of $U_-$ determine subspaces $Z$, $Z'$ and $Z_\#$ of $X$ which are, respectively, diffeomorphic to $[2,\infty) \times Y $, $(-\infty,-2]\times Y'$ and $(-\infty,-3]\times Y_\#$. The projection maps from $Z$, $Z'$ and $Z_\#$ to $Y$, $Y'$ and $Y_\#$ are respectively denoted by $\pi$, $\pi'$ and $\pi_\#$. The fixed Riemannian metrics on $\Sigma$, $Y$ and $Y'$ in Subsection \ref{3-man-bdles} and the metric $g_-$ on $U_-$ give rise to a Riemannian metric on $X$, which we denote by $g_X$.  Moreover, the $\SO(3)$ bundles $E$, $E'$ and $F$ determine an $\SO(3)$-bundle on $X$, which we denote by $V$.

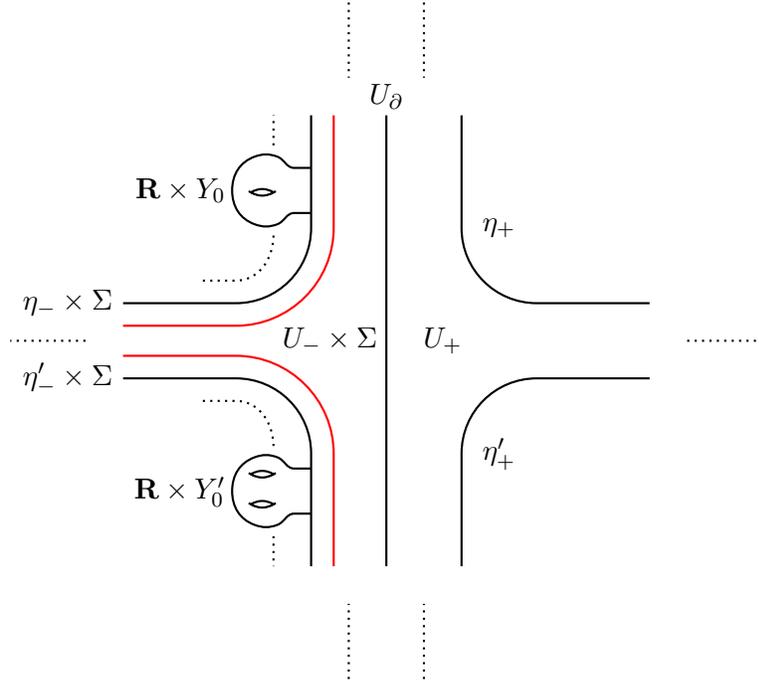
\begin{figure}
	\begin{center}
	\begin{tikzpicture}[thick]
\pic [scale=0.15](lower) at (-1.15,2) {handle};
\pic [scale=0.15](lower) at (-1.65,2.35) {hole};
\pic [scale=0.15](lower) at (-1.15,-2) {handle};
\pic [scale=0.15](lower) at (-1.65,-1.8) {hole};
\pic [scale=0.15](lower) at (-1.65,-1.4) {hole};
\draw (0,3)  --  (0,-3)
	(-1,3)  --  (-1,1.5)
	(-1,1.5) coordinate (-left) 
	to [out=-90, in=0] (-2,.5)
	 (-2,.5) -- (-3.5,0.5)
	 (-2,-.5) -- (-3.5,-0.5)
	(-2,-.5) coordinate (-left) 
	to [out=0, in=90] (-1,-1.5)
	to (-1,-3);
\draw[red]	(-.70,3) to (-.70,1.5)
	to [out=-90, in=0] (-2,.2)
	to (-3.5,0.2)
	(-3.5,-0.2) to (-2,-.2)
	to [out=0, in=90] (-.70,-1.5)
	to (-0.7,-3); 
\draw (1,3)  --  (1,1.5)
	(1,1.5) coordinate (-left) 
	to [out=-90, in=180] (2,.5)
	 (2,.5) -- (3.5,0.5)
	 (2,-.5) -- (3.5,-0.5)
	(2,-.5) coordinate (-left) 
	to [out=180, in=90] (1,-1.5)
	to (1,-3);
\node at (-.75,0){$U_-\times \Sigma$};	
\node at (.75,0){$U_+$};	
\node at (1.5,1.5){$\eta_+$};	
\node at (1.5,-1.5){$\eta_+'$};
\node[left] at (-3.5,.5){$\eta_-\times \Sigma$};
\node[left] at (-3.5,-.5){$\eta_-'\times \Sigma$};
\node[left] at (-2,2){$\R\times Y_0$};
\node[left] at (-2,-2){$\R\times Y'_0$};
\node at (0,3.25){$U_\partial$};
\draw[dotted] (-.5,4.5)  --  (-0.5,3.5);	
\draw[dotted] (.5,4.5)  --  (0.5,3.5);	
\draw[dotted] (-.5,-4.5)  --  (-0.5,-3.5);	
\draw[dotted] (.5,-4.5)  --  (0.5,-3.5);	
\draw[dotted] (4,0)  --  (5,0)
		     (-4,0)  --  (-5,0)
		     (-1.5,2.6) -- (-1.5,3)
		     (-1.5,1.4) to [out=-90, in=0] (-2,.8)
		     to (-2.5,0.8)
		     (-1.5,-2.6) -- (-1.5,-3)
		     (-1.5,-1.4) to [out=90, in=0] (-2,-.8)
		     to (-2.5,-0.8);	
	\end{tikzpicture}
	\end{center}
		\caption{ The Riemannian $4$-manifold $X$ and $U_+$}
	\label{special-quin}
\end{figure}

\begin{remark}
	The 4-manifold $X$ contains two subspaces which are naturally parametrized as $\R\times Y_0$ and $[2,\infty)\times Y$ and their intersection is $[2,\infty)\times Y_0$. Similarly, there are two subspaces diffeomorphic to $\R\times -Y_0'$ and 
	$(-\infty,-2]\times Y'$, whose intersection is the subspace $(-\infty,-2]\times Y'_0$ of $(-\infty,-2]\times Y'$ which is identified with the subspace $[2,\infty)\times -Y_0'$ of $\R\times -Y_0'$ using the orientation preserving map
	$(t,y) \to (-t,y)$ with $(t,y)\in (-\infty,-2]\times Y'_0$. To avoid confusion in the rest of the paper, 
	we write $\{t\}\times -Y_0'$ for the subspace of $\R\times -Y_0'$ with $t\in \R$, and $\{t\}\times Y'$ for the subspace of $(-\infty,-2]\times Y'$ with $t\in (-\infty,-2]$. 
	In particular, $\{-t\}\times -Y_0'$ is a subspace of $\{t\}\times Y'$ for any $t\in (-\infty,-2]$.
\end{remark}

Let $S$ be the surface given by the domain $U_+$. Using the notation in \eqref{boundary-S}, the boundary components $U_\partial$, $\eta_+$ and $\eta_+'$ of $S$ respectively play the roles of $\gamma$, $\eta_1$ and $\eta_2$. We associate the 3-manifold Lagrangians $L(Y,E)$, $L(Y',E')$ to the boundary components $\eta_+$, $\eta_+'$ of $U_+$.
We write $\mathbb L$ for these two Lagrangians together. The quintuple $\fq_{s}:=(X,V,S,M,\mathbb L)$ is called the {\it special quintuple associated to $(Y,E)$ and $(Y',E')$}. The subspaces $(-\infty,-3]\times Y_\#$ and $[3,\infty)\times[-1,1]$ are respectively 
called the gauge theoretic and symplectic ends of $X$ and $S$. Moreover, the subspaces $[2,\infty)\times Y\subset X$ and $[0,2]\times [2,\infty) \subset S$ together are called the mixed end associated to $(Y,E)$. The mixed end associated to $(Y',E')$ is defined in an analogous way.

We shall need a slight modification of the mixed equation in \eqref{mixed-eq-pre} associated to the special quintuple $\fq_{s}$. First, we fix a family of compatible almost complex structures $\{J_{(s,\theta)}\}_{(s,\theta)\in U_+}$ on the symplectic manifold $\mathcal M(\Sigma,F)$ such that $J_{(s,\theta)}$ is equal to the standard complex structure $J_*$ for $s\leq 1$ and is equal to $J_\theta$, the complex structure given by Lemma \ref{alm-cx-str} for $s\geq 1$. Moreover, $J_{s,\theta}$ is constant in the $\theta$ direction if $|\theta|>2$. For a connection $A$ on $V$ and a map $u:U_+\to \mathcal M(\Sigma,F)$, which satisfy matching and Lagrangian boundary conditions, we define the mixed equation as
\begin{equation}\label{mixed-eq}
	\left\{
	\begin{array}{l}
		F^+_A+(*_3\nabla_{A_t}h)^++(*_3\nabla_{A_t'}h')^+=0\\
		\frac{\partial u}{\partial s}+ J_{(s,\theta)}\frac{\partial u}{\partial \theta}=0 
	\end{array}
	\right.
\end{equation}
Here the self-dual part of the 2-forms in the first equation are defined with respect to the Riemannian metric $g_X$. For each $t\in \R$, $A_t$ (resp. $A_t'$) denotes the restriction of $A$ to $\{t\}\times Y_0$ (resp. $\{t\}\times -Y'_0$), and the perturbation terms $(*_3\nabla_{A_t}h)^+$ and $(*_3\nabla_{A_t}h)^+$ are defined as in \eqref{ASD} and are respectively supported in the interior of $\R\times Y_0$ and $\R\times Y'_0$. The Cauchy-Riemann equation in \eqref{mixed-eq} is defined with respect to the family of complex structures $\{J_{(s,\theta)}\}_{(s,\theta)\in U_+}$.

\subsection{Moduli spaces associated to special quintuples and the isomorphism $\bN$}\label{bN-definition}

Working with the space of all solutions of \eqref{mixed-eq} is unmanageable due to non-compactness of $X$ and $S$, and we need to impose some decay conditions on the ends to obtain a well-behaved moduli space. Suppose $(A,u)$ is a pair of a connection on $V$ and a map $U_+\to \mathcal M(\Sigma,F)$.  The {\it analytical energy} of the pair $(A,u)$ is defined as
\begin{equation}\label{analytical-energy}
  \fE(A,u):=\int_{X}\vert F_A+*_3\nabla_{A_t}h+*_3\nabla_{A_t'}h'\vert^2 \dvol_X + \int_{U_+}\vert d u\vert_{J_{(s,\theta)}}^2 ds d\theta.
\end{equation}
where $\vert d u\vert_{J_{(s,\theta)}}^2=\Omega(\frac{\partial u}{\partial s},J_{(s,\theta)}\frac{\partial u}{\partial s})+\Omega(\frac{\partial u}{\partial \theta},J_{(s,\theta)}\frac{\partial u}{\partial \theta})$.
\begin{prop}\label{exp-decay}
	There is a positive real number $\delta_0$ such that the following holds.
	Suppose $(A,u)$ is a pair of an $L^2_{2,loc}$ connection on $V$ and a continuous map $U_+\to \mathcal M(\Sigma,F)$ such that $d u$ belongs to the Sobolev space $L^2_{1,loc}$.
	Suppose $(A,u)$ is a solution of the mixed equation in \eqref{mixed-eq}, satisfies the matching and Lagrangian boundary conditions, and $\fE(A,u)$ is finite.
	Then $u$ is smooth and there is a section $g$ of $V\times_{ad} \SU(2)$ such that $\widetilde A:=g^*A$ is also smooth. Moreover, the following properties for any positive integer $l$ hold.
	\begin{itemize}
		\item[(i)] There is a representative $\alpha$ for an element of $\fC_G$ such that the difference $a:=\widetilde A-\pi_\#^*(\alpha)$, defined on the end $(-\infty,-3]\times Y_\#$, is in $L^2_l$.
		\item[(ii)] There is an element $\beta$ of $\fC_S$ such that $\lim_{s\to \infty} u(s,\theta)= \beta$ and $d u$ is in $L^2_{l-1}$.
		\item[(iii)] On the mixed cylinder associated to $(Y,E)$ (resp. $(Y',E')$), there is an element 
		$q\in L(Y,E)$ (resp. $q'\in L(Y',E')$) with a representative connection $B$ on $E$ (resp. $B'$ on $E'$)
		such that $\widetilde A-\pi^*(B)$ (resp. $\widetilde A-(\pi')^*(B')$) is in 
		$L^2_{l,\delta_0}$, $\lim_{\theta\to \infty} u= q$ (resp. $\lim_{\theta\to -\infty} u= q'$)
		and $d u$ is in $L^2_{l-1,\delta_0}$.
	\end{itemize}
\end{prop}
The weighted Sobolev norms in Theorem \ref{exp-decay} are defined as follow. Let $\tau:[2,\infty)\times Y \to \R^{\geq 0}$ be given by projection to the second factor. Then the $L^2_{l,\delta}$ norm of a function $f$ on $[2,\infty)\times Y$ is the $L^2$ norm of $e^\tau \cdot f$. Similarly,  the $L^2_{l,\delta}$ norm on $(-\infty,-2]\times Y'$ is defined using $\tau':(-\infty,-2]\times Y' \to \R^{\geq 0}$ given by the magnitude of projection to the second factor. These definitions extend to sections of bundles in the obvious way. The $L^2_{l-1,\delta}$ norm of $d u$ over the ends $[0,2]\times [2,\infty)$ and $[0,2]\times (-\infty,-2]$ are also defined in a similar fashion. The proof of the above theorem will be given in Section \ref{reg-comp-exp-dec-sec} based on results of \cite{DFL:mix}.

Theorem \ref{exp-decay} can be used as a guide to define a configuration space where the mixed equation for the special quintuple is defined. Fix $\alpha\in \fC_G$ and $\beta \in \fC_S$. We assume that a connection on $E_\#$ representing $\alpha$ is fixed, and with a slight abuse of notation, we denote this connection by $\alpha$. Let $\bA(\alpha,\beta)$ be the space of all pairs $(A,u)$ which are in $L^2_{l,loc}$, satisfy the matching and Lagrangian conditions, satisfy $(i)$ and $(ii)$ of Theorem \ref{exp-decay} for the given $\alpha$ and $\beta$. Moreover, property $(iii)$ of Theorem \ref{exp-decay} is satisfied for some choice of $q$, $q'$ and their representatives $B$, $B'$ (which might vary from one element of $\bA(\alpha,\beta)$ to another one) with $\delta_0$ being replaced with a positive constant $\delta<\delta_0$, which will be fixed later. In particular, any element of $\bA(\alpha,\beta)$ has finite analytical energy.

Suppose $\mathcal G(V)$ is the space of all sections $g$ of $V\times_{ad} \SU(2)$ such that for an element $(A_0,u_0)\in \bA(\alpha,\beta)$, the 1-form $(\nabla_{A_0} g)g^{-1}$ is in $L^2_{l,loc}$ and its restriction to the end $(-\infty,-3]\times Y_\#$ is in $L^2_l$. Moreover, there are $\fg\in \mathcal G(E)$ and $\fg'\in \mathcal G(E')$ such that the 1-forms $(\nabla_{A_0} g)g^{-1}-(\nabla_{A_0} \fg)\fg^{-1}$ and $(\nabla_{A_0} g)g^{-1}-(\nabla_{A_0} \fg')\fg'^{-1}$ on $[2,\infty)\times Y$ and $(-\infty,-2]\times Y'$ are in $L^2_{l,\delta}$. Here we regard $\fg$ and $\fg'$ as gauge transformations over $[2,\infty)\times Y$ and $(-\infty,-2]\times Y'$ by pulling them back using the projection maps $\pi$ and $\pi'$. There is an obvious map $\ff: \mathcal G(V)\to \mathcal G(E)\times \mathcal G(E')$. The group $\mathcal G(V)$ acts on $\bA(\alpha,\beta)$ and the quotient space is denoted by $\bB(\alpha,\beta)$. We may use the Sobolev norms to topologize the space $\bB(\alpha,\beta)$ in the obvious way. In particular, if $[A_i,u_i]\in \bB(\alpha,\beta)$ is convergent to $[A_0,u_0]\in \bB(\alpha,\beta)$, then the points $q_i\in L(Y,E)$ and $q_i'\in L(Y',E')$ associated to $[A_i,u_i]$ are convergent to $q_0\in L(Y,E)$ and $q_0'\in L(Y',E')$ associated to $[A_0,u_0]$.

\begin{remark}
	Note that the space $\bB(\alpha,\beta)$ is essentially independent of the choice of a representative for $\alpha$ because any element of $\mathcal G(E_\#)$ and the involution $\iota$ can be extended into $V$.
\end{remark}

The spaces $\bA(\alpha,\beta)$ and $\bB(\alpha,\beta)$ are smooth infinite dimensional spaces. To state this claim in a more precise way, we need to introduce some Banach spaces.
\begin{definition}\label{eighted-sob}
	Let $\tau:X\to \R$ be a smooth function on $X$ whose restrictions to $(-\infty,-3]\times Y_\#$, $[2,\infty)\times Y$ and $(-\infty,-2]\times Y'$ are respectively equal to $0$, projection to the first factor and the magnitude of the projection to the first factor. 
	For a vector bundle $E$ on $X$, the weighted Sobolev space $L^2_{k,\delta}(X,E)$ is defined as the space of sections $s$ of $E$ such that $e^{\tau}s$ is in the Sobolev space $L^2_{k}(X,E)$. For a vector bundle $E$ over $U_+$, the weighted sobolev space 
	$L^2_{k,\delta}(U_+,E)$ is defined in a similar way. Thus, roughly speaking, an element of $L^2_{k,\delta}(U_+,E)$ is in $L^2_k$ and is required to have exponential decay along the ends $[0,2]\times [2,\infty)$ and $[0,2]\times (-\infty,-2]$.
\end{definition}
\begin{definition}\label{E-A-u}	
	Let $(A,u)\in \bA(\alpha,\beta)$ be a mixed pair which is asymptotic to $(B,q)$ and $(B',q')$ on the mixed ends associated to $(Y,E)$ and $(Y',E')$.
	Define $E^k_{(A,u)}$ as the space of all
	\[
	  (\zeta,\nu) \in L^2_{k,loc}(X,\Lambda^1\otimes V)\times L^2_{k,loc}(U_+,u^*T\mathcal M(\Sigma,F))
	\]
	such that 
	\begin{itemize}
		\item[(i)] $\zeta|_{(-\infty,-3]\times Y_\#}$ and $\nu\vert_{[3,\infty)\times [-1,1]}$ have finite $L^2_{k}$ norms.
		\item[(ii)] There are $b\in \mathcal H^1_h(Y;B)$ and $b'\in \mathcal H^1_{h'}(Y';B')$ such that 
		\[\zeta-\pi^*(b)|_{[2,\infty)\times Y}\hspace{1cm}
		\text{and}\hspace{1cm} \zeta-\pi^*(b')|_{(-\infty,-2]\times Y'}\] 
		have finite $L^2_{k,\delta}$ norms. Let $s$ and $s'$ be tangent vectors
		to $\mathcal M(\Sigma,F)$ at the points $q$ and $q'$ given by restriction of $b$ and 
		$b'$ to the boundary. Then 
		\[\nu-\pi^*(s)|_{[0,2]\times [2,\infty)}\hspace{1cm} \text{and} \hspace{1cm} \nu-\pi^*(s')|_{[0,2]\times (-\infty,-2]}\] 
		also have finite $L^2_{k,\delta}$ norms.
		\item[(iii)] $*\zeta\vert_{U_{\partial} \times \Sigma}=0$, $d_{A_\theta}\zeta_{\theta}=0$ and 
		$[\zeta_\theta]=\nu(0,\theta)$ where $A_\theta$ and $\zeta_\theta$ are restrictions of $A$ and $\zeta$ to 
		$\{(0,\theta)\}\times \Sigma\subset X$, and $[\zeta_\theta]$ is the element of $\mathcal H^1(\Sigma;A_\theta)$ represented by $\zeta_\theta$.
		\item[(iv)] $ \nu\vert_{\eta_+}\in u^*TL(Y,E)$, $\nu\vert_{\eta'_+}\in u^*TL(Y',E')$.
	\end{itemize}
\end{definition}

The proof of the following proposition is discussed in Subsection \ref{conf-mixed}.
\begin{prop}\label{conf-space-smooth}
	The space $\bA(\alpha,\beta)$ is a Banach manifold and $\mathcal G(V)$ is a Banach Lie group 
	which acts smoothly on $\bA(\alpha,\beta)$, and the stabilizer of any element of $\bA(\alpha,\beta)$ is $\pm 1$. 
	The quotient space $\bB(\alpha,\beta)$ is also a 
	Banach manifold. 
	Let $(A,u)\in \bA(\alpha,\beta)$ be a mixed pair which is asymptotic to $(B,q)$ and $(B',q')$ on the mixed ends associated to $(Y,E)$ and $(Y',E')$.
	Then the tangent space to the point $[A,u]$ of $\bB(\alpha,\beta)$ can be identified with the 
	kernel of the surjective operator
	\[d_A^*:E^l_{(A,u)}\to L^2_{l-1,\delta}(X,V).\]
\end{prop}

For any element $[A,u]$ of the configuration space $\bB(\alpha,\beta)$, define the {\it topological energy} of $[A,u]$ as
\begin{equation*}\label{top-energy}
	\mathcal E(A,u):=\frac{1}{8\pi^2}\int_X\tr\((F_A+*_3\nabla_{A_t}h+*_3\nabla_{A_t'}h')\wedge (F_A+*_3\nabla_{A_t}h+*_3\nabla_{A_t'}h')\)
	+\frac{1}{4\pi^2}\int_{U_+}u^*\Omega.
\end{equation*}
Note that if $[A,u]$ satisfy the mixed equation in \eqref{mixed-eq}, then $\fE(A,u)=8\pi^2\mathcal E(A,u)$. Thus, the topological energy is non-negative for the solutions of \eqref{mixed-eq}. As it is justified by the following lemma, topological energy can be regarded as a soft variation of analytical energy. 

\begin{lemma}\label{top-energy-constant}
	The topological energy $\mathcal E(A,u)$ depends only on the connected component of $\bB(\alpha,\beta)$ that
	contains $(A,u)$.
\end{lemma}
\begin{proof}
	It suffices to show that for a $1$-parameter family $(A(s),u(s))$ of elements of $\bA(\alpha,\beta)$, depending smoothly on $s$, $\mathcal E(A(s),u(s))$
	is constant with respect to $s$. Since this is equivalent to vanishing of $\frac{d\mathcal E(A(s),u(s))}{ds}$,
	the claim follows if the expression 
	\begin{align}\label{top-energy-change}
	  \int_X\tr\((F_A+*_3\nabla_{A_t}h+*_3\nabla_{A_t'}h')\wedge (d_A\zeta+*_3{\rm Hess}_{A_t}h(\zeta_t)+\right.&\left.*_3{\rm Hess}_{A_t'}h'(\zeta_t')\)+ \hspace{2cm}\nonumber\\ 
	  &\hspace{1cm}+\int_{U_+}d \iota_\nu u^*\Omega
	\end{align}
	vanishes for any $(A,u)\in \bA(\alpha,\beta)$ and $(\zeta,\nu) \in E^l_{(A,u)}$. Here
	$\zeta_t$ and $\zeta_t'$ are respectively restrictions of $\zeta$ to $\{t\}\times Y_0$ and $\{t\}\times -Y'_0$, and 
	$d \iota_\nu u^*\Omega$ is the exterior derivative of the 1-form
	\[
	  \Omega(\nu,\frac{\partial u}{\partial s})ds+\Omega(\nu,\frac{\partial u}{\partial \theta})d\theta.
	\]
	Without loss of generality,  in the following we may assume that the restrictions of 
	$A$ to $\R\times Y_0$ and $\R\times Y_0'$ are in temporal gauge.

	We start by analyzing the first integral of \eqref{top-energy-change} over the sub-manifold $\R\times Y_0$ of $X$. 
	Note that $h'$ vanishes on this space.
	Therefore, the integrand over  $\R\times Y_0$ can be simplified to
	\[
	  \tr\(F_A \wedge d_A\zeta+ dt\wedge \frac{dA_t}{dt}\wedge *_3{\rm Hess}_{A_t}h(\zeta_t)+
	  *_3\nabla_{A_t}h\wedge dt \wedge \frac{d\zeta_t}{dt}\).
	\]
	By applying Stokes theorem, Bianchi identity and Lemma \ref{H-prop}, 
	the integral of the above expression over $\R\times Y_0$ is equal to
	\begin{align*}
	  \int_{\R\times Y_0} d\tr(F_A \wedge \zeta&)+\int_\R dt \(\int_{Y_0} \tr(*_3{\rm Hess}_{A_t}h( \frac{dA_t}{dt})\wedge \zeta_t+
	  *_3\nabla_{A_t}h \wedge \frac{d\zeta_t}{dt})\)=\hspace{3cm} \\
	  \hspace{2.5cm}& =\int_{\R\times \partial Y_0} \tr(F_A \wedge \zeta)+\lim_{t\to \infty}\int_{Y_0} \tr(F_{A_t} \wedge \zeta_t)+\int_\R dt \frac{d}{dt} \(\int_{Y_0} 
	  \tr(*_3\nabla_{A_t}h \wedge \zeta_t)\)\\
	  &=\int_{\R\times \partial Y_0} \tr(F_A \wedge \zeta)+\lim_{t\to \infty}\int_{Y_0} \tr((F_{A_t}+*_3\nabla_{A_t}h) \wedge \zeta_t).
	\end{align*}
	Note that we did not include the integrals of $\tr(F_{A_t} \wedge \zeta_t)$ and $\tr(*_3\nabla_{A_t}h \wedge \zeta_t)$ over $Y_0$ as $t\to -\infty$ in the second and the third identities because of the decay of $\zeta_t$ on the gauge theoretical end. 
	As $t\to \infty$, $A_t$ and $\zeta_t$ are convergent to $B$ and an element of $\mathcal H^1_h(Y;B)$. In particular, the integral of $\tr((F_{A_t}+*_3\nabla_{A_t}h) \wedge \zeta_t)$ over $Y_0$ as $t\to \infty$ is trivial.
	Consequently, the contribution of $ \R \times Y_0$ to \eqref{top-energy-change} equals the integral of $\tr(F_A \wedge \zeta)$ over 
	$ \R \times \partial Y_0$. A similar claim holds about $ \R \times -Y'_0$.

	The first integrand of \eqref{top-energy-change} over $U_-\times \Sigma$ simplifies to $\tr(F_A\wedge d_A\zeta)$. Thus, by Stokes theorem  
	and Bianchi identity this integral is equal to:
	\[
	  \int_{\partial (U_-\times \Sigma)} \tr(F_A\wedge \zeta).
	\]
	Assuming that the restriction of $A$ to $U_\partial\times \Sigma$ has the from $A_\theta+\phi ds+\psi d\theta$, we can summarize our simplifications as 
	\begin{align}
	   \int_X\tr\left((F_A+*_3\nabla_{A_t}h+*_3\nabla_{A_t'}h')\wedge \right. &(d_A\zeta+*_3{\rm Hess}_{A_t}h(\zeta_t)+
	  \left.*_3{\rm Hess}_{A_t'}h'(\zeta_t')\right)=
	    \int_{U_\partial \times \Sigma} \tr(F_A\wedge \zeta)\nonumber\\
	    &=\int_{-\infty}^\infty d\theta \(\int_\Sigma \tr(\partial_\theta {A_\theta} \wedge \zeta_\theta)-\tr(d_{A_\theta} \psi \wedge \zeta_\theta)\)\nonumber\\
	    &=\int_{-\infty}^\infty d\theta \int_\Sigma \tr(\partial_\theta {A_\theta} \wedge \zeta_\theta)\label{gauge-theory-exp}
	\end{align}
	The second identity is a consequence of the Stokes theorem and the assumption $d_{A_\theta} \zeta_\theta=0$. Another application of the Stokes theorem also shows that the second integral in \eqref{top-energy-change} can be simplified to
	\begin{equation}\label{symp-theory-exp}
	   -\int_{U_\partial}\iota_\nu u^*\Omega=-\int_{-\infty}^\infty  d\theta \int_{\Sigma}\tr(\frac{\partial u}{\partial \theta}\wedge \nu(0,\theta)).
	\end{equation}
	Using the matching conditions for $(A,u)$ and $(\zeta,\nu)$, the expressions \eqref{gauge-theory-exp} and \eqref{symp-theory-exp} cancel out each other and \eqref{top-energy-change} vanishes.
\end{proof}

In fact, Lemma \ref{top-energy-constant} can be strengthened as follows. The proof of the this lemma will be given in Subsection \ref{ind-mixed-op-subs}.
\begin{lemma}\label{top-energy-constant-reverse}
	If $[A,u],[A',u']\in \bB(\alpha,\beta)$, then $2(\mathcal E(A,u)-\mathcal E(A',u'))$ is an integer. Moreover, if $\mathcal E(A,u)=\mathcal E(A',u')$,
	then $[A,u]$ and $[A',u']$ belong to the same connected component of $\bB(\alpha,\beta)$.
\end{lemma}

Let $\bM(\alpha,\beta)$ be the subspace of $\bB(\alpha,\beta)$ given by the solutions of \eqref{mixed-eq}. The local behavior of this moduli space around a solution $(A,u)$ is governed by the {\it linearization} of the mixed equation. Define a linear operator 
\[
  L_{A,u}:E^l_{(A,u)} \to L^2_{l-1,\delta}(X,\Lambda^+\otimes V)\oplus L^2_{l-1,\delta}(U_+,u^*T\mathcal M(\Sigma,F)),
\]
as 
\begin{equation}\label{Linearized}
	L_{(A,u)}(\zeta,\nu):=(d_A^+\zeta+(*_3{\rm Hess}_{A_t}h(\zeta_t))^+
	+(*_3{\rm Hess}_{A_t'}h'(\zeta_t'))^+,\mathcal D_u(\nu)), 
\end{equation}	
where $\zeta_t$ (resp. $\zeta_t'$) is the restriction of $\zeta$ to $\{t\} \times Y_0$ (resp. $\{t\} \times -Y_0'$), and analogous to  \eqref{linear-CR}, $\mathcal D_u$ is the linearization of the Cauchy-Riemann operator 
\[
  \mathcal D_u(\nu):=\nabla_\theta \nu-J_{s,\theta}(u)\nabla_{s}\nu-(\nabla_\nu J_{s,\theta})\frac{du}{ds},
\]
with $\nabla$ being the Levi-Civita connection again. Using Proposition \ref{conf-space-smooth}, the linearization of the mixed equation is given by the restriction of $L_{(A,u)}$ to the kernel of the operator $d_A^*$. It is more convenient to combine these operators 
and define the {\it mixed operator}
\[\mathcal D_{(A,u)}:E^l_{(A,u)} \to L^2_{l-1,\delta}(X,V)\oplus L^2_{l-1,\delta}(X,\Lambda^+\otimes V)\oplus L^2_{l-1,\delta}(U_+,u^*T\mathcal M(\Sigma,F))\]
as $\mathcal D_{(A,u)}:=(d_A^*,L_{(A,u)})$. Since the operator $d_A^*$ is surjective, the kernels and co-kernels of the linearization of the mixed equation and $\mathcal D_{(A,u)}$ can be identified with each other.
\begin{definition}\label{reg}
	An element $[A,u]\in \bB(\alpha,\beta)$ is {\it regular} if $\mathcal D_{(A,u)}$ is surjective.
\end{definition}

The proof of the following proposition will be given in Subsection \ref{fred-adjoint}:
\begin{prop}\label{fred-prop-main-thm}
	After possibly decreasing the constant $\delta_0$ of Theorem \ref{exp-decay}, the following claim holds for any $\delta<\delta_0$. Suppose $\bA(\alpha,\beta)$ is defined using $\delta$ and $[A,u]$ is a smooth element of $\bA(\alpha,\beta)$ 
	that satisfies property (iii) of Theorem \ref{exp-decay}. 
	Then the operator $\mathcal D_{(A,u)}$ is Fredholm. 
	If $(A,u)$ represents a regular point of the mixed moduli space $\bM(\alpha,\beta)$, then $\bM(\alpha,\beta)$ 
	is a smooth manifold of dimension $\ind(\mathcal D_{(A,u)})$ in a neighborhood of $[A,u]$.
\end{prop}
\noindent
For the rest of this section, we assume that the constant $\delta$ used in the definition of $\bA(\alpha,\beta)$ is given by Proposition \ref{fred-prop-main-thm}.

To each $\alpha\in  \fC_S\cong \fC_G$, we can associate an element of the moduli space ${\bf M}(\alpha,\alpha)$ as we explain now. Fix a representative connection for $\alpha$ which is also denoted by $\alpha$. By definition there are connections $B$ and $B'$ on $(Y,E)$ and $(Y',E')$ which satisfy the equations $\phi_h(B)=0$ and $\phi_{h'}(B')=0$ of Subsection \ref{3-man-lag-sec}, and the restrictions of $B$ and $B'$ to collar neighborhoods of the boundaries of $Y$ and $Y'$ are given by the pullbacks of $\alpha$. Then the pullbacks of $B$ to $\R\times Y_0$, $B'$ to $\R\times Y_0'$ and $\alpha$ to $\Sigma\times U_-$ give rise to a connection $A_\alpha$ on $X$. We also define $u_\alpha:U_+\to \mathcal M(\Sigma,F)$ to be the constant map given by $\alpha$. The pair $(A_\alpha,u_\alpha)$, which is called the {\it constant pair}, clearly satisfies the mixed equation in \eqref{mixed-eq} and hence it represents an element of the moduli space ${\bf M}(\alpha,\alpha)$. Notice that the topological energy of any constant pair is zero. In fact, constant pairs are characterized as solutions of \eqref{mixed-eq} with vanishing topological energy (or equivalently analytical energy). The proof of the following proposition will be given in Subsection \ref{cons-map-ind-subs}.

\begin{prop}\label{cons-map-ind}
	The index of the mixed operator $\mathcal D_{(A_\alpha,u_\alpha)}$ is $0$. Moreover, the kernel 
	and the cokernel of the operator $\mathcal D_{(A_\alpha,u_\alpha)}$ are trivial.
\end{prop}

The following proposition generalizes the index computation of Proposition \ref{cons-map-ind} to the case of arbitrary mixed pairs.
\begin{prop} \label{ind-mixed-op}
	The index of the mixed  operator associated to a mixed pair $(A,u)\in \bA(\alpha,\beta)$ satisfies
	\begin{equation}\label{ind-mixed-op-formula}
	  \ind(\mathcal D_{(A,u)})=8\mathcal E(A,u)+\epsilon(\beta)-\epsilon(\alpha),
	\end{equation}
	where the constants $\epsilon(\alpha)$ and $\epsilon (\beta)$ are given by Proposition \ref{top-energy-index-gauge}.
\end{prop}

The proof of the above proposition will be given in Subsection \ref{ind-mixed-op-subs}. Notice that the first part of Proposition \ref{cons-map-ind} is a special case of this proposition. However, the proof of Proposition \ref{ind-mixed-op} relies on Proposition \ref{cons-map-ind} as an essential input. The other input is the {\it mixed shifting operator}, which is introduced in Subsection \ref{mixed-shit-op}.

For any integer $d$, let $\bM(\alpha,\beta)_d$ denote the subspace of $\bM(\alpha,\beta)$ consisting of index $d$ solutions. 
Suppose $(A,u)$ represents an element of $\bM(\alpha,\beta)_0$. Since the topological energy of $(A,u)$ is non-negative, the index formula implies that $\epsilon(\alpha)\geq \epsilon(\beta)$, and the equality holds if and only if $\alpha=\beta$. The latter claim holds because any element of $\bM(\alpha,\beta)$ with vanishing topological energy is a constant pair. In summary, $\bM(\alpha,\beta)_0$ is non-empty only if $\epsilon(\alpha)> \epsilon(\beta)$ or $\alpha=\beta$. In the latter case, there is exactly one element in $\bM(\alpha,\alpha)_0$ which is regular. 

We have shown that the constant solutions of the mixed solution are regular. However, not all elements of $\bM(\alpha,\beta)$ are regular. In Section \ref{pert-sect}, we introduce perturbations  of mixed equation by deforming the family of almost complex structures $\{J_{(s,\theta)}\}_{(s,\theta)\in U_+}$ and adding a term to the ASD equation:
\begin{equation}\label{mixed-eq-pert}
	\left\{
	\begin{array}{l}
		F^+_A+(*_3\nabla_{A_t}h)^++(*_3\nabla_{A_t'}h')^++\eta(A)=0\\
		\frac{\partial u}{\partial s}+ J_{(s,\theta)}\frac{\partial u}{\partial \theta}=0
	\end{array}
	\right.
\end{equation}
Here the term $\eta(A)$ is invariant with respect to the action of $\mathcal G(V)$. Thus, the solutions of the above equations determine a subspace of $\bB(\alpha,\beta)$, denoted by $\bM_{\eta}(\alpha,\beta)$. 

\begin{prop}\label{pert}
	There are secondary perturbations of the mixed equation satisfying the following properties.
	\begin{itemize}
		\item[(i)] There is a compact subset $K_-\subset U_-$ away from 
		the matching line $U_\partial$ such that $\eta(A)$ 
		depends on $A|_{K_-}$ and is supported in $K_-$. 
		There is a compact subset $K_+\subset U_+$ away from 
		the matching line $U_\partial$ such that the deformation of $J_{(s,\theta)}$ is trivial on the complement of $K_+$.
		In particular, the deformed almost complex structure agrees with the standard complex structure $J_*$ in a neighborhood of the matching line.
		\item[(ii)] The moduli spaces with expected dimension at most $3$ are regular.
		\item[(iii)] Any element of $\bM_{\eta}(\alpha,\beta)$ has non-negative topological energy. The moduli space $\bM_{\eta}(\alpha,\beta)_0$ is non-empty only if $\epsilon(\alpha)>\epsilon(\beta)$
		or $\alpha=\beta$. Moreover, $\bM_{\eta}(\alpha,\alpha)_0$ consists of only one element for each $\alpha\in \fC_S$.
	\end{itemize}
\end{prop}

This proposition is proved in Section \ref{pert-sect} after introducing an appropriate family of perturbation terms $\eta$. For any solution $(A,u)$ of \eqref{mixed-eq-pert} we have
\begin{equation}\label{rel-mixed-top-ana-energy}
  \fE(A,u)=8\pi^2\mathcal E(A,u)+2\vert\!\vert\eta(A)\vert\!\vert_{L^2(X)}^2.
\end{equation}

\begin{prop}\label{orientation}
	Let $\eta$ be given by Proposition \ref{pert}. Then the moduli spaces 
	$\bM_\eta(\alpha,\beta)_d$ with $d\leq 3$ are orientable $d$-dimensional manifolds.
\end{prop}

\begin{prop}\label{comp}
	The perturbation $\eta$ in Proposition \ref{pert} can be chosen such that the following holds.
	\begin{itemize}
		\item[(i)] The moduli spaces of the form $\bM_\eta(\alpha,\beta)_0$ are compact.
		\item[(ii)] The moduli spaces of the form $\bM_\eta(\alpha,\beta)_1$ 
		can be compactified into compact 1-manifolds by
		adding points in correspondence to the $0$-dimensional spaces
		\begin{equation}\label{ends-1dim}
		\bM_\eta(\alpha,\gamma)_0\times\breve \rM_{S}(\gamma,\beta)_p,\hspace{1cm}
		\breve \rM_{G}(\alpha, \gamma)_p \times \bM_\eta(\gamma,\beta)_0,
		\end{equation}
		where $\gamma\in \fC_G\cong \fC_S$, and in both cases $p$ denotes a path of index $1$.
	\end{itemize}
	Moreover, the orientations of the moduli spaces $\bM_\eta(\alpha,\beta)_d$ with $d\leq 1$ provided
	by Proposition \ref{orientation} can be chosen such that the induced orientation on the boundary components
	of the compactified moduli space
	$\bM_\eta(\alpha,\beta)_1$  (using outward-normal-first convention) agree with the product orientation on the first term in \eqref{ends-1dim}
	and disagrees with the induced orientation on the second term in \eqref{ends-1dim}.
\end{prop}

The essential step in the orientability of mixed moduli spaces is discussed in Subsection \ref{subsection-ori}. The proof of the compactness claims in Proposition \ref{comp} uses results of \cite{DFL:mix}, and is given in Subsection \ref{comp-sect}. The rest of the above two propositions is verified in Section \ref{pert-sect}.

We define a homomorphism $\bN:C_{G}(Y_\#,E_\#) \to C_S((Y,E),(Y',E'))$ using $0$-dimensional moduli spaces $\bM_\eta(\alpha,\beta)_0$. First we pick the perturbation $\eta$ such that the claims of Propositions \ref{pert} and \ref{comp} hold. Then define
\begin{equation}
	\bN(\alpha):=\sum_{\beta\in \fC_S}\#\bM_\eta(\alpha,\beta)_0 \cdot \beta,
\end{equation}
where $\#\bM_\eta(\alpha,\beta)_0$ denotes the signed count of the points in the $0$-dimensional moduli space $\bM_\eta(\alpha,\beta)_0$. Our main theorem is a consequence of the following result.

\begin{theorem} \label{main-thm-detailed}
	The map $\bN$ is an isomorphism and a chain map.
\end{theorem}
\begin{proof}
	By Proposition \ref{pert}, $\bM_\eta(\alpha,\beta)_0$ is non-empty only if $\epsilon(\alpha)>\epsilon(\beta)$ or 
	$\alpha=\beta$. In the latter case, $\bM_\eta(\alpha,\beta)_0$ consists of only one element. Thus $\bN$ 
	is an isomorphism. The chain map property of $\bN$ follows from a standard argument using the second part of 
	Proposition \ref{comp}.
\end{proof}

\section{Linear analysis}\label{lin-analysis}
In this section, we verify several claims in Sections \ref{HF} and \ref{main-thm-sec} related to the linear analysis of the mixed equation. During this section $(Y,E)$ and $(Y',E')$ are fixed as in Subsection \ref{3-man-bdles}, and we fix Lagrangian 3-manifolds associated to these pairs that have transversal intersection and the claim of Lemma \ref{reg-ASD} holds. We continue to drop $h$ and $h'$ from our notations for the 3-manifolds Lagrangians, and denote them by $L(Y,E)$ and $L(Y',E')$.

\subsection{The configuration space of mixed pairs}\label{conf-mixed}

In Subsection \ref{flat-surface}, we introduced the space of connections $\mathcal A(\Sigma,F)$, which is an affine space modeled on the Banach space $\fB:=L^2_{l-1}(\Sigma,\Lambda^1\otimes F)$. For any positive constant $\epsilon$, we write $\fB_{<\epsilon}$ for the subspace of elements of $\fB$ with $L^2$ norm less than $\epsilon$. We write $\mathcal A_{\rm fl}(\Sigma,F)$ for the subspace of $\mathcal A(\Sigma,F)$ given by flat connections. If we want to be specific about the Sobolev exponent in the definition of $\mathcal A(\Sigma,F)$ and $\mathcal A_{\rm fl}(\Sigma,F)$, we denote them by $\mathcal A^{l-1}(\Sigma,F)$ and $\mathcal A^{l-1}_{\rm fl}(\Sigma,F)$. The space of $L^2_l$ gauge transformations of $F$ are also denoted by $\mathcal G_l(F)$. The following lemma provides an exponential map for the tangent vectors of $\mathcal A(\Sigma,F)$, which is invariant with respect to $\mathcal G(F)$ and induces an exponential map on $\mathcal A_{\rm fl}(\Sigma,F)$.
\begin{lemma}\label{exp-flat-matching}
	There are a positive constant $\epsilon$ and a smooth map 
	$\rE:\mathcal A(\Sigma,F)\times \fB_{<\epsilon}
	\to \mathcal A(\Sigma,F)$ 
	satisfying the following properties.
	\vspace{-8pt}
	\begin{itemize}
		\item[(i)] $\rE$ is $\mathcal G(F)$-equivariant where we use the diagonal action on $\mathcal A(\Sigma,F)\times \fB_{<\epsilon}$.
		\item[(ii)] $\rE(\sigma,0)=\sigma$.
		\item[(iii)] For any $\sigma\in \mathcal A(\Sigma,F)$, the differential $\left.D_{(\sigma,0)}\rE\right\vert_{\{0\}\times \fB}:\fB\to \fB$ is identity. The map $\rE$ determines a diffeomorphism
		from $\{\sigma\}\times \fB_{<\epsilon}$ to a neighborhood of $\sigma$. 
		\item[(iv)] For any $\alpha \in \mathcal A_{\rm fl}(\Sigma,F)$
		and $c \in \fB_{<\epsilon}$ with $d_\alpha c=0$, the connection 
		$\rE(\sigma,c)$ is flat. Moreover, if $c\in \mathcal H^1(\Sigma;\alpha)$ and $c':=c+d_{\alpha}\zeta$ with $\zeta \in L^2_{l}(\Sigma,F)$,
		then $\rE(\alpha,c'):=g^*\rE(\alpha,c)$ where $g$ is obtained from exponentiating $\zeta$.
	\end{itemize}
\end{lemma}

We may restrict a map $\rE$ provided by Lemma \ref{exp-flat-matching} to the configuration of flat connections $\mathcal A_{\rm fl}(\Sigma,F)$, and then use gauge equivariance of $\rE$ to obtain a map from  $T_{\epsilon}\mathcal M(\Sigma,F)$ to $\mathcal M(\Sigma,F)$, where $T_{\epsilon}\mathcal M(\Sigma,F)$ denotes tangent vectors to $\mathcal M(\Sigma,F)$ with length at most $\epsilon$. In fact, it is useful to fix one such map before constructing $\rE$. To do this, let $\re:T\mathcal M(\Sigma,F) \to \mathcal M(\Sigma,F)$ be the exponential map with respect to the chosen metric on $\mathcal M(\Sigma,F)$. 

\begin{lemma}\label{almost-flat-nbhd}
	There is a constant $\kappa$ such that the following holds. Suppose $U_\kappa$
	denotes the subspace of $L^2_1$ connections $\sigma$ on $F$ with $|\!|F_\sigma|\!|_{L^2}<\kappa$.
	Then there is a smooth $\mathcal G(F)$-equivariant map 
	\[
	  P:U_\kappa \to \mathcal A_{\rm fl}^1(\Sigma,F)\times L^2_{2}(\Sigma,F)
	\]
	such that if $P(\sigma)=(\alpha,\zeta)$, then $\sigma=\alpha+*d_\alpha\zeta$. Moreover, if $\sigma$
	is in $L^2_k$ for $k\geq 1$, then $\alpha\in L^2_{k}$ and $\zeta\in L^2_{k+1}$.
\end{lemma}

\begin{proof}
	Consider the $\mathcal G(\Sigma,F)$-equivariant map
	\[
	  \Psi:\mathcal A^1_{\rm fl}(\Sigma,F)\times L^2_{2}(\Sigma,F)\to \mathcal A^1(\Sigma,F)
	\]
	given by $\Psi(\alpha,\zeta)=\alpha+*d_\alpha\zeta$. Inverse function theorem and 
	Uhlenbeck compactness theorem imply that there are $\kappa>0$ and a 
	$\mathcal G_2(F)$-invariant neighborhood $V$ of $\mathcal A^1_{\rm fl}(\Sigma,F)\times\{0\}$
	such that $\Psi$ induces a diffeomorphism from $V$ to $U_\kappa$.
	Then we define $P:U_\kappa \to \mathcal A_{\rm fl}^1(\Sigma,F)\times L^2_{2}(\Sigma,F)$ to be the inverse of 
	this map. Now suppose $\sigma\in U_\kappa$ is an $L^2_k$ connection with $k\geq 2$ and 
	$P(\sigma)=(\alpha,\zeta)$.
	There is an $L^2_2$ automorphism $g$ of $F$ such that $\alpha'=g^*\alpha$ is a smooth 
	flat connection.
	Moreover, $F_{g^*\sigma}=gF_\sigma g^{-1}$ is in $L^2_1$, and if $\zeta':=g^*(\zeta)$, then
	\[d_{\alpha'}*d_{\alpha'}\zeta'=-*d_{\alpha'}\zeta'\wedge *d_{\alpha'}\zeta'+F_{g^*\sigma}.\]
	By applying elliptic regularity for the Laplacian operator $d_{\alpha'}*d_{\alpha'}$ twice, we may conclude that $\zeta'$ is in $L^2_{3}$
	and hence $g^*\sigma$ is in $L^2_2$. This implies that $g$ is in fact an $L^2_{3}$ gauge transformation 
	of $F$, $\alpha$ is an $L^2_2$ flat connection and $\zeta$ is in $L^2_3$. 
	Iterations of the above argument shows that $g$ is in fact an $L^2_{k+1}$ gauge transformation,
	$\alpha$ is an $L^2_k$ flat connection and $\zeta$ is in $L^2_{k+1}$. 
\end{proof}

\begin{proof}[Proof of Lemma \ref{exp-flat-matching}]
	First we define $\rE(\alpha,c)$ in the case that $\alpha$ belongs to $\mathcal A_{\rm fl}(\Sigma,F)$. 
	The  1-form $c$ can be uniquely decomposed as 
	\[
	  c=c_0+d_\alpha \zeta+*d_\alpha \zeta'
	\]
	with $c_0\in \mathcal H^1(\Sigma;\alpha)$, $\zeta,\zeta'\in L^2_{l}(\Sigma,F)$. 
	The 1-form $c_0$ determines an element of $T_{[\alpha]}\mathcal M(\Sigma,F)$ and 
	\[
	  \gamma(t):=\re([\alpha],tc)\in \mathcal M(\Sigma,F)
	\]
	defines a path from $[\alpha]$ to $\re([\alpha],c)$. 
	Let $\widetilde \gamma:[0,1]\to \mathcal A_{\rm fl}(\Sigma,F)$ be the unique path satisfying 
	\begin{itemize}
		\item[(i)] $\widetilde \gamma(t)$ is a flat connection representing $\gamma(t)$;
		\item[(ii)] $d_{\widetilde \gamma(t)}^*\left(\frac{d}{dt}\widetilde \gamma(t)\right)=0$. 		
	\end{itemize}
	Let also $g$ be the gauge transformation in $\mathcal G (F)$ given by exponentiating 
	$\zeta$. Then we define \[\rE(\alpha,c)=g^*\widetilde \gamma(1)+*d_\alpha \zeta'.\]
	Thus, we obtain a map $\rE(\alpha,c)$, for flat $\alpha$, that satisfies properties (i)-(iv).
	Compactness of $\mathcal M(\Sigma,F)$
	and the inverse function theorem can be used to find $\epsilon$ such that
	for any $\alpha\in \mathcal A_{\rm fl}(\Sigma,F)$ and $t\in [0,1]$ the map 
	\[
	  c\to \alpha+c+t(\rE(\alpha,c)-\alpha-c)
	\]
	sends $\fB_{<\epsilon}$ to a neighborhood of $\alpha\in \mathcal A(\Sigma,F)$ by a diffeomorphism.

	Next, we extend $\rE(\sigma,c)$ to the case that $\sigma$ is an arbitrary element of 
	$\mathcal A(\Sigma,F)$. Suppose $\tau:[0,1]\to [0,1]$ is a smooth function, which is equal to 
	$1$ in a neighborhood of $0$ and evaluates to $0$ in a neighborhood of $1$. Let $U_\kappa$ be given by Lemma \ref{almost-flat-nbhd}. Suppose $\sigma\in U_\kappa$
	and $P(\sigma)=(\alpha,\zeta)$. For any $c\in \fB_{<\epsilon}$ define
	\begin{equation}\label{rE}
		\rE(\sigma,c)=\sigma+c+
		\tau(\kappa^{-1}\vert\!\vert F(\sigma)\vert\!\vert_{L^2})(\rE(\alpha,c)-\alpha-c).
	\end{equation}
	and extend \eqref{rE} to the case that $\sigma\in \mathcal A(\Sigma,F)\setminus U_\kappa$ as
	$\rE(\sigma,c)=\sigma+c$.	
\end{proof}

Next, we need deformations of $\re$, which are well-behaved with respect to $L(Y,E)$ and $L(Y',E')$.

\begin{lemma}
	For any $-1\leq s\leq 1$, there is a smooth map 
	$\re_s:T\mathcal M(\Sigma,F) \to \mathcal M(\Sigma,F)$, 
	depending smoothly on $s$,
	such that the following properties hold.
	\vspace{-8pt}
	\begin{itemize}
		\item[(i)] $\re_0=\re$.
		\item[(ii)] $\re_s$ maps the zero section of $T\mathcal M(\Sigma,F)$ to 
		$\mathcal M(\Sigma,F)$ by the identity map
		\item[(iii)] The derivative of $\re_s$ at any point $x$ in the zero section 
		and along the fiber $T_x\mathcal M(\Sigma,F)$ is given by the identity map. 
		\item[(iv)] $\re_{1}$ maps the subspace $TL(Y,E)$ of $T\mathcal M(\Sigma,F)$
		to $L(Y,E)$ and $\re_{-1}$ maps the subspace $TL(Y',E')$ of 
		$T\mathcal M(\Sigma,F)$ to $L(Y',E')$.
	\end{itemize}
\end{lemma}

By applying the argument of Lemma \ref{exp-flat-matching} to the family of maps $\re_s$ provided by the above lemma, we may construct a family of maps $\rE_s:\mathcal A(\Sigma,F)\times \fB_{<\epsilon}\to \mathcal A(\Sigma,F)$, which satisfies the properties (i)-(iv) of Lemma \ref{exp-flat-matching}. Moreover, if $\alpha$ is a flat connection on $F$ representing an element of $L(Y,E)$ and $c$ is a $d_\alpha$-closed $1$-form representing a tangent vector to $L(Y,E)$, then $\rE_1(\sigma,c)$ also represents an element in $L(Y,E)$. The map $\rE_{-1}$ has a similar property with respect to $L(Y',E')$.

\begin{proof}
	Fix a metric on $\mathcal M(\Sigma,F)$ such that $L(Y,E)$ is totally geodesic with respect to this metric.
	Then the exponential map with respect to this metric gives $\re_1$. A homotopy from this metric 
	and the standard Riemannian metric on $\mathcal M(\Sigma,F)$ can be used in a similar way to define the maps $\re_t$ for $t\in [0,1]$.
	The maps $\re_t$ for $t\in [-1,0]$ can be constructed in an analogous way.
\end{proof}

Our next goal is to give a chart for a neighborhood of a mixed pair $(A,u) \in \bA(\alpha,\beta)$. Before giving a description of such a chart, we need to fix another additional piece of data. The mixed pair $(A,u)$ is convergent to pairs $(B,q)$ and $(B',q')$ as $\theta\to \infty$ and $\theta\to -\infty$, where $q\in L(Y,E)$, $q'\in L(Y',E')$, and $B$, $B'$ are respectively connections on $E$, $E'$ representing $q$, $q'$. The restrictions of $B$, $B'$ to $\Sigma$ is denoted by $\alpha$, $\alpha'$. Let $b\in \mathcal H^1_{h}(Y;B)$ and $c$ denote the restriction of $b$ to $\Sigma$. Then $\alpha_b:=\rE_1(\alpha,c)$ is a flat connection on $F$ which represents an element of $L(Y,E)\subset \mathcal M(\Sigma,F)$. After possibly decreasing $\epsilon$, we fix a connection $B_b$, for $|b|<\epsilon$, such that
\vspace{-8pt}
\begin{itemize}
	\item[(i)] $B_b$ depends smoothly on $b$;
	\item[(ii)] $B_0=B$;
	\item[(iii)] $B_b$ represents an element of $L(Y,E)$, and its restriction to the boundary is equal to 
	$\alpha_b$.
\end{itemize}
\vspace{-8pt}
Similarly, we fix a smooth family of connections $\{B_{b'}'\}$ for $b'\in \mathcal H^1_{h}(Y';B')$ with $|b'|<\epsilon$.

Suppose $B^l_{(A,u)}$ is the space of all $\zeta \in L^2_{l,loc}(X,\Lambda^1\otimes V)$, $\nu\in L^2_{l,loc}(U_+,u^*T\mathcal M(\Sigma,F))$, which satisfy the following properties.
\vspace{-5pt}
\begin{itemize}
	\item[(i)] $\zeta|_{(-\infty,-3]\times Y_\#}$ and $\nu\vert_{[3,\infty)\times [-1,1]}$ have finite $L^2_{l}$ norms.
	\item[(ii)] There are $b\in \mathcal H^1_h(Y;B)$ and $b'\in \mathcal H^1_{h'}(Y';B')$ such that 
		\[\zeta-\pi^*(b)|_{Y\times [2,\infty)}\hspace{1cm}
		\text{and}\hspace{1cm} \zeta-\pi^*(b')|_{Y'\times (-\infty,-2])}\] 
		have finite $L^2_{l,\delta}$ norms where $\delta$ is a small positive constant, which will be fixed in the next subsection. Let $s$ and $s'$ be tangent vectors
		to $\mathcal M(\Sigma,F)$ at the points $q$ and $q'$ given by restrictions of $b$ and 
		$b'$ to the boundary. Then 
		\[\nu-\pi^*(s)|_{[0,2]\times [2,\infty)}\hspace{1cm} \text{and} \hspace{1cm} \nu-\pi^*(s')|_{[0,2]\times (-\infty,-2]}\] 
		also have finite $L^2_{l,\delta}$ norms.
	\item[(iii)] $d_{A_\theta}\zeta_{\theta}=0$ and 
		$[\zeta_\theta]=\nu(0,\theta)$ where $A_\theta$ and $\zeta_\theta$ are restrictions of $A$ and $\zeta$ to 
		$\{(0,\theta)\}\times \Sigma\subset X$, and $[\zeta_\theta]$ is the element of $\mathcal H^1(\Sigma;A_\theta)$ represented by $\zeta_\theta$.
	\item[(iv)] $ \nu\vert_{\eta_+}\in u^*TL(Y,E)$, $\nu\vert_{\eta'_+}\in u^*TL(Y',E')$.
\end{itemize}
Then $B^l_{(A,u)}$ is a Banach space where the norm is defined as 
\begin{align}
  \left|(\zeta,\nu)\right|_{B^l_{(A,u)}}:=&
  |\!|\zeta|\!|_{L^2_l(X^\circ)}+|\!|\nu|\!|_{L^2_l(U_+^\circ)}+|\!|\zeta-\pi^*(b)|\!|_{L^2_{l,\delta}(Y\times [2,\infty))}+|\!|\zeta-\pi^*(b')|\!|_{L^2_{l,\delta}(Y'\times (-\infty,-2])}+\nonumber\\
  &+|\!|\nu-\pi^*(s)|\!|_{L^2_{l,\delta}([0,2]\times [2,\infty))}+|\!|\nu-\pi^*(s') |\!|_{L^2_{l,\delta}([0,2]\times (-\infty,-2])}+|s|+|s'|\label{norm-fB}
\end{align}
with 
\[X^\circ:=X\setminus Y\times [2,\infty)\setminus Y'\times (-\infty,-2],\hspace{1cm}U_+^\circ:=U_+\setminus [0,2]\times [2,\infty)\setminus [0,2]\times (-\infty,-2].\]
In the following, we fix a constant $\kappa_0$ such that if $\left|(\zeta,\nu)\right|_{B^l_{(A,u)}}<\kappa_0$, then for any $(s,\theta)\in U_-$, the restriction of $\zeta$ to 
$\{(s,\theta)\}\times \Sigma$ belongs to $\fB_{<\epsilon}$ and for any $(s,\theta)\in U_-$, $\nu(s,\theta)\in T_{\epsilon}\mathcal M(\Sigma,F)$.

A neighborhood of $(A,u)$ in $\bA(\alpha,\beta)$ can be parametrized by the product of a small ball centered at the origin in $B^l_{(A,u)}$ and the group $\mathcal G(E)\times \mathcal G(E')$, as it is explained in the following. First we fix a smooth function $\tau:U_+ \to \R$ which in a neighborhood of $\eta_+$ is equal to $1$, in a neighborhood of $\eta_+'$ is equal to $-1$, on $[0,2]\times [2,\infty)$ is equal to $1$, and on $[0,2]\times (-\infty,-2]$ is equal to $-1$. Moreover, $\tau(s,\theta)$ on $[0,2]\times [2,\infty)$ and $[0,2]\times (-\infty,-2]$ depends only on $s$ and on $[3,\infty)\times [-1,1]$ depends only on $\theta$. We also fix a smooth {\it cutoff} function $\rho:[-2,0]\to [0,1]$ which is equal to $0$ in a neighborhood of $-2$ and is equal to $1$ in a neighborhood of $0$. For $(\zeta,\nu)$ in $B^l_{(A,u)}$ with norm less than $\kappa_0$ define
\[
  \widehat A_\zeta=A+\zeta+\rho(s) (\rE_{\tau(0,\theta)}(A_{s,\theta},\zeta_{s,\theta})-\zeta_{s,\theta}),\hspace{1cm} u_\nu(s,\theta)=\re_{\tau(s,\theta)}(\nu(s,\theta)).
\]
where $A_{s,\theta}$, $\zeta_{s,\theta}$ are the restrictions of $A$, $\zeta$ to $\{(s,\theta)\}\times \Sigma$. 
With a slight abuse of notation, $\rho$ in the definition of $\widehat A_\zeta$ denotes the induced function $X\to [0,1]$ which vanishes outside of $[-2,0]\times \R\times \Sigma$ and equals $\rho(s)$ for $(s,\theta,x)\in [-2,0]\times \R\times \Sigma$. Then $\widehat A_\zeta$ is respectively asymptotic to the connections 
\[
  \widehat B_{b}:=B+b+\rho(s)(\rE_{1}(B(s),b(s))-b(s)),\hspace{0.7cm}\widehat B'_{b'}:=B'+b'+\rho(s)(\rE_{-1}(B'(s),b'(s))-b'(s)),
\]
as $\theta \to \infty$ and $-\infty$. Here $B(s)$ and $b(s)$ are the restrictions of $B$ and $b$ to $\{s\}\times \Sigma \subset Y$, and $B'(s)$ and $b'(s)$ are defined similarly. The map $\rho$ is interpreted as a function on $Y$ and $Y'$ by composing the function $\rho:X\to [0,1]$ with the inclusion of $Y$ and $Y'$ as $Y\times \{2\}$ and $Y'\times \{-2\}$ in $X$.

The connections $\widehat B_{b}$ and $\widehat B'_{b'}$ do not necessarily represent elements of $L(Y,E)$ and $L(Y',E')$. We fix this issue by modifying $A^0_\zeta$ as
\[
  A_\zeta:=\widehat A_\zeta+\varphi_{+} \cdot (B_{b}-\widehat B_{b})+\varphi_{-} \cdot  (B_{b'}'-\widehat B'_{b'}).
\]
Here $\varphi_+:X\to \R$ (resp. $\varphi_-:X\to \R$) is a fixed cutoff function which is equal to $1$ on $Y\times [3,\infty)$ (resp. $Y'\times (-\infty,-3]$) and $0$ on the complement of $Y\times (2,\infty)$ (resp. $Y'\times (-\infty,-2)$). Given $(\fg,\fg')\in \mathcal G(E)\times \mathcal G(E')$, we define
\[
  A_{\zeta,\fg,\fg'}:=A_\zeta-\varphi_{+}\cdot  (\nabla_{B_{b}}\fg)\fg^{-1}-\varphi_{-} \cdot (\nabla_{B_{b'}'}\fg')\fg'^{-1}.
\]
The connection $A_{\zeta,\fg,\fg'}$ is asymptotic to $\fg^*B_{b}$ and $\fg'^*B'_{b'}$ as $\theta \to \infty$ and $\theta \to -\infty$.
We define a map $P$ from the product of the ball of radius $\kappa_0$ centered at the origin in $B^l_{(A,u)}$ and $\mathcal G(E)\times \mathcal G(E')$ to $\bA(\alpha,\beta)$ by mapping $(\zeta,\nu,\fg,\fg')$ to $P(\zeta,\nu,\fg,\fg'):=(A_{\zeta,\fg,\fg'},u_\nu)$. The map $P$ gives a chart for a neighborhood of $(A,u)$ in $\bA(\alpha,\beta)$. By a slight abuse of notation, $P(\zeta,\nu)$ in what follows denotes $P(\zeta,\nu,1,1)$. It is a straightforward (but daunting) task to check that the transition maps associated to these charts for different $(A,u)$ in $\bA(\alpha,\beta)$ are smooth.

The above discussion can be modified using {\it Coulomb gauge fixing condition} to define a chart for the configuration space $\bB(\alpha,\beta)$, which is obtained from $\bA(\alpha,\beta)$ by taking the quotient with respect to the action of the gauge group $\mathcal G(V)$. The following proposition is a consequence of Coulomb gauge fixing for the action of $\mathcal G(V)$ on $\bA(\alpha,\beta)$.
\begin{prop}
	For any $(A,u)\in \bA(\alpha,\beta)$, there is a constant $\kappa_0$
	such that the following holds.
	Let $\fU_{\kappa_0}$ denote the space of $(\zeta,\nu)\in B^l_{(A,u)}$ such that 
	$\left|(\zeta,\nu)\right|_{B^l_{(A,u)}}<\kappa_0$ and 
	\begin{equation}\label{coulomb-gauge-chart}
		*\zeta|_{U_\partial \times \Sigma}=0,\hspace{1cm}d_A^*\zeta=0.
	\end{equation}
	Then the map
	\[
	  (g,(\zeta,\nu)) \in \mathcal G(V)\times \fU_{\kappa_0} \to g^*P(\zeta,\nu)
	\]
	gives a diffeomorphism onto a neighborhood of $(A,u)$ in $\bA(\alpha,\beta)$.
\end{prop}

This proposition together with a standard argument can be used to show that $\bB(\alpha,\beta)$ is a Banach manifold modeled on the closed subspace of $B^l_{(A,u)}$ consisting of the elements which satisfy \eqref{coulomb-gauge-chart}. Recall that $E^l_{(A,u)}$ is the subspace of $B^l_{(A,u)}$ given by elements which satisfy the first identity in \eqref{coulomb-gauge-chart}, and we equip this space with the Banach space structure using the norm in \eqref{norm-fB}. Thus, the above proposition implies that a neighborhood of $[A,u]$ in $\bB(\alpha,\beta)$ can be parametrized by the kernel of $d_A^*$ acting on $E^l_{(A,u)}$. To complete the proof of Proposition \ref{conf-space-smooth}, we need to show that the operator $d_A^*$ is surjective. To see this note that if $\xi$ is in the $L^2$-orthogonal of the image of $d_A^*$, then $\xi$ is in the kernel of $d_A$. Since $A$ is irreducible, $\xi$ has to be zero, which verifies the claim.

\subsection{Fredholm property of the mixed operator}\label{fred-adjoint}

In this subsection, we study the Fredholm properties of the mixed operator $\mathcal D_{(A,u)}$. As it is mentioned in Section \ref{main-thm-sec}, the domain of $\mathcal D_{(A,u)}$ is $E^l_{(A,u)}$ (equipped with the norm in \eqref{norm-fB}) and its target is given by
\begin{equation}\label{domain-D-A-u}
  L^2_{l-1,\delta}(X,(\Lambda^+\oplus \Lambda^0)\otimes V)\oplus  L^2_{l-1,\delta}(U_+,u^*T\mathcal M(\Sigma,F)).
\end{equation}
For our purposes, it is useful to consider another operator $\mathcal D^*_{(A,u)}$. The following definition is the counterpart of Definition \ref{E-A-u}, and it provides a function space which serves as the domain of $\mathcal D^*_{(A,u)}$. 

\begin{definition}\label{K-A-u}	
	Let $(A,u)\in \bA(\alpha,\beta)$ be a mixed pair which is asymptotic to $(B,q)$ and $(B',q')$ on the mixed ends associated to $(Y,E)$ and $(Y',E')$.
	For any positive integer $k$, define $K^k_{(A,u)}$ as the space of all
	\[
	  (\mu,\xi,z)\in L^2_{k,loc}(X,(\Lambda^+\oplus \Lambda^0)\otimes V)\oplus  L^2_{k,loc}(U_+,u^*T\mathcal M(\Sigma,F))
	\]
	such that 
	\vspace{-8pt} 
	\begin{itemize}
		\item[(i)] $(\mu,\xi)|_{(-\infty,-3]\times Y_\#}$ and $z\vert_{[3,\infty)\times [-1,1]}$ have finite $L^2_{k}$ norms.
		\item[(ii)] The restrictions of $(\mu,\xi)$ to $[2,\infty)\times Y$ and $(-\infty,-2]\times Y'$ and the restrictions of $z$ to $[0,2]\times [2,\infty)$ and $[0,2]\times (-\infty,-2]$ have finite $L^2_{k,\delta}$ norms. 
		\item[(iii)] The restriction of $\mu$ to $U_\partial\times \Sigma$ has the form $\frac{1}{2}d\theta \wedge c$, where $c$ is a section of $\Lambda^1\Sigma\otimes F$ over $U_\partial\times \Sigma$. Moreover, if 
		$c_\theta$ denotes the restriction of $c$ to $\{(0,\theta)\}\times \Sigma\subset X$, then $d_{A_{\theta}}c_\theta=0$ and $z(0,\theta)$ is equal to the element of $T_{u(0,\theta)}\mathcal M(\Sigma,F)$ represented by $c_\theta$.
		\item[(iv)] The $(0,1)$-form $zd\theta+J_{s,\theta}zds$ maps $T\eta_+$ to $u^*TL(Y,E)$ and $T\eta_+'$ to $u^*TL(Y',E')$.
	\end{itemize}
	The norm on $ K^k_{(A,u)}$ is given by the weighted Sobolev norm $L^2_{k,\delta}$.
\end{definition}
For $(\mu,\xi,z)\in K^l_{(A,u)}$, define 
\begin{equation}\label{mixed-op-dual}
	\mathcal D^*_{(A,u)}(\mu,\xi,z):=(\mathcal D^*_A(\mu,\xi),\mathcal D_u^*(z)),
\end{equation}
where $\mathcal D^*_A$ and $\mathcal D_u^*$ are formal adjoints of the ASD operator and the Cauchy-Riemann operator.  Thus, these are the unique operators satisfying
\begin{equation}\label{L-2-adjoint}
    \int_{X}\langle \zeta, \mathcal D^*_A(\mu,\xi)\rangle=\int_{X}\langle \mathcal D_A(\zeta), (\mu,\xi)\rangle,\hspace{1cm}\int_{U_+}\langle \nu, \mathcal D_u^*(z)\rangle=\int_{U_+}\langle \mathcal D_u(\nu), z\rangle,
\end{equation}
for any $\xi$, $\zeta$, $\mu$, which are respectively smooth sections of $V$, $\Lambda^1\otimes V$, $\Lambda^+\otimes V$ compactly supported in the interior of $X$, and any $\nu$, $z$, which are smooth sections of $u^*T\mathcal M(\Sigma,F)$ compactly supported in the interior of $U_+$. To be more specific, the $L^2$ pairing for the second term in \eqref{L-2-adjoint} is defined using the metric 
\[
  \int_{U_+}\langle \nu,\nu'\rangle:=\int_{U_+}\Omega(\nu,J_{s,\theta} \nu') ds\wedge d\theta,
\]
where $\nu$ and $\nu'$ are sections of $u^*T\mathcal M(\Sigma,F)$. We have
\[
 \mathcal D^*_A(\mu,\zeta)=d_{A}\xi+d_{A}^*\mu+{\rm Hess}_{A_t}h(*_3\mu_t)+{\rm Hess}_{A_t'}h'(*_3\mu_t'),
\] 
with $\mu_t$, $\mu_t'$ being the restrictions of $\mu$ to $\{t\}\times Y_0$, $\{t\}\times -Y'_0$. Using this notation, we may write the self-dual 2-form $\mu$ on $\R\times Y_0$ and $\R\times -Y'_0$ as $\mu_t-*_3\mu_t\wedge dt$ and $\mu_t'-*_3\mu_t'\wedge dt$. The target of the operator $\mathcal D^*_{(A,u)}$ is $L^2_{l-1,\delta}(X,\Lambda^1\otimes V)\oplus L^2_{l-1,\delta}(U_+,u^*T\mathcal M(\Sigma,F))$ where our convention for the weighted Sobolev space $L^2_{l-1,\delta}$ is fixed in Definition \ref{eighted-sob}.
Proposition \ref{fred-prop-main-thm} is a consequence of the following theorem.
\begin{theorem}\label{Fredholm-D-}
	There is $\delta_0$ such that the following holds. For $\delta<\delta_0$, suppose$\bA(\alpha,\beta)$ is defined using $\delta$ and $[A,u]$ is a smooth element of $\bA(\alpha,\beta)$ that satisfies property (iii) of
	Theorem \ref{exp-decay} for $\delta_0$. Then the operators $\mathcal D_{(A,u)}$ and $\mathcal D^*_{(A,u)}$ are Fredholm.
	Furthermore, the cokernel (resp. the kernel) of $\mathcal D^*_{(A,u)}$ can be identified with the kernel (resp. the cokernel) of $\mathcal D_{(A,u)}$. In particular, 
	$\ind(\mathcal D_{(A,u)})=-\ind(\mathcal D^*_{(A,u)})$.
\end{theorem}

In order to fix the constant $\delta_0$ in Theorem \ref{Fredholm-D-}, we need to look more closely at the mixed operator on the mixed ends. This will be addressed in Subsection \ref{Fredholm-mixed-cyl}, where we also review some of the results of \cite{DFL:mix} relevant to the Fredholm property of mixed operators. We will come back to the proof of Theorem \ref{Fredholm-D-} in Subsection \ref{proof-thm-Fredholm-D}.

\subsubsection{Mixed cylinders and mixed operators}\label{Fredholm-mixed-cyl}

Suppose $(Y,E)$ is as in the previous sections and $I$ is a Riemannian connected 1-dimensional manifold. Thus, $I$ is either an open interval in $\R$ or $S^1$ with a fixed length. The {\it cylinder quintuple} associated to $I$ is given as 
\[
  \fc_I:=(I\times Y,E\times I,[0,2]\times I,\mathcal M(\Sigma,F),L(Y,E)).
\]
We fix the product metric on $I\times Y$ and a family of almost complex structures $\{J_s\}_{s\in [0,2]}$ on $\mathcal M(\Sigma,F)$ such that $J_s=J_*$ for $s<1$. This family induces a family of almost complex structures parametrized by $[0,2]\times I$, which is constant with respect to the second component. For instance, the restriction of a mixed pair for the special quintuple to the mixed end associated to $Y$ determines a mixed pair for the cylinder quintuples $\fc_{(2,\infty)}$. 

As in the case of special quintuples, we may associate a mixed operator $\mathcal D_{(A,u)}$ to any mixed pair $(A,u)$ on the cylinder quintuple $\fc_I$. The domain $E_{(A,u)}^k(I)$ of this operator consists of 
\begin{equation}\label{domain-cylinder-mixed-op}
	\zeta \in L^2_{k,loc}(I\times Y,\Lambda^1\otimes E), \hspace{1cm} \nu\in L^2_{k,loc}([0,2]\times I,u^*T\mathcal M(\Sigma,F)),
\end{equation}	
such that $*\zeta|_{\Sigma\times I}=0$, and for any $\theta\in I$, we have
\begin{equation}\label{bdry-matching-cylinder}
  \nu(2,\theta)\in T_{u(2,\theta)}L(Y,E),\hspace{1cm} d_A\zeta\vert_{\Sigma\times \{(0,\theta)\}}=0,\hspace{1cm} [\zeta\vert_{\Sigma\times \{(0,\theta)\}}]=\nu(0,\theta)
\end{equation}
In the case that $I$ is an infinite interval, we demand that an element $(\zeta,\nu)\in E_{(A,u)}^k(I)$ has a finite weighted Sobolev norm with respect to the weight $e^{\delta \tau}$. Here $\delta$ is a real number, and $\tau:I\times Y\sqcup [0,2]\times I\to \R$, is the projection map  to $I$. 
For any $(\zeta,\nu)\in E_{(A,u)}^k(I)$, we have
\[
  \mathcal D_{(A,u)}(\zeta,\nu)=(d_{A}^*\zeta, d_{A}^+\zeta+(*_3{\rm Hess}_{A_\theta}h(\zeta_\theta))^+,\mathcal D_u\nu).
\]
where the Cauchy-Riemann operator $\mathcal D_u$ is defined as in \eqref{linear-CR}. 
The target of $\mathcal D_{(A,u)}$ is the space 
\begin{equation}\label{target-cylinder-mixed-op}
  L^2_{k-1,\delta}(I\times Y,(\Lambda^0\oplus \Lambda^+)\otimes E)\oplus L^2_{k-1,\delta}([0,2]\times I,u^*T\mathcal M(\Sigma,F)).
\end{equation}
where the wighted Sobolev space is defined again using the weight $e^{\delta \tau}$. Of course, if $I$ is a finite interval the weight does not play any role and we may replace the weighted Sobolev space $L^2_{k-1,\delta}$ with $L^2_{k-1}$.

There is a useful reparametrization of the target of $\mathcal D_{(A,u)}$ in the case of cylinder quintuples. Any section $\zeta$ of $(\Lambda^0\oplus \Lambda^+)\otimes E$ over $I\times Y$ has the form $(\varphi,\frac{1}{2}(d\theta\wedge b+*_3b))$ where $b$ is a section of $\Lambda^1(Y)\otimes E$ over $I\times Y$. In particular, we may associate $b-\varphi d\theta$, a section of  $\Lambda^1\otimes E$, to $\zeta$. This allows us to identify the target of $\mathcal D_{(A,u)}$ in \eqref{target-cylinder-mixed-op} with 
\begin{equation}\label{target-cylinder-mixed-op-repar}
  L^2_{k-1,\delta}(I\times Y,\Lambda^1\otimes E)\oplus L^2_{k-1,\delta}([0,2]\times I,u^*T\mathcal M(\Sigma,F)).
\end{equation}
We will use this reparametrization of the target of $\mathcal D_{(A,u)}$ in the rest of this subsection.

\begin{theorem}\label{Fred-mixed-op}
	Suppose $I=(a,b)$, $J=(c,d)$ are finite intervals with $a<c<d<b$. Suppose $(A,u)$ is a smooth mixed pair associated to the cylinder quintuple $\fc_I$.
	\vspace{-5pt}
	\begin{itemize}
		\item[(i)] Suppose $k\geq 1$ is an integer, $(\zeta,\nu)\in E^1_{(A,u)}(I)$ and $\mathcal D_{(A,u)}(\zeta,\nu)$ is in $L^2_{k-1}$. Then $(\zeta,\nu)\in E^k_{(A,u)}(J)$. There is also a constant $C$, independent of $(\zeta,\nu)$, such that
		\begin{equation}\label{Fred-ineq}
			|\!|(\zeta,\nu)|\!|_{L^2_{k}(J)}\leq C\left(|\!|\mathcal D_{(A,u)}(\zeta,\nu)|\!|_{L^2_{k-1}(I)}+|\!|(\zeta,\nu)|\!|_{L^2(I)}\right).
		\end{equation}
		\item[(ii)] Suppose $\zeta \in L^2(I\times Y,\Lambda^1\otimes E)$ and $\nu\in L^2([0,2]\times I,u^*T\mathcal M(\Sigma,F))$ satisfy
		\[
		  |\langle (\zeta,\nu),\mathcal D_{(A,u)}(\zeta',\nu')\rangle| \leq \kappa |\!|(\zeta',\nu') |\!|_{L^2}
		\]
		for any smooth element $(\zeta',\nu')$ of $ E^1_{(A,u)}(I)$ with compact support and for a fixed constant $\kappa$. Then $(\zeta,\nu)\in E^1_{(A,u)}(J)$.
	\end{itemize}
\end{theorem}
\begin{proof}
	This theorem in the absence of the perturbation term $(*_3{\rm Hess}_{A_\theta}h(\zeta_\theta))^+$ is proved in \cite[Theorem 5]{DFL:mix}. 
	This special case and the property of 
	cylinder functions mentioned in part (iii) of Proposition \ref{hol-pert-properties} allows us to conclude the general case.
\end{proof}

\begin{remark}\label{fred-st-family}
	The mixed operators associated to special quintuples satisfy a uniform version of Theorem \ref{Fred-mixed-op}.
	To be more precise, suppose $(A_i,u_i)$ is a sequence of smooth mixed pairs associated to the cylinder quintuple $\fc_I$ that are $C^\infty$-convergent to $(A,u)$. Then there is a constant $C$ such that for any $i$ and any 
	$(\zeta,\nu)\in E^k_{(A_i,u_i)}(I)$, we have
		\begin{equation*}
			|\!|(\zeta,\nu)|\!|_{L^2_{k}(J)}\leq C\left(|\!|\mathcal D_{(A_i,u_i)}(\zeta,\nu)|\!|_{L^2_{k-1}(I)}+|\!|(\zeta,\nu)|\!|_{L^2(I)}\right).
		\end{equation*}
	As in Theorem \ref{Fred-mixed-op}, this is again a consequence of the results of \cite{DFL:mix}	 and Proposition \ref{hol-pert-properties}. (See \cite[Remarks 5.73 and 5.76]{DFL:mix}.) 
\end{remark}

The required result from \cite{DFL:mix} in the proof of Theorem \ref{Fred-mixed-op} uses a description of the mixed operator in terms of a {\it dimensionally reduced mixed operator}. First we review this description in a simpler case. Any $q\in L(Y,E)$ determines a mixed pair associated to the cylinder quintuple $\fc_I$. Suppose $B_q$ is a connection on $E$ that represents $q$, and its restriction to the boundary is $\alpha_q$. Let $A_q$ be the pullback of $B_q$ to $I\times Y$ and $u_q:[0,2]\times I\to \mathcal M(\Sigma,F)$ be the constant map to $q$. The pair $(A_q,u_q)$ defines a mixed pair for the cylinder quintuple $\fc_I$, which can be regarded as the counterpart of constant pairs for special quintuples. In fact, the restriction of a constant pair to the mixed end associated to $Y$ determines an element of the form $(A_q,u_q)$ for the interval $I=(2,\infty)$. For any 	
\begin{equation}\label{triple-Hilbert}
	  (\varphi,b,\nu)\in \Omega^0(Y,E)\oplus \Omega^1(Y,E)\oplus \Omega^0([0,1],T_q\mathcal M(\Sigma,F)),
\end{equation}
define
\[
  \fD_{q}(\varphi,b,\nu)=(d_{B_q}^*b,-*_3d_{B_q}b+ {\rm Hess}_{B_q}h(b)+d_{B_q}\varphi ,J_s\frac{d\nu}{ds}),
\]
which is again an element of $\Omega^0(Y,E)\oplus \Omega^1(Y,E)\oplus \Omega^0([0,1],T_q\mathcal M(\Sigma,F))$. Then we have
\[
  \mathcal D_{(A_q,u_q)}=\frac{d}{d\theta}-\fD_{q}.
\] 
Here we again use the identification of a $1$-form on $I\times Y$ with a map from $I$ to the space of sections of $\Lambda^0\oplus \Lambda^1$ on $Y$.

\begin{prop}\label{inv-mixed-op-const-sol}
	There is a positive constant $\delta_0$ such that if $0<\delta<\delta_0$ or $-\delta_0<\delta<0$ , then the operator 
	\[
	  \mathcal D_{(A_q,u_q)}:E_{(A_q,u_q)}^1(\R) \to L^2_{\delta}(\R\times Y ,\Lambda^1\otimes E)\oplus L^2_{\delta}([0,2]\times \R,u^*T\mathcal M(\Sigma,F)).\] 
	 is an isomorphism. Moreover, there is $C$ such that for any $(\zeta,\nu) \in E_{(A_q,u_q)}^1(\R)$, we have
	\[
	  C^{-1}|\!|(\zeta,\nu)|\!|_{L^2_{1,\delta}} \leq |\!|\mathcal D_{(A_q,u_q)}(\zeta,\nu)|\!|_{L^2_\delta}\leq C|\!|(\zeta,\nu)|\!|_{L^2_{1,\delta}}.
	\] 
\end{prop}

\begin{proof}
	Suppose $\mathcal H_q$ is the completion of the space of $(\varphi,b,\nu)$ as in \eqref{triple-Hilbert} with respect to the $L^2$ norm
	\begin{equation}
		 \langle (\varphi_0,b_0,\nu_0),(\varphi_1,b_1,\nu_1)\rangle_{L^2}:=\int_Y\tr\left(\varphi_0\wedge *\varphi_1+b_0\wedge *b_1\right)
		 +\int_{0}	^{1}\Omega(\nu_0(s),J_s\nu_1(s) )ds.
	\end{equation}
	Let $\mathcal W_q$ denote the $L^2_1$ completion of the space of all triples $(\varphi,b,\nu)$ as in \eqref{triple-Hilbert} such that
	\begin{equation}\label{match-bdry-cond}
		*b|_{\Sigma}=0,\hspace{1cm}d_{\alpha_q}b|_{\Sigma}=0,\hspace{1cm} [b|_{\Sigma}]=\nu(0),\hspace{1cm} \nu(1)\in T_qL(Y,E).
	\end{equation}
	Given a 1-parameter family $\{(\varphi_\theta,b_\theta)\}_{\theta\in \R}$, we may define a $1$-form on $\R\times Y $ as $b_\theta-\varphi_\theta d\theta$.
	Using this identification, we have 
	\[
	  E_{(A_q,u_q)}^1(\R)=L^2_{1,\delta}(\R,\mathcal H_q)\cap L^2_\delta(\R,\mathcal W_q),
	\]
	and 
	\[
	 L^2_{\delta}(\R\times Y ,\Lambda^1\otimes E)\oplus L^2_{\delta}([0,2]\times \R,T_q\mathcal M(\Sigma,F))=L^2_\delta(\R,\mathcal H_q).
	\]
	It is shown in \cite{DFL:mix} that $\fD_q:\mathcal H_q\to \mathcal H_q$ is an (unbounded) self-adjoint Fredholm operator with domain $\mathcal W_q$ and a discrete spectrum that has a 
	finite intersection with any finite interval. The kernel of $\fD_{q}$ can be identified with $T_qL(Y,E)$.
	Moreover, the operator $\mathcal D_{(A_q,u_q)}$ is invertible if and only if $\delta$ is not in the spectrum of $\fD_q$ (see \cite[Proposition 5.79]{DFL:mix}).
	Thus, it suffices to show that 
	\begin{equation}\label{delta-0-pos}
	  \delta_0:=\inf_{B_q}\{\delta_q\mid \delta_q \text{ is the smallest magnitude of a non-zero eigenvalue of $\fD_q$}\}>0
	\end{equation}
	Although $\fD_q$ is defined in terms of $B_q$, it essentially depends only on $q$, the gauge equivalence class of $B_q$ up to conjugation. In fact, the Hilbert spaces
	$\mathcal H_q$ and $\mathcal W_q$ define Hilbert space bundles $\bH$ and $\bW$ on $L(Y,E)$. (It is clear that $\bH$ is locally trivial, and local trivializations of $\bW$ is given by \cite[Proposition 5.27]{DFL:mix}.) 
	Then $\{\fD_q\}_q$ define a smooth family of Fredholm operators from the fibers of 
	$\bW$ to the fibers of $\bH$. The claim in \eqref{delta-0-pos} follows because the dimension of the kernels of these operators is 
	independent of $q$ and $L(Y,E)$ is compact.
\end{proof}

\begin{cor}\label{cyl-cokernel-const-pair}
	Suppose $\delta_0$ is as in Proposition \ref{inv-mixed-op-const-sol} and $0<\delta<\delta_0$. Suppose 
	\begin{equation}\label{domain-cylinder-mixed-op-loc-cor}
		\zeta \in L^2_{1,loc}(\R\times Y ,\Lambda^1\otimes E), \hspace{1cm} \nu\in L^2_{1,loc}([0,2]\times \R,T_q\mathcal M(\Sigma,F)),
	\end{equation}	
	such that $*\zeta|_{\Sigma\times \R}=0$, and the identities in \eqref{bdry-matching-cylinder} hold for any $\theta\in \R$ and the mixed pair $(A_q,u_q)$.
	Suppose also $(\zeta,\nu)\in L^2_{-\delta}$ and $(\zeta',\nu'):=\mathcal D_{(A_q,u_q)}(\zeta,\nu)\in L^2_\delta$, where $L^2_{-\delta}$ is the weighted Sobolev norm defined using the negative exponent $-\delta$. 
	Then there is $(\zeta_1,\nu_1)$, which is the pullback of an element of the kernel of $\fD_{q}$, such that $(\zeta-\zeta_1,\nu-\nu_1)$
	has finite $L^2_{1,\delta}$ norm.
\end{cor}
\begin{proof}
	Proposition \ref{inv-mixed-op-const-sol} implies that there is $(\zeta_0,\nu_0)\in E_{(A_q,u_q)}^1(\R)$ such that $\mathcal D_{(A_q,u_q)}(\zeta_0,\nu_0)=(\zeta',\nu')$.
	In particular, $(\zeta_1,\eta_1):=(\zeta-\zeta_0,\nu-\nu_0)$ belongs to the kernel of $\mathcal D_{(A_q,u_q)}$. Moreover, $e^{-|\theta|\delta}(\zeta_1,\eta_1)$ has a finite $L^2$ norm.
	There is a complete eigenspace
	decomposition $\{f_i\}_i$ associated to the operator $\fD_{q}$ which provides an orthonormal basis for $\mathcal H$. Using this eigenspace decomposition, we have 
	\begin{equation}\label{zeta-nu-ev-decom}
	  (\zeta_1,\nu_1)=\sum_{i}c_ie^{\lambda_i\theta}f_i
	\end{equation}
	where $\lambda_i$ is the eigenvalue of $f_i$ and $c_i\in \R$. Our assumption on $e^{-|\theta|\delta}(\zeta_1,\eta_1)$ implies that $c_i=0$ unless $\lambda_i=0$. This gives the claim.
\end{proof}

In fact, we shall need a generalization of Corollary \ref{cyl-cokernel-const-pair} where $(A_q,u_q)$ is replaced with a more general mixed pair $(A,u)$.

\begin{cor}\label{cyl-cokernel}
	Suppose $\delta_0$ is as in Proposition \ref{inv-mixed-op-const-sol} and $0<\delta<\delta_0$. For $q\in L(Y,E)$, suppose $(A,u)$ is a smooth mixed pair for the cylinder quintuple $\fc_{(0,\infty)}$ such that $A-A_q\in L^2_{l,\delta}$, $u$
	converges to $q$ as $\theta\to \infty$ and $d u$ is in $L^2_{l-1,\delta}$. Suppose 
	\begin{equation}\label{domain-cylinder-mixed-op-loc-cor}
		\zeta \in L^2_{1,loc}((0,\infty)\times Y,\Lambda^1\otimes E), \hspace{1cm} \nu\in L^2_{1,loc}([0,2]\times (0,\infty),u^*T\mathcal M(\Sigma,F)),
	\end{equation}	
	such that $*\zeta|_{\Sigma\times (0,\infty)}=0$, and the identities in \eqref{bdry-matching-cylinder} hold for the mixed pair $(A,u)$ and any $\theta\in (0,\infty)$.
	Suppose also $(\zeta,\nu)\in L^2_{1,-\delta}$ and $(\zeta',\nu'):=\mathcal D_{(A,u)}(\zeta,\nu)\in L^2_\delta$. 
	Then there is $(\zeta_1,\nu_1)$, which is the pullback of an element of the kernel of $\fD_{q}$, such that $(\zeta-\zeta_1,\nu-\nu_1) \in E^1_{(A,u)}((0,\infty))$.
\end{cor}
\begin{proof}
	Suppose $(B_\theta,u_\theta)$ denotes the restriction of $A$, $u$ to $\{\theta\}\times Y$ and $[0,2]\times\{\theta\}$. Analogous to $\mathcal H_q$ and $\mathcal W_q$ in the proof of 
	Proposition \ref{inv-mixed-op-const-sol}, we may use $(B_\theta,u_\theta)$ to define the completions $\mathcal H_\theta$ and and $\mathcal W_\theta$ of 
	\[
	  \Omega^0(Y,E)\oplus \Omega^1(Y,E)\oplus \Omega^0([0,1],u_\theta^*T\mathcal M(\Sigma,F)).
	\]
	We may use local trivializations of the Hermitian bundles $(T\mathcal M(\Sigma,F),\Omega,J_s)$ in a neighborhood of $q$, to identity $\mathcal H_\theta$ and $\mathcal H_q$ in the obvious way.
	This allows us to drop $\theta$ from our notation for $\mathcal H_\theta$, and denote it by $\mathcal H$.
	
	 To prove the claim, it suffices to show that there is $(\zeta_1,\nu_1)$ as above such that for some $T_0>0$ the restriction of $(\zeta-\zeta_1,\nu-\nu_1)$ to $(T_0,\infty)\times Y$ and $[0,2]\times (T_0,\infty)$ is in $L^2_{1,\delta}$. In particular, by taking 
	$T_0$ large enough, we may assume that $(B_\theta,u_\theta)$ is in a neighborhood of $(B_q,u_q)$ such that we can apply \cite[Proposition 5.27]{DFL:mix} 
	and show that there are isomorphisms 
	\[
	  Q_\theta: \mathcal H \to \mathcal H
	\]
	such that $Q_\theta$ maps $\mathcal W_q$ to $\mathcal W_\theta$. 
	Moreover, $Q_\theta$ maps the subspace of $L^2_k$ elements of $\mathcal H_\theta$ isomorphically onto the subspace of $L^2_k$ elements of $\mathcal H_q$, 
	and satisfies
	\begin{equation}\label{Q-theta-}
		C_k^{-1}|\!|(\varphi,b,\nu) |\!|_{L^2_k}  \leq |\!|Q_\theta(\varphi,b,\nu)|\!|_{L^2_k} \leq C|\!|(\varphi,b,\nu)|\!|_{L^2_k}
	\end{equation}	
	for a constant $C_k$  independent of $\theta$ and for any $(\varphi,b,\nu)$. 
	In fact, the operator norm of $Q_\theta-{\rm Id}$ with respect to the $L^2_k$ norm is bounded by $C_ke^{-\delta \theta}$. The map $\theta \to Q_\theta$ as a map from $(T_0,\infty)$ to the space $B(\mathcal H)$
	of bounded operators of $\mathcal H$ is smooth and its derivatives satisfy the analogue of \eqref{Q-theta-}. In particular, the operators $Q_\theta$ can be put together to define 
	\begin{align}
	  \bQ:L^2_{k,loc}((T_0,\infty)\times Y,\Lambda^1\otimes E)\oplus L^2_{k,loc}([0,2]\times (T_0,\infty),T_q\mathcal M(\Sigma,F)) \to &\hspace{4cm}\nonumber\\
	  \hspace{3cm} \to L^2_{k,loc}((T_0,\infty)\times Y,\Lambda^1\otimes E)\oplus L^2_{k,loc}([0,2]\times (T_0,\infty),u^*T&\mathcal M(\Sigma,F))\label{bQ}
	\end{align}
	for any $k$  \cite[Lemma 5.67]{DFL:mix}. 
	
	The operator $\bQ$ maps the domain and the target of $\mathcal D_{(A_q,u_q)}$ respectively to the the domain and the target of $\mathcal D_{(A,u)}$, and we have
	\begin{equation}\label{change-basis-mixed-op}
	  \bQ^{-1} \circ \mathcal D_{(A,u)}\circ \bQ=\frac{d}{d\theta}-\fD_{q}-S_\theta,
	\end{equation}
	where $S_\theta:\mathcal W_q\to \mathcal H_q$ is a bounded linear operator whose norm is bounded by $C e^{-\delta \theta}$.
	
	Using \eqref{change-basis-mixed-op}, we may write
	\[
	  \mathcal D_{(A_q,u_q)}(\bQ^{-1} (\zeta,\nu))=\bQ^{-1}(\zeta',\nu')+\bS \bQ^{-1} (\zeta,\nu).
	\]
	where $\bS$ is defined using the operators $S_\theta$. By assumption the first term on the left hand side has a finite $L^2_\delta$ norm and the second term has a finite $L^2$ norm. 
	Using a cutoff function $\rho:(T_0,\infty)\to \R$
	which vanishes for $\theta<T_0+\frac{1}{2}$ and is equal to $1$ for $\theta>T_0+1$, we may extend $\bQ(\zeta,\nu)$ to 
	\[
	  (\widetilde \zeta,\widetilde \nu) \in L^2_{1,loc}(\R\times Y ,\Lambda^1\otimes E) \oplus L^2_{1,loc}([0,2]\times \R,T_q\mathcal M(\Sigma,F)),
	\]
	such that $\mathcal D_{(A_q,u_q)}(\widetilde \zeta,\widetilde \nu)$ has a finite $L^2_{-\delta/2}$ norm. (In fact, $L^2_{-\delta/2}$ can be replaced with $L^2$.) By applying the argument in 
	the proof of Corollary \ref{cyl-cokernel-const-pair}, we may conclude that $(\widetilde \zeta,\widetilde \nu)$ has finite $L^2_{1,-\delta/2}$ norm. Thus, the same claim holds for $(\zeta,\nu)$.
	By iterating the same argument, we can show that now that $\mathcal D_{(A_q,u_q)}(\bQ^{-1}(\zeta,\nu))$ has a finite $L^2_{\delta/2}$ norm. Using 
	Corollary \ref{cyl-cokernel-const-pair} again we may conclude that there is $(\zeta_1,\nu_1)$, which is the pullback of an element of the kernel of $\fD_{q}$, such that $\bQ^{-1}(\zeta,\nu)-(\zeta_1,\nu_1)$
	has finite $L^2_{\delta/2}$ norm. Iterating this argument once more, we conclude that $\bQ^{-1}(\zeta,\nu)-(\zeta_1,\nu_1)$ has finite $L^2_{\delta}$ norm. Our assumption on $\bQ$ implies that 
	$(\zeta,\nu)-(\zeta_1,\nu_1)$ also has finite $L^2_{\delta}$ norm.
\end{proof}

\begin{remark}\label{adjoint-Fred}
	The analogues of the results of this subsection hold for the adjoint of the mixed operator. In fact, the adjoint of the mixed 
	operator for cylinder quintuples have a similar form as 
	the mixed operator (see \cite[Section 5]{DFL:mix} for more details), and the results of this section would immediately 
	imply the corresponding results for the adjoint of the mixed operator.
\end{remark}

\subsubsection{Proof of Theorem \ref{Fredholm-D-}}\label{proof-thm-Fredholm-D}

In this subsection, we prove Theorem \ref{Fredholm-D-} on Fredholmness of mixed operators associated to the special quintuple where $\delta_0$ is given by Proposition \ref{inv-mixed-op-const-sol}. Suppose $(A,u)$ is a smooth mixed operator that satisfies the assumption of Theorem \ref{Fredholm-D-}.  Let also $X_T$ denote the compact subspace of $X$ given as the complement of the subspaces $(T,\infty)\times Y$, $(-\infty,-T)\times Y'$ and $(-\infty,-T)\times Y_\#$ in $X$. Similarly, let $U_T$ be the compact subspace of $U_+$ given as the complement of $[0,2]\times (T,\infty)$, $[0,2]\times (-\infty,-T)$ and $(T,\infty)\times [-1,1]$.

\begin{lemma}\label{mixed-op-reg-lemma}
	For $(A,u)$ as above and any $k\geq 1$, there are constants $C$ and $T_0$ such that the following holds.
	Suppose $(\zeta,\nu)\in E^1_{(A,u)}$ and $\mathcal D_{(A,u)}(\zeta,\nu)\in L^2_{k-1,\delta}$. Then $(\zeta,\nu)\in E^k_{(A,u)}$, and we have
	\begin{equation}\label{mixed-op-reg}
		|\!|(\zeta,\nu)|\!|_{E^k_{(A,u)}} \leq C(|\!|\mathcal D_{(A,u)}(\zeta,\nu)|\!|_{L^2_{k-1,\delta}}+ |\!|\zeta|\!|_{L^2(X_{T_0})}+|\!|\nu|\!|_{L^2(U_{T_0})}+|s|+|s'|).
	\end{equation}
	A similar result holds for $\mathcal D^*_{(A,u)}$.
\end{lemma}

\begin{proof}
	Theorem \ref{Fred-mixed-op} and standard regularity results about the linearized ASD and CR equations imply that $(\zeta,\nu)\in L^2_{k,loc}$. Moreover, for any $T_0>1$, there is $C$ such that 
	\begin{equation}\label{mixed-op-reg-comp-sub}
	  |\!|\zeta|\!|_{L^2_k(X_{T_0+1})}+|\!|\nu|\!|_{L^2_k(U_{T_0+1})}\leq C(|\!|\mathcal D_{(A,u)}(\zeta,\nu)|\!|_{L^2_{k-1,\delta}}+ |\!|\zeta|\!|_{L^2(X_{T_0+2})}+|\!|\nu|\!|_{L^2(U_{T_0+2})}).
	\end{equation}
	Next, we obtain control over the decay of $(\zeta,\nu)$ on the mixed end associated to 
	$Y$. Suppose $(A,u)$ is asymptotic to $(B_q,q)$ on this mixed end for $q\in L(Y,E)$. 
	We use a similar construction as in the proof of Corollary \ref{cyl-cokernel}. 
	Suppose $(B_\theta,u_\theta)$, $\mathcal H$ and $\mathcal W_\theta$
	are as in there, and for any $\theta\in (T_0,\infty)$, let the isomorphism  
	$Q_\theta:\mathcal H\to \mathcal H$ be given by \cite[Proposition 5.27]{DFL:mix}. 
	Suppose $\bQ$ is obtained from $Q_\theta$ analogous to \eqref{bQ}. Let $(\widehat \zeta,\widehat \nu)$ be the result of applying $\bQ^{-1}$ to the restriction of $(\zeta,\nu)$ to $(T_0,\infty)\times Y$ 
	and $[0,2]\times (T_0,\infty)$. We have
	\[
	  \mathcal D_{(A_q,u_q)}(\widehat \zeta,\widehat \nu)=\bQ^{-1} \mathcal D_{(A,u)}(\zeta,\nu)+\bS\bQ^{-1}(\zeta,\nu),
	\]
	where $\mathcal D_{(A_q,u_q)}$ is the mixed operator associated to the pullback of $(B_q,q)$ on the cylinder quintuple, and $\bS$ is given by a family of operators $S_\theta$ defined as in \eqref{change-basis-mixed-op}. 
	In particular, for a given positive constant $\epsilon$, we may assume that $T_0$ is chosen such that
	\begin{equation}\label{fixing-T}
	  |\!|\mathcal D_{(A_q,u_q)}(\widehat \zeta, \widehat \nu)|\!|_{L^2_{k-1,\delta}(T_0,\infty)}\leq C_0 |\!|\mathcal D_{(A,u)}(\zeta,\nu)|\!|_{L^2_{k-1,\delta}(T_0,\infty)}+\epsilon|\!|(\zeta,\nu)|\!|_{E^k_{(A,u)}(T_0,\infty)},
	\end{equation}
	where $C$ is a constant independent of $(\zeta,\nu)$. Here $|\!|\mathcal D_{(A,u)}(\zeta,\nu)|\!|_{L^2_{k-1,\delta}(T_0,\infty)}$ denotes the $L^2_{k,\delta}$ norm of the restriction of $\mathcal D_{(A,u)}(\zeta,\nu)$ to $(T_0,\infty)\times Y$. 

	Suppose $(\zeta,\nu)$ is asymptotic to $(b,s)$ on the mixed end associated to $Y$. Theorem \ref{Fred-mixed-op} for the pair $(A_q,u_q)$ implies that there is a constant 
	$C$ such that for any $T$, we have
	\begin{align*}
	  |\!|(\widehat \zeta-\pi^*b,\widehat \nu-\pi^*s)&|\!|_{L^2_{k}(T-1,T+1)}\leq \\
	  &C(|\!|\mathcal D_{(A_q,u_q)}(\widehat \zeta,\widehat \nu)|\!|_{L^2_{k-1}(T-2,T+2)}+|\!|(\widehat \zeta-\pi^*b,\widehat \nu-\pi^*s)|\!|_{L^2(T-2,T+2)}).
	\end{align*}	
	A weighted sum of these inequalities imply that
	\begin{align*}
	  |\!|(\widehat \zeta-\pi^*b,\widehat \nu-\pi^*s)&|\!|_{L^2_{k,\delta}(T_0+1,\infty)}\leq C(|\!|\mathcal D_{(A_q,u_q)}(\widehat \zeta,\widehat \nu)|\!|_{L^2_{k-1,\delta}(T_0,\infty)}+|\!|(\widehat \zeta-\pi^*b,\widehat \nu-\pi^*s)|\!|_{L^2_\delta(T_0,\infty)}).
	\end{align*}		
	The last term in the above inequality can be controlled by $\mathcal D_{(A_q,u_q)}(\widehat \zeta,\widehat \nu)$. In fact, by multiplying $(\widehat \zeta-\pi^*b,\widehat \nu-\pi^*s)$
	by a cutoff function $g:(T_0,\infty)\to \R$ satisfying $g(\theta)=1$ for $\theta\geq T_0+1$ and $g(\theta)=0$ for $\theta\leq T_0+1/2$, we may regard it as an element of $E_{(A_q,u_q)}^k(\R)$. In particular, applying 
	Proposition \ref{inv-mixed-op-const-sol} implies that 
	\begin{align*}
	  |\!|(\widehat \zeta-\pi^*b,\widehat \nu-\pi^*s)&|\!|_{L^2_{k,\delta}(T_0+1,\infty)}\leq \\
	  &C(|\!|\mathcal D_{(A_q,u_q)}(\widehat \zeta,\widehat \nu)|\!|_{L^2_{k-1,\delta}(T_0,\infty)}+|\!|(\widehat \zeta-\pi^*b,\widehat \nu-\pi^*s)|\!|_{L^2(T_0,T_0+1)}+|s|).
	\end{align*}	
	
	Our assumption on the exponential decay of $(A,u)$ over the mixed end and the properties of the map $\bQ$ imply that 
	\[
	  |\!|(\zeta-\pi^*b,\nu-\pi^*s)|\!|_{L^2_{k,\delta}(T_0+1,\infty)}\leq C(|\!|(\widehat \zeta-\pi^*b,\widehat \nu-\pi^*s)|\!|_{L^2_{k,\delta}(T_0+1,\infty)}+|s|).
	\]
	In particular, from the previous two inequalities and \eqref{fixing-T}, we conclude that
	\begin{align*}
	  |\!|(\zeta,\nu)&|\!|_{E^k_{(A,u)}(T_0+1,\infty)}
	  \leq C( |\!|\mathcal D_{(A_q,u_q)}(\widehat \zeta,\widehat \nu)|\!|_{L^2_{k-1,\delta}(T_0,\infty)}+|\!|(\widehat \zeta-\pi^*b,\widehat \nu-\pi^*s)|\!|_{L^2(T_0,T_0+1)}+|s|)\\
	  &\leq C( C_0|\!|\mathcal D_{(A,u)}(\zeta,\nu)|\!|_{L^2_{k-1,\delta}}+\epsilon|\!|(\zeta,\nu)|\!|_{E^k_{(A,u)}(T_0,\infty)}+|\!|\zeta|\!|_{L^2(X_{T_0+1})}+|\!|\nu|\!|_{L^2(U_{T_0+1})}+|s|).
	\end{align*}	
	By picking $\epsilon$ small enough, we may rearrange the terms and use \eqref{mixed-op-reg-comp-sub} to remove $\epsilon|\!|(\zeta,\nu)|\!|_{E^k_{(A,u)}(T_0,\infty)}$ from the above inequality. In summary, we have
	\begin{equation}\label{control-mixed-end-compact-part}
	  |\!|(\zeta,\nu)|\!|_{E^k_{(A,u)}(T_0+1,\infty)}\leq C( |\!|\mathcal D_{(A,u)}(\zeta,\nu)|\!|_{L^2_{k-1,\delta}}+|\!|\zeta|\!|_{L^2(X_{T_0+2})}+|\!|\nu|\!|_{L^2(U_{T_0+2})}+|s|).
	\end{equation}
	We obtain a similar inequality for the mixed end associated to $Y'$.

	Non-degeneracy of $\alpha\in \fC_G$, $\beta\in \fC_S$  and standard results about solutions of the ASD and the CR equations on cylinders
	imply that
	\begin{align*}
	  |\!|\zeta|\!|_{L^2_k(-\infty,-T_0-1)\times Y_\#}+|\!|\nu&|\!|_{L^2_k(T_0+1,\infty)\times [-1,1]}\leq \\&C( |\!|\mathcal D_{(A,u)}(\zeta,\nu)|\!|_{L^2_{k-1,\delta}}+|\!|\zeta|\!|_{L^2(X_{T_0+2})}+|\!|\nu|\!|_{L^2(U_{T_0+2})}+|s|).
	\end{align*}
	In fact, this inequality can be verified following a similar strategy analogous to \eqref{control-mixed-end-compact-part}. 
	Combining this inequality, \eqref{control-mixed-end-compact-part} and its counterpart for $Y'$ gives us the desired result after replacing $T_0$ with $T_0+2$.
	The proof of the analogous result for the adjoint operator 
	$\mathcal D^*_{(A,u)}$ where we replace Theorem \ref{Fred-mixed-op} and Proposition \ref{inv-mixed-op-const-sol} with the corresponding result for the adjoint operator (see Remark \ref{adjoint-Fred}).
\end{proof}

As a consequence of Lemma \ref{mixed-op-reg-lemma}, the operators $\mathcal D_{(A,u)}$ and $\mathcal D_{(A,u)}^*$ have finite dimensional kernels and closed images. Moreover, in order to show that the cokernel of $\mathcal D_{(A,u)}$ is finite dimensional, it suffices to show that for $k=1$, the cokernel of $\mathcal D_{(A,u)}$ can be identified with the kernel of $\mathcal D_{(A,u)}^*$. An element of the cokernel of $\mathcal D_{(A,u)}$ in this case is given by $(\mu,\xi,z)$ such that for any $(\zeta,\nu)\in E^1_{(A,u)}$, we have 
\begin{equation}\label{cokernel-identity}
	\langle (\mu,\xi,z),\mathcal D_{(A,u)}(\zeta,\nu)\rangle_{L^2}:=2\int_{X}\langle (\mu,\xi),\mathcal D_A(\zeta)\rangle+\int_{U_+}\langle z,\mathcal D_u(\nu)\rangle=0.
\end{equation}
Moreover, $\xi$, $\mu$ and $z$ belong to $L^2_{loc}$, the restrictions of $\xi$ and $\mu$ to $(-\infty,-3]\times Y_\#$ and $z$ to $[3,\infty)\times [-1,1]$ are in $L^2$, the restrictions of $\xi$, $\mu$ and $z$ to the mixed ends are in $L^2_{-\delta}$. In particular, an element of this space is allowed to have an exponential growth by a controlled quantity over the mixed ends. We included a factor $2$ in our convention for the $L^2$ pairing so that after integration by parts the boundary terms behave in the desired form.

Theorem \ref{Fred-mixed-op} and standard results on Fredholm theory of the adjoints of linearized ASD and CR operators imply that $(\mu,\xi,z)$ is in fact in $L^2_{1,loc}$ and properties (i) (for $k=1$), (iii) and (iv) of Definition \ref{K-A-u} hold for $(\mu,\xi,z)$. The proof of property (i) of Definition \ref{K-A-u} for $(\mu,\xi,z)$ uses the fact that $\alpha$ and $\beta$ are non-degenerate elements of $\fC_G$ and $\fC_S$. Now applying \eqref{cokernel-identity} to $\zeta$ and $\nu$ which are compactly supported in the interior of $X$ and $U_+$ implies that $\mathcal D_{(A,u)}^*(\mu,\xi,z)=0$. In particular, the inequality in \eqref{Fred-ineq} in Theorem \ref{Fred-mixed-op} can be used to show that the restrictions of $\xi$, $\mu$ and $z$ to
the mixed ends are in fact in $L^2_{1,-\delta}$. So, we can apply Corollary \ref{cyl-cokernel} to the restrictions of $(\mu,\xi,z)$ to the mixed ends and conclude that the following holds. There are $\fb\in \mathcal H^1_h(Y;B)$ and $\fb'\in \mathcal H^1_{h'}(Y';B')$ such that 
\[
  \mu-\frac{1}{2}(d\theta \wedge \pi^*(\fb)+*_3\pi^*(\fb))|_{Y\times [2,\infty)}\hspace{1cm}\text{and}\hspace{1cm} \mu-\frac{1}{2}(d\theta \wedge \pi^*(\fb')+*_3\pi^*(\fb'))|_{Y'\times (-\infty,-2])}
\] 
have finite $L^2_{1,\delta}$ norms. Let $\fs$ and $\fs'$ be tangent vectors to $\mathcal M(\Sigma,F)$ at the points $q$ and $q'$ given by restriction of $\fb$ and $\fb'$ to the boundary. Then 
\[z-\pi^*(\fs)|_{[0,2]\times [2,\infty)}\hspace{1cm} \text{and} \hspace{1cm} z-\pi^*(\fs')|_{[0,2]\times (-\infty,-2])}\] 
also have finite $L^2_{1,\delta}$ norms. The Stokes' thoerem together with $\mathcal D_{(A,u)}^*(\mu,\xi,z)=0$ shows that for an arbitrary element $(\zeta,\nu)$ of $E^1_{(A,u)}$ which is asymptotic to $(b,s)$ on the mixed end associated to $Y$ and is asymptotic to $(b',s')$ on the mixed end associated to $Y'$, we have
\[
   \langle (\mu,\xi,z),\mathcal D_{(A,u)}(\zeta,\nu)\rangle_{L^2}=\langle (b,s),(\fb,\fs)\rangle_{L^2}-\langle (b',s'),(\fb',\fs')\rangle_{L^2}.
\]
In particular, \eqref{cokernel-identity} implies that $(\fb,\fs)=0$ and $(\fb',\fs')=0$. Thus, property (ii) of Definition \ref{K-A-u} holds for $(\mu,\xi,z)$, and hence $(\mu,\xi,z)\in K^1_{(A,u)}$. In summery, cokernel of $\mathcal D_{(A,u)}$ can be identified with the kernel of $\mathcal D_{(A,u)}^*$. In particular, it is finite dimensional. This completes the proof of the claim that $\mathcal D_{(A,u)}$ is Fredholm. Similarly, one can show that the cokernel of $\mathcal D_{(A,u)}^*$ can be identified with the kernel of $\mathcal D_{(A,u)}$ and $\mathcal D_{(A,u)}^*$ is Fredholm.

\subsection{Proof of Proposition \ref{cons-map-ind}}\label{cons-map-ind-subs}

The goal of this subsection is to show that the mixed operator $\mathcal D_{(A_\alpha,u_\alpha)}$ associated to a constant solution $(A_\alpha,u_\alpha)$ is an isomorphism. We start with the following general result about the elements of $E^1_{(A_\alpha,u_\alpha)}$. 

\begin{lemma}
	For any $(\zeta,\nu)\in E^1_{(A_\alpha,u_\alpha)}$, the expression
	\begin{equation}\label{lin-top-energy}
		\int_{X}\tr(d_{A_\alpha}^{h,h'}\zeta\wedge d_{A_\alpha}^{h,h'}\zeta)+\int_{U_+}\Omega(d \nu, d\nu)
	\end{equation}
	vanishes, where 
	\begin{equation}\label{pert-ex-der}
		d_{A_\alpha}^{h,h'}\zeta:=d_{A_\alpha}\zeta+*_3{\rm Hess}_{B_\alpha}h(\zeta_t)+
		*_3{\rm Hess}_{B_\alpha'}h'(\zeta_t').
	\end{equation}
	with $B_\alpha$ (resp. $B_\alpha'$) being the the restrictions of $A_\alpha$ to $\{t\}\times Y_0\subset X$ (resp. $\{t\}\times -Y_0'\subset X$), which is independent of $t$. The term $d\nu$ in \eqref{lin-top-energy} denotes the exterior derivative of 
	$\nu:U_+ \to T_\alpha\mathcal M(\Sigma,F)$.
\end{lemma}
In \eqref{pert-ex-der},  $*_3{\rm Hess}_{B_\alpha}h(\zeta_t)$ and $*_3{\rm Hess}_{B_\alpha'}h'(\zeta_t')$ are defined as in \eqref{Linearized}, and in what follows, they are respectively denoted by $H(\zeta_t)$ and $H(\zeta_t')$.
\begin{proof}
	Stokes theorem and the decay constraints on $\nu$, given as part of the definition of $E^1_{(A_\alpha,u_\alpha)}$, imply that
        \begin{align}
        \int_{U_+}\Omega(d \nu, d\nu)&=
        \int_{U_\partial}\Omega (\nu, d \nu)+\int_{\eta_+\sqcup\eta_+'} 
        	\Omega(\nu, d \nu)\nonumber\\
        	&=\int_{-\infty}^{\infty}d\theta\int_{\Sigma}\tr(\nu(0,\theta) \wedge \partial_\theta \nu(0,\theta))\label{dAh-s}
        \end{align}
        The second identity is due to the Lagrangian boundary condition satisfied by $\nu$. We shall show 
        that the contribution from the first integral in \eqref{lin-top-energy} cancels out the integral in \eqref{dAh-s}.

        We decompose the domain of the first integral in \eqref{lin-top-energy} into $ \R \times Y_0$, $ \R \times -Y'_0$ 
        and $U_-\times \Sigma$, and study the contributions from each of them separately. 
        The restriction of $\zeta$ to $\{t\}\times Y_0$ has the form
        \begin{equation}\label{decom-R-Y}
          \zeta=\zeta_t+\phi_tdt
        \end{equation}
        where $\phi_t$ is a 0-form with values in $E$. Therefore, we have
        \begin{equation}\label{dAh}
        	d_{A_\alpha}^{h,h'}\zeta=d_{B_\alpha}\zeta_t+*_3{\rm Hess}_{B_\alpha}h(\zeta_t)+dt\wedge \(\frac{d \zeta_t}{dt}-d_{B_\alpha}\phi_t\)
        \end{equation}
        where $d_{B_\alpha}\zeta_t$ and $d_{B_\alpha}\phi_t$ denote the three dimensional exterior derivatives of 
        $\zeta_t$ and $\phi_t$ with respect to $B_\alpha$. Identity \eqref{dAh} implies that
        \begin{equation}\label{dAh-1-1}
          \tr(d_{A_\alpha}^{h,h'}\zeta\wedge d_{A_\alpha}^{h,h'}\zeta)=2dt\wedge \tr\((\frac{d \zeta_t}{dt}-d_{B_\alpha}\phi_t)
          \wedge(d_{B_\alpha}\zeta_t+H(\zeta_t))\).
        \end{equation}
        Using Stokes theorem and Lemma \ref{H-prop}, we have
        \begin{equation}\label{dAh-1-2}
        	2\int_{Y_0} \tr\(\frac{d \zeta_t}{dt}\wedge(d_{B_\alpha}\zeta_t+H(\zeta_t))\)=
        	\frac{d}{dt}\int_{Y_0} \tr\(\zeta_t \wedge(d_{B_\alpha}\zeta_t+H(\zeta_t))\)
        	-\int_{\partial Y_0}\tr(\zeta_t\wedge \frac{d}{dt}\zeta_t).
        \end{equation}
        Stokes theorem, vanishing of $H(\zeta_t)$ on $\partial Y$ and Proposition \ref{linearized-3-man-Lag} give        \begin{align}
        	2\int_{Y_0} \tr\(d_{B_\alpha}\phi_t\wedge(d_{B_\alpha}\zeta_t+H(\zeta_t))\)
	&=2\int_{\partial Y_0} \tr\(\phi_t\wedge d_{B_\alpha}\zeta_t\)\nonumber
        	\\
        	&=\int_{\partial Y_0} \tr\(\phi_t\wedge d_{B_\alpha}\zeta_t\)+
	\int_{\partial Y_0} \tr\(\zeta_t\wedge d_{B_\alpha}\phi_t\).\label{dAh-1-3}
        \end{align}

        We can use \eqref{dAh-1-1}, \eqref{dAh-1-2} and \eqref{dAh-1-3} 
        and the exponential convergence of $\zeta$ to an element of $\mathcal H^1_h(Y;B)$ to conclude that
        \begin{equation}\label{dAh-1}
        	\int_{\R\times Y_0}\tr(d_{A_\alpha}^{h,h'}\zeta\wedge d_{A_\alpha}^{h,h'}\zeta)=
        	\int_{\partial (\R\times Y_0)}\tr(\zeta\wedge d_{A_\alpha}\zeta).
        \end{equation}
        A similar argument shows that
        \begin{equation}\label{dAh-2}
        	\int_{\R\times -Y_0'}\tr(d_{A_\alpha}^{h,h'}\zeta\wedge d_{A_\alpha}^{h,h'}\zeta)=
        	\int_{\partial (\R\times -Y_0')}\tr(\zeta\wedge d_{A_\alpha}\zeta).
        \end{equation}
        Since $A_\alpha$ is flat on $U_-\times \Sigma$, Stokes theorem implies that
        \begin{equation}\label{dAh-3}
        	\int_{U_-\times \Sigma}\tr(d_{A_\alpha}^{h,h'}\zeta\wedge d_{A_\alpha}^{h,h'}\zeta)=
        	\int_{\partial (U_-\times \Sigma)}\tr(\zeta\wedge d_{A_\alpha}\zeta)
        \end{equation}
        By adding up the above three equations, we have the following simple form for the first term in \eqref{lin-top-energy}:
        \begin{align}
        	\int_{X}\tr(d_{A_\alpha}^{h,h'}\zeta\wedge d_{A_\alpha}^{h,h'}\zeta)
        	&=\int_{U_\partial\times \Sigma}\tr(\zeta\wedge d_{A_\alpha}\zeta)\nonumber\\
        	&=-\int_{-\infty}^\infty d\theta\int_\Sigma\tr(a(\theta)\wedge \partial_\theta a(\theta))\label{dAh-123}
        \end{align}
        To clarify the notation in the second line, note that the restriction of $\zeta$ to $U_\partial\times \Sigma$ 
        has the form $\zeta=a(\theta)+\psi(\theta)d\theta$ where $a(\theta)$ and $\psi(\theta)$ 
        are respectively 1- and 0-forms on $\Sigma$ with values in $F$. The second identity is a consequence of the
         assumption that for each $\theta$ the 2-dimensional exterior derivative $d_{\alpha}a(\theta)$ vanishes. 
        Since $a(\theta)-v(0,\theta)$ is $d_{\alpha}$-exact, identities \eqref{dAh-s} and 
        \eqref{dAh-123} imply that \eqref{lin-top-energy} vanishes.
\end{proof}

Now we assume that $(\zeta,\nu)$ belongs to the kernel of the mixed operator $\mathcal D_{(A_\alpha,u_\alpha)}$: 
\[
  d^*_{A_\alpha}\zeta=0,\hspace{1cm} d_{A_\alpha}^+\zeta+(*_3{\rm Hess}_{B_\alpha}h(\zeta_t))^+
	+(*_3{\rm Hess}_{B_\alpha'}h'(\zeta_t'))^+=0,\hspace{1cm}\partial_\theta \nu- J_{s,\theta}\partial_s\nu=0.
\]
Thus we have
\[
  \tr(d_{A_\alpha}^{h,h'}\zeta\wedge d_{A_\alpha}^{h,h'}\zeta)=\vert d_{A_\alpha}^{h,h'}(\zeta)\vert^2 \dvol_X,
\]
and
\[
  \Omega(d \nu, d \nu)(s,\theta)=2\vert\partial_s\nu(s,\theta)\vert^2 ds\wedge d\theta.
\]
We conclude from these identities and the vanishing of \eqref{lin-top-energy} that
\begin{equation}\label{der-vanish}
  d_{A_\alpha}^{h,h'}\zeta=0,\hspace{1cm} d\nu=0.
\end{equation}
In particular, $\nu$ is constant, which implies that $\nu=0$ due to its decay on the symplectic end. In particular, $\zeta$ has exponential decay on the mixed ends associated to $Y$ and $Y'$.

\begin{prop}
	For $\zeta$ as above, there is an $L^2_{1,\delta}$ section $\eta$ of the bundle $V$ over $X$ such that $\zeta=d_{A_\alpha}\eta$.
\end{prop}

\begin{proof}
	We construct $\eta$ on the subspaces $ \R \times Y_0$, $ \R \times -Y_0'$ and $U_-\times \Sigma$ separately. 
	For $(\tau,y)\in \R\times Y_0$, let
	\[
	  \eta_1(\tau,y):=\int_{-\infty}^\tau \phi_t(y) dt,
	\]
	where $\phi_t$ is given in \eqref{decom-R-Y}. Vanishing of $d_{A_\alpha}^{h,h'}\zeta$ implies that 
	$\frac{d \zeta_t}{dt}=d_{B_\alpha}\phi_t$. This observation and the decay of $\zeta$ on the end that 
	$t\to -\infty$ imply that $d_{A_\alpha} \eta_1$ is equal to $\zeta$ over the subspace $\R\times Y_0$. 
	Similarly, we define $\eta_2$ on $ \R \times -Y_0'$. 
	On the subspace $U_-\times \Sigma$, we have $d_{A_\alpha} \zeta=0$. 
	The element of $\mathcal H^1(\Sigma;\alpha)$ represented by the 1-form $\zeta(s,\theta)$ for $(s,\theta)\in U_-$ 
	is independent of the choice of $(s,\theta)$. In particular, this cohomology class is trivial 
	because of the decay assumption on $\zeta$ as $s\to -\infty$. Thus for any $(s,\theta)$, 
	there is a unique $\eta_3(s,\theta)$ such that $d_\alpha\eta_3(s,\theta)=\zeta(s,\theta)$. 
	It is also straightforward to see $d_{A_\alpha}\eta_3=\zeta$ because $d_{A_\alpha}\eta_3-\zeta$ is 
	$d_{A_{\alpha}}$-closed and its restriction to $\{(s,\theta)\}\times \Sigma$ for any $(s,\theta)$ vanishes.
	Since the restriction of $A_\alpha$ to the overlaps of $ \R \times Y$ and $U_-\times \Sigma$ 
	(resp. $ \R \times Y$ and $U_-\times \Sigma$) is still irreducible, the sections $\eta_1$ (resp. $\eta_2$) and 
	$\eta_3$ agree on the overlap regions. In particular, we obtain a section $\eta$ of $V$ over $X$ such that 
	$\zeta=d_{A_\alpha}\eta$. It is straightforward to check that $\eta$ is in $L^2_{1,loc}$. For any $\theta\in [2,\infty)$, if $\eta_\theta$ is the restriction of $\eta$ to $\{\theta\}\times Y$, then $d_{B_\alpha}\eta_\theta$ equals the restriction of 
	$\zeta$ to $\{\theta\}\times Y$. Since $B_\alpha$ is irreducible, we may conclude that $\eta$ on the mixed end associated to $Y$ belongs to $L^2_{1,\delta}$ because $\zeta$ satisfies a similar exponential decay. Similar argument shows the 
	decay of $\zeta$ on the mixed end associated to $Y'$ and the gauge theoretical end.
\end{proof}

The identity $\zeta=d_{A_\alpha}\eta$, Stokes theorem and the boundary condition $*\zeta|_{U_\partial \times \Sigma}=0$ implies that
\[
  \int_{X}\langle \zeta ,\zeta\rangle= \int_{X}\langle d^*_{A_\alpha}d_{A_\alpha}\eta,\eta\rangle.
\]
Since $d^*_{A_\alpha}\zeta=0$, we conclude that $\zeta=0$. Thus, the kernel of $\mathcal D_{(A_\alpha,u_\alpha)}$ is trivial.

Next, we show that the cokernel of the operator $\mathcal D_{(A_\alpha,u_\alpha)}$ is trivial. Let $(\mu,\xi,z)\in K^l_{(A_\alpha,u_\alpha)}$ belongs to the kernel of $\mathcal D^*_{(A_\alpha,u_\alpha)}$. This implies that $z$ is a map from $U_+$ to $\mathcal H^1(\Sigma;\alpha)$. These terms satisfy
\begin{equation}\label{coker-eq}
	d_{A_\alpha}\xi+d_{A_\alpha}^*\mu+{\rm Hess}_{B_\alpha}h(*_3\mu_t)+{\rm Hess}_{B_\alpha'}h'(*_3\mu_t')=0,\hspace{1cm}\partial_\theta(J_{s,\theta}z)-\partial_s z=0.
\end{equation}
Moreover, $(\mu,\xi,z)$ satisfy the conditions spelled out in Definition \ref{K-A-u}.

First we show that $d_{A_\alpha}\xi$ vanishes, which immediately implies that $\xi=0$, because $A_\alpha$ is an irreducible connection. In fact, we have the following identities for the $L^2$ norm of $d_{A_\alpha}\xi$:
\begin{align*}
	\int_{X}\langle d_{A_\alpha}\xi,d_{A_\alpha}\xi \rangle
	&=\int_{X}\tr \(d_{A_\alpha}\xi\wedge *d_{A_\alpha}^*\mu \)+\int_{-\infty}^\infty\int_{Y_0}\tr \(d_{A_\alpha}\xi_t \wedge *_3{\rm Hess}_{A_t}h(*_3\mu_t)\)dt \\
	&+\int_{-\infty}^\infty\int_{-Y'_0}\tr \(d_{A_\alpha}\xi_t \wedge *_3{\rm Hess}_{A_t'}h'(*_3\mu_t') \)dt \\
	&=\int_{X}\tr \(d_{A_\alpha}\xi\wedge d_{A_\alpha}\mu \)-\int_{-\infty}^\infty\int_{Y_0}\tr \(\xi_t \cdot d_{A_\alpha}(*_3{\rm Hess}_{A_t}h(*_3\mu_t))\)dt \\
	&-\int_{-\infty}^\infty\int_{-Y'_0}\tr \(\xi_t \cdot d_{A_\alpha}(*_3{\rm Hess}_{A_t'}h'(*_3\mu_t')) \)dt \\
	&=\int_{U_\partial \times \Sigma}\tr \(\xi\cdot  d_{A_\alpha}\mu \)-\int_{X}\tr \(\xi\cdot  [F_{A_\alpha}, \mu] \)+\int_{-\infty}^\infty\int_{Y_0}\tr \(\xi_t \cdot [F_{A_\alpha}, *_3\mu_t]\)dt \\
	&+\int_{-\infty}^\infty\int_{-Y'_0}\tr \(\xi_t \cdot [F_{A_\alpha}, *_3\mu_t'] \)dt. 
\end{align*}
We use Stokes theorem in the last two identities, and Proposition \ref{linearized-3-man-Lag} is used in the third identity. The assumption on the restriction of $\mu$ to $U_\partial \times \Sigma$ and the assumption that $(A_\alpha,u_\alpha)$ is a constant solution to the mxied equation imply that the last expression is zero. Thus $\xi$ vanishes.

We introduced $d_{A_\alpha}^{h,h'}$ in \eqref{pert-ex-der}, as a deformation of the exterior derivative operator $d_{\alpha}$ acting on sections of $\Lambda^1\otimes V$, and now we define a similar operator for sections of $\Lambda^2\otimes V$. For a section $\kappa$ of $\Lambda^2\otimes V$ over $X$, let
\begin{equation}\label{dAh-2form}
	d_{A_\alpha}^{h,h'}\kappa :=d_{A_\alpha}\kappa-*_3{\rm Hess}_{B_\alpha}h((\iota_{\partial_t}\kappa)_t)\wedge dt-*_3{\rm Hess}_{B_\alpha'}h'((\iota_{\partial_t}\kappa)'_t)\wedge dt,
\end{equation}
where $(\iota_{\partial_t}\kappa)_t$ is obtained by contracting $\kappa\vert_{\R\times Y_0}$ with respect to $\partial_t$ and then restricting it to $\{t\}\times Y_0$. The 1-form $(\iota_{\partial_t}\kappa)_t'$ is defined similarly by replacing $Y_0$ with $-Y_0'$. A straightforward calculation using Proposition \ref{linearized-3-man-Lag} shows that $d_{A_\alpha}^{h,h'}d_{A_\alpha}^{h,h'}\zeta=0$ for a section of $\Lambda^1\otimes V$. As a consequence of the first identity in \eqref{coker-eq} and the vanishing of $\xi$, $d_{A_\alpha}^{h,h'}\mu$ vanishes.

\begin{lemma}
	There is $(\zeta,\nu)\in E^l_{(A_\alpha,u_\alpha)}$ such that:
	\begin{equation}\label{primitive-mu-nu}
		2\mu=d_{A_\alpha}^{h,h'}\zeta,\hspace{1cm} zd\theta+J_{s,\theta}zds=d \nu.
	\end{equation}
\end{lemma}

This lemma allows us to conclude the triviality of the kernel of $\mathcal D^*_{(A_\alpha,u_\alpha)}$ as in the case of the kernel of $\mathcal D_{(A_\alpha,u_\alpha)}$ in the following way. On one hand, the expression in \eqref{lin-top-energy} vanishes for the pair $(\zeta,\nu)$ produced by the lemma because $(\zeta,\nu)\in E^l_{(A_\alpha,u_\alpha)}$. On the other hand, $d_{A_\alpha}^{h,h'}\zeta$ is self-dual and $d\nu$ at $(s,\theta)$ is a $(0,1)$-form with respect to $J_{s,\theta}$, and a similar argument as in the previous case shows that the expression in \eqref{lin-top-energy} is non-positive and it is equal to zero if and only if $\zeta$ and $\nu$ vanish. This shows that $(\mu,\xi,z)$ is trivial.

\begin{proof}
	The pair of $\mu$ and $\eta:=zd\theta+J_{s,\theta}zds$ satisfies:
	\begin{itemize}
		\item[(i)] $\mu \in L^2_{l,\delta}(X,\Lambda^2\otimes V)$ and  $\eta \in L^2_{k,\delta}(U_+,\Lambda^1\otimes T_\alpha\mathcal M(\Sigma,F))$;
		\item[(ii)] $d_{A_\alpha}^{h,h'}\mu=0$ and $d \eta=0$;
		\item[(iii)]  At any point $(0,\theta)\in U_\partial$, we have $2\iota_{\partial_\theta} \mu(0,\theta)$ is $d_\alpha$-closed and represents the same cohomology class as $\iota_{\partial_\theta} \eta(0,\theta)$.
	\end{itemize}
	We prove a more general result showing that for any $\mu$ and $\eta$ as above there is $(\zeta,\nu)\in E^l_{(A_\alpha,u_\alpha)}$ such that \eqref{primitive-mu-nu} holds.

	The transversality of the Largrangians $L(Y,E)$ and $L(Y',E')$ implies that there are $c\in T_{\alpha}L(Y,E)$ and $c' \in T_{\alpha}L(Y',E')$ such that
	\begin{equation}\label{eta-int-mathicng-line}
	  \int_{-\infty}^\infty \iota_{\partial_\theta}\eta(0,\theta) d\theta=c-c'.
	\end{equation}
	Note that our assumption on $z$ implies that the integral on the left exists. Then $c$ and $c'$ determine $b\in \mathcal H^1_h(Y;B_\alpha)$ and $b'\in \mathcal H^1_{h'}(Y';B_\alpha')$. 
	We define a section $\zeta_0$ of $\Lambda^1\otimes V$ which is supported in 	
	$(2,\infty)\times Y$, and over this subspace of $X$ it is given by 
	\[
	  \zeta_0(\theta,y)=2 f(\theta) \(\frac{1}{2}\pi^*(b)- \int_\theta^\infty(\iota_{\partial_\theta}\mu(y,\tau)) d\tau\),
	\]
	where $f:\R\to \R$ is a smooth function that is equal to $1$ on $(3,\infty)$ and vanishes  on $(-\infty, 5/2)$. Similarly define
	\[
	  \nu_0(s,\theta)= f(\theta) \(c- \int_\theta^\infty \iota_{\partial_\theta}\eta (s,\tau) d\tau\).
	\]	
	Then $(\zeta_0,\nu_0)\in E^l_{(A_\alpha,u_\alpha)}$. Moreover, the identities in \eqref{coker-eq} and the decay assumptions on $\mu$ and $z$ imply that over the space $Y\times(3,\infty)$ we have $2\mu=d_{A_\alpha}^{h,h'}\zeta_0$ and 
	$\eta=d \nu_0$. Let $f':\R\to \R$ be another bump function that equals $1$ on $(-\infty,-3)$ and vanishes over $(-5/2,\infty)$. Define 
	\[
	  \zeta_0'(\theta,y)=2 f'(\theta) \(\frac{1}{2}\pi^*(b')+\int_{-\infty}^\theta(\iota_{\partial_\theta}\mu(y,\tau)) d\tau\),\hspace{1cm}\nu_0'(s,\theta)= f'(\theta) \(c'+\int_{-\infty}^\theta \iota_{\partial_\theta}\eta (s,\tau) d\tau\).
	\]	
	We again have $(\zeta_0',\nu_0')\in E^l_{(A_\alpha,u_\alpha)}$. Moreover, the pair  
	\[
	  \mu_0:= \mu-d_{A_\alpha}^{h,h'}( \zeta_0+\zeta_0'),\hspace{1cm} \eta_0:=\eta-d(\nu_0+\nu_0'),
	\]
	satisfies properties (i)-(iii) stated above, $\mu_0$ vanishes on $(3,\infty)\times Y$ and $(-\infty,-3)\times Y'$, and $\eta_0$ vanishes on $[0,2]\times (3,\infty)$ and $[0,2]\times (-\infty,-3)$.
	Moreover, there is a section $\lambda$ of $F$ over $\Sigma\times [-2,0] $ such that for any $x\in \Sigma$, $s\in [-2,0]$ and $s'\in [0,2]$, we have
	\begin{equation}\label{mu-0-eta-0}
	  \int_{-\infty}^\infty \iota_{\partial_\theta}\mu_0|_{\{(s,\theta)\}\times \Sigma} d\theta=d_\alpha\lambda(x,s),\hspace{1cm}\int_{-\infty}^\infty \iota_{\partial_\theta}\eta_0(s',\theta) d\theta=0.
	\end{equation}
	The second identity in \eqref{mu-0-eta-0} for $s'=0$ follows readily from \eqref{eta-int-mathicng-line}. We obtain the identity for all values of $s'$ using the assumption that $\eta_0$ is closed. The first identity in \eqref{mu-0-eta-0}
	for $s=0$ follows from the second one and property (iii) of $(\mu_0,\eta_0)$. This can be extended to all values of $s$ using $d_{A_\alpha}^{h,h'}\mu=0$ and the Stokes' theorem.

	Next, we modify $\mu_0$ and $\eta_0$ such that in addition to the properties mentioned in the previous paragraph, they vanish in a neighborhood of the matching line $U_\partial$. 
	Fix a bump function $f_\partial:\R\to\R$ that is equal to $1$ on the interval $[-1,1]$ and vanishes outside the interval $[-2,2]$. Let also 
	$h:\R\to \R$ be a compactly supported bump function with support in $[-2,2]$ whose integral over $\R$ equals $1$. Define
	\[
	  \zeta_\partial(x,s,\theta):=2f_\partial(s) \int_{-\infty}^\theta-h(\tau)d_\alpha\lambda(x,s)+\iota_{\partial_\theta}\mu_0(x,s,\tau) d\tau, \,\,\,\,\nu_\partial(s,\theta):=f_\partial(s) \int_{-\infty}^\theta\iota_{\partial_\theta}\eta_0 (s,\tau)d\tau
	\]
	where $(x,s,\theta)\in \Sigma\times [-2,0]\times[-3,3]$ and $(s,\theta)\in [0,2]\times [-3,3]$. Extend $\zeta_\partial$ in the trivial way to the rest of $X$, and extend $\nu_\partial$ in the trivial way to the rest of $U_+$. Then 
	we can see $(\zeta_\partial,\nu_\partial)$ belongs to $E^l_{(A_\alpha,u_\alpha)}$ using the identities in \eqref{mu-0-eta-0}. From the definition, it is clear that the support of $\zeta_\partial$ is contained in $\Sigma\times [-2,0]\times[-3,3]$ 
	and the support of $\nu_\partial$ is contained in $ [0,2]\times [-3,3]$. 
	If we define
	\[
	  \mu_1:= \mu_0-d_{A_\alpha}^{h,h'}( \zeta_\partial),\hspace{1cm} \eta_1:=\eta_0-d\nu_\partial,
	\]	
	then $(\mu_1,\eta_1)$ satisfies (i)-(iii), $\mu_1$ vanishes on $(3,\infty)\times Y$, $ (-\infty,-3)\times \times$ and the neighborhood $\Sigma\times [-1,0]\times U_\partial$ of the matching line, and $\eta_1$ vanishes on $[0,2]\times (3,\infty)$, 
	$[0,2]\times (-\infty,-3)$ and the neighborhood $[0,1]\times U_\partial$ of the matching line.

	The support of $\mu_1$ is contained in an open subspace $K_-$ of $X$ which is diffeomorphic to $(-\infty,3)\times Y_\#$ (see Figure \ref{support-K}). 
	We may assume that the diffeomorphism from $K_-$ to $(-\infty,3)\times Y_\#$ is given by the identity map on
	 $(-\infty,-3)\times Y_\#$, $(-\infty,3)\times Y_0$ and $(-\infty,3)\times -Y_0'$.  We use this diffeomorphism to 
	 identify $K_-$ with $(-\infty,3)\times Y_\#$ and for any $(t,y_\#)\in (-\infty,3)\times Y_\#$, we define 
	 \[
	   \zeta_\#'(t,y_\#):=2\int_{-\infty}^t(\iota_{\partial_\theta}\mu_1(\tau,y_\#)) d\tau
	 \]
	Then $d_{A_\alpha}^{h,h'}\zeta_\#'=\mu_1$. In particular, the restriction of $\zeta_\#'$ to $\{3\}\times Y_\#$ is in the kernel of $d_{B_\alpha}+*_3 {\rm Hess}_{B_\alpha}h+*_3{\rm Hess}_{B_\alpha'}h'$. Non-degeneracy of $\alpha\in \fC_G$ implies that 
	there is a section $\phi$ of $E_\#$ such that $\zeta_\#'(3,\cdot)=d_{B_\alpha} \phi$. Let $f_\#:(-\infty,3)\to \R$ be a function which equals $0$ on $(-\infty,2)$ and equals $1$ in a neighborhood of $3$, and modify $ \zeta_\#'$ as 
	$ \zeta_\#= \zeta_\#'-d_{A_\alpha}(f_\#\phi)$. Then $\zeta_\#$ vanishes in a neighborhood of $\{3\}\times Y_\#$ and we may extend it to $X$ trivially. Now $(\zeta_\#,0)\in E^l_{(A_\alpha,u_\alpha)}$ and $d_{A_\alpha}^{h,h'}\zeta_\#=\mu_1$.
	Similarly, we may find $\nu_\#$ such that $(0,\nu_\#)\in E^l_{(A_\alpha,u_\alpha)}$ and $d\nu_\#=\eta_1$. Consequently, the pair 
	\[
	  \zeta:=\zeta_0+\zeta_0'+\zeta_\partial+\zeta_\#,\hspace{1cm}\nu:=\nu_0+\nu_0'+\nu_\partial+\nu_\#,
	\]
	gives the desired claim.	
	\begin{figure}
	\begin{center}
	\begin{tikzpicture}[thick]
\pic [scale=0.15](lower) at (-1.15,2) {handle};
\pic [scale=0.15](lower) at (-1.65,2.35) {hole};
\pic [scale=0.15](lower) at (-1.15,-2) {handle};
\pic [scale=0.15](lower) at (-1.65,-1.8) {hole};
\pic [scale=0.15](lower) at (-1.65,-1.4) {hole};
\draw (0,3)  --  (0,-3)
	(-1,3)  --  (-1,1.5)
	(-1,1.5) coordinate (-left) 
	to [out=-90, in=0] (-2,.5)
	 (-2,.5) -- (-3.5,0.5)
	 (-2,-.5) -- (-3.5,-0.5)
	(-2,-.5) coordinate (-left) 
	to [out=0, in=90] (-1,-1.5)
	to (-1,-3);
\draw (1,3)  --  (1,1.5)
	(1,1.5) coordinate (-left) 
	to [out=-90, in=180] (2,.5)
	 (2,.5) -- (3.5,0.5)
	 (2,-.5) -- (3.5,-0.5)
	(2,-.5) coordinate (-left) 
	to [out=180, in=90] (1,-1.5)
	to (1,-3);
\shade[left color=lime,right color=lime] (-1,-1.5)--(-1,1.5)  arc[start angle=40, end angle=-40,radius=2.3] -- (-1,-1.5);
\shade[left color=lime,right color=lime] (-1,-1.5)--(-.75,-1)--(-.75,1)--(-1,1.5) coordinate (-left) 
	to [out=-90, in=0] (-2,.5)--(-2,-.5)coordinate (-left) 
	to [out=0, in=90] (-1,-1.5);
\shade[left color=lime,right color=lime](-3.5,0.5)--(-3.5,-0.5)--(-2,-0.5)--(-2,0.5);
\shade[left color=yellow,right color=yellow] (1,-1.5)--(1,1.5)  arc[start angle=140, end angle=220,radius=2.3] -- (1,-1.5);
\shade[left color=yellow,right color=yellow] (1,-1.5)--(.75,-1)--(.75,1)--(1,1.5) coordinate (-left) 
	to[out=-90, in=180] (2,.5)--(2,-.5)coordinate (-left) 
	to [out=180, in=90] (1,-1.5);
\shade[left color=yellow,right color=yellow](3.5,0.5)--(3.5,-0.5)--(2,-0.5)--(2,0.5);
\node at (-1.25,0){$K_-$};	
\node at (1.25,0){$K_+$};	
\draw[dotted] (-.5,4.5)  --  (-0.5,3.5);	
\draw[dotted] (.5,4.5)  --  (0.5,3.5);	
\draw[dotted] (-.5,-4.5)  --  (-0.5,-3.5);	
\draw[dotted] (.5,-4.5)  --  (0.5,-3.5);	
\draw[dotted] (4,0)  --  (5,0)
		     (-4,0)  --  (-5,0)
		     (-1.5,2.6) -- (-1.5,3)
		     (-1.5,1.4) to [out=-90, in=0] (-2,.8)
		     to (-2.5,0.8)
		     (-1.5,-2.6) -- (-1.5,-3)
		     (-1.5,-1.4) to [out=90, in=0] (-2,-.8)
		     to (-2.5,-0.8);			     
	\end{tikzpicture}
	\end{center}
		\caption{ The support of $\mu_1$ is contained in the set $K_-$ and the support of $\eta_1$ is contained in $K_+$}
	\label{support-K}
\end{figure}
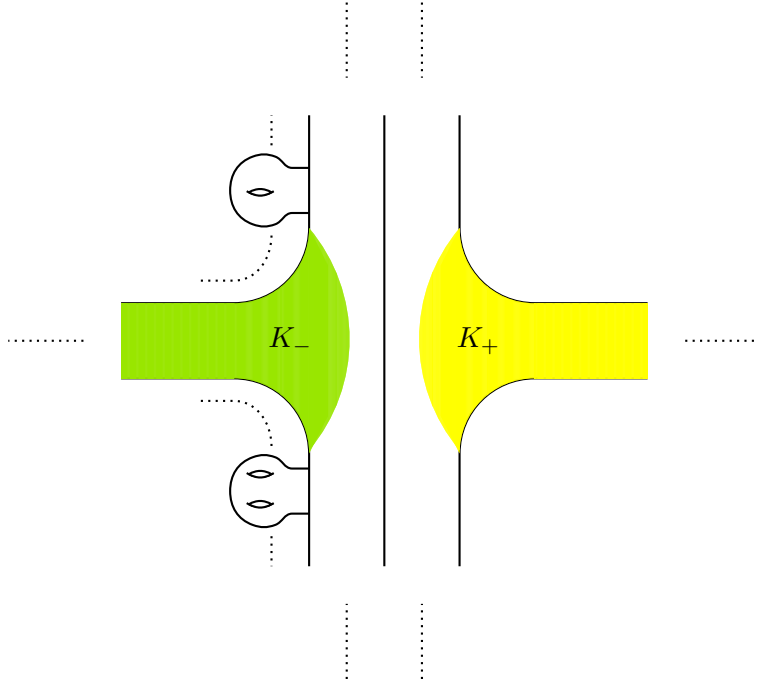
\end{proof}

\subsection{Mixed shifting}\label{mixed-shit-op}

We start this subsection by introducing a special type of mixed pairs.
\begin{definition}
	An element $(A,u)\in \bA(\alpha,\beta)$ is {\it symplectically constant} if $(A,u)$, restricted to the complement of $(-\infty,-3]\times Y_\#$, is equal to a constant pair $(A_\beta,u_\beta)$ associated to $\beta$. 
	In particular, the map $u$ is a constant map to $\beta$.
\end{definition}

\begin{lemma}\label{part-symp-constant}
	For any $(A,u)\in \bA(\alpha,\beta)$, there is a path from $(A,u)$ to a symplectically constant pair in $\bA(\alpha,\beta)$.
\end{lemma}

To verify the lemma, it is helpful to give a different parametrization of the $4$-manifold $X$ and the 2-dimensional domain $U_+$. Identify $\R \times [-1,1]$ in the standard way with the subspace of $\C$ given by numbers whose imaginary parts are in $[-1,1]$. There is a diffeomorphism $\Phi_+:U_+\to \([0,\infty)\times [-1,1]\)\setminus \{\pm \bi\}$ which satisfy the following conditions (see Figure \ref{special-quin-old-new-par}).
\begin{itemize}
	\item[(i)] The restriction of $\Phi_+$ to the subspace $[3,\infty)\times [-1,1]$ of $U_+$ is given by the identity map.
	\item[(ii)] On the subspace $[0,2]\times [2,\infty)$ of $U_+$, we have
	\[\Phi_+(s,\theta)=\bi+e^{-\theta+\pi \bi \frac{s-2}{4}}.\]
	\item[(iii)] For any $(s,\theta)\in U_+$, $\Phi_+(s,-\theta)$ is equal to the complex conjugate of $\Phi_+(s,\theta)$.
	In particular, on the subspace $[0,2]\times (-\infty,-2]$ of $U_+$, we have: 
	\[\Phi_+(s,\theta)=-\bi+e^{\theta+\pi \bi \frac{2-s}{4}}.\]
\end{itemize}
There is also a diffeomorphism $\Phi_-:X\to \left((-\infty,0]\times Y_\#\right)\setminus \left( \{0\}\times Y_0\cup \{0\}\times  -Y_0'\right)$ such that the following hold.
\begin{itemize}
	\item[(i)] The restriction of $\Phi_-$ to the subspace $U_-\times \Sigma$ of $X$ is given by
	\[\Phi_-(s,\theta,x)=(-\Phi_{+,s}(-s,\theta),\Phi_{+,\theta}(-s,\theta),x),\]
	where $\Phi_+(s,\theta)=(\Phi_{+,s}(s,\theta),\Phi_{+,\theta}(s,\theta))\in [0,\infty)\times [-1,1]$.
	\item[(ii)] On the subspace $\R\times Y_0$ of $X$, we have
	\[\Phi_-(\tau,y)=(f(\tau),y),\]
	where $f:\R\to (-\infty,0)$ is an increasing smooth function which is determined by the 
	restriction of $\Phi_-$ to the subspace $\eta_-\times \Sigma$ and satisfies
	\begin{equation*}
		f(\tau)=\left\{
		\begin{array}{ll}
			\tau&\tau\leq -3,\\
			-e^{-\tau}&\tau\geq 2.\\
		\end{array}
		\right.
	\end{equation*}
	\item[(iii)] On the subspace $\R\times -Y_0'$ of $X$, we have
	\[\Phi_-(\tau,y')=(f(\tau),y').\]
\end{itemize}

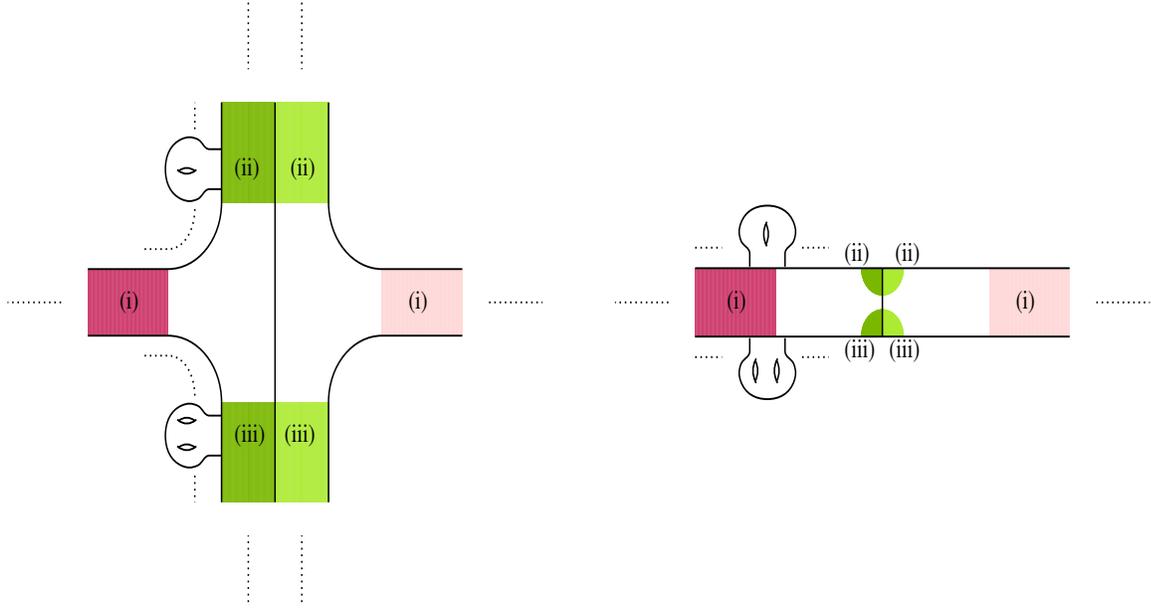
\begin{figure}
	\begin{minipage}{.5\linewidth}
	\resizebox{0.9\linewidth}{8cm}{
	\begin{tikzpicture}[thick]
\pic [scale=0.15](lower) at (-1.15,2) {handle};
\pic [scale=0.15](lower) at (-1.65,2.35) {hole};
\pic [scale=0.15](lower) at (-1.15,-2) {handle};
\pic [scale=0.15](lower) at (-1.65,-1.8) {hole};
\pic [scale=0.15](lower) at (-1.65,-1.4) {hole};
\shade[left color=black!20!lime,right color=black!20!lime](-1,3)--(-1,1.5)--(0,1.5)--(0,3);	
\shade[left color=white!20!lime,right color=white!20!lime](1,3)--(1,1.5)--(0,1.5)--(0,3);	
\shade[left color=black!20!lime,right color=black!20!lime](-1,-3)--(-1,-1.5)--(0,-1.5)--(0,-3);	
\shade[left color=white!20!lime,right color=white!20!lime](1,-3)--(1,-1.5)--(0,-1.5)--(0,-3);	
\shade[left color=white!25!purple,right color=white!25!purple](-3.5,0.5)--(-3.5,-0.5)--(-2,-0.5)--(-2,0.5);
\shade[left color=white!40!pink,right color=white!40!pink](3.5,0.5)--(3.5,-0.5)--(2,-0.5)--(2,0.5);
\draw (0,3)  --  (0,-3)
	(-1,3)  --  (-1,1.5)
	(-1,1.5) coordinate (-left) 
	to [out=-90, in=0] (-2,.5)
	 (-2,.5) -- (-3.5,0.5)
	 (-2,-.5) -- (-3.5,-0.5)
	(-2,-.5) coordinate (-left) 
	to [out=0, in=90] (-1,-1.5)
	to (-1,-3);
\draw (1,3)  --  (1,1.5)
	(1,1.5) coordinate (-left) 
	to [out=-90, in=180] (2,.5)
	 (2,.5) -- (3.5,0.5)
	 (2,-.5) -- (3.5,-0.5)
	(2,-.5) coordinate (-left) 
	to [out=180, in=90] (1,-1.5)
	to (1,-3);
\draw[dotted] (-.5,4.5)  --  (-0.5,3.5);	
\draw[dotted] (.5,4.5)  --  (0.5,3.5);	
\draw[dotted] (-.5,-4.5)  --  (-0.5,-3.5);	
\draw[dotted] (.5,-4.5)  --  (0.5,-3.5);	
\draw[dotted] (4,0)  --  (5,0)
		     (-4,0)  --  (-5,0)
		     (-1.5,2.6) -- (-1.5,3)
		     (-1.5,1.4) to [out=-90, in=0] (-2,.8)
		     to (-2.5,0.8)
		     (-1.5,-2.6) -- (-1.5,-3)
		     (-1.5,-1.4) to [out=90, in=0] (-2,-.8)
		     to (-2.5,-0.8);	
\node[left] at (3,0){(i)};	
\node[left] at (-2.4,0){(i)};	
\node[left] at (.9,2){(ii)};
\node[right] at (-.9,2){(ii)};	
\node[left] at (.9,-2){(iii)};
\node[right] at (-.9,-2){(iii)};		     
	\end{tikzpicture}}
	\end{minipage}
	\begin{minipage}{.5\linewidth}
		\resizebox{0.9\linewidth}{2.6cm}{
	\begin{tikzpicture}[thick]
\pic [scale=0.15,rotate=-90,yscale=1.1,xscale=0.85](lower) at (-2.15,.65) {handle};
\pic [scale=0.15](lower) at (-1.8,1) {thole};
\pic [scale=0.15,rotate=90,yscale=1.1,xscale=0.85](lower) at (-2.15,-.65) {handle};
\pic [scale=0.15](lower) at (-1.6,-1) {thole};
\pic [scale=0.15](lower) at (-2,-1) {thole};
\shade[left color=white!25!purple,right color=white!25!purple](-3.5,0.5)--(-3.5,-0.5)--(-2,-0.5)--(-2,0.5);
\shade[left color=white!40!pink,right color=white!40!pink](3.5,0.5)--(3.5,-0.5)--(2,-0.5)--(2,0.5);
\shade[left color=white!20!lime,right color=white!20!lime] (0,0.5) -- (.4,0.5) arc[start angle=0, end angle=-90,radius=0.4] -- (0,0);
\shade[left color=white!20!lime,right color=white!20!lime] (0,-0.5) -- (.4,-0.5) arc[start angle=0, end angle=90,radius=0.4] -- (0,0);
\shade[left color=black!20!lime,right color=black!20!lime] (0,0.5) -- (-.4,0.5) arc[start angle=-180, end angle=-90,radius=0.4] -- (0,0);
\shade[left color=black!20!lime,right color=black!20!lime] (0,-0.5) -- (-.4,-0.5) arc[start angle=180, end angle=90,radius=0.4] -- (0,0);
\draw (0,0.5)  --  (0,-0.5)
	 (3.5,.5) -- (-3.5,0.5)
	 (3.5,-.5) -- (-3.5,-0.5);
\draw[dotted] (4,0)  --  (5,0)
		     (-4,0)  --  (-5,0)
		     (-3.5,0.8) -- (-3,.8)
		      (-1.5,0.8) -- (-1,.8)
		     (-3.5,-0.8) -- (-3,-.8)
		      (-1.5,-0.8) -- (-1,-.8);	
\node[left] at (3,0){(i)};	
\node[left] at (-2.4,0){(i)};		     
\node[left] at (.85,.7){(ii)};
\node[right] at (-.85,.7){(ii)};
\node[left] at (.85,-.7){(iii)};
\node[right] at (-.85,-.7){(iii)};
	\end{tikzpicture}}
	\end{minipage}
	\caption{The old and the new parametrizations of $U_+$ and $X$: The diffeomorphisms $\Phi_+$ and $\Phi_-$ map the spaces on the left to the spaces on the right while mapping each colored region with a label to a region with the same color and label.}
	\label{special-quin-old-new-par}
\end{figure}

Before delving into the technical aspects of the proof of Lemma \ref{part-symp-constant}, we discuss the main idea of the construction of a path from a mixed pair $(A,u)\in \bA(\alpha,\beta)$ to a symplectically constant pair. Using the above reparametrization, we may regard $u$ as a map from $[0,\infty)\times [-1,1]$ to $\mathcal M(\Sigma,F)$ and $A$ as a connection on $(-\infty,0]\times Y_\#$. (Strictly speaking, we have to remove $\pm \bi$ from the domain of $u$ and the subspace $\{0\}\times Y_0\cup \{0\}\times  -Y_0'$ from $(-\infty,0]\times Y_\#$.) Let $A_u$ be a connection on $[0,\infty)\times Y_\#$ such that for any $(s,\theta)\in [0,\infty)\times [-1,1]$, the restriction of $A_u$ to $\{(s,\theta)\}\times \Sigma\subset \{s\}\times Y_\#$ represents $u(s,\theta)$, its restriction to $\{0\}\times Y_\#$ agrees with the restriction of $A$ to $\{0\}\times Y_\#$ and for any $s\in[0,\infty)$, the restriction of $A_u$ to $\{s\}\times Y_0\subset \{s\}\times Y_\#$ (resp. $\{s\}\times -Y_0'\subset \{s\}\times Y_\#$) represents $u(s,1)\in L(Y,E)$ (resp. $u(s,-1)\in L(Y',E')$). Then we may shift the mixed pair to the gauge theory side and define the pair $(A_\tau,u_\tau)$ for any $\tau\in[0,\infty)$
\begin{equation}\label{shifting-op}
  A_\tau(s,y):=\left\{
  \begin{array}{cc}
  	A (s+\tau,y) & s\leq -\tau\\
	A_u(s+\tau,y)& s> -\tau
  \end{array}
  \right.,
  \hspace{1cm}
  u_\tau(s,y):= u(s+\tau,y).
\end{equation}
As $\tau$ tends to infinity, the pair $(A_\tau,u_\tau)$ converges to a symplectically constant pair.

The above argument needs to be modified to guarantee that the mixed pairs $(A_\tau,u_\tau)$ belong to the function space used in the definition of $\bA(\alpha,\beta)$. Before applying the above shifting construction, we pick a path from $(A,u)$ to a smooth mixed pair $(A',u')$ which satisfies the following additional assumptions. (In the following, we use the old parametrization of the spaces $X$ and $U_+$.)
\begin{itemize}
	\item[(i)] The restriction of $u'$ to the subspace $[3,\infty)\times [-1,1]$ of $U_+$ is the constant map 
	to $\beta$.
	\item[(ii)] On a tubular neighborhood of $\eta_+$ (resp. $\eta_+'$) identified with 
	$\eta_+\times (\frac{1}{2},1]$ (resp. $\eta_+'\times [-1,-\frac{1}{2})$) the map $u$ is equal to the 
	pullback of a smooth map from $\eta_+$ to $L(Y,E)$ (resp. $\eta_+'$ to $L(Y',E')$). 
	\item[(iii)] The restriction to $[0,2]\times [2,\infty)$ (resp. $[0,2]\times (-\infty,-2]$) of $u$ is the 
	constant map to an element $q\in L(Y,E)$ (resp. $q'\in L(Y',E')$).
	\item[(iv)] There is a smooth function $w:[-2,2]\to \mathcal M(\Sigma,F)$ such that for 
	$(s,\theta)\in [0,1]\times [-2,2]$, 
	$u(s,\theta)=w(\theta)$.
	\item[(v)] The restriction of $A$ to the subspaces $[2,\infty)\times Y$ 
	(resp. $(-\infty,-2]\times Y'$) is the 
	pull-back of a connection $B$ (resp. $B'$) on $E$ (resp. $E'$) which
	is a representative for $q$ (resp. $q'$). 
	\item[(vi)] The restriction of $A$ to $[-1,0]\times [-2,2]\times \Sigma$ is the pullback of
	 a smooth connection $B_\#$ on$[-2,2]\times \Sigma$ with a vanishing $d\theta$ component. In particular, $B_\#\vert_{\{\theta\}\times \Sigma}$ is flat and is a representative for 
	 $w(\theta)$.
\end{itemize}

Next, we wish to lift $u'$ to a connection $A_u'$ on $[0,\infty)\times Y_\#$. In the following, we use the new reparametrization of $X$ and $U_+$. In particular, $A'$ can be identified as a connection on $((-\infty,0]\times Y_\#)\setminus (\{0\}\times Y_0\cup \{0\}\times  -Y_0')$, which can be extended smoothly to a connection on $(-\infty,0]\times Y_\#$. For any $(r,t)\in [0,\infty)\times [-1,1]$, let $\alpha(r,t)$ be the unique connection on $F$ which satisfies 
\begin{itemize}
	\item[(i)] $\alpha(r,t)$ is a flat connection representing $u(r,t)$;
	\item[(ii)] $\alpha(0,t)=A'\vert_{ \Sigma\times \{(0,t)\}}$ for any $t$;
	\item[(iii)] $d_{\alpha(r,t)}^*\partial_r\alpha(r,t)=0$.
\end{itemize}
For any $r\in [0,\infty)$, we fix smooth connections $B(r)$ and $B'(r)$ on $Y$ and $Y'$ such that 
\begin{itemize}
	\item[(i)] $B(r)$ (resp. $B'(r)$) represents an element of $L(Y,E)$ (resp. $L(Y',E')$) whose restriction to the tubular neighborhood of the boundary of $Y$ (resp. $Y'$) is determined by $\alpha(r,1)$ (resp. $\alpha(r,-1)$);
	\item[(ii)] the restriction of $B(0)$ to $Y_0$ (resp. $B'(0)$ to $Y'_0$) is equal to the restriction of $A'$ to $\{0\}\times Y_0$ (resp. $\{0\}\times -Y_0'$).
\end{itemize}
The flat connections $\alpha(r,t)$ determine a smooth connection on $\Sigma\times [0,\infty)\times [-1,1]$ with vanishing $dr$ and $dt$ components, and the connections $B(r)$, $B'(r)$ determine connections on $[0,\infty)\times Y_0$, $[0,\infty)\times -Y_0'$ with vanishing $dr$ components. Gluing these connections determines  the desired connection $A_u'$ on $[0,\infty)\times Y_\#$. Now it is easy to see that the above shifting operation in \eqref{shifting-op} applied to $A'$ and $A'_u$ provides a smooth path in $\bA(\alpha,\beta)$ from $(A',u')$ to a symplectically constant pair. In fact, the same argument addresses the family version of Lemma \ref{part-symp-constant}.
 
\begin{lemma}\label{part-symp-constant-family}
	For a compact space $T$, suppose $f:T\to \bB(\alpha,\beta)$ is a smooth map.
	Then there is a smooth map $F:T\times [0,1]\to \bB(\alpha,\beta)$
	such that for any $x\in T$, $F(x,0)=f(x)$ and $F(x,1)$ 
	is a symplectically constant pair. Moreover, if $f(x)$ is already symplectically constant pair,
	then $F(x,t)$ is a symplecically constant pair for any $t$.
\end{lemma}

\subsection{Proof of Proposition \ref{ind-mixed-op}} \label{ind-mixed-op-subs}
Suppose $A\in \mathcal A_G(\alpha,\beta)$ is a connection on the cylindrical manifold $\R\times Y$ and $(A',u)\in \bA(\beta,\gamma)$ is a mixed pair. We assume that $\mathcal A_G(\alpha,\beta)$ and $\bA(\beta,\gamma)$ are defined using the same representative for $\beta$. For any $T\in [3,\infty)$, we can glue $A$ and $(A',u)$ to define an element $(A_T,u)\in  \bA(\alpha,\gamma)$. The connection $A_T$ is defined as follows.
\begin{itemize}
	\item[(i)] On the cylinder  $(-\infty,-2T]\times Y_\#$, $A_T$ is equal to $\tau_{4T}^*(A)$, the translation of the restriction $A$ over $(-\infty,2T]\times Y_\#$ by $4T$.
	\item[(ii)] On the complement of $(-\infty,-T]\times Y_\#$, $A_T$ is equal to $A'$.
	\item[(iii)] On the cylinder $(-2T,T)\times Y_\#$, $A_T$ is equal to $\rho(\frac{s+2T}{T})\cdot \tau_{4T}^*(A)+(1-\rho(\frac{s+2T}{T}))\cdot A'$ where $\rho:[0,3]\to [0,1]$ is a fixed smooth function with $\rho(t)=1$ if $t\leq 1$ and 
	$\rho(t)=0$ if $t\geq 2$.
\end{itemize}
By putting the connections $A$ and $A'\vert_{(-\infty,-3]\times Y_\#}$ in the temporal gauge, the above gluing construction descends to gluing an element of $\mathcal B_G(\alpha,\beta)$ and $\bB(\alpha,\beta)$.

\begin{prop}\label{additivity-gauge}
	The topological energy of $(A_T,u)$ and the index of  $\mathcal D_{(A_T,u)}$ are given by
	\begin{equation}\label{top-energy-glu-gauge}
	  \mathcal E(A_T,u)=\mathcal E(A)+\mathcal E(A',u),
	\end{equation}
	and 
	\begin{equation}\label{ind-glu-gauge}
	  \ind(\mathcal D_{(A_T,u)})=\ind(\mathcal D_A)+\ind(\mathcal D_{(A',u)}).
	\end{equation}
\end{prop}
\begin{proof}
	Lemma \ref{top-energy-constant} implies that $\mathcal E(A_T)$ is independent of $T$. Thus the identity in \eqref{top-energy-glu-gauge} can be obtained by taking the limit $T\to \infty$.
	The additivity formula in \eqref{ind-glu-gauge} is the counterpart of the additivity of the index of the ASD operator with respect to gluing \cite[Section 3.3]{Don:YM-Floer} and a similar argument can be used to prove \eqref{ind-glu-gauge}.
\end{proof}

\begin{proof}[Proof of Proposition \ref{ind-mixed-op}]
	Since the index of the mixed operator and topological energy of mixed pairs are locally constant, Lemma \ref{part-symp-constant} implies that it suffices to prove Proposition \ref{ind-mixed-op} 
	for symplectically constant pairs. A symplectically constant pair $(A,u)\in \bA(\alpha,\beta)$ can be obtained by gluing a constant pair $(A_\beta,u_\beta)$ and a connection $A\in \mathcal A_G(\alpha,\beta)$. 
	Now Propositions \ref{top-energy-index-gauge}, \ref{cons-map-ind} and \ref{additivity-gauge} give the index formula for mixed operators.
\end{proof}

\begin{proof}[Proof of Lemma \ref{top-energy-constant-reverse}]
	Suppose $[A,u],[A',u']\in \bB(\alpha,\beta)$. 
	Using Lemma \ref{part-symp-constant} we may assume that 
	$[A,u]$ and $[A',u']$ are symplectically constant pairs without changing their topological energies. Thus, after picking appropriate representatives for 
	the connections $A$ and $A'$, we may assume that they agree on the complement of 
	$(-\infty,-3]\times Y_\#$. In particular, these two connections induce connections 
	$A_G$ and $A_G'$ on $\R\times Y_\#$ which represent elements of $\mathcal B_G(\alpha,\beta)$. Characterization of the components of 
	$\mathcal B_G(\alpha,\beta)$ implies that $2(\mathcal E(A_G)-\mathcal E(A_G'))$ is an integer, and hence a similar result holds for $2(\mathcal E(A,u)-\mathcal E(A',u'))$.
	Moreover, if $\mathcal E(A_G)=\mathcal E(A_G')$, then $A_G$ and $A_G'$ can be connected to each other by a path 
	of connections which is fixed on $[-3,\infty)\times Y_\#$. This induces a path between the mixed pairs 
	$[A,u]$ and $[A',u']$.
\end{proof}

The following is a consequence of Proposition \ref{ind-mixed-op} and Lemma \ref{top-energy-constant-reverse}.
\begin{cor} \label{ind-loop-four}
	For any smooth $(A,u)\in \bA(\alpha,\alpha)$, the index $\mathcal D_{(A,u)}$ is a multiple of $4$.
\end{cor}

There is a variant of Proposition \ref{additivity-gauge} where a mixed pair $(A,u)\in \bA(\alpha,\beta)$ is glued to a map $u':\R\times [-1,1]\to \mathcal M(\Sigma,F)$ representing a path from $\beta\in \fC_S$ to $\gamma\in \fC_S$. After arranging an appropriate chart for a neighborhood of $\beta$ in $\mathcal M(\Sigma,F)$ and the Lagrangians $L(Y,E)$ and $L(Y',E')$, we may follow a similar process as in the previous case to define $(A,u_T)\in \bA(\alpha,\gamma)$ for $T$ large enough. The proof of the following proposition is similar to Proposition \ref{additivity-gauge}. 
\begin{prop}\label{additivity-symplectic}
	The topological energy of $(A,u_T)$ and the index of the mixed operator $\mathcal D_{(A,u_T)}$ is given by
	\begin{equation}\label{top-energy-glu-symp}
	  \mathcal E(A,u_T)=\mathcal E(A,u)+\frac{1}{4\pi^2}\int_{\R\times Y}(u')^*\Omega
	\end{equation}
	and 
	\begin{equation}\label{ind-glu-symp}
	  \ind(\mathcal D_{(A,u_T)})=\ind(\mathcal D_{(A,u)})+\ind(\mathcal D_{u'}).
	\end{equation}
\end{prop}

\subsection{Proof of Propositions \ref{monotone-lag} and \ref{linear-AF-gradings}}\label{monotone-lag-subs}

\begin{proof}[Proof of Proposition \ref{monotone-lag}]
	Suppose $\alpha\in \fC_S$ and $u:\R\times [-1,1]\to \mathcal M(\Sigma,F)$ is a smooth map that is the constant map to $\alpha$ on the complement of the compact region $[-1,1]\times [-1,1]$.
	This map determines an element $\gamma_u$ of $\pi_1(\Omega(L,L'),o_\alpha)$ in an obvious way and we have
	\begin{equation}\label{rel-closed}
	  [\Omega](\gamma_u)=\int_{\R\times [-1,1]}u^*\Omega,\hspace{1cm}\mu(\gamma_u)=\ind(\mathcal D_u).
	\end{equation}
	We may glue $u_\alpha$ to the constant map $(A_\alpha,u_\alpha)$ to define a mixed pair $(A_\alpha,u')\in \bA(\alpha,\alpha)$ which satisfies the following properties by 
	Propositions \ref{cons-map-ind},
	\ref{additivity-symplectic} and the identities in \eqref{rel-closed}
	\begin{equation}\label{rel-to-mixed}
	   \mathcal E(A_\alpha,u')= \frac{1}{4\pi^2}[\Omega](\gamma_u), \hspace{1cm} 
	    \ind(\mathcal D_{(A_\alpha,u')})=\mu(\gamma_u).
	\end{equation}
	As a consequence of Proposition \ref{ind-mixed-op} we can conclude that
	\[
	  \mu(\gamma_u)= \frac{2}{\pi^2}[\Omega](\gamma_u).
	\]
	This implies that $(L(Y,E),L(Y',E'))$ is a monotone pair. The second identity in \eqref{rel-to-mixed} and 
	Corollary \ref{ind-loop-four} imply that the minimal Masolv number of the pair is divisible by $4$. 
	
	The minimal Masolv number of the pair $(L(Y,E),L(Y',E'))$ is in fact equal to $4$. This follows from the well-known
	fact that $c_2(\mathcal M(\Sigma,F))$ is twice the generator of $H^2(\mathcal M(\Sigma,F),\Z)$ \cite{Ram:can-cl,AB:YM}. 
	(This fact can be also derived from the arguments used in this section.)
	Since $\mathcal M(\Sigma,F)$ is simply connected, there is an element of $\pi_2(\mathcal M(\Sigma,F))$ whose
	pairing with $c_2(\mathcal M(\Sigma,F))$ is any given even integer. Thus we may change the Maslov number of 
	$u:\R\times [-1,1]\to \mathcal M(\Sigma,F)$ as above by any multiple of $4$ after gluing to a sphere in 
	$\mathcal M(\Sigma,F)$.
	
	Assuming that $\fC_S$ is non-empty, monotonicity of the pair $(L(Y,E),L(Y',E'))$ implies 
	that each of $L(Y,E)$ and $L(Y',E')$ is an oriented monotone Lagrangian with minimal Maslov number $4$.  
	To avoid the assumption on $\fC_S$, we may assume $(Y',E')=(Y,E)$ and use two different perturbation terms 
	$h$, $h'$ for $L(Y,E)$ such that the corresponding Lagrangians intersect non-trivially and transversely. 
	In any case, replacing $L(Y',E')$ with $L(Y,E)$ turns out to be unnecessary because $\rI_*(Y_\#,E_\#)$ is 
	always non-trivial \cite{km-propertyp,km-sutures}. Next, let $l$ be a loop in $L(Y,E)$. Let 
	$v:D^2\to \mathcal M(\Sigma,F)$ with $v\vert_{\partial D^2}=l$. Since the Maslov index of the disc 
	$v$ is an even integer, $TL(Y,E)$ is orientable. Thus $L(Y,E)$ and similarly $L(Y',E')$ are orientable.
\end{proof}

\begin{proof}[Proof of Proposition \ref{linear-AF-gradings}]
	Suppose $\alpha,\beta\in \fC_{S,o}$ where $o$ denotes a connected component of the path space $\Omega(L(Y,E),L(Y',E'))$, and $u:\R\times [-1,1]\to \mathcal M(\Sigma,F)$ is a smooth map 
	representing a path from $\alpha$ to $\beta$. Thus $u$ satisfies
	\begin{equation}\label{constant-ends}
	  \hspace{2cm}  u(-s,\theta)=\alpha,\hspace{.5cm}u(s,\theta)=\beta, \hspace{2cm}\forall (s,\theta)\in [1,\infty)\times [-1,1]
	\end{equation}
	and
	\begin{equation}\label{BC-2}
	  u\vert_{\R\times \{1\}}\subset L(Y,E),\hspace{1cm} u\vert_{\R\times \{-1\}}\subset L(Y',E').
	\end{equation}
	Gluing the constant pair $(A_{\alpha},u_{\alpha})$ to $u$
	produces a constant pair $(A_{\alpha},u_T)$ which can be connected to a symplectically constant pair $(A,u_{\beta})$ by Lemma \ref{part-symp-constant}. 
	The latter mixed pair is obtained by gluing a connection 
	$A'\in \mathcal A_G(\alpha,\beta)$ to the constant pair $(A_{\beta},u_{\beta})$. A similar argument as above using Propositions \ref{cons-map-ind}, 
	\ref{additivity-gauge} and \ref{additivity-symplectic} shows
	\[
	  \ind(\mathcal D_u)=\ind(\mathcal D_{A'}).
	\]
	The above identity implies that the relative grading of $\alpha$ and $\beta$ with respect to $\deg_S$ and $\deg_G$ agree with each other. This completes the proof of Proposition \ref{linear-AF-gradings}.
\end{proof}

\subsection{Orientability of the mixed determinant lines}\label{subsection-ori}

The smooth elements of $\bB(\alpha,\beta)$ parametrize a family of Fredholm operators given by the mixed operators. Associated to this family of Fredholm operators, we can associate a determinant line bundle $\delta^M$ over the subspace of $\bB(\alpha,\beta)$ given by smooth elements, where the fiber of $\delta^M$ over $[A,u]$ can be identified with
\[
  \Lambda^{\rm max}\ker(\mathcal D_{(A,u)})\otimes (\Lambda^{\rm max}\coker(\mathcal D_{(A,u)}))^*.
\]
We shall show in the next section that the elements of $\bM(\alpha,\beta)$ are smooth. In particular, $\delta^M$ induces a line bundle on $\bM(\alpha,\beta)$ whose restriction to the open subspace $\bM^{\rm reg}(\alpha,\beta)$ of regular elements of $\bM(\alpha,\beta)$ can be naturally identified with the orientation bundle of the manifold $\bM^{\rm reg}(\alpha,\beta)$. Therefore, we are interested in trivializing $\delta^M$ to orient the moduli spaces of solutions to the mixed equation. Moreover, we use orientability of $\delta^M$ to verify the claim in Proposition \ref{orientaion-delta-p}.

To prove triviality of $\delta^M$, it suffices to show that its restriction to any loop $\gamma:S^1\to \bB(\alpha,\beta)$ is orientable. Using Lemma \ref{part-symp-constant-family}, we may assume that $\gamma$ parametrizes an $S^1$-family of symplectically constant pairs. In particular, there is a connection $A_\beta$ representing $\beta$ and a loop $\gamma_G:S^1\to \mathcal B_G(\alpha,\beta)_p$ such that $\gamma$ is obtained by gluing $\gamma_\beta$ to $[A_\beta,u_\beta]$ in the same way as in Subsection \ref{ind-mixed-op-subs}. The family version of \eqref{ind-glu-gauge} in Proposition \ref{additivity-gauge}, which can be proved again using essentially the same arguments as in \cite[Section 3.3]{Don:YM-Floer}, implies that the restriction of $\delta^M$ to the family given by the loop $\gamma$ can be identified with the the tensor product $\delta_p^G\vert_{\gamma_G}\otimes \delta^M|_{[A_\beta,u_\beta]}$.  In particular, orientability of $\delta_p^G$ implies that $\delta^M$ is orientable.

The above argument can be also used to fix an orientation of $\delta^M$. First we fix an orientation of the lines bundles $\delta^G_p$ over the configuration spaces $\mathcal B_G(\alpha,\beta)_p$. For any connected component of $\bB(\alpha,\beta)$, we fix a symplectically constant pair $[A,u]$ which exists according to Lemma \ref{part-symp-constant}. Thus, $(A,u)$ is obtained from gluing a connection $A_G\in \mathcal A_G(\alpha,\beta)_p$ over $\R\times Y_\#$ to a constant pair $(A_\beta,u_\beta)$. The fiber of $\delta^M$ over $[A,u]$ is isomorphic to $\delta_p^G\vert_{[A_G]}\otimes \delta^M|_{[A_\beta,u_\beta]}$, and the isomorphism is canonical up to multiplication by a positive constant. Since the kernel and the cokernel of the mixed operator associated to the mixed solution are trivial, $\delta^M|_{[A_\beta,u_\beta]}$ can be naturally identified with $\R$. Therefore, the fixed orientation of $\delta^G_p$ determines an orientation of $\delta^M\vert_{[A,u]}$. This orientation is independent of the choice of $[A,u]$. If $[A',u']$ is another element in the same connected component of $\bB(\alpha,\beta)$, then there is a path $\gamma:[0,1]\to \bB(\alpha,\beta)$ from $[A,u]$ to $[A',u']$. Using Lemma \ref{part-symp-constant-family}, we may assume that $\gamma$ is in fact a path in the subspace of symplectically constant pairs. Therefore, orientations of $\delta^M$ induced by $[A,u]$ and $[A',u']$ agree with each other.

We use a similar trick to show that the line bundles $\delta_p^S$ over the configuration spaces of strips $\mathcal B_S(\alpha,\beta)_{p}$ are trivial, and then fix a trivialization of these line bundles. Let $\gamma_S:S^1\to \mathcal B_S(\alpha,\beta)_{p}$ be a loop. By changing this loop using a homotopy, we may assume that $\gamma_S$ is represented by a smooth map $U:\R\times [-1,1]\times S^1\to \mathcal M(\Sigma,F)$ such that for $s\geq 1$, we have $U(-s,\theta,t)=\alpha$ and $U(s,\theta,t)=\beta$. We may glue this loop to the constant mixed pair $(A_\alpha,u_\alpha)$ as in Subsection \ref{ind-mixed-op-subs} to define $\gamma^T:S^1\to \bA(\alpha,\beta)$ for $T$ large enough. The family version of Proposition \ref{additivity-symplectic} implies that $\delta^M\vert_{\gamma^T}$  is isomorphic to the tensor product $ \delta^M|_{[A_\alpha,u_\alpha]}\otimes \delta_p^S\vert_{\gamma}$. In particular, $ \delta_p^S\vert_{\gamma}$ is trivial, which verifies our claim.

We may fix an orientation of $ \delta_p^S$ in the same was as in the case of $\delta^M$. Given a strip $u:\R\times [-1,1]\to \mathcal M(\Sigma,F)$ satisfying \eqref{constant-ends} and \eqref{BC-2}, we may glue  the constant mixed pair $(A_\alpha,u_\alpha)$  to $u$ to define another mixed pair $(A_\alpha,u')$. Since $\delta^M\vert_{[A_\alpha,u']}$ and $\delta^M|_{[A_\alpha,u_\alpha]}\otimes \delta_p^S\vert_{u}$ are isomorphic and $\delta^M\vert_{[A_\alpha,u']}$ and $ \delta^M|_{[A_\alpha,u_\alpha]}$ have fixed orientations, we obtain an orientation of $\delta_p^S|_{u}$. This induces a well-defined orientation of $\delta_p^S$. These orientations are compatible with the strip gluing maps in \eqref{strip-gluing-ori} because the fixed orientations on $\delta_p^G$ are compatible with the cylinder gluing maps in \eqref{cyl-gluing-ori}. In summary, we obtain a coherent system of orientations for the line bundles $\delta_p^S$.

\begin{remark}\label{comp-Z-summand}
	We use gluing theory of various indices to define orientations of the line bundles $\delta^M$ and $\delta_p^S$ in terms of the orientations of the line bundles $\delta_p^G$. Recall that we had a degree of 
	freedom to orient $\delta_p^G$. To define this orientation, we fixed an orientation of $\delta_{p_0}^G$ where $p_0$ is a path from a fixed $\alpha_0\in \fC_G$ to itself, whose index has the form $8k+4$. 
	One such path $p_0$ can be fixed as follows. Let $s:S^2\to \mathcal M(\Sigma,F)$ represent an element of $\pi_2(\mathcal M(\Sigma,F))$ which is associated to one of the connected components of $\Sigma$ 
	and is introduced at the end of Subsection \ref{flat-surface}. 
	Gluing $s$ to the constant strip mapped to $\alpha_0$ determines $u:\R\times [-1,1]\to \mathcal M(\Sigma,F)$ with $\ind(\mathcal D_u)=4$. Applying the mixed shifting operation of Subsection \ref{mixed-shit-op} gives a connection on $\R\times Y_\#$, which represents the desired path $p_0$.
	The sphere gluing map $\Psi_{p,s}$ in Remark \ref{deficiency-ori} allows us to define an orientation of the index of $\mathcal D_u$. Then using the construction of this subsection in the reverse order, 
	we may fix an orientation of $\delta^G_{p_0}$. Using this orientation of $\delta^G_{p_0}$, one may easily see that the induced coherent system of orientations of the line bundles $\delta_p^S$ is 
	compatible with $\Psi_{p,s}$ when $s$ is the above element of $\pi_2(\mathcal M(\Sigma,F))$.
\end{remark}

\section{Non-linear analysis}\label{reg-comp-exp-dec-sec}
In this section we shall prove Proposition \ref{exp-decay} and part of Proposition \ref{comp}. Our primary tools are the compactness and regularity theorems of \cite{DFL:mix} together with some standard results about the solutions of ASD and pseudo-holomorphic curve equations. As in the previous section, $(Y,E)$ and $(Y',E')$ are fixed as in Subsection \ref{3-man-bdles}, and we fix Lagrangian 3-manifolds associated to these pairs that have transversal intersection and the claim of Lemma \ref{reg-ASD} holds. We continue to drop $h$ and $h'$ from our notations for the 3-manifolds Lagrangians, and denote them by $L(Y,E)$ and $L(Y',E')$.

\subsection{Review of the results of \cite{DFL:mix} on regularity and compactness}\label{review-DFL:mix}

The following compactness theorem from \cite[Theorem 3]{DFL:mix} can be regarded as a common generalization of Uhlenbeck and Gromov compactness theorems to moduli spaces of solutions to the mixed equation.

\begin{theorem}\label{compactness-X-S}
	Suppose $\fq=(X,V,S,\mathcal M(\Sigma,F),{\mathbb L})$ is a quintuple as in Subsection \ref{quintuples}.  
	There is a constant $\hbar$ such that the following holds. Suppose $\{(A_i,u_i)\}$ is a sequence of solution of the mixed equation \eqref{mixed-eq-pre} associated to $\fq$ such that 
	\[
	  \mathcal E(A_i,u_i)\leq \kappa
	\]
	for a fixed constant $\kappa$. Then there are
	\begin{itemize}
		\item[(i)]  a subsequence $\{(A_i^\pi,u_i^\pi)\}$ of $\{(A_i,u_i)\}$,
		\item[(ii)] a solution of the mixed equation $(A_0,u_0)$ for the quintuple $\fq$,
		\item[(iii)] finite sets $\sigma_-\subset {\rm int}(X)$, $\sigma_\partial \subset \gamma$ and $\sigma_+\subset S\setminus \gamma$,
	\end{itemize}
	such that the following holds.
	\begin{itemize}
		\item[(i)] The pair $(A_0,u_0)$ satisfies the energy bound \[\mathcal E(A_0,u_0)\leq \limsup_{i}\mathcal E(A_i,u_i).\]
		If any of the sets $\sigma_-$, $\sigma_\partial$ and $\sigma_+$ is nonempty, then the above inequality can be improved by subtracting $\hbar$ from the right hand side.
		\item[(ii)] $u_i^\pi$ is $C^\infty$-convergent to $u_0$ on any compact subspace of $S\setminus (\sigma_+\cup \sigma_\partial)$.
		\item[(iii)] There are gauge transformations $g_i^\pi$ defined over $X\setminus (\sigma_\partial\times \Sigma\cup \sigma_-)$ such that $(g_i^\pi)^*A_i^\pi$ is $C^\infty$ convergent to
		$A_0$ on any compact subspace of  $X\setminus (\sigma_\partial\times \Sigma\cup \sigma_-)$.
	\end{itemize}
\end{theorem}

In the above theorem, one should think about $\sigma_-$, $\sigma_+$ and $\sigma_\partial$ as the sets where the bubbling phenomenon happens. We have bubbling of the ASD equation on $\sigma_-$, bubbling of the holomorphic curve equation on $\sigma_+$ and mixed bubbling on $\sigma_\partial$. We need a slightly more general version of this compactness theorem where the mixed equation is perturbed by terms similar to the perturbation terms that appear in \eqref{mixed-eq-pert}. To be more specific, we consider a generalization of the mixed equation where holomorphic equation part of the mixed equation is defined using a family of domain dependent almost complex structures and the ASD equation is deformed by holonomy perturbations. We shall make the precise type of such perturbations clear in the subsequent section. For now, we just point out that we only consider perturbations that in a neighborhood of $\gamma$ in $S$ the almost complex structure is the standard one $J_*$, and in a neighborhood of $\gamma \times \Sigma$ in $X$, the holonomy perturbation of the ASD equation is trivial. We call any such perturbation a {\it standard perturbation} of the mixed equation, which is {\it trivial in a neighborhood of the matching line}.

\begin{theorem}\label{compactness-X-S-pert}
	Suppose $\fq=(X,V,S,\mathcal M(\Sigma,F),{\mathbb L})$ is given as in Theorem \ref{compactness-X-S}. Suppose the mixed equation associated to $\fq$ is deformed by a standard perturbation,
	which is trivial in a neighborhood of the matching line, and $\{(A_i,u_i)\}$ is a sequence of solutions to the perturbed mixed equation. 
	Then the same claim as in Theorem \ref{compactness-X-S} holds except that the last part of the claim should be replaced with 
	\vspace{-.25cm}
	\begin{itemize}
		\item[(iii)'] There are gauge transformations $g_i^\pi$ defined over $X\setminus (\sigma_\partial\times \Sigma\cup \sigma_-)$ such that for any $p$,
		the connections $(g_i^\pi)^*A_i^\pi$ are $L^p_1$ convergent to
		$A_0$ on any compact subspace of  $X\setminus (\sigma_\partial\times \Sigma\cup \sigma_-)$. This convergence can be improved to $C^\infty$ if $\sigma_-$ is empty.
	\end{itemize}
\end{theorem}

\begin{proof}
	Theorem \ref{compactness-X-S} has a local nature. First, one obtains a compactness theorem for {\it nice} neighborhoods of points in ${\rm int}(X)$, $\gamma$ and 
	$S\setminus \gamma$. (By a neighborhood around a point $p\in \gamma$, we mean the disjoint union of a neighborhood of $\{p\}\times \Sigma$ in $X$ and a neighborhood of $p$ in $S$.)
	A neighborhood around a given point is nice if $|\!|F_{A_i})|\!|_{L^2}$ and $|\!|\nabla u_i|\!|_{L^2}$ are universally bounded by a specific constant $\hbar$ in the neighborhood. Then a patching 
	argument as in \cite[Chapter 4]{DK} allows us to obtain the global compactness theorem. As a result the argument of the proof of Theorem \ref{compactness-X-S} can be easily adapted to prove this variation.
	For points in $\gamma$, we may use the assumption to find neighborhoods where the ASD equation is not deformed and the chosen family of complex structures on $\mathcal M(\Sigma,F)$ is 
	the constant family given by $J_*$. For points in $S\setminus \gamma$, we may find neighborhoods contained in $S\setminus \gamma$ where we can use the Gromov compactness theorem 
	for the pseudo-holomorphic curve equation with respect to a domain dependent almost complex structure (see, for example, \cite{Gr:comp}). For points in ${\rm int}(X)$, we may use compactness 
	theorem for the deformation of the ASD equation (see \cite{Uh:com,Uh:remov-sing,DK,K:higher}). Here due to the non-local nature of holonomy perturbations one can only obtain $L^p_1$ convergence
	in the presence of bubbles. A detailed treatment of this issue can be found in \cite{K:higher} (in the more general case of ${\bf PU}(N)$-connections.)
\end{proof}

Next, we turn to regularity of solutions of the mixed equation. First we focus on quintuples which capture all novel issues for the moduli of solutions to the mixed equation. Suppose $B_r$ is the unit disc of radius $r$ centered at the origin in the $(s,\theta)$-plane, and $D_+(r)$, $D_-(r)$ denote the intersections with the half planes $s\geq 0$ and $s\leq 0$. Let also $U_\partial(r)$ denote the intersection of $D_+(r)$ and $D_-(r)$. Consider the quintuple
\[
  \fQ(r):=(D_-(r)\times \Sigma,D_-(r)\times F,D_+(r),\mathcal M(\Sigma,F),\emptyset).
\]
The standard metric on $D_-(r)$ and the fixed metric on $\Sigma$ induce the product metric on $D_-(r)\times \Sigma$. Suppose $(A,u)$ is a solution of the mixed equation \eqref{mixed-eq} associated to the quintuple $\fQ(r)$ such that $A$ satisfies the Coulomb gauge fixing condition
\begin{equation}\label{coulomb-gauge}
  d^*_{A_0}(A-A_0)=0,\hspace{1cm}*(A-A_0)\vert_{U_\partial(r)\times \Sigma}=0.
\end{equation}
Here $A_0$ is an auxiliary smooth connection on $D_-(r)\times \Sigma$. The following is Theorem 1 in \cite{DFL:mix}.

\begin{theorem}\label{reg-mixed}
	Suppose $p>2$ and $(A,u)$ is an $L^p_1$ solution of the mixed equation associated to $\fQ(r)$ satisfying \eqref{coulomb-gauge}.
	Then $(A,u)$ is smooth. 
\end{theorem}

Suppose that $(A,u)$ is a solution of the mixed equation given as in the statement of Proposition \ref{exp-decay}. Let $x$ be a point in the matching line $U_\partial$, and $D_-(r)\times \Sigma$ (resp. $D_+(r)$) is a neighborhood of $\{x\}\times \Sigma$ (resp. $x$) which embeds into $X$ (resp. $S$). For a fixed $2<p<4$, we may find an $L^p_2$ gauge transformation $h$ and a smooth connection $A_0$ on $D_-(r)\times \Sigma$ such that $h^*A$ satisfies \eqref{coulomb-gauge} with respect to the connection $A_0$. Then $(h^*A|_{D_-(r)\times \Sigma},u|_{D_+(r)})$ satisfies the assumptions of Theorem \ref{reg-mixed}, and hence this pair is smooth. Standard regularity results for the solutions of the ASD equation (perturbed by a holonomy perturbation) and pseudo-holomorphic curves (with respect to a domain dependent almost complex structure) allow us to obtain similar results for the interior points of $X$ and $S$. Thus, $u$ is smooth and for any point $x\in X$, there is a gauge transformation $g$ on a neighborhood of $x$ such that $\widetilde A:=g^*A$ is smooth. We may use the patching argument of \cite{DK} to obtain a global gauge transformation $g$ on $X$ such that $g^*A$ is smooth. 

In the process of the construction of the gauge transformation $g$, we may obtain a stronger regularity result on the the gauge theoretic end of $X$. Since the restriction of $A$ to $(-\infty,-3]\times Y_\#$ is an ASD connection, we may find a gauge transformation $h$ on this end and a connection $\alpha$ representing an element of $\fC_G$ such that $h^*A-\pi_\#^*(\alpha)$ is in $L^2_l$ for any $l$ \cite[Chapter 4]{Don:YM-Floer}. Since the Lagrangians $L(Y,E)$ and $L(Y',E')$ intersect transversely, there is $\beta\in \fC_S$ such that $u(s,\theta)\to \beta$ as $s\to \infty$ and the restriction of $du$ to the symplectic end $[3,\infty)\times [-1,1]$ is in $L^2_{l-1}$ for the given $l$ \cite{Fl:HF}. This verifies all parts of Proposition \ref{exp-decay} except the last part about the behavior of the solutions of the mixed equation on the mixed ends, which will be taken up in Subsection \ref{exp-dec-subsec}. In fact, the same argument proves the generalization  of these parts of Proposition \ref{exp-decay} in the case that the mixed equation in \eqref{mixed-eq} is perturbed by a standard perturbation, which is trivial in a neighborhood of the matching line.

\subsection{Mixed Chern-Simons functional}
We start this part by defining the 3-dimensional analogue of the configuration space of mixed pairs. Suppose $c_0$ denotes one of the connected components of $L(Y,E)$. Let $\bA_{c_0}(Y,E)$ be the space of all pairs $(B,q)$ where $B$ is an $L^2_l$ connections on the bundle $E$ over $Y$, and $q:[0,2]\to \mathcal M(\Sigma,F)$ is an $L^2_l$ path such that the restriction of $B$ to $\partial Y=\Sigma$ is flat and represents the flat connection $q(0)$, and $q(2)$ belongs to the connected component $c_0$ of $\mathcal M(\Sigma,F)$. The space of $L^2_{l+1}$ automorphisms of the bundle $E$ acts on $\bA_{c_0}(Y,E)$ in the obvious way, and we let $\bB_{c_0}(Y,E)$ be the quotient space.  A pair $(B,q)$ representing an element of $\bB_{c_0}(Y,E)$ is called a {\it flat mixed pair} if $q$ is a constant map to an element $z\in L(Y,E)$ and $B$ represents $z$. In particular, $B$ satisfies the equation 
\[
    *_3F_B+\nabla_{B} h=0,
\]
and the subspace of $\bB_{c_0}(Y,E)$ given by flat mixed pairs can be identified with the connected component $c_0$ of $L(Y,E)$.

Fix an arbitrary flat mixed pair $(B_0,q_0) \in \bA_{c_0}(Y,E)$. Given $(B,q) \in \bA_{c_0}(Y,E)$, there is a connection $A$ on $[-1,1]\times Y$ and a map $u:[0,2]\times [-1,1]\to \mathcal M(\Sigma,F)$ such that
\begin{itemize}
	\item[(i)] $A|_{ \{-1\}\times Y}=B$ and $A|_{\{1\}\times Y}=B_0$;
	\item[(ii)] $u|_{[0,2]\times \{-1\}}=q$ and $u|_{[0,2]\times \{1\}}=q_0$;
	\item[(iii)] $q|_{\{2\}\times I}\subset L(Y,E)$;
	\item[(iv)] For any $\theta\in [-1,1]$, $A|_{\Sigma\times \{\theta\}}$ is flat and represents $u(0,\theta)$.
\end{itemize}
Define the {\it mixed Chern-Simons} of the pair $(B,q)$ as 
\begin{equation}\label{mixed-CS}
	CS_h(B,q)=\frac{1}{8\pi^2}\int_{[-1,1]\times Y}\tr\((F_A+*_3\nabla_{A_t}h)\wedge (F_A+*_3\nabla_{A_t}h)\)+\frac{1}{4\pi^2}\int_{ [0,2]\times [-1,1]}u^*\Omega,
\end{equation}
where $A_t$ denotes the restriction of $A$ to $\{t\}\times Y_0$. A priori, the value of $CS_h(B,q)$ depends on $(A,u)$. However, the following lemma asserts that $CS_h(B,q)$ does not change by a local deformation of $(A,u)$.

\begin{lemma}
	The mixed Chern-Simons functional $CS_h(B,q)$ depends only on the homotopy class of the path $\theta\in [-1,1]\to [A|_{\{\theta\}\times Y},u|_{[0,2]\times \{\theta\}}]\in \bB_{c_0}(Y,E)$ 
	among the paths from $[B,q]$ to $c_0$. 
\end{lemma} 
\begin{proof}
	As in Lemma \ref{top-energy-constant}, 
	if we vary $(A,u)$ while preserving $(A\vert_{\{\pm 1\}\times Y},u|_{[0,2]\times \{\pm 1\}})$, the expression on the left hand side of \eqref{mixed-CS} does not change.
	Next, note that for any path $\gamma:[-1,1] \to L(Y,E)$ there is a pair $(A,u)$
	satisfying (i)-(iv) such that for $(s,\theta)\in [0,2]\times [-1,1]$, 
	$u(s,\theta)=\gamma(\theta)$ and $(A_{\{\theta\}\times Y},u|_{\{[0,2]\times \{\theta\}})$ is a flat mixed pair. For any such pair, the left hand side 
	of \eqref{mixed-CS} vanishes. These two observations allow us to complete the proof.
\end{proof}

\begin{remark}
	We may alter the definition of $CS_h(B,q)$ by dropping the terms involving $h$ from the first integrand in \eqref{mixed-CS}. Then the change in $CS_h(B,q)$ equals $\frac{1}{4\pi^2}(\widetilde h(B_0)-\widetilde h(B))$ which depends only
	$[B,q]\in \bB_{c_0}(Y,E)$. Thus, in studying the dependence of $CS_h(B,q)$ on the homotopy class of the path from $[B,q]$ to $c_0$, which is our next goal, we may work with this alternative definition of $CS_h(B,q)$ that has a simpler form.
	However, we keep working with our original definition $CS_h(B,q)$, which has a more canonical role.
\end{remark}

To see how $CS_h(B,q)$ depends on the homotopy class of the path from $[B,q]$ to $c_0$, it is helpful to introduce an alternative characterization of the mixed Chern-Simons functional. Suppose the pair $(\widetilde Y,\widetilde E)$ of a closed Riemannian 3-manifold and an $\SO(3)$-bundle is defined in the same way as $(Y_\#,E_\#)$ in Subsection \ref{3-man-bdles} except that we replace $(Y',E')$ with $(Y,E)$. In particular, we have the 3-manifold decomposition  
\begin{equation}\label{deco-tilde-Y}
  \widetilde Y=Y_0\cup [-2,2]\times \Sigma\cup -Y_0.
\end{equation}
In \eqref{deco-tilde-Y}, we slightly diverged from our convention in Subsection \ref{3-man-bdles}. The intersection of $Y_0$ (resp. $-Y_0$) with $[-2,2]\times \Sigma$ is $\{-2\}\times \Sigma$ (resp. $\{2\}\times \Sigma$), and we identify $Y$ with the subset $Y_0\cup [-2,0]\times \Sigma$ of $\widetilde Y$. There is an obvious orientation reversing involution $\iota:\widetilde Y\to \widetilde Y$, and the Riemannian metric on $\widetilde Y$ is invariant with respect to this involution.

The function $h:\mathcal B(Y,E)\to \R$ induces a function $\widetilde h$ on the space of connections on $\widetilde E$ whose value at a connection $B$ on $\widetilde E$ depends only on the gauge equivalence class of $B$ over $Y_0$ and $-Y_0$. Any connection $\widetilde B$ on $\widetilde Y$ induces connections $B$ and $B'$ on $Y_0$ and $-Y_0$, and $\widetilde h(\widetilde B)$ is equal to $h(B)-h(B')$. We perturb the Chern-Simons functional of the closed pair $(\widetilde Y,\widetilde E)$ using $\widetilde h$. This induces a perturbation of the flat equation on the space of connections on $\widetilde E$ as
\begin{equation}\label{flat-3d-closed-double}
	\phi_{h}(\widetilde B)=*_3F_{\widetilde B}+\nabla_{\widetilde B} \widetilde h
\end{equation}
in the same way as in \eqref{flat-3d-closed}, which satisfies $\phi_h\circ \iota^*=-\iota^*\circ \phi_h$. The solutions of \eqref{flat-3d-closed-double}, $\fC(\widetilde Y,\widetilde E)$, can be identified with the intersection of $L(Y,E)$ with itself, and hence, is equal to $L(Y,E)$.

Let $(A,u)$ be a pair as above which connects $(B,q)$ to a flat mixed pair $(B_0,q_0)$. We may use $(A,u)$ to define a connection $\widetilde A$ on $[-1,1]\times \widetilde Y$. First we regard $A$ as a connection on $[-1,1] \times Y\subset [-1,1]\times \widetilde Y$. By applying a gauge transformation to $A$ if necessary, we can assume that the $d\theta$ component of $A$ on $[-1,1]\times \{0\}\times \Sigma\subset [-1,1] \times Y$ and the $ds$ component of $A$ on $[-1,1]\times [-2,0]\times \Sigma \subset [-1,1]\times Y$ vanish. (As before, $s$ denotes the coordinate on the interval $[-2,0]$ and $\theta$ denotes the coordinate on $[-1,1]$.) For any $(s,\theta)\in [0,2]\times [-1,1]$, let $\alpha(s,\theta)$ be the connection on $F$ such that
\begin{itemize}
	\item[(i)] $\alpha(s,\theta)$ is a flat connection representing $u(s,\theta)$;
	\item[(ii)] $\alpha(0,\theta)=A\vert_{\{0\}\times \Sigma\times \{\theta\}}$ for any $\theta$;
	\item[(iii)] $d_{\alpha(s,\theta)}^*\partial_s\alpha(s,\theta)=0$.
\end{itemize}
For any $\theta\in [-1,1]$, $\alpha(2,\theta)\in L(Y,E)$, and we fix a representative connection $B'(\theta)$ on $-Y_0\subset \widetilde Y$ for $\alpha(2,\theta)$ such that $B'(\theta)$ satisfies \eqref{flat-3d-closed-double} over $-Y_0$, it has a vanishing $ds$ component along the boundary of $-Y_0$ and $B'(1)=\iota^*B_0$. The connections $A$, $\alpha(s,t)$ and $B'(t)$ determine a connection $\widetilde A$ on $[-1,1]\times \widetilde Y$ which has vanishing $ds$ and $d\theta$ components on $[0,2]\times \Sigma\times [-1,1]$ and vanishing $d\theta$ component on $[-1,1]\times -Y_0$. Although $\widetilde A$ might not be smooth, it is in $L^p_1$ for any $p$.

\begin{lemma}\label{CS-alternative-lemma}
	Suppose $CS_h(B,q)$ is defined using $(A,u)$. Then we have
	\begin{equation}\label{CS-alternative}
	  CS_h(B,q)=\frac{1}{8\pi^2}\int_{[-1,1]\times \widetilde Y}\tr\((F_{\widetilde A}+*_3\nabla_{\widetilde A_\theta}\widetilde h)\wedge (F_{\widetilde A}+*_3\nabla_{\widetilde A_\theta}\widetilde h)\).
	\end{equation}
	Here $\widetilde A_\theta$ denotes the restriction of $\widetilde A$ to $\widetilde \{\theta\}\times Y$.
	\end{lemma}
\begin{proof}
	It is clear that the restriction of the integrand in \eqref{CS-alternative} vanishes over $[-1,1]\times -Y_0$, and its restriction over $I\times Y$ integrates to the first term on the 
	right hand side of \eqref{mixed-CS}. Thus \eqref{CS-alternative} follows if we show that 
	\begin{equation}\label{CS-alternative-middle}
		\int_{[-1,1]\times [0,2]\times \Sigma }\tr\(F_{\widetilde A}\wedge F_{\widetilde A}\)=2\int_{ [0,2]\times [-1,1]}u^*\Omega.
	\end{equation}
	The restriction of $F_{\widetilde A}$ to $[-1,1]\times [0,2]\times \Sigma $ of $[-1,1]\times \widetilde Y $ is given by 
	$ds\wedge \partial_s\alpha(s,\theta)+d\theta\wedge \partial_\theta\alpha(s,\theta)$. Now it is easy to verify \eqref{CS-alternative-middle} using (iii) and the fact that 
	$ \partial_s\alpha(s,\theta)$ and $ \partial_\theta\alpha(s,\theta)$ are $d_{\alpha(s,\theta)}$ closed.
\end{proof}

Let $\widetilde B=\widetilde A_{-1}$ and $\widetilde B_0=\widetilde A_{1}$ be given as in Lemma \ref{CS-alternative-lemma}, and $\widetilde b=\widetilde B-\widetilde B_0$. A straightforward examination of \eqref{CS-alternative} shows that we may integrate out the variable $\theta$ in \eqref{CS-alternative}. After using the assumption that $*_3\nabla_{\widetilde B_0} \widetilde h$ equals $-F_{\widetilde B_0}$, we have
\begin{equation}\label{CS-2nd-alternative}
	  CS_h(B,q)=\frac{1}{8\pi^2}\int_{\widetilde Y}\tr(2\widetilde b\wedge F_{\widetilde B_0}+\widetilde b\wedge d_{\widetilde B_0} \widetilde b+
	  \frac{2}{3}\widetilde b\wedge \widetilde b\wedge \widetilde b))+\frac{1}{4\pi^2}(\widetilde h(\widetilde B)-\widetilde h(\widetilde B_0)).
\end{equation}
Note that $\widetilde B$ is essentially determined by $(B,q)$. The restriction of $\widetilde B$ to $Y$ is $B$, and its restriction to $[0,2]\times \Sigma$ is given by the flat connections $\alpha(s)$ for $s\in [0,2]$, which are uniquely determined by 
\begin{itemize}
	\item[(i)] $\alpha(s)$ is a flat connection representing $q(s)$;
	\item[(ii)] $\alpha(0)=B\vert_{\{0\}\times \Sigma}$;
	\item[(iii)] $d_{\alpha(s)}^*\partial_s\alpha(s)=0$.
\end{itemize}
The only ambiguity in determining $\widetilde B$ is to pick a representative $B'$ for $\alpha(2)$ over $-Y_0$. In particular, any two different choices for $\widetilde B$ are related to each other by applying an element of $\mathcal G(\widetilde E)$ and then possibly the involution $\iota$, defined as in Subsection \ref{Fl-hom-subsec} using the decomposition of $\widetilde Y$ as $Y\cup -Y$. In particular, $\widetilde B$ as an element of $\mathcal B(Y,E)/\iota$, completed with respect to the Sobolev norm $L^2_1$, is well-defined.

The following proposition shows that $2CS_h$ as a map $\bB_{c_0}(Y_0,E)\to \R/\Z$ is well-defined.
\begin{prop}\label{CS-h-wd-mod-Z}
	The value of $2CS_h(B,q)$ mod integers depends only on $[B,q]\in \bB_{c_0}(Y,E)$ and is independent of the choice of the mixed pair $(A,u)$.
\end{prop}
\begin{proof}
	Suppose $(A,u)$ and $(A',u')$ are two pairs as above connecting a flat mixed pair $(B_0,q_0)$ to $(B,q)$ and $(g^*B,q)$ for a gauge transformation $g\in \mathcal G(E)$. 
	Suppose also $\widetilde A$ and $\widetilde A'$ are defined in the same way as above. Suppose $\widetilde B$ and $\widetilde B'$ denote the restrictions of $\widetilde A$ and $\widetilde A'$
	to $\{-1\}\times \widetilde Y$. Then $\widetilde B$ and $\widetilde B'$ over the subspaces $Y$ and $-Y$ of $\widetilde Y$ are gauge equivalent to each other using elements of 
	$\mathcal G(E)$. Thus $\widetilde B$ and $\widetilde B'$ are equivalent to each other either using a determinant one gauge transformation or the composition of a gauge transofmration
	with an in involution $\iota$, defined as in Subsection \ref{Fl-hom-subsec} using the decomposition of $\widetilde Y$ as $Y\cup -Y$. In the first case, we can glue $\widetilde A$
	and $\widetilde A'$ to define an $\SO(3)$ connection $\widehat A$ on $S^1 \times \widetilde Y $ whose $w_2$ is the the pullback of $w_2(\widetilde E)$. We also have
	\begin{align}
		CS_h(B,q)-CS_h(g^*B,q)=&
		\frac{1}{8\pi^2}\int_{S^1 \times \widetilde Y}\tr\((F_{\widehat A}+*_3\nabla_{\widehat A_\theta}\widetilde h)\wedge (F_{\widehat A}+*_3\nabla_{\widehat A_\theta}\widetilde h)\)\nonumber\\
		=&\frac{1}{8\pi^2}\int_{S^1 \times \widetilde Y}\tr\(F_{\widehat A}\wedge F_{\widehat A}\).\label{energy-4D}
	\end{align}
	The last identity is a consequence of the Stokes theorem. In general, \eqref{energy-4D} equals $-p_1(\widehat E)/4$, where $\widehat E$ denotes the underlying $SO(3)$-bundle of $\widehat E$. Therefore, \eqref{energy-4D} 
	is an integer because $w_2(\widehat A)$ is the pullback of $w_2(\widetilde E)$. A similar argument can be 
	applied to the second case with the difference that $w_2(\widehat A)$ is the sum of the Poincar\'e dual of $\Sigma_0$ (one of the connected components of $\Sigma$) and 
	the pullback of $w_2(E)$. Thus \eqref{energy-4D} is a half integer in the second case.
\end{proof}

The following lemma, which gives a relation between the mixed Chern-Simons functional and its gradient, plays a crucial role in showing that the solutions of the mixed equation satisfies an exponential decay on cylindrical ends.

\begin{lemma}\label{mixed-grad-CS}
	There is a positive constant $\kappa$ and a neighborhood $U$ of the space of flat mixed pairs of $\bB_{c_0}(Y,E)$ such that the following holds.
	There is a real lift $\widetilde{CS}_{h}(B,q)$ of $CS_h(B,q)$ for any $[B,q]\in U$ such that
	\begin{equation}\label{ineq-mixed-grad-CS}
	  \widetilde{CS}(B,q)\leq \kappa(\vert\!\vert F_B+*_3\nabla_{B} h\vert\!\vert_{L^2(Y)}^2+\vert\!\vert d q\vert\!\vert_{L^2([0,1])}^2).
	\end{equation}
\end{lemma}
This lemma is an (infinite dimensional) instance of a general property for Morse-Bott functions. Suppose $f:M\to \R$ is a function on a finite dimensional manifold and $\fC\subset M$ is a submanifold of $M$ which gives a connected component of the critical points of $f$ consisting of Morse-Bott critical points. Then there are a constant $C$ and  a neighborhood $U$ of $\fC$ such that for any $x\in U$ we have
\[
  |f(x)-f_0|\leq C |\nabla f(x)|^2,
\]
where $f_0$ is the value of $f$ on $C$.

\begin{proof}
	This lemma can be proved using an analogous result which holds for the Chern-Simons functional of connections on closed 3-manifolds. 	
	Define the configuration space of connections $\mathcal B(\widetilde Y,\widetilde E)$ using the $L^2_1$ norm. Fix a connection $\widetilde B_0$ satisfying \eqref{flat-3d-closed-double} and for a connection 
	$\widetilde B=\widetilde B_0+\widetilde b$ on the bundle $\widetilde E$ over $\widetilde Y$, representing an element of $\mathcal B(\widetilde Y,\widetilde E)$, define
	\begin{equation}\label{CS-h-double}
		CS_h(\widetilde B):=\frac{1}{8\pi^2}\int_{\widetilde Y}\tr(\widetilde b\wedge d_{\widetilde B_0} \widetilde b+
		\frac{2}{3}\widetilde b\wedge \widetilde b\wedge \widetilde b)+
		\frac{1}{4\pi^2}\int_{0}^1dt\int_{\widetilde Y}\tr\(\widetilde b\wedge (*_3(\nabla_{\widetilde B_0+t\widetilde b}\widetilde h-\nabla_{\widetilde B_0}\widetilde h))\).
	\end{equation}
	Note that this definition agrees with the left hand side of the expression in \eqref{CS-2nd-alternative}. The critical points of the Chern-Simons functional $CS_h$ satisfy \eqref{flat-3d-closed-double}.
	In fact, with the same argument as in Proposition \ref{CS-h-wd-mod-Z}, $2CS_h$ induces a map $\mathcal B(\widetilde Y,\widetilde E)/\iota\to \R/\Z$, and the critical locus of this functional, denoted by 
	$\fC(\widetilde Y,\widetilde E)$, can be identified
	with $L(Y,E)$. A similar argument as in Lemma \ref{comp-per-reg} 
	shows that for any $\widetilde B_0$ representing an element 
	of $\fC(\widetilde Y,\widetilde E)\cong L(Y,E)$, the vector space
	\begin{equation}\label{tang-closed-double}
		\mathcal H^1_{\widetilde h}(\widetilde Y;\widetilde B_0):=\{\widetilde b \in \Omega^1(\widetilde Y,\widetilde E) \mid d_{\widetilde B_0}^*\widetilde b=0,\,
		*d_{\widetilde B_0}(\widetilde b)+{\rm Hess}_{\widetilde B_0}\widetilde h(\widetilde b)=0 \}
	\end{equation}
	is isomorphic to the tangent space of $\fC(\widetilde Y,\widetilde E)$ at $\widetilde B_0$. That is, the Chern-Simons functional $CS_h$ of $(\widetilde Y,\widetilde E)$ is Morse-Bott. 
	
	Suppose $\widetilde U$ is the subspace of $\mathcal B(\widetilde Y,\widetilde E)/\iota$ represented by connections $\widetilde B$ of the form $\widetilde B_0+\widetilde b$ where $\widetilde B_0$ 
	represents an element 
	of $\fC(\widetilde Y,\widetilde E)$, $\vert\!\vert \widetilde b\vert\!\vert_{L^2_1}<\epsilon'$, 
	$d_{\widetilde B_0}^*\widetilde b=0$ and $\widetilde b$ is $L^2$
	orthogonal to $\mathcal H^1_{\widetilde h}(\widetilde Y;\widetilde B_0)$. A straightforward application of implicit function theorem shows for $\epsilon'$ small enough, $\widetilde U$ 
	determines an open neighborhood of $\fC(\widetilde Y,\widetilde E)$, and the representation of an element of $\widetilde U$ as $\widetilde B_0+\widetilde b$ is unique up to the action of 
	the gauge group. In particular, \eqref{CS-h-double} gives a well-defined real valued function on $\widetilde U$. Moreover, there is a universal constant $\delta_1$ such that 
	\begin{equation}\label{L21-L2dB}
	  \delta_1 \vert \! \vert\widetilde  b\vert \! \vert_{L^2_1(\widetilde Y)}\leq \vert \! \vert d_{\widetilde B_0}\widetilde b +
	  *_3{\rm Hess}_{\widetilde B_0}\widetilde h(\widetilde b)\vert \! \vert_{L^2(\widetilde Y)}.
	\end{equation}	
	This follows from the fact that the Chern-Simons functional $CS_h$ of $(\widetilde Y,\widetilde E)$ is Morse-Bott. The constant $\delta_1$ can be made independent of $\widetilde B_0$
	because $\fC(\widetilde Y,\widetilde E)$ is compact.

	The inequality in \eqref{L21-L2dB} allows us to control the $L^2_1$ norm of $\widetilde b$ by the norm of the gradient of the perturbed Chern-Simons functional $CS_{h}$. We have
	\begin{align*}
		\vert\!\vert F_{\widetilde B}+*_3\nabla_{\widetilde B} \widetilde h\vert\!\vert_{L^2}
		&=\vert\!\vert \(F_{\widetilde B}+*_3\nabla_{\widetilde B} \widetilde h\)-\(F_{\widetilde B_0}+*_3\nabla_{\widetilde B_0} \widetilde h\)\vert\!\vert_{L^2}\\
		&\geq \vert\!\vert d_{\widetilde B_0}\widetilde b+*_3{\rm Hess}_{\widetilde B_0}\widetilde h(\widetilde b)\vert\!\vert_{L^2}-
		\vert\!\vert\widetilde b \wedge \widetilde b\vert\!\vert_{L^2}-
		\vert\!\vert \nabla_{\widetilde B} \widetilde h-\nabla_{\widetilde B_0} \widetilde h-{\rm Hess}_{\widetilde B_0}\widetilde h(\widetilde b)\vert\!\vert_{L^2}.
	\end{align*}
	Since we have
	\begin{align*}
	  \vert\!\vert \nabla_{\widetilde B} \widetilde h-\nabla_{\widetilde B_0} \widetilde h-{\rm Hess}_{\widetilde B_0}\widetilde h(\widetilde b)\vert\!\vert_{L^2}
	  &\leq \vert\!\vert \(\int_{0}^1 {\rm Hess}_{\widetilde B_0+t\widetilde b}\widetilde h(\widetilde b) dt\)-{\rm Hess}_{\widetilde B_0}\widetilde h(\widetilde b)\vert\!\vert_{L^2},\\
	  &\leq\(\int_{0}^1\vert\!\vert {\rm Hess}_{\widetilde B_0+t\widetilde b}\widetilde h-{\rm Hess}_{\widetilde B_0}\widetilde h\vert\!\vert_{L^2} dt\) \vert\!\vert\widetilde b\vert\!\vert_{L^2},
	\end{align*}
	Corollary \ref{Hess-cont} implies that after decreasing the value of $\delta_1$ and shrinking $\widetilde U$ we have
	\begin{align*}
		\vert\!\vert F_{\widetilde B}+*_3\nabla_{\widetilde B} \widetilde h\vert\!\vert_{L^2}
		&\geq \delta_1 \vert \! \vert \widetilde b\vert \! \vert_{L^2_1} - \vert\!\vert\widetilde b\vert\!\vert_{L^4}^2\\
		&\geq (\delta_1-\vert\!\vert\widetilde b\vert\!\vert_{L^4}) \vert \! \vert\widetilde b\vert \! \vert_{L^2_1}.
	\end{align*}	
	Thus, if $\epsilon'$ is small enough, then there is a constant $\kappa_0$ such that
	\begin{equation}\label{control-L2-curv}
	    \vert \! \vert\widetilde  b\vert \! \vert_{L^2_1(\widetilde Y)}\leq \kappa_0 \vert\!\vert F_{\widetilde B}+*_3\nabla_{\widetilde B} \widetilde h\vert\!\vert_{L^2(\widetilde Y)}.
	\end{equation}
	
	The Chern-Simons functional of $\widetilde B$ in \eqref{CS-h-double} can be bounded in the following way:
	\begin{align}
	  CS_h(\widetilde B)
	  &\leq \frac{1}{8\pi^2}\vert \int_{\widetilde Y}\tr(\widetilde b\wedge d_{\widetilde B_0} \widetilde b+
  	  \frac{2}{3}\widetilde b\wedge \widetilde b\wedge \widetilde b)\vert+
	  \frac{1}{4\pi^2}\int_{0}^1dt\left\vert \int_{\widetilde Y}\tr\(\widetilde b\wedge (*_3(\nabla_{\widetilde B_0+t\widetilde b}\widetilde h-\nabla_{\widetilde B_0}\widetilde h))\)\right\vert\nonumber\\
	 & \leq C \(\vert \! \vert \widetilde b\vert \! \vert_{L^2(\widetilde Y)}\vert \! \vert d_{\widetilde B_0}\widetilde b\vert \! \vert_{L^2(\widetilde Y)}+
	 \vert \! \vert \widetilde b\vert \! \vert_{L^3(\widetilde Y)}^3+\vert \! \vert\widetilde b\vert \! \vert_{L^2(\widetilde Y)}^2\)\nonumber\\ 	  
	 &\leq   \kappa \vert\!\vert F_{\widetilde B}+*_3\nabla_{\widetilde B} \widetilde h\vert\!\vert_{L^2(\widetilde Y)}^2,\label{ineq-closed-CS-grad}
	\end{align}
	where the second inequality follows from the general property of cylinder functions which is stated in part (ii) of Proposition \ref{hol-pert-properties}:
	\[
	  |\!|\nabla_{\widetilde B} \widetilde h-\nabla_{\widetilde B'} \widetilde h|\!|_{L^2}\leq C|\!|B-B'|\!|_{L^2}.
	\]	
	The last inequality in \eqref{ineq-closed-CS-grad} is a consequence of \eqref{control-L2-curv} and the assumption that $\epsilon'$ is small enough. 
	
	The upper bound on $CS_h(\widetilde B)$ in \eqref{ineq-closed-CS-grad} for a connection $\widetilde B$ on $\widetilde E$ allows us to verify our main claim. There is a neighborhood $U$ of flat mixed pairs such that 
	for any $[B,q]$, the associated element $[\widetilde B]$ belongs to $\widetilde U$. Since we have $CS_h([B,q])=CS_h(\widetilde B)$ and 
	\[
	   \vert\!\vert F_{\widetilde B}+*_3\nabla_{\widetilde B} \widetilde h\vert\!\vert_{L^2(\widetilde Y)}^2=
	   \vert\!\vert F_{B}+*_3\nabla_{B} h\vert\!\vert_{L^2(Y)}^2+\vert\!\vert d q\vert\!\vert_{L^2([0,1])}^2,
	\]
	\eqref{ineq-closed-CS-grad} gives us the desired inequality in \eqref{ineq-mixed-grad-CS}.
\end{proof}

\subsection{Exponential decay}\label{exp-dec-subsec}

Our next goal is to show that the solutions of the mixed equation for the special quintuples have exponential decay on the mixed ends. This subsection follows a similar scheme as the proofs of the the corresponding results in the context of Yang-Mills gauge theory in \cite[Section 4]{Don:YM-Floer}. To obtain the desired exponential decay results and complete the proof of Proposition \ref{exp-decay}, we may focus on solutions of the mixed equation on the cylinder quintuple, which is introduced in Subsection \ref{Fredholm-mixed-cyl}.

\begin{lemma}\label{weak-conv-cylinder}
	Suppose $c_0$ is a connected component of $L(Y,E)$. For any open neighborhood $U$ of the space of flat mixed pairs in $\bB_{c_0}(Y,E)$, there is $\epsilon$ such that the following holds. If $(A,u)$ is a solution of the mixed 
	equation on the mixed cylinder $\fc_{(-1,1)}$ with $\mathcal E(A,u)<\epsilon$, then for any $\theta\in (-\frac{1}{2},\frac{1}{2})$, the pair $(B_\theta,q_\theta):=(A\vert_{\{\theta\}\times Y},u\vert_{[0,2]\times \{\theta\}})$ belongs to $U$. 
\end{lemma}
\begin{proof}
	If the claim does not hold, there is a sequence $(A_i,u_i)$ of solutions of the mixed equation on $\fc_{(-\frac{1}{2},\frac{1}{2})}$ such that $\mathcal E(A_i,u_i)\to 0$ as $i\to \infty$, and the class
	in $\bB_{c_0}(Y,E)$ represented by the restriction of $(A_i,u_i)$ to $(\{0\}\times Y, [0,2]\times \{0\})$ does not belong to $U$. On the other hand, 
	Theorem \ref{compactness-X-S-pert} implies that there is a subsequence $\{(A_i^\pi,u_i^\pi)\}$ and gauge transformations $g_i^\pi$ such that 
	$((g_i^\pi)^*A_i^\pi,u_i^\pi))$ is $C^\infty$ convergent to the solution $(A_0,u_0)$ of the mixed equation on $\fc_{(-\frac{1}{2},\frac{1}{2})}$.
	In particular, $\mathcal E(A_0,u_0)=0$, and after applying a gauge transofrmation
	 $(A_0,u_0)$ is the pullback of a flat mixed pair to the quintuple $\fc_{(-\frac{1}{2},\frac{1}{2})}$, which is a contradiction.
\end{proof}

\begin{prop}\label{exp-decay-mixed-cylinder-energy}
	There are positive constants $\epsilon$, $\delta_0$ and $C$ such that the following holds. Suppose $(A,u)$ is a solution of 
	the mixed equation on the mixed cylinder $\fc_{(0,\infty)}$ such that $\mathcal E(A,u)<\epsilon$. Then for any $\theta\in (1,\infty)$, we have
	\begin{equation}\label{L2-band}
		\vert\!\vert F_A+*_3\nabla_{A_t}h\vert\!\vert_{L^2((\theta-\frac{1}{2},\theta+\frac{1}{2})\times Y)}+
		\vert\!\vert d u\vert\!\vert_{L^2([0,2]\times (\theta-\frac{1}{2},\theta+\frac{1}{2}))}\leq Ce^{-\delta_0\theta},
	\end{equation}
	where for any $t\in (0,\infty)$, as usual $A_t$ denotes the connection $A\vert_{\{t\}\times Y_0}$.
\end{prop}

\begin{proof}
	Let $c_0$ be the connected component of $L(Y,E)$ determined by $u(2,\theta)$ for any $\theta\in (0,\infty)$. Suppose $(B_\theta,q_\theta)$ is the pair obtained by restricting $A$ and $u$ to $\{\theta\}\times Y$ and 
	$[0,2]\times \{\theta\}$. Then $(B_\theta,q_\theta)$ are elements of $\bA_{c_0}(Y,E)$. 
	Suppose $U$ and $\kappa$ are given as in Lemma \ref{mixed-grad-CS}. 
	Using Lemma \ref{weak-conv-cylinder}, we pick $\epsilon<\frac{1}{2}$ such that for any $\theta\in [\frac{1}{2},\infty)$,
	$(B_\theta,q_\theta)$ represents an element of $U$. Let also $\delta_0=\frac{1}{4\pi^2\kappa}$.
	Define $P:(0,\infty) \to \R$ by
	\[
	  P(\theta)=\frac{1}{8\pi^2}\int_{(\theta,\infty)\times Y}\tr\((F_A+*_3\nabla_{A_t}h)
	  \wedge (F_A+*_3\nabla_{A_t}h)\)+\frac{1}{4\pi^2}\int_{(\theta,\infty)\times [-1,1]}u^*\Omega.
	\]
	The assumption on $\mathcal E(A,u)$ implies that $P(\theta)$ is the lift of $CS_{h}(B_\theta,q_\theta)$ in $[0,\frac{1}{2})$.
	Since $(A,u)$ is a solution of the mixed equation, we have
	\[
	  \frac{dP}{d\theta}(\theta)=-\frac{1}{4\pi^2}(\vert\!\vert F_{B_\theta}+*_3\nabla_{B_\theta} h\vert\!\vert_{L^2(Y)}^2+
		\vert\!\vert d q_\theta\vert\!\vert_{L^2([0,1])}^2).
	\]
	Combining these observations and Lemma \ref{mixed-grad-CS}, we conclude that
	\[
	  \hspace{2cm}P\leq -4\pi^2 \kappa \frac{dP}{d\theta}\hspace{1cm}\text{for $\theta\geq \frac{1}{2}$}.
	\]
	In particular, for any $\theta\geq \frac{1}{2}$, we have the following inequality which gives the desired claim:
	\begin{equation}\label{dec-exp}
	  P(\theta)\leq Ce^{-\delta_0\theta}.
	\end{equation}
	Here we can take $C=e^{\delta_0/2}\epsilon$, because it is greater than $e^{\delta_0/2}P(\frac{1}{2})$. 
\end{proof}

To improve the exponential decay of Proposition \ref{exp-decay-mixed-cylinder-energy}, we need a variation of Theorem \ref{Fred-mixed-op}.

\begin{prop}\label{Fred-mixed-op-ind-A-u}
	Suppose $I$ and $J$ are respectively the intervals $(-1/2,1/2)$ and $(-1/4,1/4)$. For any $k\geq 1$, there are constants $\epsilon_k$ and $C_k$ such that the following holds. 
	Suppose $(A,u)$ is a solution of the mixed equation associated to the cylinder quintuple
	 $\fc_I$ with $\mathcal E(A,u)<\epsilon_k$. Then for any $(\zeta,\nu)\in E^k_{(A,u)}(I)$ we have
	\begin{equation}\label{Fred-estimate-ind-A-u}
		|\!|(\zeta,\nu)|\!|_{L^2_{k,A}(J)}\leq C_k\left(|\!|\mathcal D_{(A,u)}(\zeta,\nu)|\!|_{L^2_{k-1,A}(I)}+|\!|(\zeta,\nu)|\!|_{L^{2}(I)}\right).
	\end{equation}
	where $L^2_{k,A}$ denotes the Sobolev norm defined using the connection $A$. In particular, only the contribution of $\zeta$ to this norm depends on $A$ and the contribution of $\nu$ is independent of $A$.
\end{prop}
\begin{proof}
	Theorem \ref{Fred-mixed-op} implies that for a fixed $(A,u)$ we may find a constant $C_k$ such that \eqref{Fred-estimate-ind-A-u} holds for any $(\zeta,\nu)\in E^k_{(A,u)}(I)$. Using Remark \ref{fred-st-family} and compactness of the 
	space of flat mixed pairs, we may in fact find a constant 
	$C_k'$ which works for any $(A,u)$ which is the pullback of a flat mixed pair. Let $C_k=2C_k'$.
	If the claim does not hold for $k$, then there is a sequence $\{(A_i,u_i)\}$ of solutions of the mixed equation associated to the quintuple $\fc_I$ such that $\mathcal E(A_i,u_i)<1/i$ and there is an element of $E^k_{(A_i,u_i)}(I)$ for which 
	the inequality in \eqref{Fred-estimate-ind-A-u} fails for the pair $(A_i,u_i)$ and the constant $C_k$. Because of our assumption about the Sobolev norms in \eqref{Fred-estimate-ind-A-u}, changing each mixed pair
	$(A_i,u_i)$ by applying a gauge transformation gives us another sequence satisfying the same property. Theorem \ref{compactness-X-S-pert} implies that $(A_i,u_i)$, after passing to a subsequence and applying gauge 
	transformations, is $C^\infty$ convergent to a mixed pair $(A_0,u_0)$ on compact subspaces of $I \times Y$ and $[0,2]\times I$. The topological energy of $(A_0,u_0)$ vanishes and hence it is the pullback of a flat mixed pair. Now, we 
	may use Remark \ref{fred-st-family} to conclude that if $i$ is large enough then the inequality in \eqref{Fred-estimate-ind-A-u} holds for $(A_i,u_i)$, which is a contradiction.
\end{proof}

\begin{prop}\label{exp-decay-mixed-cylinder}
	For any non-negative integer $l$, there are positive constants $\epsilon$, $\delta_0$ and $C$ such that the following holds. Suppose $(A,u)$ is a solution of 
	the mixed equation on the mixed cylinder $\fc_{(0,\infty)}$ such that $\mathcal E(A,u)<\epsilon$ and $A$ is in temporal gauge. Then there is a flat mixed pair $(B,q)$ such that
	for any $\theta\in (1,\infty)$ and any $k\leq l$ we have 
	\begin{equation}\label{exp-decayall-der-conn}
		\vert \nabla^k (A-\pi^*B)\vert_{\{\theta\}\times Y}\leq Ce^{-\delta_0\theta},
	\end{equation}
	$q_\theta$ is $C^0$-convergent to the constant map to $q$, and 
	\begin{equation}\label{u-Cl}
		\hspace{2cm}\vert \nabla^{k-1} (du)\vert_{\{\theta\}\times [0,1]}\leq C e^{-\delta_0\theta}\hspace{1cm}\text{ for $1\leq k\leq l$}.
	\end{equation}
\end{prop}
\begin{proof}
	Let $\delta_0$ be given by Proposition \ref{exp-decay-mixed-cylinder-energy}, and decrease the value of $\epsilon$ in this proposition so that it becomes smaller than the constant $\epsilon_{l+3}$ provided by Proposition 
	\ref{Fred-mixed-op-ind-A-u}. Let also $C$ be the constant given by Proposition \ref{exp-decay-mixed-cylinder-energy}. In the following we might increase the value of $C$ from each line to the next while keeping 
	it independent of $(A,u)$. Suppose $(B_\theta,q_\theta)$ is given as before, and for any $\theta\in (1,\infty)$, let $A(\theta)$ and $u(\theta)$ denote the restriction of $A$ and $u$ to $(\theta-\frac{1}{2},\theta+\frac{1}{2})\times Y$ and 
	$[0,2]\times (\theta-\frac{1}{2},\theta+\frac{1}{2})$. By induction on $k$, we show that for any $0\leq k\leq l+3$ we have
	\begin{equation}\label{sobolev-exp-dec}
	  \vert\!\vert \frac{dA}{d\theta}\vert\!\vert_{L^2_k((\theta-\frac{1}{2},\theta+\frac{1}{2})\times Y)}+\vert\!\vert \frac{d u}{d\theta}\vert\!\vert_{L^2_k([0,2]\times (\theta-\frac{1}{2},\theta+\frac{1}{2}))}\leq Ce^{-\delta_1\theta}.
	\end{equation}
	This claim in the case that $k=0$ is proved in Proposition \ref{exp-decay-mixed-cylinder-energy}. Assuming \eqref{sobolev-exp-dec}, we may integrate this inequality from $\theta$ to $\infty$ in the case that $k=l+3$ and show see that there is a flat mixed pair $(B,q)$  
	such that the claim of this proposition holds.

	The derivatives $\frac{d A}{d\theta}$ and $\frac{d u}{d\theta}$ over $(\theta-\frac{1}{2},\theta+\frac{1}{2})\times Y$ and 
	$[0,2]\times (\theta-\frac{1}{2},\theta+\frac{1}{2})$ define an 	
	element $(\zeta(\theta),\nu(\theta))$ of $E^k_{(A(\theta),u(\theta))}(I)$, where $A(\theta)$ and $u(\theta)$ are the restrictions of $A$ and $u$ to $(\theta-\frac{1}{2},\theta+\frac{1}{2})\times Y$ and 
	$[0,2]\times (\theta-\frac{1}{2},\theta+\frac{1}{2})$ and $I=(-\frac{1}{2},\frac{1}{2})$. 
	The translation invariance of the mixed equation for cylinder quintuples implies that  $(\zeta(\theta),\nu(\theta))$ is in the kernel of $\mathcal D_{(A(\theta),u(\theta))}$. In particular, Proposition \ref{Fred-mixed-op-ind-A-u} and \eqref{sobolev-exp-dec}
	for $k=0$ imply that
	\begin{equation}\label{exp-decay-rest-theta}
		|\!|(\zeta(\theta),\nu(\theta))|\!|_{L^2_{k,A(\theta)}(J)}\leq Ce^{-\delta_0\theta}.
	\end{equation}
	Using the induction assumption we may replace the left hand side of \eqref{exp-decay-rest-theta} with $|\!|(\zeta(\theta),\nu(\theta))|\!|_{L^2_{k}(J)}$. This allows us to prove \eqref{sobolev-exp-dec} for the given $k$.
\end{proof}

\begin{proof}[proof of Proposition \ref{exp-decay}]
	Most steps of Proposition \ref{exp-decay} are already addressed in Subsection \ref{review-DFL:mix}. The only missing part is the exponential decay of solutions of the mixed equation 
	on the mixed ends of special quintuples, which follows from Proposition \ref{exp-decay-mixed-cylinder}.
\end{proof}

\subsection{Compactness}\label{comp-sect}
We shall consider the compactness aspects of Proposition \ref{comp} in this section. 
Suppose $\{(A_i,u_i)\}$ is a sequence in $\bA(\alpha,\beta)$ representing elements of $\bM_\eta(\alpha,\beta)_d$ with $d\leq 1$, which does not have any subsequence convergent to an element of $\bM_\eta(\alpha,\beta)_d$. Here $\eta$ gives a standard perturbation of the mixed equation as in \eqref{mixed-eq-pert}, which is trivial in a neighborhood of the matching line and is provided by Proposition \ref{pert}. In particular, $\bM_\eta(\alpha,\beta)_d$ is empty for negative values of $d$.  Since the indices of the mixed operators $\mathcal D_{(A_i,u_i)}$ are $d$, Proposition \ref{ind-mixed-op} implies that $\mathcal E(A_i,u_i)$ is constant and hence bounded. Therefore, the analytical energy terms $\fE(A_i,u_i)$ are also bounded, and we may apply Theorem \ref{compactness-X-S-pert}.

Theorem \ref{compactness-X-S-pert} implies that after passing to a subsequence and applying gauge transformations there is a sequence, still denoted by $\{(A_i,u_i)\}$, which is convergent in a weak sense. To be more detailed, there is a solution of the mixed equation $(A_0,u_0)$ associated to the special quintuples $\fq_s$ and finite sets $\sigma_-\subset {\rm int}(X)$, $\sigma_\partial \subset \gamma$ and $\sigma_+\subset S\setminus \gamma$ such that $\fE(A_0,u_0)$ is bounded, $A_i$ is $C^\infty$-convergent to $A_0$ on compact subspaces of $X\setminus (\sigma_\partial\times \Sigma\cup \sigma_-)$ and $u_i$ is convergent to $u_0$ on compact subspaces of $S\setminus (\sigma_+\cup \sigma_\partial)$. Proposition \ref{exp-decay} implies that $[A_0,u_0]\in \bM_\eta(\alpha',\beta')_{d'}$ for some choices of $\alpha'\in \fC_G$, $\beta'\in \fC_S$ and $d'\geq 0$. To be more precise, we use adaptation of Proposition \ref{exp-decay}  to the perturbed mixed equation in \eqref{mixed-eq-pert}. Since the secondary perturbations provided by by Proposition \ref{pert} are compactly supported, we may use the exponential decay results of the previous section to prove part (iii) of Proposition \ref{exp-decay} in this more general setup. As it is mentioned in Subsection \ref{review-DFL:mix}, the remaining part of this generalization of Proposition \ref{exp-decay} can be proved as in the original case. 

Fix $T_0$ such that the intersections of the the finite sets $\sigma_-$, $\sigma_\partial$ and $\sigma_+$ with the subspaces 
\begin{equation}\label{G-ends}
  (-\infty,-T_0]\times Y_\#\hspace{0.2cm} \sqcup \hspace{0.2cm} [T_0,\infty)\times Y \hspace{0.2cm}\sqcup\hspace{0.2cm} (-\infty,-T_0]\times Y',
\end{equation}
and 
\begin{equation}\label{S-ends}
  [T_0,\infty)\times [-1,1] \hspace{0.2cm}\sqcup\hspace{0.2cm} [0,2]\times [T_0,\infty)\sqcup \hspace{0.2cm} [0,2]\times (-\infty,-T_0],
\end{equation}
of the ends of $\fq_s$ are empty and the secondary perturbations of Proposition \ref{pert} on these sets are empty. For any $T\geq 2$, let $(A_i(T),u_i(T))$ denote the restrictions  of $(A_i,u_i)$ to the mixed end $((T,\infty)\times Y,[0,2]\times (T,\infty))$. Then $(A_i(T_0),u_i(T_0))$ gives a solution of the mixed equation associated to the cylinder quintuple $\fc_{(T_0,\infty)}$. There are gauge transformations $g_i$ over $(T_0,\infty)\times Y$ such that $A_i(T_0)$ is in temporal gauge. Proposition \ref{exp-decay-mixed-cylinder} implies that $(g_i^*A_i(T_0),u_i(T_0))$ is exponentially asymptotic to a flat mixed pair $(B_i,q_i)$. Similarly, there is $g_0$ such that $(g_0^*A_0(T),u_0(T))$ is exponentially asymptotic to a flat mixed pair $(B_0,q_0)$. Using patching arguments of \cite[Section 4.4.2]{DK}, we may assume that the gauge transformations $g_i$ are trivial after changing the gauge equivalence classes of connections $A_i$.

After passing to a further subsequence, we may assume that $(B_i,q_i)$ is $C^\infty$ convergent to a flat mixed pair $(B,q)$ because the gauge equivalence classes of flat mixed pairs is a compact set. A priori, $(B,q)$ might not be necessarily equal to $(B_0,q_0)$ because we only know that $A_i(T)-\pi^*(B_i)$ is $C^\infty$-convergent to $A_0(T)-\pi^*(B)$ on compact subspaces of $((T,\infty)\times Y,[0,2]\times (T,\infty))$. However, if there is $T\geq T_0$ such that the topological energies of $(A_i(T),u_i(T))$ for large enough values of $i$ is less than the constant $\epsilon$ given by Proposition \ref{exp-decay-mixed-cylinder}, then Proposition \ref{exp-decay-mixed-cylinder} and the dominated convergence theorem imply that 
\begin{equation}\label{L2-ext-delta}
  \hspace{1cm}\lim_{i\to \infty} \int_{(T,\infty)\times Y}\vert \nabla^k \left ((A_0(T)-\pi^*B)-(A_i(T)-\pi^*(B_i))\right )\vert^2 e^{\delta_0\theta}=0,\hspace{.5cm}0\leq k \leq l.
\end{equation}
In particular, $B=B_0$ and hence $q=q_0$. Moreover, by working in a smooth chart about $q\in \mathcal M(\Sigma,F)$, we obtain a similar exponential convergence of $u_i(T)$ to $u_0(T)$.

Now, suppose there is no $T$ satisfying the above properties. After passing to a subsequence, we may assume that there is a sequence $\{T_i\}_i$ converging to $\infty$ such that $\mathcal E(A_i(T_i),u_i(T_i))$ is equal to $\epsilon$. In particular, we have
\begin{align}\label{top-energy-ineq}
  \mathcal E(A_0(T_0),u_0(T_0))&\leq \limsup_{i\to \infty}\mathcal E(A_i(T_0),u_i(T_0))-\epsilon.
\end{align}
(In fact, we may replace $\epsilon$ with $\frac{1}{2}$ using the fact that the minimal topological energy of a solution of the mixed equation on $\fc_{\R}$ with finite energy is $\frac{1}{2}$. However, we do not need this stronger upper bound.) We obtain a similar dichotomy for the mixed end associated to $Y'$: either the solutions $(A_i,u_i)$ restricted to the mixed end $((-\infty,-T_0) \times Y',[0,2]\times (-\infty,-T_0))$ is convergent to $(A_0,u_0)$ as in \eqref{L2-ext-delta}, or there is a loss of topological energy by at least $\epsilon$ on the mixed end analogous to \eqref{top-energy-ineq}.

A similar analysis can be applied to study the behavior of the sequence $\{(A_i,u_i)\}$ on the gauge theoretical and symplectic ends of $\fq_s$, and one can even obtain more efficient results. For instance, we may apply the results of \cite[Section 5]{Don:YM-Floer} to the sequence of (perturbed) ASD connections $A_i\vert_{(-\infty,-T_0)\times Y_\#}$ and obtain a sequence 
\[
 A_1^G\in \breve \rM_G(\alpha,\alpha_1)_{p_1},\,A_2^G\in \breve \rM_G(\alpha_1,\alpha_2)_{p_2},\,\dots,\,A_n^G\in \breve \rM_G(\alpha_{n-1},\alpha')_{p_n}
\]
such that $[A_i\vert_{(-\infty,-T_0)\times Y_\#}]$ is {\it chain convergent} to $[A_1^G,A_2^G,\dots,A_n^G,A_0\vert_{(-\infty,-T_0)\times Y_\#}]$ on the complement of a finite set of bubbling points in the sense of \cite[Section 5]{Don:YM-Floer}. This means that in addition to the convergent of $A_i$ to $A_0$ on compact subspaces of $(-\infty,-T_0)\times Y_\#$, there are finite subsets $\sigma_i\subset \R\times Y_\#$ and a sequence of real numbers 
\[
  s_0^i=0<s_1^i<s_2^i<\dots<s_n^i
\]
such that for any $1\leq j \leq n$, we have $s_j^i-s_{j-1}^i \to \infty$ as $i\to \infty$ and $\tau_{s_j^i}^*(A_i)$ is $C^\infty$ convergent to $A_j^G$ on the compact subspaces of $\R\times Y_\#\setminus \sigma_i$. Here $\tau_{s_j^i}^*(A_i)$ denotes the translate of $A_i\vert_{(-\infty,-T_0)\times Y_\#}$ by $s_j^i$, which is a connection on $(-\infty,-T_0+s_j^i)\times Y_\#$. If $k_G$ is the sum of the size of the bubbling sets $\sigma_i$, the chain convergence implies that we have
 \begin{equation}\label{top-energy-ineq-gauge-theory}
  \mathcal E(A_0\vert_{(-\infty,-T_0)\times Y_\#})+\sum_{i=1}^n\mathcal E(A_i^G)\leq \limsup_{i\to \infty}\mathcal E(A_i\vert_{(-\infty,-T_0)\times Y_\#})-k_G.
\end{equation}
Similar argument shows that on the gauge theory side there is an integer $k_S$ and a sequence of holomorphic strips 
\[
 u_1^S\in \breve \rM_S(\beta',\beta_1)_{p_1'},\,u_2^S\in \breve \rM_S(\beta_1,\beta_2)_{p_2'},\,\dots,\,u_m^S\in \breve \rM_S(\beta_{m-1},\beta)_{p_m'}
\]
such that the sequence of pseudo-holomorphic maps $u_i\vert_{(T_0,\infty)\times [-1,1]}$ are chain convergent to the broken pseudo-holomorphic map $(u_0\vert_{(T_0,\infty)\times [-1,1]}, u_1^S,u_2^S,\dots,u_m^S)$ on the complement of a set of $k_S$ bubbling points. Consequently, we have
 \begin{equation}\label{top-energy-ineq-symplectic}
  \mathcal E(u_0\vert_{(T_0,\infty)\times [-1,1]})+\sum_{i=1}^m\mathcal E(u_i^S)\leq \limsup_{i\to \infty}\mathcal E(u_i\vert_{(T_0,\infty)\times [-1,1]})-\frac{k_S}{2},
\end{equation}
where for a map $u$ from an oriented Riemann surface $S$ to $\mathcal M(\Sigma,F)$, we define $\mathcal E(u)$ as the integral of $u^*\Omega$ over $S$ divided by $4\pi^2$. To get the term $k_s/2$ in \eqref{top-energy-ineq-symplectic}, we use the fact that the pairing of $\Omega$ with an element of $\pi_2(\mathcal M(\Sigma,F))$ is a positive integer multiple of $2\pi^2$ (see Subsection \ref{flat-surface}).

Theorem \ref{compactness-X-S-pert} implies that the topological energy of $(A_0,u_0)$ over the complement of the ends of $\fq_s$ in \eqref{G-ends} and \eqref{S-ends} is bounded by the $\limsup$ of the topological energies of $(A_i,u_i)$ over the same set, and this inequality can be improved by $\hbar$ unless $\sigma_-$, $\sigma_\partial$ and $\sigma_+$ are empty. Combining all of the above inequalities we conclude that 
\begin{equation}\label{ineq-top-energy}
  \mathcal E(A_0,u_0)+\sum_{i}\mathcal E(A_i^G)+\sum_{j}\mathcal E(u_j^S)\leq \limsup_{i}\mathcal E(A_i,u_i).
\end{equation}
Since $\mathcal E(A_i,u_i)$ is constant, we may replace the right hand side with $\mathcal E(A_i,u_i)$ for any $i\geq 1$. The inequality in \eqref{ineq-top-energy} is strict unless $\sigma_-$, $\sigma_\partial$ and $\sigma_+$ are empty, $(A_i,u_i)$ is convergent to $(A_0,u_0)$ with respect to the $L^2_{l,\delta}$ on the mixed ends as in \eqref{L2-ext-delta}, and $k_S=k_G=0$. If the inequality is strict, then Lemma \ref{top-energy-constant-reverse} implies that in fact the difference between the two sides of \eqref{ineq-top-energy} is at least $\frac{1}{2}$. To see the latter claim note that we can glue $(A_0,u_0)$, the connections $A_i^G$ and the maps $u_j^S$ to obtain an element of $\bB(\alpha,\beta)$, and the topological energy is additive with respect to gluing. Using Proposition \ref{ind-mixed-op} we also have 
\begin{equation}\label{ineq-index}
  \ind(\mathcal D_{(A_0,u_0)})+\sum_{i}\ind(\mathcal D_{A_i^G})+\sum_{j}\ind(\mathcal D_{u_j^S})\leq d,
\end{equation}
and if the inequality in \eqref{ineq-top-energy} is strict, then the left hand side of \eqref{ineq-index} is at most $d-4$, which is impossible because $d\leq 1$ and all the indices on the left hand side of \eqref{ineq-index} are non-negative. In fact, $\ind(\mathcal D_{A_i^G})$ and $\ind(\mathcal D_{u_j^S})$ are at least one. Thus, in the case that $d=0$, $m=n=0$, which shows that $[A_0,u_0]\in \bM_\eta(\alpha,\beta)_0$, and $[A_i,u_i]$ is a sequence in $\bM_\eta(\alpha,\beta)_0$ convergent to $[A_0,u_0]$. This is a contradiction, and hence $\bM_\eta(\alpha,\beta)_0$ is compact. In the case that $d=1$, we conclude that either $n=0$ or $m=0$. In the first case, $\beta=\beta'$, $[A_0,u_0]\in \bM_\eta(\alpha',\beta)_0$, $A_1^G$ is in the $0$-dimensional moduli space $\breve \rM_G(\alpha,\alpha')_{p_1}$, and $[A_i,u_i]$ is chain convergent to $[A_1^G,(A_0,u_0)]$. In the case that $m=0$, $\alpha=\alpha'$, $[A_0,u_0]\in \bM_\eta(\alpha,\beta')_0$, $u_1^S$ is in the $0$-dimensional moduli space $\breve \rM_S(\beta',\beta)_{p_1'}$, and $[A_i,u_i]$ is chain convergent to $[(A_0,u_0),u_1^S]$. Therefore, we can compactify $\bM_\eta(\alpha,\beta)_1$ by adding the points 
\begin{equation}\label{ends-1dim-com-sec}
	\breve \rM_{G}(\alpha, \alpha')_p \times \bM_\eta(\alpha',\beta)_0,\hspace{1cm}
	\bM_\eta(\alpha,\beta')_0\times\breve \rM_{S}(\beta',\beta)_p,
\end{equation}
as it is stated in Proposition \ref{comp}.

To complete the proof of part (ii) of Proposition \ref{comp}, we need a gluing theorem as an inverse to the above compactness theorem, which shows that the non-compact ends of $\bM_\eta(\alpha,\beta)_1$ can be parametrized by gluing the {\it broken solutions} of the mixed equation in \eqref{ends-1dim-com-sec}. Such gluing theory are standard in the context of instanton Floer homology \cite{floer:inst1}, \cite[Section 4]{Don:YM-Floer} and Lagrangian Floer homology \cite{Fl:LFH}. In the present setup, we need to either glue solutions of (perturbed) ASD equation on $\R\times Y_\#$ to solutions of the mixed equation on the gauge theoretical end of $\fq_s$ or glue pseudo-holomorphic strips in $\mathcal M(\Sigma,F)$ to solutions of the mixed equation on the symplectic end of $\fq_s$. In these two cases the proofs in the gluing theory of instanton Floer homology and Lagrangian Floer homology can be adapted without any essential change to prove the desired result. Using similar arguments as in instanton Floer homology and Lagrangian Floer homology, one can also see easily that the conventions in Subsection \ref{subsection-ori} determine orientations of the moduli spaces that satisfy the claim in Proposition \ref{comp}.

\section{Perturbations}\label{perturbations}

In several stages in the definition of instanton Floer homology group $\rI_*(Y_\#,E_\#)$, its symplectic counterpart $\rSI_*(Y_\#,E_\#)$ and the isomorphism between them, we had to use perturbations of the relevant equations. Up until this point we treat such perturbations as blackboxes, and only exploited their properties to obtain the required results. In this section, we recall the definition of these perturbations \cite{donaldson:orientations,floer:inst1,Tau:Cass,He:3-man-lag,KM:YAFT} and collect the properties which are used in the earlier sections. The first subsection reviews the definition of cylinder functions. We used the formal gradients of cylinder functions to perturb the defining equation of 3-manifold Lagrangians (Subsection \ref{3-man-lag-sec}) and the ASD equation in the definition of $\rI_*(Y_\#,E_\#)$ (Subsection \ref{Fl-hom-subsec}). The primary version of the mixed equation in \eqref{mixed-eq} is also defined using cylinder functions $h$ and $h'$. Then we use secondary perturbations to deform the mixed equation. These secondary perturbations are the main subject of the second subsection of the present section.

In this section, we slightly change our viewpoint on the configuration spaces of connections. As it is mentioned in the proof of Lemma \ref{fixed-iota}, the $\SO(3)$-bundle $E_\#$ over $Y_\#$ is the adjoint bundle associated to a $\U(2)$-bundle $\widetilde E_\#$ over $Y_\#$. Suppose $B_0$ is a fixed connection on the determinant bundle $\Lambda^2\widetilde E_\#$ of $\widetilde E_\#$. Then the configuration space of $\SO(3)$-connections on $E_\#$ modulo the determinant one gauge group $\mathcal G(E_\#)$ can be identified with the quotient by $\mathcal G(E_\#)$ of the space of $\U(2)$-connections on $\widetilde E_\#$ with $B_0$ being the induced connection on the determinant bundle. There is a similar description for the configuration spaces of connections on $Y$, $Y'$ and $X$ in terms of $\U(2)$-connections and we use this viewpoint through this section. (We assume that the $\U(1)$-connections fixed on these four spaces are compatible with each other in the obvious way.) In particular, holonomy of connections  will be regraded as a section of a fiber bundle with fiber $\U(2)$. This will give us some flexibility in the construction of holonomy perturbations.

\subsection{Cylinder functions and holonomy perturbations}\label{hol-pert-sect}

In this subsection we prove Proposition \ref{family-lag} and Lemma \ref{reg-ASD}. We first recall the definition of cylinder functions and then collect some of their properties from the literature which are used in the earlier subsections. As before $(Y,E)$ and $(Y',E')$ denote pairs which are introduced in Subsection \ref{3-man-bdles}.

Let $\gamma:S^1\times D^2\to Y$ be a smooth immersion, supported in the interior of $Y_0 \subset Y$. Given a connection $B\in \mathcal A(Y,E)$ and $z\in D^2$, let $\tau_z(B)$ denote the trace of the holonomy of $B$ along the loop $\gamma(S^1\times \{z\})$. Fix a compactly supported 2-from $\mu$ on $D^2$ with integral $1$. Define 
\[
  \tau_\gamma(B)=\int_{D^2} \tau_z(B)\mu.
\]
Now suppose $\lambda$ is the data of a positive integer $n$, smooth immersions $\gamma_i:S^1\times D^2\to Y_0$ for $1\leq i \leq n$, and a smooth function $G:[-3,3]^n \to \R$. The {\it cylinder function} $h_\lambda$ associated to $\lambda$ is defined as
\[
  h_\lambda(B):=G(\tau_{\gamma_1}(B),\dots,\tau_{\gamma_n}(B)).
\]
The formal gradient of the cylinder function $h_\lambda$ can be regarded as a section $\nabla h_\lambda$ of the tangent bundle $\mathcal T_l$ of  $\mathcal A(Y,E)$, which is the Banach space bundle $\mathcal A(Y,E)\times L^2_l(Y,\Lambda^1\otimes E)$. Clearly, $h_\lambda$ is invariant with respect to the action of $\mathcal G(E)$ and induces a map $\mathcal B(Y,E)\to \R$, still denoted by $h_\lambda$. The formal gradient $\nabla h_\lambda$ is also $\mathcal G(E)$-invariant and its value at any $B\in \mathcal A(Y,E)$ belongs to $X_B\subset L^2_l(Y,\Lambda^1\otimes E)$.

Following \cite{KM:YAFT}, we may form a Banach space $\mathcal P$ of perturbations from cylinder functions. Fix once and for all a sequence $\{\lambda_i\}_{i \in \N}$ where $\lambda_i$ is the information of a positive integer $n_i$, immersions $\gamma_{i,j}:S^1\times D^2\to Y_0$ for $1\leq j \leq n_i$ and a smooth function $G_i:\R^{n_i}\to \R$. We require this sequence to be dense in the following dense: given any $(n,\{\gamma_i\},G)$ as above there is a subsequence $\{\phi_{i_k}\}_{k \in \N}$ of  $\{\lambda_i\}$ such that $n_{i_k}=n$, $\gamma_{i_k,j}$, for any $j$, is convergent to $\gamma_j$ in the $C^\infty$ topology and $G_{i_k}$ is $C^\infty$ convergent to $G$. For a sequence of real numbers $\rho=\{a_i\}_{i\in \N}$, define $h_{\rho}:\mathcal A(Y,E)\to \R$ as
\[
  h_\rho=\sum_{i=1}^\infty a_ih_{\lambda_i}.
\]
To have convergent cylinder functions $h_\rho$ belonging to appropriate function spaces, we need to control the growth of the sequence  $\{a_i\}$ by a condition of the form
\[
  |\rho|:=\sum_{i=1}^\infty C_i|a_i|<\infty
\]
for a carefully chosen sequence of positive real numbers $C_i$. 

\begin{prop}{\cite[Proposition 3.7]{KM:YAFT}}\label{hol-pert-properties}
	The constants $C_i$ can be chosen such that the cylinder functions $h_\rho$ and their formal gradients $\nabla h_\rho$ satisfy the following properties.
	\begin{itemize}
	\item[(i)] The association 
	\[
	  (B,\rho)\to \nabla_Bh_\rho
	\]
	determines a smooth map $H:\mathcal A(Y,E)\times \mathcal P\to\mathcal T_l$.
	\item[(ii)] There is a constant $C$ such that for any $B\in \mathcal A(Y,E)$
	\[
	  |h_\rho(B)|\leq C|\rho|,\hspace{1cm} |\!|\nabla_B h_\rho|\!|_{L^\infty}\leq C|\rho|,
	\]
	and for any $B, B'\in \mathcal A(Y,E)$ and $1\leq p \leq \infty$
	\[
	  |\!|\nabla_B h_\rho-\nabla_{B'} h_\rho|\!|_{L^p}\leq C|\rho||\!|B-B'|\!|_{L^p}.
	\]
	\item[(iii)] The derivative $DH$ of $H$ with respect to the component in $\mathcal A(Y,E)$ defines a smooth section of $ \mathcal A(Y,E)\times \mathcal P\to Hom (\mathcal T_l,\mathcal T_l)$ which extends smoothly to 
	a smooth section of $\mathcal A(Y,E)\times \mathcal P\to Hom (\mathcal T_k,\mathcal T_k)$ for any $0\leq k\leq l$. Analogous to $\mathcal T_l$, the Banach space bundle $\mathcal T_k$ over $\mathcal A(Y,E)$ is given by 
	$\mathcal A(Y,E)\times L^2_k(Y,\Lambda^1\otimes E)$.
	\end{itemize}
\end{prop}

The following corollary is used in the proof of Proposition \ref{mixed-grad-CS}.
\begin{cor}\label{Hess-cont}
	For any $B\in \mathcal A(Y,E)$, $\rho\in \mathcal P$ and any constant $\epsilon$, there is a neighborhood $U$ of $B$ such that for any $B'\in U$ we have
	\[
	  \vert\!\vert {\rm Hess}_{B}h_\rho-{\rm Hess}_{B'}h_\rho\vert\!\vert_{L^2}<\epsilon.
	\]
\end{cor}	
\begin{proof}	
	This follows immediately from the special case of part (iii) of Proposition \ref{hol-pert-properties} that $k=0$.
\end{proof}

Recall that $\bL(Y,E)$ is the subspace of $\mathcal B(Y,E)\times \mathcal P$ consisting of pairs $([B],\rho)$ such that $\Phi(B,\rho)=0$ where $\Phi$ is defined in \eqref{flat-pert-3d}. Let $\pi:\bL(Y,E) \to \mathcal P$ be the projection map sending $([B],\rho)\in \bL(Y,E)$ to $\rho$. Given $([B_0],\rho_0)\in \bL(Y,E)$, the space $\bL(Y,E)$ in a neighborhood of a point $([B_0],\rho_0)$ can be identified with the solutions of the equation
\[
  \Pi_{B_0}\circ \Phi:\mathcal B(Y,E)\times \mathcal P\to \ker(d_{B_0}^*),
\]
with $\Pi_{B_0}$ being projection into $\ker(d_{B_0}^*)$. The linearization of this equation acts as
\begin{equation}\label{linearized-family}
  (b,\sigma) \to*d_{B_0}b+{\rm Hess}_{B_0}h_{\rho_0}(b)+\nabla_{B_0}h_\sigma
\end{equation}
on the elements of $X_{B_0}\times \mathcal P$. Restricting to pairs $(b,0)$ in \eqref{linearized-family} gives the Fredholm operator $L_{B_0}$ in \eqref{linearized} which governs the local behavior of the space $L_{h_{\rho_0}}(Y,E)=\pi^{-1}(\rho_0)\subset \mathcal B(Y,E)$. According to Proposition \ref{linearized-3-man-Lag}, the kernel and the cokernel of the operator $L_{B_0}$ are given by $\mathcal H^1_{h_{\rho_0}}(Y;B_0)$ and $\mathcal H^1_{h_{\rho_0}}(Y,\Sigma;B_0)$, and it has Fredholm index $-\frac{3}{2}\chi(\Sigma)$. We sketch a slight modification of an argument of \cite{He:3-man-lag} showing that the operator in \eqref{linearized-family} is surjective.  Therefore, the implicit function theorem implies that $\bL(Y,E)$ is a smooth Banach manifold and $\pi:\bL(Y,E) \to \mathcal P$ is a Fredholm map with index $-\frac{3}{2}\chi(\Sigma)$.

To verify surjectivity of \eqref{linearized-family}, it suffices to show that for any non-zero $b\in \mathcal H^1_{h_{\rho_0}}(Y,\Sigma;B_0)$, there is $\sigma\in \mathcal P$ such that $\langle \nabla_{B_0}h_\sigma,b\rangle $, which is equal to the derivative $D_{B_0}h_\sigma(b)$  of $h_\sigma$ at the connection $B_0$ evaluated at $b$, is non-zero. Our assumption on the density of the sequence $\{\rho_i\}$ implies that this claim holds unless the derivative $D_{B_0}\tau_\gamma$ vanishes at $b$ for any $\gamma$ as above. The latter condition holds only if the derivative $D_{B_0}{\rm Hol}_\ell$ evaluates to zero at $b$ for any closed loop $\ell$ in $Y_0$ \cite[Section 5.5]{Don:YM-Floer}. Here ${\rm Hol}_\ell:\mathcal A(Y,E)\to \U(2)$ is the holonomy of a connection on $Y$ along $\ell$. The same claim holds for loops in the collar neighborhood of the boundary $Y\setminus Y_0$ because $d_{B_0}b$ vanishes on $Y\setminus Y_0$ and $B_0$ restricts to a flat connection on $Y\setminus Y_0$. Finally vanishing of the derivative of ${\rm Hol}_\ell$ for closed loops $\ell$ in $Y_0$ and $Y\setminus Y_0$ can be used to conclude the vanishing of these derivatives for any closed loop in $Y$.

We claim that if $D_{B_0}{\rm Hol}_\ell$ vanishes at $b$, then there is a section $\zeta$ of $E$ on $Y$ such that $d_{B_0}\zeta=b$. Fix a basepoint $p_0\in \Sigma$, and for any $p\in Y$, take a path $\ell:[0,1] \to Y$ with $\ell(0)=p_0$, $\ell(1)=p$. Trivialize $E\vert_\ell$ by parallel transport and define
\begin{equation}\label{deta-def}
  \zeta(p):=\int_0^1 \ell^*(b).
\end{equation}
The assumption on $b$ implies that \eqref{deta-def} is independent of the choice of $\ell$ and hence is well-defined. From the definition, it is straightforward to check that $d_{B_0}\zeta=b$. In particular, $\zeta$ restricts to zero on $\Sigma$ because the same property holds for $b$ and the restriction of $B_0$ is irreducible. From this we obtain
\begin{align*}
	\int_Y\langle b,b\rangle=&-\int_Y\tr(d_{B_0}\zeta\wedge *b)\\
	=& -\int_{\Sigma}\tr(\zeta\wedge *b)-\int_Y\tr(\zeta\wedge *d_{B_0}^*b)=0,
\end{align*}
where in the second identity we use $d_{B_0}^*b=0$ and the vanishing of $\zeta$ on $\Sigma$. This shows that $b=0$, which is a contradiction and hence the operator in \eqref{linearized-family} is surjective.

\begin{prop}\label{familytransversality}
	Suppose $(B_0,\rho_0)$ is as above. Suppose $a$ is an $L^2_{l-1}$ section of $\Lambda^1\otimes E$ with $d_{B_0}^*a=0$. Then there are $\sigma\in \mathcal P$ and $b\in L^2_l(Y,\Lambda^1 \otimes E)$ with 
	$d_B^*b=0$ and $*b|_{\Sigma}=0$ such that 
	\begin{equation}\label{a-b-sigma}
		a=*d_{B_0}b+{\rm Hess}_{B_0}h_{\rho_0}(b)+\nabla_{B_0}h_\sigma
	\end{equation}
	Similarly, if  $a$ is an $L^2_{l-1}$ section of $\Lambda^1\otimes E$ with $d_{B_0}^*a=0$ and $*a|_{\Sigma}=0$, then there are $\sigma\in \mathcal P$ and $b\in L^2_l(Y,\Lambda^1 \otimes E)$ with 
	$d_B^*b=0$ and $b|_{\Sigma}=0$ such that \eqref{a-b-sigma} holds.
\end{prop}
\begin{proof}
	We already addressed the first part of this lemma. To prove the second part, we consider the map in \eqref{linearized-family} as an operator from $X_{B_0}'\times \mathcal P$ with 
	\[
	  X_{B_0}':=\{b \in L^2_l(Y,\Lambda^1 \otimes E) \mid d_B^*b=0,\, b|_{\Sigma}=0\},
	\]
	to the subspace of $\ker(d_{B_0}^*)$ given by $a$ with $*a|_{\Sigma}=0$. Restricting to pairs $(b,0)$ determines an operator $L_{B_0}'$ with the kernel and the cokernel $\mathcal H^1_{h_{\rho_0}}(Y,\Sigma;B_0)$ and $\mathcal H^1_{h_{\rho_0}}(Y;B_0)$.
	Now a similar argument as above can be used to complete the proof.
\end{proof}

\begin{prop}[cf. \cite{He:3-man-lag}]
	The map $r:\bL(Y,E) \to \mathcal M(\Sigma,F)$ which maps any element $([B_0],\rho_0)$ to the restriction $\alpha_0$ of $B_0$ to the boundary is a submersion.\label{res-par-submer}
\end{prop}
This proposition together with Proposition \ref{familytransversality} completes the proof of Proposition \ref{family-lag}.
\begin{proof}
	 Suppose $c$ is a 1-form with values in $F$ over $\Sigma$ which represents an element of $T_{\alpha_0}\mathcal M(\Sigma,F)$. That is to say, $d_{\alpha_0}c=0$. Extend $c$ to a smooth section $b_0$ of $\Lambda^1\otimes E$ over $Y$. 
	 Then $a:=*d_{B_0}(b_0)+{\rm Hess}_{B_0}h_{\rho_0}(b_0)$ is in the kernel of $d_{B_0}^*$ and hence Proposition \ref{familytransversality} applied to $a$ implies that there are $\sigma\in \mathcal P$ and $b_1\in L^2_l(Y,\Lambda^1 \otimes E)$ with $d_B^*b_1=0$ and 
	 $b_1|_{\Sigma}=0$ such that
	\[
	  *d_{B_0}(b_1-b_0)+{\rm Hess}_{B_0}h_{\rho_0}(b_1-b_0)+\nabla_{B_0}h_\sigma=0.
	\]
	In particular, the restriction of $b_0-b_1$ to $\Sigma$ is equal to $c$. There is an $L^2_{l+1}$ section $\zeta$ of $E$ such that $b:=b_1-b_0-d_{B_0}\zeta$ is in the Coulomb gauge
	\[
	  d_{B_0}^*b=0,\hspace{1cm}*b|_{\Sigma}=0.
	\]
	Thus $(b,\sigma)$ gives a vector tangent to $\bL(Y,E)$ whose restriction to the boundary represents the same element as $c$. 
\end{proof}

Next, we turn into the proof of Lemma \ref{reg-ASD}. As before, suppose $(Y_\#,E_\#)$ is an admissible pair with an admissible splitting
\[
  (Y,E)\cup_{(\Sigma,F)}(Y',E').
\]
Associated to the pairs $(Y,E)$, $(Y',E')$ we may form Banach spaces $\mathcal P$ and $\mathcal P'$ parametrizing perturbations. Since the map $\pi:\bL(Y,E)\to \mathcal P$ is Fredholm, the Sard-Smale theorem implies that there is a residual subset $\mathcal P_{\rm reg}$ of $\mathcal P$ such that for any $\sigma\in \mathcal P_{\rm reg}$, the space $L_{h_\sigma}(Y,E)$ is a smooth Lagrangian in $\mathcal M(\Sigma,F)$. We fix one such element $\sigma$ of $\mathcal P_{\rm reg}$. 
Similarly we may form a residual subset $\mathcal P_{\rm reg}'$ of $\mathcal P'$. The restriction map $r':\bL(Y',E') \to \mathcal M(\Sigma,F)$ is transversal to $L_{h_0}(Y,E)$ because the former map is a submersion. Therefore, by passing to a smaller residual subset, we may assume that any $\sigma'\in \mathcal P_{\rm reg}'$ has the property that $L_{h_0}(Y,E)$ and $L_{h_{\sigma'}}(Y',E')$ are transversal to each other. We fix one such $\sigma'$. 

The functions $h_\sigma$ and $h_{\sigma'}$ induce a perturbation of the flat equation for connections on $E_\#$ as in \eqref{flat-3d-closed}. We follow the same notation as before to denote the solutions of this equation with $\fC_G$. According to Lemma \ref{comp-per-reg} elements of $\fC_G$ are regular. We form \eqref{ASD} which is a perturbation of the ASD equation on configuration spaces $\mathcal B(\alpha,\beta)_p$ with $\alpha,\beta\in \fC_G$. Regularity of elements of $\fC_G$ implies that these equations are Fredholm (see, for example, \cite[Chapter 4]{Don:YM-Floer}). Note that the perturbation of the ASD equation in \eqref{ASD} are supported in the subset $\R\times (Y_0\sqcup Y_0')$ of the cylinder $\R\times Y$. Lemma \ref{reg-ASD} is a consequence of the following proposition.

\begin{prop}\label{reg-gauge-4D-mod-space}
	The Riemannian metrics on $Y$, $Y'$ and cylinder functions $h$ and $h'$ supported in the interior of $Y_0\subset Y$ and $Y_0'\subset Y'$ can be chosen such that the following conditions are satisfied.
	\begin{itemize}
		\item[(i)] Solutions of the perturbation of the flat equation \eqref{flat-3d-closed} for the functions $h$ and $h'$ agree with $\fC_G$.
		\item[(ii)] Solutions of the perturbation of the ASD equation \eqref{ASD} for the functions $h$ and $h'$ with index at most seven are regular.
	\end{itemize}
\end{prop}

Here the assumption on the index of the solutions of \eqref{ASD} is not essential and we are making this assumption to avoid the bubbling phenomena.
\begin{proof}
	The proof is just a slight modification of a similar result in \cite[Section 5.5]{Don:YM-Floer}. The main difference is that we want to 
	guarantee our perturbation vanishes in the complement of $Y_0\sqcup Y_0'$. Following the notation of Subsection \ref{3-man-bdles},
	let $Y_1$ (resp. $Y_1'$) be the union of $Y_0$ and $[-1,0]\times \Sigma$ (resp. $Y_0'$ and $[0,1]\times \Sigma$).
	Then the intersection of $Y_1$ and $Y_1'$ is a copy of $\Sigma$ with the collar neighborhood $[-1,1]\times \Sigma$.  
	Fix Riemannian metrics on $Y$ and $Y'$ in the same way as before. We show inductively that for any $k\leq 7$, there is a pair 
	$(h_k,h_k')$ of cylinder functions such that $h_k$ and $h_k'$ agree with $h_\sigma$ and $h_\sigma'$ in a neighborhood of $\fC_G$, 
	$\nabla_B h_k$ (resp. $\nabla_{B'} h_k'$) is compactly supported in the interior of $Y_1$ (resp. $Y_1'$) for any $B$ (resp. $B'$), 
	solutions of the perturbation of the flat equation \eqref{flat-3d-closed} for the pair $(h_k,h_k')$ agree with $\fC_G$, and the moduli space $M_G(\alpha,\beta)_p$ of solutions to the perturbed ASD equation associated to 
	$(h_k,h_k')$ is regular for any path $p$ of index at most $k$. In fact, the pair $h_k-h_\sigma$, $h_k'-h_{\sigma'}$ are sums of finitely many cylinder functions where our assumption on cylinder functions is slightly relaxed 
	and we allow cylinder functions
	 associated to immersion into the interior of $Y_1$ and $Y_1'$. In particular, $h_k$ and $h_k'$ vanish in the regular neighborhood 
	$[-\epsilon,\epsilon] \times \Sigma$ of $Y_1\cap Y_1'$ if $\epsilon$ is small enough. Note that if $k$ is small enough, then solutions of the perturbed ASD equation with index at most $k$ should have negative topological energy and hence these 
	moduli spaces are empty. Thus the claim for such values of $k$ holds if we set $(h_k,h_k')=(h_\sigma,h_{\sigma'})$.

	Next, we show that the claim holds for $k$ assuming that it already holds for $k-1$. We find a collection of cylinder functions 
	$(h_\rho,h_{\rho}')$ where $\rho$ belongs to a finite dimensional vector space $\mathcal P_\#$, $(h_\rho,h_{\rho}')$ depends linearly on
	$\rho$, $(h_\rho,h_\rho')$ is a  cylinder functions associated to an immersion into the interior of $Y_1$ and $Y_1'$, and for small values of $\rho$, 
	solutions of the perturbation of the flat equation \eqref{flat-3d-closed} for the pair $(h_{k-1}+h_\rho,h_{k-1}'+h'_{\rho})$ agree with 
	$\fC_G$. Moreover, for any path $p$ of index $k$, the family moduli space $\mathbb M_G(\alpha,\beta)_p\subset \mathcal B_G(\alpha,\beta)_p\times \mathcal P_\#$, which is the union of 
	the moduli spaces $M_G(\alpha,\beta)_p$ of solutions to the perturbed ASD equation associated to $(h_{k-1}+h_\rho,h_{k-1}'+h'_{\rho})$ for all $\rho\in \mathcal P_\#$, is cut down transversely at elements of the form $([A],0)$.

	Suppose $([A],0)$ is an element of a moduli space $\mathbb M_G(\alpha,\beta)_p$ with index $k$ which is not cut down transversely.
	Our induction assumption and the assumption $k\leq 7$ imply that the space of all such non-regular elements of $\mathbb M_G(\alpha,\beta)_p$ is compact. Assume that $A$ is given 
	in the temporal gauge, and let $A_t$ denote the restriction of $A$ to $\{t\}\times Y_\#$. 
	A non-trivial element in the cokernel of the linearized operator for $\mathbb M_G(\alpha,\beta)_p$ is given by a smooth family of $1$-forms $\{\phi_t\}_{t\in \R}$ on 
	$Y_\#$ with values in $E_\#$ such that 
	\begin{equation}\label{ODE-phi}
	  \frac{d}{dt}\phi_t=-L_t\phi_t,
	\end{equation}
	\begin{equation}\label{phi-2-eq}
	  d_{A_t}^*\phi_t=0,\hspace{2cm}\int_{-\infty}^\infty \langle \nabla_{A_t}h_\rho + \nabla_{A_t'}h_\rho',\phi_t\rangle =0,\hspace{1cm} \forall \rho\in \mathcal P_\#,
	\end{equation}
	and the $L^2$ norm of $\phi_t$ converges to zero as $|t| \to \infty$. Here $L_t$ is a self-adjoint operator (with respect to the $L^2$ norm) defined on the sections of 
	$\Lambda^1(Y_\#)\otimes E_\#$. The operator $L_t$ depends on $A_t$, $(h_0+h_\rho,h_0'+h_\rho')$ and $L_t(b)$ is equal to $*_3d_{A_t}b$ outside the supports of $h_0+h_\rho$ and $h_0'+h_\rho'$. 
	Unique continuation of the solutions 
	of the equations of the form \eqref{ODE-phi} \cite[Lemma 7.1.3]{km:monopole} implies that $\phi_t$ is non-zero for all values of $t$.
	For $\dot A_t:=\frac{d}{dt}A_t$, we have
	\begin{equation}\label{ODE-Adot}
	  \frac{d}{dt}\dot A_t=L_t\dot A_t.
	\end{equation}
	Using \eqref{ODE-phi},  \eqref{ODE-Adot} and the decay of $\phi_t$ and $\frac{d}{dt}\dot A_t$, we can see that $\phi_t$ is $L^2$-orthogonal to $\frac{d}{dt}\dot A_t$ for all $t$. This claim can be proved by differentiating the inner product of 
	$\phi_t$ and $\dot A_t$ with respect to $t$.
	
	Since the restrictions of $\phi_t$ and $\dot A_t$ to $Y_\#\setminus \Sigma$ are linearly independent, we may use the same argument as in \cite[Proposition 5.17]{Don:YM-Floer} to find cylinder 
	functions $h_{\rho}$, $h_{\rho}'$ supported in the interior of $Y$, $Y'$ such that $h_{\rho}$, $h_{\rho}'$ vanish in a neighborhood of $\fC_G$ and 
	\[
	  \int_{-\infty}^\infty \langle \nabla_{A_t}h_\rho + \nabla_{A_t'}h_\rho',\phi_t\rangle >0.
	\]
	In particular, if we enlarge $\mathcal P_\#$ using $(h_\rho, h_\rho')$, then the dimension of the cokernel of the linearization of the moduli space $\mathbb M_G(\alpha,\beta)_p$
	at $([A],0)$ decreases by one. Because of the compactness of the space of the non-regular elements of $\mathbb M_G(\alpha,\beta)_p$ with index $k$, we may iterate this process and 
	modify $\mathcal P_\#$ such that $\mathbb M_G(\alpha,\beta)_p$ for any $p$ with index $k$ is regular. 	
	Now, a standard application of Sard's theorem shows that for a generic small $\rho\in \mathcal P_\#$, the pair $(h_k,h_k')=(h_{k-1}+h_\rho,h_{k-1}'+h'_{\rho})$ verifies the claim for $k$.
	
	Let $(h_7,h_7')$ be the cylinder functions given for $k=7$. Then $(h_7,h_7')$ is supported in the complement of $\Sigma\subset Y_\#$. In fact, these functions are supported in the 
	complement of a regular neighborhood of $\Sigma$ because  $h_\sigma$, $h_\sigma'$ are already supported in $Y_0$, $Y_0,$ 
	and $h_7-h_\sigma$ and $h'_7-h_\sigma'$ are defined using finitely many immersions into the interior of $Y_1$ and $Y_1'$. If we rescale the metric on $Y_\#$ by a constant, the same assumption on the regularity of the moduli spaces hold. 
	Moreover, we may assume that the rescaling constant 
	is large enough so that $(h_0+h_{1},h_0'+h_{1}')$ is supported in the complement of a copy of $(-1,1)\times \Sigma$ equipped with the product metric of the standard metric on $(-1,1)$ and some Riemannian metric on $\Sigma$. 
\end{proof}

\subsection{Secondary perturbations}\label{pert-sect}
The purpose of this subsection is to prove Proposition \ref{pert} using a secondary perturbation of the mixed equation. The perturbed equation has the form in \eqref{mixed-eq-pert}, copied below again for the reader's convenience:
\begin{equation}\label{mixed-eq-pert-2}
	\left\{
	\begin{array}{l}
		F^+(A)+(*_3\nabla_{A_t}h)^++(*_3\nabla_{A_t'}h')^++\eta(A)=0,\\
		\overline \partial_{J} u=0.
	\end{array}
	\right.
\end{equation}
The perturbation of the ASD equation is given by the holonomy perturbation term $\eta(A)$, and the perturbation of the CR equation is provided by perturbing the complex structure $J$. Let $X_c$ and $U_c$ be the sunspaces of $X$ and $U_+$ obtained by the complement of the gray region sketched in Figure \ref{support}. Since we established the analysis of mixed equation in a neighborhood of the matching line only in the unperturbed case, we limit ourselves to holonomy perturbations which are supported in $X_c$ and complex structures $J$ which differ from the standard complex structure $J_*$ only in $U_c$.
\begin{figure}
	\begin{center}
	\begin{tikzpicture}[thick]
\pic [scale=0.15](lower) at (-1.15,2) {handle};
\pic [scale=0.15](lower) at (-1.65,2.35) {hole};
\pic [scale=0.15](lower) at (-1.15,-2) {handle};
\pic [scale=0.15](lower) at (-1.65,-1.8) {hole};
\pic [scale=0.15](lower) at (-1.65,-1.4) {hole};
\shade[left color=gray,right color=gray](0,3)--(0,-3)--(0.3,-3)--(0.3,3);		     
\shade[right color=gray,left color=gray](0,3)--(0,-3)--(-0.3,-3)--(-0.3,3);	
\draw (0,3)  --  (0,-3)
	(-1,3)  --  (-1,1.5)
	(-1,1.5) coordinate (-left) 
	to [out=-90, in=0] (-2,.5)
	 (-2,.5) -- (-3.5,0.5)
	 (-2,-.5) -- (-3.5,-0.5)
	(-2,-.5) coordinate (-left) 
	to [out=0, in=90] (-1,-1.5)
	to (-1,-3);
\draw[red]	(-.70,3) to (-.70,1.5)
	to [out=-90, in=0] (-2,.2)
	to (-3.5,0.2)
	(-3.5,-0.2) to (-2,-.2)
	to [out=0, in=90] (-.70,-1.5)
	to (-0.7,-3);
\draw (1,3)  --  (1,1.5)
	(1,1.5) coordinate (-left) 
	to [out=-90, in=180] (2,.5)
	 (2,.5) -- (3.5,0.5)
	 (2,-.5) -- (3.5,-0.5)
	(2,-.5) coordinate (-left) 
	to [out=180, in=90] (1,-1.5)
	to (1,-3);
\draw[dotted] (-.5,4.5)  --  (-0.5,3.5);	
\draw[dotted] (.5,4.5)  --  (0.5,3.5);	
\draw[dotted] (-.5,-4.5)  --  (-0.5,-3.5);	
\draw[dotted] (.5,-4.5)  --  (0.5,-3.5);	
\draw[dotted] (4,0)  --  (5,0)
		     (-4,0)  --  (-5,0)
		     (-1.5,2.6) -- (-1.5,3)
		     (-1.5,1.4) to [out=-90, in=0] (-2,.8)
		     to (-2.5,0.8)
		     (-1.5,-2.6) -- (-1.5,-3)
		     (-1.5,-1.4) to [out=90, in=0] (-2,-.8)
		     to (-2.5,-0.8);		     		     
	\end{tikzpicture}
	\end{center}
		\caption{ Support of the secondary perturbation terms is in the complement of the gray region.}
	\label{support}
\end{figure}
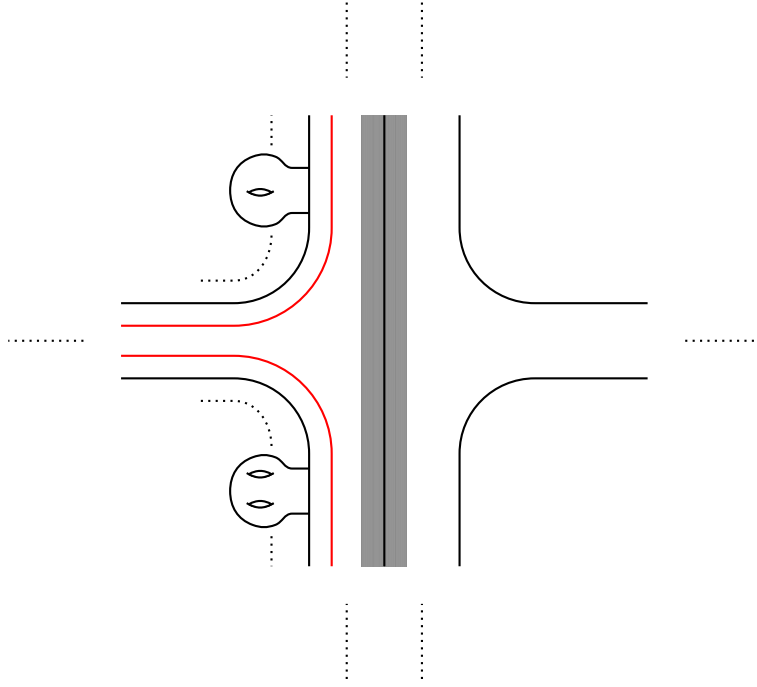

The definition of the holonomy perturbation term $\eta(A)$ is analogous to the definition of cylinder functions. Given a Riemannian 4-manifold $M$, let $\gamma:S^1\times D^4\to M$ be a smooth submersion such that $\gamma$ restricted to $\{1\}\times D^4$ is an embedding of $D^4$ into $M$. Let also $\omega$ be a self-dual 2-form on the image of $\gamma(\{1\}\times D^4)$. For any connection $A$ on a $U(2)$ bundle $E$ over $M$, the holonomy of $A$ along the loop $\gamma(S^1\times\{z\})$ is an element of the fiber of ${\rm End}(E)$ over the point $\gamma(1,z)$. Let $H_z(A)$ be the image of this holonomy with respect to the homomorphism ${\rm End}(E) \to \su(E)$ induced by the projection of ${\rm End}(\C^2) \to \su(2)$. Then $H_z(A)$ determines a section of $\su(E)$ over $\gamma(\{1\}\times D^4)$. Define
\begin{equation}
  P_{\gamma,\omega}(A):=H_z(A)\otimes \omega
\end{equation}
which gives a gauge invariant map from the space of connections on $E$ to the space of sections of $\Lambda^+\otimes \su(E)$. Note that the $|P_{\gamma,\omega}(A)|$ is bounded by the $C^0$ norm of $\omega$.

Now, we turn into the proof of Proposition \ref{pert}. We shall define the perturbation term $\eta$ as a linear combination
\begin{equation}\label{eta-form}
  \eta=\sum_{i=1}^NP_{\gamma_i,\omega_i}
\end{equation}
for a finite collection of $\gamma_i:S^1\times D^4 \to X_c$ and self-dual 2-forms $\omega_i$. For any solution $(A,u)$ of \eqref{mixed-eq-pert-2}, we have
\[
  \mathcal E(A,u) =\frac{1}{8\pi^2} \fE(A,u)-\frac{1}{4\pi^2}\vert\!\vert\eta(A)\vert\!\vert_{L^2(X)}^2\geq -C\sum_{i=1}^N|\!|\omega_i|\!|_{L^2}^2
\] 
where $C$ is a fixed constant, independent of $(\gamma_i,\omega_i)$. Thus, if the sum of $|\!|\omega_i|\!|_{L^2}^2$ is less than $\epsilon$ for a small enough $\epsilon$, then the topological energy of any solution $(A,u)$ of \eqref{mixed-eq-pert-2} is greater than $-\delta$ for a given positive constant $\delta$. Since the set of possible values for the topological energy $\mathcal E(A,u)$ for elements of the configuration spaces $\bB(\alpha,\beta)$ is a discrete subset of $\R$, we can pick $\epsilon$ such that the topological energy of any solution of \eqref{mixed-eq-pert-2} is non-negative. To prove Proposition \ref{pert}, we follow a similar strategy as in Proposition \ref{reg-gauge-4D-mod-space}. By induction on the expected dimension of the moduli spaces of mixed equation, we show that for any $k\leq 3$, there are $\eta_k$ and a family complex structures $\mathcal J_{k}=\{J_{(s,\theta)}\}_{(s,\theta)\in U_+}$ as in Subsection \ref{quintuples} such that all moduli spaces $\bM_{\eta_k}(\alpha,\beta)_d$, with $d\leq k$ and defined with respect to $\mathcal J_k$, are cut down transversely, and the moduli spaces $\bM_{\eta_k}(\alpha,\alpha)_0$ consists of a single regular element. Moreover, the perturbation term $\eta$ has the form in \eqref{eta-form} with the sum of $|\!|\omega_i|\!|_{L^2}^2$ being less than $\epsilon-1/2^{k}$. If $k$ is small enough, then the trivial perturbation of the ASD equation and a family of almost complex structures as in Subsection \ref{quintuples} satisfy this claim.

Next, we show that the claim holds for $k$ assuming that it already holds for $k-1$. Suppose $\eta_{k-1}$ and $\mathcal J_{k-1}$ are chosen satisfying the above properties. We show that there is a collection $\{(\gamma_i,\omega_i)\}_{i=1}^N$, an open neighborhood $\mathcal U$ of the origin in a Euclidean space, and for each $x\in \mathcal U$, a family of compatible almost complex structures $\{J^x_{(s,\theta)}\}_{(s,\theta)\in U_+}$ on $\mathcal M(\Sigma,F)$ with $\{J^0_{(s,\theta)}\}_{(s,\theta)\in U_+}=\mathcal J_{k-1}$ such that the following claim holds. Suppose $\fM(\alpha,\beta)$ is the subspace of $([A,u],r_1,\dots,r_N,x)\in \bB(\alpha,\beta)\times \R^{N}\times \mathcal U$ where $(A,u)$ is a solution of \eqref{mixed-eq-pert-2} defined using almost complex structures $\{J^x_{(s,\theta)}\}_{(s,\theta)\in U_+}$ and the perturbation 
\[\eta=\eta_{k-1}+\sum_{i=1}^Nr_iP_{\gamma_i,\omega_i}.\]
Then any solution $(A,u,0,0,\dots,0,\vec 0)\in \fM(\alpha,\beta)$ with $\ind(\mathcal D_{(A,u)})=k$ is cutdown transversely. Here $J^x_{(s,\theta)}$ depends smoothly on $(x,s,\theta)$, is equal to the standard complex structure $J_*$ for $s\leq 1$ and is equal to $J_\theta$, the complex structure given by Lemma \ref{alm-cx-str} for $s\geq 3$. Moreover, $J^x_{(s,\theta)}$ is constant in the $\theta$ direction if $|\theta|>2$.

Assuming the claim of the previous paragraph, a standard application of Sard's theorem shows that there are $\br=(r_1,\dots,r_N)\in \R^{N}$ and $x\in \mathcal U$ with with arbitrary small norms such that if we define $\bM_{\eta}(\alpha,\beta)_d$ using $\eta=\eta_{k-1}+\sum_{i=1}^Nr_iP_{\gamma_i,\omega_i}$ and the family of compatible almost complex structures $\{J^x_{(s,\theta)}\}_{(s,\theta)\in U_+}$, then any such moduli space is regular if $d\leq k$. Since any moduli space of the form $\bM_{\eta_{k-1}}(\alpha,\alpha)_0$ already contains a unique regular element, if $\br$ is small enough, then the moduli space $\bM_{\eta_{k-1}}(\alpha,\alpha)_0$ contains a unique regular element, too. Moreover, a small enough $\br$ allows us to guarantee that $\eta$ has the desired form in \eqref{eta-form} with the sum of $|\!|\omega_i|\!|_{L^2}^2$ being less than $\epsilon-1/2^{k}$. This completes the proof of Proposition \ref{pert}.

Now, we turn to the construction of the family of perturbations of the ASD equation and the compatible almost complex structures on $\mathcal M(\Sigma,F)$. Suppose $(A,u,0,0,\dots,0,\vec 0)\in \fM(\alpha,\beta)$ is a non-regular solution of $\fM(\alpha,\beta)$ with the index of $\mathcal D_{(A,u)}$ being $k$. The induction assumption and the compactness results of Subsection \ref{comp-sect} imply that the set of all such non-regular solutions is compact. Fix a non-zero element $(\mu,\xi,z)$ in the cokernel of the linearization of $\fM(\alpha,\beta)$ at $(A,u,0,0,\dots,0,\vec 0)$, where as before $\mu$ and $\xi$ respectively denote a $0$-form and a self-dual 2-form on $X$, and $z$ denotes a a section of  $u^*T\mathcal M(\Sigma,F)$ over $U_+$.

If $(\mu,\xi)$ is non-trivial, then unique continuation implies that the restriction of $\mu$ to $X_c$ is nontrivial \cite{Ar:Unique-cont} because the restriction of $A$ to the complement of $X_c$ satisfies the (unperturbed) ASD equation. Since $A$ over $X_c$ is irreducible, we may find $(\omega,\gamma)$ as above such that $P_{\gamma,\omega}(A)\neq 0$ (See, for example, the proof of Lemma 13 in \cite{K:higher}). Therefore, if we add $(\omega,\gamma)$ to the collection $\{(\gamma_i,\omega_i)\}_{i=1}^N$, then the dimension of the cokernel of $\fM(\alpha,\beta)$ at $(A,u,0,0,\dots,0,\vec 0)$ decreases by $1$. In the case that $(\mu,\xi)$ is trivial, the restriction of $z$ to $U_\partial$ needs to be trivial. If $u$ is a constant map, then one can use the arguments of Subsection \ref{cons-map-ind-subs} to see that the vanishing of $z$ on $U_\partial$ implies that $z$ vanishes globally. This is in contradiction with non-triviality of $(\mu,\xi,z)$, and hence $u$ is a non-constant map. In particular, there are points in $U_c$ where the derivative of $u$ does not vanish. This implies that we may enlarge $\mathcal U$ by adding another direction to deform $\{J^x_{(s,\theta)}\}_{(s,\theta)\in U_+}$ such that the deformation is compactly supported in $U_c$ and the dimension of of the cokernel of $\fM(\alpha,\beta)$ at $(A,u,0,0,\dots,0,\vec 0)$ decreases by $1$. (See, for example, the proof of Proposition 6.7.7. in \cite{MS:J-holo}. Note also that the domain $U_+$ is rigid and does not have any non-trivial automorphism mapping $U_\partial$, $\eta_+$ and $\eta_+'$ to themselves.) By iterating this process finitely many times, we may guarantee the regularity of $(A,u,0,0,\dots,0,\vec 0)$. Using the compactness of the space of all such elements of $\fM(\alpha,\beta)$, we can more generally achieve regularity at all points $(A,u,0,0,\dots,0,\vec 0)\in \fM(\alpha,\beta)$ with the same topological energy. This completes the proof of our claim.

\section{Extensions of the main theorem}
The goal of this section is to show that the isomorphism $\bN$ of Theorem \ref{main-thm} is compatible with certain additional structures on instanton Floer homology and its symplectic variant. In the first subsection, we define the structure relevant for Theorem \ref{framed-AF} in the introduction and then prove this theorem. A more precise version of Theorem \ref{framed-AF} is stated as Theorem \ref{enriched-AF}. In the second subsection, we review the definition of the operators $m_\sigma^G$, $m_\sigma^S$ and then prove Theorem \ref{module-over-quantum}.

\subsection{Filtered framed Floer homology groups}\label{action}
Topological energy of solutions to the mixed equation plays a key role in the proof of our main theorem in Section \ref{main-thm-sec}. In fact, we can use the notion of topological energy to define an additional structure on instanton Floer homology and its symplectic version. We call this additional structure the {\it Chern--Simons} filtration. 
In the discussion below, we follow similar conventions as in \cite{DS1}.

To define the Chern--Simons filtration, it is convenient to introduce $\overline \fC_G$, as a variation of $\fC_G$. Fix an element $\alpha_0$ of $\fC_G$, and let $\overline \fC_G$ consist of pairs $\overline \alpha=(\alpha,p)$ where $\alpha\in \fC_G$ and $p$ is a path from $\alpha$ to $\alpha_0$. We call $\overline \alpha$ a {\it lift} of $\alpha$. Define $\deg_I(\overline \alpha)$, the I-grading of $\overline \alpha$, to be the topological energy of the path $p$ defined in \eqref{top-energy-connection}. Any $\overline \alpha=(\alpha,p)$ in $\overline \fC_G$ is determined by $\alpha$ and $\deg_I(\overline \alpha)$. Moreover, for any two different lifts $\overline \alpha$ and $\overline \alpha'$ of $\alpha$, the expression $2(\deg_I(\overline \alpha)-\deg_I(\overline \alpha'))$ is an integer. We define a bijection $U:\overline \fC_G \to \overline \fC_G$ by requiring that for a lift $\overline \alpha$ of $\alpha$ we have 
\[\deg_I(U(\overline \alpha))-\deg_I(\overline \alpha)=\frac{1}{2}.\]
Similarly, define $\overline C_G(Y_\#,E_\#)$ to be the variation of $C_G(Y_\#,E_\#)$ which is the free abelian group generated by the elements of $\overline \fC_G$. The bijection $U$ defines the structure of a $\Z[U^{\pm1}]$-module on $\overline C_G(Y_\#,E_\#)$. We also modify the differential on $C_G(Y_\#,E_\#)$ in the following way to get a differential $d:\overline C_G(Y_\#,E_\#) \to \overline C_G(Y_\#,E_\#)$:
\begin{equation}\label{diff-filtered-gauge}
	d(\overline \alpha):=\sum_{p:\alpha\to \beta}\#\breve \rM_G(\alpha,\beta)_p \cdot \overline \beta.
\end{equation}
Here $\overline \alpha, \overline \beta\in \overline \fC_G$ are lifts of $\alpha, \beta\in \fC_G$ which are related by the path $p$. That is to say, the path from $\alpha$ to $\alpha_0$ given by $\overline \alpha$ is equal to the composition of $p$ and the path provided by $\overline \beta$. The differential map in \eqref{diff-filtered-gauge} is a $\Z[U^{\pm1}]$-module homomorphism. The following lemma implies that the I-grading defines a filtration on $\overline C_G(Y_\#,E_\#)$, which we call the Chern--Simons filtration. 

\begin{lemma}\label{ASD-pos-energy}
	Suppose  $\overline \alpha, \overline \beta\in \overline \fC_G$ are lifts of $\alpha, \beta\in \fC_G$ which are related by the path $p$ and the moduli space $\rM_G(\alpha,\beta)_p $ is non-empty. Then 
	\[
	  \deg_I(\overline \alpha)\geq \deg_I(\overline \beta),
	\]
	and the equality holds if and only if $\overline \alpha=\overline \beta$ and $p$ is the constant path.
\end{lemma}

\begin{proof}
	This is a consequence of the fact that for $[A]\in \rM_G(\alpha,\beta)_p $ the topological energy $\mathcal E(A)$ is non-negative because it is equal to $\frac{1}{8\pi^2}\fE(A)$ where 
	\[
	  \fE(A):=\int_{\R\times Y_\#}\vert F_A+*_3\nabla_{A_t}h+*_3\nabla_{A_t'}h'\vert^2 \rd t\,\dvol_{Y_\#}.
	\]
\end{proof}

We may define a Chern--Simons filtration on symplectic instanton Floer homology in a similar way. The sets $\fC_G$ and $\fC_S$ are naturally identified with each other and we define the $\Z$-covering $\overline \fC_S$ of $\fC_S$ to be the same as $\overline \fC_G$ with the same I-grading as above. We define $\overline C_S(Y_\#,E_\#)$ to be the abelian group generated by $\overline \fC_S$. Suppose $\alpha,\beta\in \fC_S$ and $p$ is a path from $\alpha$ to $\beta$ represented by a strip $u:\R\times [-1,1]\to \mathcal M(\Sigma,F)$. By gluing $u$ to the constant mixed pair $(A_\alpha,u_\alpha)$ as in Subsection \ref{ind-mixed-op-subs} and then applying mixed shifting we obtain a connection $A$ on $\R\times Y_\#$ from $\alpha$ to $\beta$ which is glued to the constant mixed pair $(A_\beta,u_\beta)$. This construction allows us to associate to $p$ a well-defined gauge theoretical path $p'$ from $\alpha\in \fC_G$ to $\beta\in \fC_G$. In particular, $p'$ can be used to assign a lift $\overline \beta$ of $\beta$ in $\overline \fC_S$ to any lift $\overline \alpha\in \overline \fC_S$ of $\alpha$, which is characterized by
\[\deg_I(\overline \alpha)=\mathcal E(u)+\deg_I(\overline \beta).\] 
Analogous to \eqref{diff-filtered-gauge}, we define a differential $d:\overline C_S(Y_\#,E_\#)\to \overline C_S(Y_\#,E_\#)$ which is a $\Z[U^{\pm1}]$-module homomorphism and the following lemma implies that it is filtered with respect to the I-grading.

\begin{lemma}\label{ASD-pos-energy}
	Suppose  $\overline \alpha, \overline \beta\in \overline \fC_S$ are lifts of $\alpha, \beta\in \fC_S$ which are related by the path $p$ and the moduli space $\rM_S(\alpha,\beta)_p $ is non-empty. Then 
	\[
	  \deg_I(\overline \alpha)\geq \deg_I(\overline\beta),
	\]
	and the equality holds if and only if $\overline \alpha=\overline \beta$ and $p$ is the constant path.
\end{lemma}

We can more generally use any mixed pair $(A,u)\in \bA(\alpha,\beta)$, or rather its connected component, to associate a lift $\overline \beta\in \overline \fC_S$ of $\beta\in \fC_S$ to any lift $\overline \alpha\in \overline \fC_G$ of $\alpha\in \overline \fC_G$. In fact, one can see that $ \deg_I(\overline \alpha)-\mathcal E(A,u)$ is the I-grading of a lift of $\beta$ by using Lemma \ref{part-symp-constant} and reducing this claim to the case that $(A,u)$ is a symplectically constant pair. Therefore, we can define $\overline \beta$ by requiring   
\[
  \deg_I(\overline \alpha)=\mathcal E(A,u)+\deg_I(\overline \beta).
\]
This allows us to lift the isomorphism $\bN:C_{G}(Y_\#,E_\#) \to C_{G}(Y_\#,E_\#) $ of Subsection \ref{bN-definition} to a $\Z[U^{\pm1}]$-module homomorphism from $\overline C_G(Y_\#,E_\#)$ to $\overline C_S(Y_\#,E_\#)$. The third part of Lemma \ref{pert} implies that if the perturbation term in the definition of $\bN$ is small enough, then for any $\overline \alpha\in \overline \fC_G$
\[
  \bN(\overline \alpha)=\overline \alpha+\sum_{\overline \beta}n_{\alpha,\beta}\overline \beta,
\]
where $\overline \beta$ appears in the sum above only if $\deg_I(\overline \beta)< \deg_I(\overline \alpha)$. In another word, $\bN$ is filtered with respect to the Chern-Simons filtration and its leading term is equal to the identity map.

\begin{definition}
	An {\it I-graded complex} is a chain complex $(C,d)$ which is freely and finitely generated over the ring $\Q[U^{\pm 1}]$ and has a $\Z\times \R$-bigrading. If $C_{i,j}$ is the subgroup of $C$ consisting of elements with bigrading $(i,j)$, 
	then we have
	\begin{itemize}
		\item[(i)] $U \, C_{i,j}\subset C_{i+4,j+\frac{1}{2}}$,
		\item[(ii)] $d \, C_{i,j}\subset \bigcup_{j'\leq j} C_{i-1,j'}$.
	\end{itemize}
	Here $i$ and $j$ are respectively called the Floer grading and the I-grading of $C_{i,j}$.
	A chain map $f:(C,d)\to (C',d')$ of I-graded complexes is of level $\lambda>0$, if it is a module homomorphism satisfying 
	\[
	  f\, C_{i,j}\subset \bigcup_{j'\leq j+\lambda} C_{i,j'}
	\]
	A level $\lambda$ chain homotopy $h:(C,d)\to (C',d')$ between two chain maps $f,\, g:(C,d)\to (C',d')$ of level $\lambda$ is a module homomorphism satisfying 
	\[
	  f-g=d'h+hd
	\]
	and
	\[
	  h \, C_{i,j}\subset \bigcup_{j'\leq j+\lambda} C_{i+1,j'}.
	\]
\end{definition}

Our discussion above shows that the instanton Floer complex $\overline C_G(Y_\#,E_\#)$  and its symplectic version $\overline C_S(Y_\#,E_\#)$ are I-graded complexes, and $\bN$ defines an I-graded chain map of level $0$ which is an isomorphism. However, this is not completely satisfactory for two reasons. First the I-gradings on $\overline C_G(Y_\#,E_\#)$ and $\overline C_S(Y_\#,E_\#)$ depend on the auxiliary choices of perturbations, the Riemann metrics and almost complex structures. (It turns out the dependence on perturbation terms $h$ and $h'$ are more serious than the other two items.) Second we need to fix a distinguished element $\alpha_0$ of $\fC_G$ to define the I-grading. Since the set $\fC_G$ changes by varying the auxiliary choices, we need to slightly modify this choice to resolve the second issue. Mostly for the ease of exposition, from now on we focus on the case of framed Floer homology discussed in the introduction. 

Recall that for any closed 3-manifold $M$, framed Floer homology is defined by introducing $Y_\#:=M\#T^3$. We also fix an $\SO(3)$ bundle $E_\#$ which is induced by the trivial bundle on $M$ and the pullback to $T^3$ of the non-trivial bundle $F_1$ on $T^2$. There is a unique flat connection on $F_1$ up to the action of the gauge group $\mathcal G(F_1)$ and the pullback of this connection to $T^3$ and the trivial connection on $M$ induce an $SO(3)$ family of flat connections on $E_\#$. Let $\alpha_0$ be an arbitrary element of this family. As it is explained in the introduction, any Heegaard splitting of $M$ induces an admissible splitting of $(Y_\#,E_\#)$ as $(Y\cup Y',E\cup E')$.  For any such admissible splitting, there is a sequence $\{(\sigma_i,\sigma_i')\}$ in the perturbation space $\mathcal P\times \mathcal P'$ which converges to zero and $(h_i,h_i'):=(h_{\sigma_i},h'_{\sigma_i'})$ satisfies the following properties.
\begin{enumerate}
	\item[(i)] The spaces $L_{h_i}(Y,E)$, $L_{h_i'}(Y',E')$ are smooth embedded Lagrangians which intersect transversely. The intersection of these two Lagrangians is denoted by $\fC_S^i$ and it can be identified with its gauge theoretical counterpart 
	$\fC_G^i$. The set $\fC_G^i$ includes the flat connection $\alpha_0$.
	\item[(ii)] The claim of Lemma \ref{reg-ASD} holds. In particular, we can use $(h_i,h_i')$ to define I-graded Floer complexes $(\overline C_G^i(Y_\#,E_\#),d^i)$ and 
	$(\overline C_S^i(Y_\#,E_\#),d^i)$, where the I-grading is defined using $\alpha_0$.
\end{enumerate}
Standard continuation maps in Floer theory provide chain maps $f_i^j:\overline C_G^i(Y_\#,E_\#)\to \overline C_G^j(Y_\#,E_\#)$ for any $i, j$ such that $f_i^i={\rm Id}$ and $f_j^k\circ f_i^j$ is chain homotopic to $f_i^k$ using a chain homotopy $l_{i,j,k}$. In fact, the chain maps $f_i^j$ and the chain homotopy $l_{i,j,k}$ for $i,j,k\geq n$ are of level $\lambda_n$ where $\lambda_n\to 0$ as $n\to \infty$. (See \cite[Subsection 2.2]{D:CS-HCG} for the proof of a similar claim in a closely related context.) In particular, the chain complexes $\overline C_G^i(Y_\#,E_\#)$ form an {\it enriched complex} in the sense of the following definition. (This is a slight variation of \cite[Definition 7.16]{DS1}, which is adapted to the case of instanton Floer homology for admissible bundles.)

\begin{definition}
	An {\it enriched complex} $\fE=\{(C^i,d^i),f_i^j,\lambda_n\}_{i,j,n}$ is a sequence of I-graded complexes $\{(C^i,d^i)\}$ and a family of chain maps $f_i^j:C^i\to C^j$ such that 
	\begin{itemize}
		\item[(i)] $f_i^j$ is of level $\lambda_n$ for any $i,j\geq n$,
		\item[(ii)] $f_i^i={\rm Id}$,
		\item[(iii)] $f_j^k\circ f_i^j$ is chain homotopic to $f_i^k$ using a chain homotopy $l_{i,j,k}$ of level $\lambda_n$ whenever $i,j,k\geq n$,
		\item[(iv)] $\lim_{n\to \infty}\lambda_n=0$.
	\end{itemize}
\end{definition}
\begin{definition}	
	If $\fC_1=\{(C^i_1,d^i_1),f_i^j,\lambda_n\}_{i,j,n}$ and $\fC_2=\{(C^i_2,d^i_2),g_i^j,\mu_n\}_{i,j,n}$ are two enriched complexes, then an enriched morphism $\fN:\fC_1\to \fC_2$ consists of level $\kappa_i$ chain maps $\bN_i:C^i_1\to C^i_2$ for
	any $i$ such that 
	$\bN_j f_i^j$ and $g_i^j \bN_i$ are chain homotopic using a chain homotopy of level $\kappa_n$ whenever $i,j\geq n$, and $\kappa_n\to 0$ as $n\to \infty$. 
	Enriched morphisms $\fN=\{\bN_i:C^i_1\to C^i_2\}$ and $\fM=\{\bM_i:C^i_1\to C^i_2\}$ are chain homotopic 
	to each other, if there is a sequence $\{\bK_i:C^i_1\to C^i_2\}$ where $\bK_i$ is a chain homotopy of level $\kappa_i$ between $\bN_i$ and $\bM_i$ with $\kappa_n\to 0$ as $n\to \infty$. .
	The enriched complexes $\fC_1$ and $\fC_2$ are chain homotopy equivalent to each other if there are enriched
	 morphisms $\fN=\{\bN_i:C^i_1\to C^i_2\}$ and $\fM=\{\bM_i:C^i_2\to C^i_1\}$ such that the enriched morphism
	 $\fM\circ \fN:=\{\bM_i\circ \bN_i:C^i_1\to C^i_1\}$ (resp. $\fN\circ \fM:=\{\bN_i\circ \bM_i:C^i_2\to C^i_2\}$) is chain homotopy equivalent to an 
	 isomorphism of $\fC_1$ (resp. $\fC_2$).
\end{definition}

The following theorem is an extension of our main theorem, which in particular shows that framed Floer homology and symplectic framed Floer homology together with their Chern-Simons filtrations are isomorphic to each other. In the statement of the theorem, we write $\fE_G(M)$ for the enriched framed Floer complex $\{(\overline C_G^i(Y_\#,E_\#),d^i),f_i^j,\lambda_n\}$ of $M$.
\begin{theorem}\label{enriched-AF}
	The I-graded complexes $(\overline C_S^i(Y_\#,E_\#),d^i)$ can be completed into an enriched complex 
	\[\fE_S(M)=\{(\overline C_S^i(Y_\#,E_\#),d^i),g_i^j,\mu_n\},\] 
	and the enriched framed Floer theories $\fE_G(M)$ and $\fE_S(M)$ are chain homotopy equivalent to each other as enriched complexes.
\end{theorem}
\begin{proof}
	The construction of Subsection \ref{bN-definition} gives a level $0$ chain map $\bN_i:\overline C_G^i(Y_\#,E_\#) \to \overline C_S^i(Y_\#,E_\#)$ for each $i$, which is an isomorphism. Then
	\[
	  g_i^j:=\bN_j\circ f_i^j\circ \bN_i^{-1}:(\overline C_S^i(Y_\#,E_\#),d_i)\to (\overline C_S^j(Y_\#,E_\#),d^j)
	\]
	is a chain map, and $g_j^k\circ g_i^j$ is chain homotopic to $g_i^k$ using the chain homotopy $\bN_k\circ l_{i,j,k}\circ \bN_i^{-1}$. By picking $\mu_n=\lambda_n$, we may easily see that $\fE_S(M)=\{(\overline C_S^i(Y_\#,E_\#),d^i),g_i^j,\mu_n\}$
	is an enriched complex. The maps $\bN_i$ and $\bM_i:=\bN_i^{-1}$ give the desired chain homotopy equivalence between $\fE_G(M)$ and $\fE_S(M)$.
\end{proof}

\begin{remark}
	In this paper we have been concerned with the instanton Floer homology for admissible bundles on 3-manifolds. The key feature of these bundles is that they do not admit reducible flat connections. However, there are versions of 
	instanton Floer homology for 3-manifolds \cite{floer:inst1,Don:YM-Floer,froyshov} 
	and knots \cite{collin-steer,DS1,SFO} where one works with bundles which admit flat reducible connections. In these cases, one still has the Chern-Simons filtration which can be used to produce numerical invariants of 3-manifolds
	\cite{D:CS-HCG,NST:filtered} and 
	knots \cite{DS1}. Although this has not been investigated in the literature, it is reasonable to expect that the 
	Chern-Simons filtration on framed Floer homology, in the form of the enriched complex $\fE_G(M)$ (or equivalently $\fE_S(M)$), could be useful in the study low of dimensional manifolds. 
\end{remark}

\subsection{Quantum cohomology and the $\mu$-operator}\label{mu-quantum-comp}

Associated to any configuration space of connections on a principal bundle, there is a universal bundle, which can be used to produce cohomology classes of the configuration space. As the first example, let $\mathcal A^*(\Sigma,F)$ be the open subspace of $\mathcal A(\Sigma,F)$ given by irreducible connections, and define $\mathcal B^*(\Sigma,F)\subset \mathcal B(\Sigma,F)$ as the quotient of $\mathcal A^*(\Sigma,F)$ by $\mathcal G(F)$. The gauge group $\mathcal G(F)$ acts in the obvious way on the product space $\mathcal A^*(\Sigma,F)\times \Sigma$, and this action can be lifted to the pullback of $F$ to $\mathcal A^*(\Sigma,F)\times \Sigma$. The stabilizer of the action of $\mathcal G(F)$ at any point of $\mathcal A^*(\Sigma,F)\times \Sigma$ is $\{\pm 1\}$ which act trivially at any point on the bundle. In particular, taking quotient with respect to $\mathcal G(F)$ defines an $\SO(3)$-bundle on $\mathcal B^*(\Sigma,F)\times \Sigma$ which is called the universal bundle associated to $F$. We write $\mathbb F$ for the restriction of this bundle to the subspace $\mathcal M(\Sigma,F)\times \Sigma$ of $\mathcal B^*(\Sigma,F)\times \Sigma$. Similar constructions give rise to the universal bundles $\mathbb E_\#$ over $\mathcal B^*(Y_\#,E_\#)\times Y_\#$ and $\mathbb V$ over $\bB(\alpha,\beta)\times X$ for any $\alpha,\beta \in \fC_G$. Here $\mathcal B^*(Y_\#,E_\#)$ is again the subspace of $\mathcal B(Y_\#,E_\#)$ given by irreducible connections.

The universal bundles $\mathbb F$, $\mathbb E_\#$ and $\mathbb V$ are related to each other. For any $(A,u)\in \bA(\alpha,\beta)$, the restriction of $A$ to $\{(0,\theta)\}\times \Sigma\subset U_\partial \times \Sigma$ gives a flat connection in $\mathcal A^*(\Sigma,F)$ and this association is equivariant. This induces a map $\rho:\bB(\alpha,\beta)\times U_\partial \times \Sigma \to\mathcal M(\Sigma,F) \times \Sigma$, and the restriction of $\mathbb V$ to $\bB(\alpha,\beta)\times U_\partial \times \Sigma$ is given by the pullback of $\mathbb F$ via the map $\rho$. Let $\widetilde X$ be the manifold obtained as the union of  $X$ and $U_+\times \Sigma$ where $U_\partial \times \Sigma \subset X\times \Sigma$ is identified with $U_\partial \times \Sigma \subset U_+ \times \Sigma$ in the obvious way. The above discussion shows that the bundles $\mathbb V$ over $\bB(\alpha,\beta)\times X$ and the pullback of $\mathbb F$ to $\bB(\alpha,\beta)\times U_+\times \Sigma$ via the map 
\begin{equation}\label{symp-rel-univ-bundles}
  ([A,u],s,\theta,x)\in \bB(\alpha,\beta)\times U_+\times \Sigma\to (u(s,\theta),x)\in \mathcal M(\Sigma,F) \times \Sigma,
\end{equation}
can be naturally identified with each other over  $U_\partial \times \Sigma$, and hence they induce a bundle over $\bB(\alpha,\beta)\times\widetilde X$, which we denote by $\widetilde {\mathbb V}$. Fix $T\geq 3$ and let $\bB^*(\alpha,\beta)$ denote the subspace of $\bB(\alpha,\beta)$ given by mixed pairs $[A,u]$ that the restriction of $A$ to $\{t\}\times Y_\#$, for any $t\in (-\infty,-T]$, is irreducible. Then the restriction of $\mathbb V$ to $\bB^*(\alpha,\beta) \times (-\infty,-T]\times Y_\#$ is the pullback of the bundle $\mathbb E_\#$ with respect to the map
\begin{equation}\label{gauge-rel-univ-bundles}
  ([A,u],t,x) \in \bB^*(\alpha,\beta) \times (-\infty,-T]\times Y_\#\to ([A|_{\{t\}\times Y_\#}],x)\in \mathcal B^*(Y_\#,E_\#) \times Y_\#.
\end{equation}
Later in this subsection we shall need a further constraint on $T$.

Universal bundles can be used to produce cohomology classes in configuration spaces of connections. For instance, in the case of $\mathcal M(\Sigma,F)$, the first Pontryagin class of $\mathbb F$ has the K\"unneth decomposition
\[
  p_1(\mathbb F)=\alpha_0\otimes \omega_0+\alpha_1\otimes \omega_1+\beta_0\otimes x_0+\beta_1\otimes x_1 +\sum_{i}\psi_j\otimes \gamma_j \in H^4(\mathcal M(\Sigma,F)\times \Sigma)
\]
 where the cohomology classes $\omega_j\in H^2(\Sigma)$, $x_j\in H^0(\Sigma)$ and $\gamma_j\subset H^1(\Sigma)$ give a basis for the corresponding cohomology groups of $\Sigma$, and $\alpha_j\in H^2(\mathcal M(\Sigma,F))$, $\beta_j \in H^4(\mathcal M(\Sigma,F))$ and $\psi_j\in H^3(\mathcal M(\Sigma,F))$. These cohomology classes of $\mathcal M(\Sigma,F)$ provide a multiplicative generating set for the cohomology ring $H^*(\mathcal M(\Sigma,F))$ \cite{AB:YM}. 
 
 Explicit representatives for these cohomology classes can be constructed in the following way. Let $\mathbb F_\C$ denote the complexification $\mathbb F \otimes \C$. Fix two sections $s_1^S$ and $s_2^S$ of $\mathbb F_\C$ and define
 \[
   Z:=\{(\alpha,x)\in \mathcal M(\Sigma,F)\times \Sigma\mid \text{$s_1^S(\alpha,x)$, $s_2^S(\alpha,x)$ are linearly dependent}\}.
 \]
If the sections $s_1^S$ and $s_2^S$ are chosen generically, then $Z$ is a codimension four compact stratified subspace of $\mathcal M(\Sigma,F)\times \Sigma$ where the top stratum $Z_0$ is a smooth submanifold and $Z\setminus Z_0$ given by the common zeros of $s_1^S$ and $s_2^S$ has codimension twelve. In particular, if we co-orient $Z$ using the product orientation on $ \mathcal M(\Sigma,F)\times \Sigma$ and the complex orientation on the fibers of $\mathbb F_\C$, then it has a well-defined fundamental class which gives the Poincar\'e dual for $c_2(\mathbb F_\C)=-p_1(\mathbb F)$. For this space $Z$ and the analogous ones defined in the following, we use the non-standard product orientation convention to get a representative for $p_1(\mathbb F)$.  Thus $Z$ can be used to produce representatives for the cohomology classes $\alpha_j$, $\beta_j$ and $\psi_j$. For instance, the projection of $Z\cap (\mathcal M(\Sigma,F)\times \Sigma_j)$ to $\mathcal M(\Sigma,F)$ gives a cycle representing $\alpha_j$ and the intersection of $Z$ with $\mathcal M(\Sigma,F)\times \{x\}$ for a generic $x\in\Sigma_j$ gives a representative for $\beta_j$. Similarly, if $\ell$ is a closed oriented loop representing $\gamma_j$, then the intersection $Z\cap (\mathcal M(\Sigma,F)\times \ell)$, after possibly a perturbation of $\ell$, is transversal and the projection of this intersection to $\mathcal M(\Sigma,F)$ gives a cycle representing $\gamma_j$.

As in Subsection \ref{Fl-hom-subsec}, for $\alpha$, $\beta\in \fC_S$, let $\rM_S(\alpha,\beta)_p$ be the moduli space of pseudo-holomorphic maps $u:\R\times [-1,1] \to \mathcal M(\Sigma,F)$ satisfying Lagrangian boundary conditions corresponding to the path $p$ from $\alpha$ to $\beta$. Let ${\rE\rv}:\rM_S(\alpha,\beta)_p\times \Sigma \to \mathcal M(\Sigma,F)\times \Sigma$ be the evaluation map that sends $(u,x)$ in $\rM_S(\alpha,\beta)_p\times \Sigma$ to $(u(0,0),x)$. We also fix an element of $H^2(\Sigma)$ represented by one of the connected components $\Sigma_\sigma$ of $\Sigma$. A generic choice of $s_1^S$ and $s_2^S$ allows us to assume that ${\rE\rv}$ is transversal to $Z$. We form the cutdown moduli space
\[
  \rL_S(\alpha,\beta)_p:=\{(u,x)\in \rM_S(\alpha,\beta)_p\times \Sigma_\sigma \mid {\rE\rv}(u,x)\in Z\}.
\]
Our transversality assumption implies that $\dim(\rL_S(\alpha,\beta)_p)=\dim(\rM_S(\alpha,\beta)_p)-2$. 

Now, we are ready to review the definition of the operator $m_\sigma^S$. If the index of $p$ is at most $1$, the moduli space $\rL_S(\alpha,\beta)_p$ is empty and if the index of $p$ is $2$, then $\rL_S(\alpha,\beta)_p$ is a compact $0$-dimensional manifold which we may orient using the orientation of $\rM_S(\alpha,\beta)_p$ and the co-orientation of $Z$ that realizes $p_1(\mathbb F)$. Orientation of these moduli spaces allows us to define a homomorphism $m_\sigma^S:C_S((Y,E),(Y',E'))\to C_S((Y,E),(Y',E'))$ as
\[
  m_\sigma^S(\alpha):=\sum_{p:\alpha\to \beta} \# \rL_S(\alpha,\beta)_p\cdot \beta
\]
where the sum is over all paths $p$ from $\alpha\in \fC_S$ to $\beta\in \fC_S$ of index two and $\#\rL_S(\alpha,\beta)_p$ denotes the signed count of the elements of $\rL_S(\alpha,\beta)_p$. An analysis of the ends of $1$-dimensional cutdown moduli spaces $\rL_S(\alpha,\beta)_p$ shows that the homomorphism $m_\sigma^S$ is a chain map, and we use the same notation to denote the induced map $m_\sigma^S:\rSI_*(Y_\#,E_\#) \to \rSI_*(Y_\#,E_\#)$ at the level of homology.

The above construction has a counterpart in the case of instanton Floer homology for admissible pairs, as we review the construction now for the pair $(Y_\#,E_\#)$. For any $\alpha,\beta\in \fC_G$, any path $p$ from $\alpha$ to $\beta$ and any $[A]\in \rM_G(\alpha,\beta)_p$, the restriction $A_t$ of $A$ to $\{t\}\times Y_\#$ for any $t\in \R$ is irreducible. Otherwise, if $A_t$ has a non-trivial stabilizer $u$, then $A$ and the pullback $u^*A$ are two solutions of the perturbed ASD equation that agree on $\{t\}\times Y_\#$. Unique continuation implies that these two connections are equal to each other which contradicts with the irreducibility of $A$. In particular, we obtain a well-defined map $\rM_G(\alpha,\beta)_p \to \mathcal B^*(Y_\#,E_\#)$ by restricting $[A]$ to $\{0\}\times Y_\#$. 

Next, we fix sections $s_1^G$ and $s_2^G$ of the complexified universal bundle $\mathbb E_\#\otimes \C$ over $\mathcal B^*(Y_\#,E_\#)\times Y_\#$ and define
\begin{align*}
  \rL_G(\alpha,\beta)_p:=\{([A],x)\in &\rM_G(\alpha,\beta)_p\times \Sigma_\sigma \mid \\
  &\text{ $s_1^G ([A\mid_{\{0\}\times Y_\#}],x)$ and $s_2^G([A\mid_{\{0\}\times Y_\#}],x)$ are linearly dependent.}\}.
\end{align*}
Again, we may assume that the space $\rL_G(\alpha,\beta)_p$ is cut down transversely. In particular, it is empty, if the index of $p$ is at most one, and it is a compact $0$-dimensional manifold if the index of $p$ is two. In the latter case, we use the product orientation of $\rM_G(\alpha,\beta)_p\times \Sigma_\sigma$ to orient $\rL_G(\alpha,\beta)_p$. Oriented $0$-dimensional moduli spaces $\rL_G(\alpha,\beta)_p$ can be used to define the operator $m_\sigma^G:C(Y_\#,E_\#) \to C(Y_\#,E_\#)$ as 
\[
  m_\sigma^G(\alpha):=\sum_{p:\alpha\to \beta} \# \rL_G(\alpha,\beta)_p\cdot \beta
\]
where the sum is over all paths $p$ from $\alpha\in \fC_G$ to $\beta\in \fC_G$ with index two. Using $1$-dimensional moduli spaces $\rL_G(\alpha,\beta)_p$ one can see again that $m_\sigma^G$ is a chain map. The induced operator acting on $\rI_*(Y_\#,E_\#)$ is denoted by the same notation.

To relate the operators $m_\sigma^G$ and $m_\sigma^S$ we need sections of the bundle $\widetilde {\mathbb V}$ interpolating between $s_1^G$, $s_2^G$ on the gauge theoretical side and $s_1^S$, $s_2^S$ on the symplectic side. Fix continuous sections  $\bs_1$ and $\bs_2$ of $\widetilde {\mathbb V}$ defined over $\bB^*(\alpha,\beta)\times \widetilde X$ which satisfy the following properties:
\begin{itemize}
	\item[(i)] The restriction of $\bs_i$ to $\bB^*(\alpha,\beta)\times X$ is smooth.
	\item[(ii)] The restriction of $\bs_i$ to $\bB^*(\alpha,\beta)\times (-\infty,-T]\times Y_\#$ is given by pulling back $s_i^G$ using the map \eqref{gauge-rel-univ-bundles}. 
	\item[(iii)] As in Subsection \ref{proof-thm-Fredholm-D}, let $X_T$ denote the the compact subspace of $X$ given as the complement of $ (T,\infty)\times T$, $(-\infty,-T)\times Y'$ and $(-\infty,-T)\times Y_\#$ in $X$.
	Then $\bs_i([A,u],x)$ for $([A,u],x)\in \bB^*(\alpha,\beta)\times X_T$ depends on the restriction of $[A]$ to $X_T$. To be more precise, the bundle $\mathbb V$ over 
	$\bB^*(\alpha,\beta)\times X_T$ is the pullback of the universal bundle over $\mathcal B^*(X_T)\times X_T$ and we demand that $\bs_i$ is the pullback of a section of this universal bundle. Here $\mathcal B^*(X_c)$ denotes the configuration 
	space of irreducible $L^2_{l}$ connections on $X_T$ and the universal bundle over this space is defined analogous to the previous instances of universal bundles.
	\item[(iv)] The restriction of $\bs_i$ to $\bB^*(\alpha,\beta)\times U_-\times \Sigma$ is given by pulling back $s_i^S$ using the map \eqref{symp-rel-univ-bundles}. 
\end{itemize}
Suppose the constant $T$ in the definition of $\bB^*(\alpha,\beta)$ is chosen such that the secondary perturbation term $\eta$ vanishes on $(-\infty,-T]\times Y_\#$. Then unique continuation again implies that the moduli space $\bM_\eta(\alpha,\beta)$ is contained in $\bB^*(\alpha,\beta)$. We may arrange the sections $s_i^G$, $s_i^S$ and $\bs_i$ so that the following subspace of $\bM_\eta(\alpha,\beta)\times \R\times \Sigma_\sigma$ for $d\leq 2$ is cut down transversely:
\begin{align*}
  \bL(\alpha,\beta)_{d}:=\{([A,u],t,x)\in &\bM_\eta(\alpha,\beta)_{d+1} \times \R\times \Sigma_\sigma \mid \\
  &\text{ $\bs_1 ([A,u],t,x)$ and $\bs_2 ([A,u],t,x)$ are linearly dependent}\}.
\end{align*}
Here we use the embedding of $\R\times \Sigma_\sigma$ into $\widetilde X$ where $(-\infty,0]\times \Sigma_\sigma$ is mapped to $(-\infty,0]\times \{0\}\times \Sigma_\sigma\subset U_-\times \Sigma$ and $[0,\infty)\times \Sigma_\sigma$ is mapped to $[0,\infty)\times \{0\}\times \Sigma_\sigma\subset U_+\times \Sigma$. By assumption, the restrictions of $\bs_i$ to $\bM(\alpha,\beta)_d \times \Sigma_\sigma\times (-\infty,0]$, $\bM(\alpha,\beta)_d \times \Sigma_\sigma\times [0,\infty)$ and $\bM(\alpha,\beta)_d \times \Sigma_\sigma\times \{0\}$ are smooth. The transversality assumption above means that the loci that $\bs_1$ and $\bs_2$ are linearly dependent over each of these subspaces is cut down transversely.  This transversality assumption implies that $\bL(\alpha,\beta)_d$ is empty for $d< 0$ and it is a $0$-dimensional manifold for $d=0$.

\begin{lemma}\label{cut-dn-compact}
	The moduli space $\bL(\alpha,\beta)_0$ is compact.
\end{lemma}
\begin{proof}
	Suppose $\{([A_i,u_i],t_i,x_i)\}$ is a sequence of elements in $\bL(\alpha,\beta)_0$. After passing to a subsequence, we may assume that $x_i$ converges to $x_0\in \Sigma_\sigma$ and $t_i$ 
	converges to $t_0$ which is either a finite real number or $\pm \infty$. The argument of Subsection \ref{comp-sect} implies that there is a solution of the mixed equation $[A_0,u_0]\in \bM_{\eta}(\alpha',\beta')_{d+1}$ with $d\leq 0$, 
	perturbed ASD connections 
	\[
	 A_1^G\in \breve \rM_G(\alpha,\alpha_1)_{p_1},\,A_2^G\in \breve \rM_G(\alpha_1,\alpha_2)_{p_2},\,\dots,\,A_n^G\in \breve \rM_G(\alpha_{n-1},\alpha')_{p_n}
	\]
	and pseudo--holomorphic strips 
	\[
	 u_1^S\in \breve \rM_S(\beta',\beta_1)_{p_1'},\,u_2^S\in \breve \rM_S(\beta_1,\beta_2)_{p_2},\,\dots,\,u_m^S\in \breve \rM_S(\beta_{m-1},\beta)_{p_m'}
	\]
	such that $[A_i,u_i]$ is chain convergent to $([A_1^G],\dots,[A_n^G],[A_0,u_0],u_1^S,\dots,u_m^S)$ on the complement of a set of bubble points. Since the index of the mixed pairs $[A_i,u_i]$ is $1$, we may show that the set of bubble points 
	is empty by arguing as in Subsection \ref{comp-sect}. If $t_0$ is a finite number, then the continuity of the sections $\bs_1$ and $\bs_2$ (together with the properties (ii) and (iii) of $\bs_i$ if $t_0<0$) implies that $\bs_1 ([A_0,u_0],t_0,x_0)$ and $\bs_2 ([A_0,u_0],t_0,x_0)$ are linearly dependent. 
	In particular, the moduli space $\bL(\alpha',\beta')_d$ is non-empty which shows that $d=0$. This in turn implies that $m=n=0$ and $\alpha=\alpha'$, $\beta=\beta'$, and $\{([A_i,u_i],t_i,x_i)\}$ is convergent to
	$([A_0,u_0],t_0,x_0)\in \bL(\alpha,\beta)_0$ with respect to the topology of $\bL(\alpha,\beta)_0$. Thus in this case $\{([A_i,u_i],t_i,x_i)\}$ is convergent after passing to a subsequence.
	
	Next, we consider the case that $t_0$ is not a finite number. First let $t_0=-\infty$. After passing to a subsequence, we may assume that $t_i\leq -T$ for any $i$. Translating the restriction of $A_i$ to $(-\infty,-T]$ by $t_i$ gives 
	a connection $A_i'$ on $(-\infty,-T-t_i]$ such that $s_1^G([A_i'\vert_{\{0\}\times Y_\#}],x_i)$ and $s_2^G([A_i'\vert_{\{0\}\times Y_\#}],x_i)$ are linearly independent. Since $t_i\to -\infty$, our assumption on the chain convergence of 
	$\{([A_i,u_i],t_i,x_i)\}$ implies that the connections $A_i'$ modulo the action of the gauge group are convergent to $B$ which is one of the connections $A_i^G$ or the pullback of one of the connections $\alpha_0=\alpha$, $\alpha_1$, $\dots$, 
	$\alpha_{n-1}$, $\alpha_n=\alpha'$. Moreover, property (ii) of the sections $\bs_i$ implies that 
	$s_1^G([B|_{\{0\}\times Y_\#}],x_0)$ and $s_2^G([B|_{\{0\}\times Y_\#}],x_0)$ are linearly dependent. In particular, $B$ equals one of the connections $A_i^G$, and this connections represents an element of $\rL_G(\alpha_{i-1},\alpha_i)_p$. 
	This implies that the index of 
	 $A_i^G$ is at least $2$. On the other hand, the sum of the indices of the connections $A_i^G$, the mixed pair $[A_0,u_0]$ and $u_j^S$ is $1$, which is a contradiction. This shows that $t_0$ cannot be $-\infty$. A similar proof rules out the 
	 case $t_0=\infty$.
\end{proof}

We orient the compact $0$-dimensional manifold $\bL(\alpha,\beta)_0$  using the orientation of $\bM_\eta(\alpha,\beta)_1$ and the induced product orientation on $\bM_\eta(\alpha,\beta)_1 \times \R\times \Sigma_\sigma$. These oriented moduli spaces allow us to define a map $K:C(Y_\#,E_\#)\to C_S((Y,E),(Y',E'))$ as
\[
  K(\alpha):=\sum_{\beta} \# \bL(\alpha,\beta)_0 \cdot \beta.
\]
To prove Theorem \ref{module-over-quantum} for the operators $m_\sigma^G$ and $m_\sigma^S$, it suffices to show that $\bN\circ m_\sigma^G-m_\sigma^S\circ \bN=dK+Kd$. The following proposition, which is the counterpart of part (ii) of Proposition \ref{comp}, shows that this relation follows from analyzing the ends of the $1$-dimensional moduli spaces $\bL(\alpha,\beta)_1$.

\begin{prop}\label{comp-cutdown}		
	The moduli spaces of the form $\bL(\alpha,\beta)_1$ can be compactified into compact 1-manifolds by adding points in correspondence to the $0$-dimensional spaces
	\begin{equation}\label{ends-1dim-bL-1}
		\bL(\alpha,\gamma)_0\times \breve \rM_{S}(\gamma,\beta)_p,\hspace{1cm}
		\breve \rM_{G}(\alpha, \gamma)_p \times \bL(\gamma,\beta)_0,
	\end{equation}
	and 
	\begin{equation}\label{ends-1dim-bL-2}
		\bM_\eta(\alpha,\gamma)_0\times \rL_{S}(\gamma,\beta)_p,\hspace{1cm}
		\rL_{G}(\alpha, \gamma)_p \times \bM_\eta(\gamma,\beta)_0.
	\end{equation}
	where $\gamma\in \fC_G\cong \fC_S$, in \eqref{ends-1dim-bL-1} $p$ denotes a path of index $1$, and in \eqref{ends-1dim-bL-2} $p$ denotes a path of index $2$.
	Moreover, the induced orientation on the boundary components 
	of the compactified moduli space
	$\bL(\alpha,\beta)_1$ agree with the product orientation on the two terms in \eqref{ends-1dim-bL-1} and the first term in \eqref{ends-1dim-bL-2}
	and disagrees with the induced orientation on the second term in \eqref{ends-1dim-bL-2}.
\end{prop}

\begin{proof}
	A straightforward adaptation of the proof of Lemma \ref{cut-dn-compact} shows that any sequence of elements in $\bL(\alpha,\beta)_1$ without any subsequence convergent to an element of this moduli space has a subsequence chain convergent 	to an element in one of the spaces in \eqref{ends-1dim-bL-1} or \eqref{ends-1dim-bL-2}. We need a gluing theory as a converse to this  compactness result to show all elements of \eqref{ends-1dim-bL-1} and \eqref{ends-1dim-bL-2} appear 
	as the ends of the moduli space $\bL(\alpha,\beta)_1$. As in the case of Proposition \ref{comp}, the desired gluing theory results concern gluing mixed pairs to ASD connections or pseudo-holomorphic strips on the gauge theoretical or 
	symplectic ends. In particular, they can be proved by a straightforward adaptation of the corresponding gluing results in the context of instanton Floer theory and Lagrangian Floer theory. The discussion of the induced orientations of the moduli 
	spaces on the boundary components is also similar to the standard corresponding results in the context of instanton Floer theory and Lagrangian Floer theory.
\end{proof}

We may follow a similar discussion to prove the variation of the above result in the case that $\sigma\in H^1(\Sigma)$ and is represented by a loop $\ell_\sigma$ in $\Sigma$. The main modifications applied to the proof are replacing $\Sigma_\sigma$ with $\ell_\sigma$ and working with the moduli spaces of instantons, pseudo-holomorphic strips and mixed pairs of one dimension higher. The above argument does not immediately generalize to the case that $\sigma \in H^0(\Sigma)$ because we need to work with the 4-dimensional moduli spaces of solutions to the mixed equation. The main obstacle in this case is that we may have bubbling along the matching line in the compactification of the moduli spaces of mixed pairs, and we have not studied the behavior of the compactified moduli spaces in a neighborhood of such bubbles.

\bibliography{references}
\bibliographystyle{hplain}
\end{document}